\documentclass{amsart}
\usepackage{cite}
\usepackage{amssymb}
\usepackage[arrow,curve,matrix,arc]{xy}
\begin{document}
\newtheorem{thm}{Theorem}[section]
\newtheorem{cor}[thm]{Corollary}
\newtheorem{prop}[thm]{Proposition}
\newtheorem{lem}[thm]{Lemma}
\newtheorem{conj}[thm]{Conjecture}
\newtheorem{cond}[thm]{Condition}
\theoremstyle{definition}
\newtheorem{dfn}[thm]{Definition}
\newtheorem{ex}[thm]{Example}
\newtheorem{rem}[thm]{Remark}
\newtheorem{conv}[thm]{Convention}
\newtheorem{quest}[thm]{Question}
\numberwithin{figure}{section}
\def\eq#1{{\rm(\ref{#1})}}
\def\dim{\mathop{\rm dim}\nolimits}
\def\vdim{\mathop{\rm vdim}\nolimits}
\def\ev{\mathop{\rm ev}\nolimits}
\def\bev{\mathop{\bf ev}\nolimits}
\def\btev{\mathop{\bf\widetilde{ev}}\nolimits}
\def\Ker{\mathop{\rm Ker}}
\def\Coker{\mathop{\rm Coker}}
\def\ind{\mathop{\rm ind}}
\def\Ind{\mathop{\rm Ind}}
\def\val{\mathop{\rm val}}
\def\mul{\mathop{\rm mul}}
\def\Im{\mathop{\rm Im}}
\def\Stab{\mathop{\rm Stab}\nolimits}
\def\vol{\mathop{\rm vol}\nolimits}
\def\rank{\mathop{\rm rank}}
\def\Hom{\mathop{\rm Hom}\nolimits}
\def\Aut{\mathop{\rm Aut}}
\def\id{\mathop{\rm id}\nolimits}
\def\rsi{{\rm si}}
\def\geo{{\rm geo}}
\def\bs{\boldsymbol}
\def\ge{\geqslant}
\def\le{\leqslant\nobreak}
\def\ma{{\rm main}}
\def\nov{{\rm nov}}
\def\CY{{\rm\sst CY}}
\def\cal{\mathcal}
\def\R{{\mathbin{\mathbb R}}}
\def\Z{{\mathbin{\mathbb Z}}}
\def\N{{\mathbin{\mathbb N}}}
\def\Q{{\mathbin{\mathbb Q}}}
\def\C{{\mathbin{\mathbb C}}}
\def\CP{{\mathbin{\mathbb{CP}}}}
\def\M{{\mathbin{\mathcal M}}}
\def\X{{\mathbin{\mathcal X}}}
\def\Y{{\mathbin{\mathcal Y}}}
\def\a{{\mathfrak a}}
\def\b{{\mathfrak b}}
\def\f{{\mathfrak f}}
\def\g{{\mathfrak g}}
\def\h{{\mathfrak h}}
\def\m{{\mathfrak m}}
\def\n{{\mathfrak n}}
\def\p{{\mathfrak p}}
\def\q{{\mathfrak q}}
\def\s{{\mathfrak s}}
\def\H{{\mathfrak H}}
\def\U{\mathbin{\rm U}}
\def\In{{\mathfrak Incl}}
\def\Ev{{\mathfrak Eval}}
\def\oM{{\mathbin{\smash{\,\,\overline{\!\!\mathcal M\!}\,}}}}
\def\tM{{\mathbin{\smash{\,\,\widetilde{\!\!\mathcal M\!}\,}}}}
\def\hM{{\mathbin{\smash{\,\,\widehat{\!\!\mathcal M\!}\,}}}}
\def\al{\alpha}
\def\be{\beta}
\def\ga{\gamma}
\def\de{\delta}
\def\io{\iota}
\def\ep{\epsilon}
\def\vep{\varepsilon}
\def\la{\lambda}
\def\ka{\kappa}
\def\th{\theta}
\def\ze{\zeta}
\def\up{\upsilon}
\def\vp{\varphi}
\def\si{\sigma}
\def\om{\omega}
\def\De{\Delta}
\def\La{\Lambda}
\def\Om{\Omega}
\def\Ga{\Gamma}
\def\Si{\Sigma}
\def\Th{\Theta}
\def\pd{\partial}
\def\ts{\textstyle}
\def\sst{\scriptscriptstyle}
\def\w{\wedge}
\def\sm{\setminus}
\def\bu{\bullet}
\def\op{\oplus}
\def\ot{\otimes}
\def\ov{\overline}
\def\bigop{\bigoplus}
\def\bigot{\bigotimes}
\def\iy{\infty}
\def\pr{{\mathop{\preceq}\nolimits}}
\def\tl{\trianglelefteq\nobreak}
\def\ls{\mathop{\lesssim\kern .05em}\nolimits}
\def\ra{\rightarrow}
\def\Ra{\Rightarrow}
\def\na{\nabla}
\def\ab{\allowbreak}
\def\hookra{\hookrightarrow}
\def\longra{\longrightarrow}
\def\t{\times}
\def\ci{\circ}
\def\ti{\tilde}
\def\d{{\rm d}}
\def\ha{{\ts\frac{1}{2}}}
\def\md#1{\vert #1 \vert}
\def\bmd#1{\big\vert #1 \big\vert}
\def\ms#1{\vert #1 \vert^2}
\def\nm#1{\Vert #1 \Vert}
\def\bnm#1{\big\Vert #1 \big\Vert}
\title{Immersed Lagrangian Floer Theory}
\author{Manabu Akaho and Dominic Joyce}
\thanks{Supported by EPSRC grant EP/D07763X/1}
\begin{abstract}
Let $(M,\om)$ be a compact symplectic $2n$-manifold, and $L$ a
compact embedded Lagrangian submanifold in $M$. Fukaya, Oh, Ohta and
Ono \cite{FOOO} construct {\it Lagrangian Floer cohomology\/} for
such $M,L$, yielding groups $HF^*(L,b;\La_\nov)$ for one Lagrangian
or $HF^*\bigl((L_1,b_1),(L_2,b_2);\La_\nov\bigr)$ for two, where
$b,b_1,b_2$ are choices of {\it bounding cochains}, and exist if and
only if $L,L_1,L_2$ have {\it unobstructed Floer cohomology}. These
are independent of choices up to canonical isomorphism, and have
important invariance properties under Hamiltonian equivalence. Floer
cohomology groups are the morphism groups in the derived Fukaya
category of $(M,\om)$, and so are an essential part of the
Homological Mirror Symmetry Conjecture of Kontsevich.

The goal of this paper is to extend \cite{FOOO} to {\it immersed\/}
Lagrangians $L$ in $M$ with immersion $\io:L\ra M$, with transverse
self-intersections. In the embedded case, Floer cohomology
$HF^*(L,b;\La_\nov)$ is a modified, `quantized' version of singular
homology $H_{n-*}(L;\La_\nov)$ over the Novikov ring $\La_\nov$. In
our immersed case, $HF^*(L,b;\La_\nov)$ turns out to be a quantized
version of $H_{n-*}(L;\La_\nov)\op\bigop_{(p_-,p_+)\in
R}\La_\nov\cdot(p_-,p_+)$, where $R=\bigl\{(p_-,p_+):p_-,p_+\in L$,
$p_-\ne p_+$, $\io(p_-)=\io(p_+)\bigr\}$ is a set of two extra
generators for each self-intersection point of $L$, and $(p_-,p_+)$
has degree $\eta_{(p_-,p_+)}\in\Z$, an index depending on how $L$
intersects itself at $\io(p_-)=\io(p_+)$.

The theory becomes simpler and more powerful for {\it graded\/}
Lagrangians in Calabi--Yau manifolds, when we can work over a
smaller Novikov ring $\La_\CY$. The proofs involve associating a
gapped filtered $A_\iy$ algebra over $\La_\nov^0$ or $\La_\CY^0$ to
$\io:L\ra M$, which is independent of nearly all choices up to
canonical homotopy equivalence, and is built using a series of
finite approximations called $A_{N,0}$ algebras
for~$N=0,1,2,\ldots$.
\end{abstract}
\maketitle

\baselineskip 11.53pt plus .47pt

\section{Introduction}
\label{il1}

Let $(M,\om)$ be a compact symplectic manifold, and $L$ a compact
embedded Lagrangian submanifold in $M$. Fukaya, Oh, Ohta and Ono
\cite{FOOO} have undertaken the mammoth task of rigorously
constructing {\it Lagrangian Floer cohomology\/} for such $M,L$. In
brief, to each Lagrangian $L$ in $M$ they associate a ({\it gapped
filtered\/}) $A_\iy$ {\it algebra\/} $(\Q\X\ot\La^0_\nov,\m)$. A
{\it bounding cochain\/} $b\in\Q\X\ot\La^0_\nov$ is a solution of
the equation $\sum_{k\ge 0}\m_k(b,\ldots,b)=0$ in
$\Q\X\ot\La^0_\nov$. Given such $b$, they define the Lagrangian
Floer cohomology $HF^*(L,b;\La_\nov)$ of $L$. If $L$ does not admit
a bounding cochain, we say that $L$ has {\it obstructed Lagrangian
Floer cohomology}. If $L_1,L_2$ are transversely intersecting
Lagrangians in $M$ with bounding cochains $b_1,b_2$, they define the
Lagrangian Floer cohomology $HF^*\bigl((L_1,b_1),
(L_2,b_2);\La_\nov\bigr)$ of $L_1,L_2$. These are the morphism
groups in the derived Fukaya category of $(M,\om)$, and so are an
essential part of the Homological Mirror Symmetry Conjecture of
Kontsevich~\cite{Kont}.

The purpose of this paper is to extend the work of Fukaya, Oh, Ohta
and Ono \cite{FOOO} to {\it immersed\/} Lagrangians $L$ in $M$ with
immersion $\io:L\ra M$, with transverse self-intersections. This was
done by the first author \cite{Akah} under the simplifying
assumption that $\pi_2(M,\io(L))=\{1\}$, which eliminates the issues
of disc bubbling, $A_\iy$ algebras and bounding cochains. We now
discuss the much more difficult general case.

Suppose $\io:L\ra M$ is a compact immersed Lagrangian in $(M,\om)$,
such that $\io^{-1}(p)$ is at most two points for each $p\in\io(L)$,
and when $\io^{-1}(p)=\{p_+,p_-\}$ is two points the two sheets of
$L$ intersect transversely at $p$, that is,
$\d\io(T_{p_+}L)\cap\d\io(T_{p_-}L)=\{0\}$ in $T_pM$. We will
construct a gapped filtered $A_\iy$ algebra $(\Q\X\ot
\La^0_\nov,\m)$ associated to $L$, independent of choices up to
canonical homotopy equivalence, which generalizes both the embedded
case in Fukaya et al.\ \cite[\S 3]{FOOO}, and the gapped filtered
$A_\iy$ category associated to finitely many embedded Lagrangian
submanifolds by Fukaya \cite{Fuka}. Thus we can define {\it bounding
cochains\/} $b$ for $L$, and {\it Lagrangian Floer cohomology
groups\/} $HF^*(L,b;\La_\nov)$ and $HF^*\bigl((L_1,b_1),(L_2,b_2);
\La_\nov\bigr)$, which are independent of choices up to canonical
isomorphism.

Fukaya et al.\ \cite{FOOO} mainly develop two subjects: geometry and
algebra. In the geometric part, they realize $A_{N,K}$ structures on
some singular chains of an embedded Lagrangian submanifold $L$
through moduli spaces of isomorphism classes of stable maps from a
genus 0 prestable bordered Riemann surface with boundary attached to
$L$. In the algebraic part, they develop the homotopy theory, or
homological algebra, of $A_{N,K}$ and gapped filtered $A_\iy$
algebras. Finally, they apply the homotopy theory to the geometric
realization, and obtain a gapped filtered $A_\iy$ algebra associated
to an embedded Lagrangian submanifold.

Here we develop a generalization of their geometry, that is, we
construct $A_{N,0}$ structures associated to an immersed Lagrangian
submanifold with transverse self-intersections. Then we apply the
homotopy theory to our generalization, and obtain a gapped filtered
$A_\iy$ algebra associated to an immersed Lagrangian submanifold.

Fukaya et al.\ also construct a gapped filtered $A_\iy$ bimodule
associated to a pair of transversely intersecting embedded
Lagrangian submanifolds \cite{FOOO}, and a gapped filtered $A_\iy$
category associated to a finite number of transversely intersecting
embedded Lagrangian submanifolds \cite{Fuka}. Regarding a finite
union of embedded Lagrangians as a single immersed Lagrangian, their
gapped filtered $A_\iy$ modules and categories become part of our
gapped filtered $A_\iy$ algebras.

Here is one reason why extending Lagrangian Floer cohomology to
immersed Lagrangians may be important. Using the embedded Lagrangian
Floer theory of \cite{FOOO}, one can define the {\it Fukaya
category} ${\rm Fuk}(M,\om)_{\rm em}$, whose objects are roughly
speaking pairs $(L,b)$ of an embedded Lagrangian and a bounding
cochain $b$ for $L$, and the {\it derived Fukaya category}
$D^b\bigl({\rm Fuk}(M,\om)_{\rm em}\bigr)$. Kontsevich's Homological
Mirror Symmetry Conjecture \cite{Kont} says (roughly) that for
$(M,\om)$ a symplectic Calabi--Yau with mirror complex Calabi--Yau
$(\check M,\check J)$, the derived Fukaya category $D^b\bigl({\rm
Fuk}(M,\om)_{\rm em}\bigr)$ should be equivalent as a triangulated
category to the derived category $D^b\bigl({\rm coh}(\check M,\check
J)\bigr)$ of coherent sheaves on~$(\check M,\check J)$.

The theory of this paper would allow us to define an {\it immersed
Fukaya category} ${\rm Fuk}(M,\om)_{\rm im}$ involving immersed
Lagrangians, and the {\it derived immersed Fukaya category}
$D^b\bigl({\rm Fuk}(M,\om)_{\rm im}\bigr)$. We could then use
$D^b\bigl({\rm Fuk}(M,\om)_{\rm im}\bigr)$ in place of
$D^b\bigl({\rm Fuk}(M,\om)_{\rm em}\bigr)$ in Homological Mirror
Symmetry. Actually, it seems likely that $D^b\bigl({\rm
Fuk}(M,\om)_{\rm im}\bigr)$ and $D^b\bigl({\rm Fuk}(M,\om)_{\rm
em}\bigr)$ are {\it equivalent categories}, although $D^b\bigl({\rm
Fuk}(M,\om)_{\rm im}\bigr)$ has more objects.

Motivated by conjectures of Thomas and Yau \cite{ThYa} and more
recent ideas of Bridgeland \cite{Brid} and the String Theorists
Douglas and Aspinwall, we can state the following (approximate)
conjecture, which is an extension of the Homological Mirror Symmetry
story: let $(M,J,\om,\Om)$ be a Calabi--Yau $n$-fold. Then there
should exist a Bridgeland stability condition $(Z,{\cal P})$ on
$D^b\bigl({\rm Fuk}(M,\om)\bigr)$ depending on the holomorphic
$(n,0)$-form $\Om$ on $M$, such that each isomorphism class of
stable objects in $D^b\bigl({\rm Fuk}(M,\om)\bigr)$ is represented
by a unique special Lagrangian.

For this conjecture to have a chance of being true, we need
$D^b\bigl({\rm Fuk}(M,\om)\bigr)$ to contain as many actual
geometric Lagrangians as possible. In particular, the conjecture
should be {\it false} for the embedded case $D^b\bigl({\rm
Fuk}(M,\om)_{\rm em}\bigr)$ when $n>2$, since then there could exist
$(L,b)$ and $(L',b')$ isomorphic in $D^b\bigl({\rm Fuk}(M,\om)_{\rm
im}\bigr)$ with $L$ embedded, and $L'$ special Lagrangian and
immersed but not embedded. Then $(L,b)$ must be stable in both
$D^b\bigl({\rm Fuk}(M,\om)_{\rm em}\bigr)$ and $D^b\bigl({\rm
Fuk}(M,\om)_{\rm im}\bigr)$, but the uniqueness argument of Thomas
and Yau \cite[Th.~4.3]{ThYa} applied in our immersed Floer
cohomology theory implies that there cannot exist $(L'',b'')$ in
$D^b\bigl({\rm Fuk}(M,\om)_{\rm em}\bigr)$ isomorphic to $(L,b)$
with $L''$ special Lagrangian. Thus, to make our modified
Thomas--Yau conjecture true we need at least to include immersed
Lagrangians in the Fukaya category, and perhaps also some classes of
singular Lagrangians as well.

We begin with some background material on Kuranishi spaces,
multisections, and virtual chains in \S\ref{il2}, and on $A_\iy$
algebras and $A_{N,K}$ algebras in \S\ref{il3}. Section \ref{il4}
introduces the moduli spaces of isomorphism classes of stable maps
from a genus 0 prestable bordered Riemann surface with boundary
attached to an immersed Lagrangian submanifold. They are Kuranishi
spaces, with boundary and corners, whose boundaries are fibre
products of other such moduli spaces. Section \ref{il5} discusses
orientations of our moduli spaces.

Sections \ref{il6}--\ref{il11} construct gapped filtered $A_\iy$
algebras from immersed Lagrangian submanifolds $\io:L\ra M$, and
show they are independent of choices such as the almost complex
structure $J$, up to canonical homotopy equivalence. First, in
\S\ref{il6}--\S\ref{il7}, we construct $A_{N,0}$ algebras
$(\Q\X_N,{\cal G},\m)$ from $\io:L\ra M$ for all $N=0,1,2,\ldots$,
involving different arbitrary choices for each $N$. In
\S\ref{il8}--\S\ref{il9}, we show that the $A_{N,0}$ algebras of
\S\ref{il6}--\S\ref{il7} are unique up to $A_{N,0}$ homotopy
equivalences ${\mathfrak j}:(\Q\X_N,{\cal G},\m)\ra(\Q\X'_N,{\cal
G},\m')$, and \S\ref{il10} proves that these $\mathfrak j$ are
unique up to homotopy.

Section \ref{il11} passes from $A_{N,0}$ algebras $(\Q\X_N,{\cal
G},\m)$ to gapped filtered $A_\iy$ algebras $(\Q\X\ot\La_\nov^0,
\m)$ by a limiting process as $N\ra\iy$, and shows that $(\Q\X
\ot\La_\nov^0,\m)$ is independent of choices up to canonical
homotopy equivalence. Section \ref{il12} defines {\it graded\/}
Lagrangians in Calabi--Yau manifolds, and explains how in the graded
case we can redo \S\ref{il6}--\S\ref{il11} using the smaller Novikov
ring $\La_\CY^0$. Finally, \S\ref{il13} defines bounding cochains
and Lagrangian Floer cohomology, discusses some applications, and
suggests some questions and conjectures for future research.

By its very nature, this paper exists wholly in the shadow of
Fukaya, Oh, Ono and Ohta's massive work \cite{FOOO}. Despite this,
we have tried hard to make our paper {\it independent of\/}
\cite{FOOO}, in the sense that our paper is self-contained,
requiring no more than the usual background for research papers in
the area, and readers should not need to read (or even open)
\cite{FOOO} to understand our paper.

We also frequently use different methods of proof to Fukaya et al.\
\cite{FOOO}. In particular, \S\ref{il8}--\S\ref{il10} is much
shorter and simpler than the parallel parts of \cite{FOOO}. The
current version of \cite{FOOO} is more than 1000 pages long, and in
this paper we not only cover a large proportion, the most important
parts, of \cite{FOOO}, but we also generalize it significantly by
extending it to immersed Lagrangians. So we maintain that our paper
is, by the standards of \cite{FOOO}, very short and succinct.

Parts of this paper will be rewritten in \cite{AkJo} using the
second author's theory of {\it Kuranishi cohomology\/}
\cite{Joyc2,Joyc3}, which simplifies issues to do with virtual
chains.
\medskip

\noindent{\it Acknowledgements.} The authors would like to thank
Kenji Fukaya, Hiroshige Kajiura, Yong-Geun Oh, Hiroshi Ohta, Kauru
Ono, Paul Seidel and Ivan Smith for useful conversations, and the
EPSRC for financial support, grant EP/D07763X/1.

\section{Background material on Kuranishi spaces and multisections}
\label{il2}

We now summarize results from Fukaya, Ono et al.\ \cite[\S 3--\S
6]{FuOn}, \cite[\S A]{FOOO} on Kuranishi spaces, multisections and
virtual chains that we will need later. Where the notation of
\cite{FuOn,FOOO} differs, for instance in whether Kuranishi
neighbourhoods are $(V,E,\Ga,s,\psi)$ with $V$ a manifold or
$(V,E,s,\psi)$ with $V$ an orbifold, we generally
follow~\cite{FOOO}.

\subsection{Kuranishi structures on topological spaces}
\label{il21}

We define {\it Kuranishi spaces}, following Fukaya, Ono et al.\
\cite[\S5]{FuOn} and~\cite[\S A1.1]{FOOO}.

\begin{dfn} Let $X$ be a compact, metrizable topological space. A
{\it Kuranishi neighbourhood\/} of $p\in X$ is a quintet $(V_p,E_p,
\Ga_p,s_p,\psi_p)$ such that:
\begin{itemize}
\setlength{\itemsep}{0pt}
\setlength{\parsep}{0pt}
\item[(i)] $V_p$ is a smooth finite-dimensional manifold, which may
or may not have boundary or corners;
\item[(ii)] $E_p\ra V_p$ is a vector bundle over $V_p$;
\item[(iii)] $\Ga_p$ is a finite group which acts smoothly on $V_p$,
and acts compatibly on $E_p$ preserving the vector bundle structure;
\item[(iv)] $s_p:V_p\ra E_p$ is a $\Ga_p$-equivariant smooth section;
and
\item[(v)] $\psi_p$ is a homeomorphism from $s_p^{-1}(0)/\Ga_p$ to a
neighbourhood of $p$ in $X$, where $s_p^{-1}(0)$ is the subset of
$V_p$ where the section $s_p$ is zero.
\end{itemize}
We call $E_p$ the {\it obstruction bundle}, and $s_p$ the {\it
Kuranishi map}.
\label{il2dfn1}
\end{dfn}

Here we follow \cite[Def.~5.1]{FuOn} in taking $E_p$ to be a vector
bundle, rather than a finite-dimensional vector space as
in~\cite[Def.~A1.1]{FOOO}.

\begin{dfn} Let $(V_p,E_p,\Ga_p,s_p,\psi_p)$ and $(V_q,E_q,\Ga_q,s_q,
\psi_q)$ be Kuranishi neighbourhoods of $p\in X$ and $q\in
\psi_p(s_p^{-1}(0)/\Ga_p)$ respectively. We call a triple
$(\hat\phi_{pq},\ab\phi_{pq},\ab h_{pq})$ a {\it coordinate change\/}
if:
\begin{itemize}
\setlength{\itemsep}{0pt}
\setlength{\parsep}{0pt}
\item[(a)] $h_{pq}:\Ga_q\ra\Ga_p$ is an {\it injective\/} group
homomorphism;
\item[(b)] $\phi_{pq}:V_q\ra V_p$ is an $h_{pq}$-equivariant smooth
{\it embedding\/};
\item[(c)] $(\hat\phi_{pq},\phi_{pq})$ is an $h_{pq}$-equivariant
smooth {\it embedding\/} of vector bundles $E_q\ra E_p$;
\item[(d)] $\hat\phi_{pq}\ci s_q\equiv s_p\ci\phi_{pq}$; and
\item[(e)] $\psi_q\equiv \psi_p\ci\phi_{pq}$.
\end{itemize}
\label{il2dfn2}
\end{dfn}

We define the notions of a {\it germ of a Kuranishi neighbourhood\/}
and a {\it germ of a coordinate change\/} in the obvious way.

\begin{dfn} A {\it Kuranishi structure\/} on $X$ assigns a germ of a
Kuranishi neighbourhood $(V_p,E_p,\Ga_p,s_p,\psi_p)$ for each $p\in X$
and a germ of a coordinate change $(\hat\phi_{pq},\phi_{pq},h_{pq})$
for each $q\in\psi_p(s_p^{-1}(0)/\Ga_p)$, such that the following
hold:
\begin{itemize}
\setlength{\itemsep}{0pt}
\setlength{\parsep}{0pt}
\item[(i)] $\dim V_p-\rank E_p$ is independent of $p$\/; and
\item[(ii)] If $q\in \psi_p(s_p^{-1}(0)/\Ga_p)$ and
$r\in\psi_q(s_q^{-1}(0)/\Ga_q)$ then $h_{pq}\ci h_{qr}=h_{pr}$,
$\phi_{pq}\ci\phi_{qr}=\phi_{pr}$
and~$\hat\phi_{pq}\ci\hat\phi_{qr}=\hat\phi_{pr}$.
\end{itemize}
We call $\vdim X=\dim V_p-\rank E_p$ the {\it virtual dimension\/}
of the Kuranishi structure. A topological space $X$ with a Kuranishi
structure is called a {\it Kuranishi space}.
\label{il2dfn3}
\end{dfn}

The point of these definitions is that in many moduli problems in
geometry in which there are obstructions, the moduli spaces can be
equipped with Kuranishi structures in a natural way. This holds for
the moduli spaces of $J$-holomorphic maps from a bordered Riemann
surface studied by Fukaya et al.\ \cite{FOOO} and Liu \cite{Liu}, as
we shall explain in~\S\ref{il4}.

\subsection{Boundaries, strongly smooth maps, and fibre products}
\label{il22}

We now define the {\it boundary\/} $\pd X$ of a Kuranishi space $X$,
which is itself a Kuranishi space of dimension $\vdim X-1$. To
understand the definition, recall that in Definition \ref{il2dfn1}(i),
$V_p$ may be a manifold {\it with boundary}, and {\it with
corners}. An $n$-manifold $M$ with boundary is locally modelled on
$[0,\ep)\t(-\ep,\ep)^{n-1}$, and an $n$-manifold $M$ with corners is
locally modelled on $[0,\ep)^k\t(-\ep,\ep)^{n-k}$, for small $\ep>0$.
If $x$ lies in a codimension $k$ corner of $M$ then $k$ different
$(n-1)$-dimensional boundary strata of $M$ meet at $x$. The
{\it boundary\/} $\pd M$ is the set of pairs $(x,B)$, where $x\in M$
and $B$ is a local choice of $(n-1)$-dimensional boundary stratum of
$M$ containing~$x$.

Thus, if $x$ lies in a codimension $k$ corner of $M$ then $x$ is
represented by $k$ {\it distinct points\/} $(x,B_i)$ in $\pd M$ for
$i=1,\ldots,k$. The point of making $\pd M$ a set of pairs $(x,B)$
and not points $x$ is that this way $\pd M$ is a manifold with
corners, but if we defined $\pd M$ as the obvious subset of $M$ it
would not be a manifold with corners near a codimension $k$ corner of
$M$ for~$k>1$.

\begin{dfn} Let $X$ be a Kuranishi space. We shall define a Kuranishi
space $\pd X$ called the {\it boundary\/} of $X$. The points of
$\pd X$ are equivalence classes $[p,v,B]$ of triples $(p,v,B)$, where
$p\in X$, $(V_p,E_p,\Ga_p,s_p,\psi_p)$ lies in the germ of Kuranishi
neighbourhoods at $p$, $v\in V_p$ with $s_p(v)=0$ and $\psi_p(\Ga_pv)
=p$, and $B$ is a local boundary stratum of $V_p$ containing~$v$.

Two triples $(p,v,B)$ and $(p,w,C)$ are {\it equivalent\/} if $p=q$
and $\ga\cdot(v,B)=(w,C)$ for some $\ga\in\Ga_p$; we also have an
obvious notion of equivalence for choices of different Kuranishi
neighbourhoods $(V_p,E_p,\Ga_p,s_p,\psi_p)$, $(V'_p,E'_p,\Ga'_p,s'_p,
\psi'_p)$ in the germ at $p$. Basically, this just means that points
of $\pd X$ are $p\in X$ together with a choice of boundary stratum of
the Kuranishi neighbourhoods $V_p$ lying over $p$, up to the action
of~$\Ga_p$.

We can then define a unique natural {\it topology\/} and {\it
Kuranishi structure\/} on $\pd X$, such that $(\pd V_p,E_p\vert_{\pd
V_p},\Ga_p,s_p\vert_{\pd V_p},\psi_p\vert_{\pd V_p/\Ga_p})$ is a
Kuranishi neighbourhood on $\pd X$ for each Kuranishi neighbourhood
$(V_p,E_p,\Ga_p,s_p,\psi_p)$ on $X$. It is easy to verify that
$\vdim\pd X=\vdim X-1$, and $\pd X$ is compact if $X$ is compact.
\label{il2dfn4}
\end{dfn}

Here is \cite[Def.~6.6]{FuOn}. The equivalent definition in
\cite[Def.~A1.13]{FOOO} instead uses {\it good coordinate systems}.
Fukaya et al. \cite[Def.~A1.13]{FOOO} use the notation {\it weakly
submersive} rather than strong submersion.

\begin{dfn} Let $X$ be a Kuranishi space, and $Y$ be a topological
space. Roughly speaking, a {\it strongly continuous map\/} $\bs f:
X\ra Y$ consists of a continuous map $f_p:V_p\ra Y$ with $f_p\ci\ga
\equiv f_p$ for all $\ga\in\Ga_p$ for each Kuranishi neighbourhood
$(V_p,E_p,\Ga_p,s_p,\psi_p)$ in $X$, such that if
$(\hat\phi_{pq},\phi_{pq},h_{pq})$ is a coordinate change between
$(V_p,E_p,\Ga_p,s_p,\psi_p)$ and $(V_q,E_q,\Ga_q,s_q,\psi_q)$, then
$f_p\ci\phi_{pq}=f_q$. But because Kuranishi spaces are defined
using germs of Kuranishi neighbourhoods, we define a strongly
continuous map $\bs f$ to be a system of germs of $\Ga_p$-invariant
continuous maps $f_p:V_p\ra Y$, satisfying $f_p\ci\phi_{pq}=f_q$ for
germs of coordinate changes. Then $\bs f$ induces a continuous map
$f:X\ra Y$ in the obvious way. If $Y$ is a smooth manifold and all
$f_p$ are smooth, we call $\bs f$ {\it strongly smooth}, and if all
$f_p$ are submersions, we call $\bs f$ a {\it strong submersion}.
\label{il2dfn5}
\end{dfn}

Fukaya et al.\ \cite[Def.~A1.37]{FOOO} define {\it fibre products\/}
of Kuranishi spaces.

\begin{dfn} Let $X,X'$ be Kuranishi spaces, $Y$ be a smooth manifold,
and $\bs f:X\ra Y$, $\bs f':X'\ra Y$ be strongly smooth maps, at
least one of which is a strong submersion, inducing continuous maps
$f:X\ra Y$ and $f':X'\ra Y$. Then we can form the {\it fibre
product\/} $X\t_YX'=\bigl\{(p,p')\in X\t X':f(p)=f'(p') \bigr\}$, a
paracompact Hausdorff topological space. We also write $X\t_YX'$ as
$X\t_{\bs f,Y,\bs f'}X'$ when we wish to specify ~$\bs f,\bs f'$.

Let $(p,p')\in X\t_YX'$, let $(V_p,E_p,\Ga_p,s_p,\psi_p)$, $(V'_{p'},
E'_{p'},\Ga'_{p'},s'_{p'},\psi'_{p'})$ be sufficiently small Kuranishi
neighbourhoods in the germs at $p,p'$ in $X,X'$, and $f_p:V_p\ra Y$,
$f'_{p'}:V'_{p'}\ra Y$ be smooth maps in the germs of $\bs f,\bs f'$
at $p,p'$ respectively. Define a Kuranishi neighbourhood in $X\t_YX'$
by
\begin{equation}
\begin{split}
\bigl(V_p\t_{f_p,Y,f^{\smash\prime}_{\smash{p'}}}V'_{p'},
&(E_p\op E'_{p'})
\vert_{V_p\t_YV^{\smash{\prime}}_{\smash{p^{\smash{\prime}}}}},
\Ga_p\t\Ga'_{p'}, \\
&(s_p\op s'_{p'})\vert_{V_p\t_YV^{
\smash{\prime}}_{\smash{p^{\smash{\prime}}}}},
(\psi_p\t\psi'_{p'})\vert_{V_p\t_YV^{\smash{\prime}}_{
\smash{p^{\smash{\prime}}}}}\bigr).
\end{split}
\label{il2eq1}
\end{equation}
Here $V_p\t_{f_p,Y,f'_{p'}}V'_{p'}$ is the fibre product of smooth
manifolds, defined as at least one of $f_p,f'_{p'}$ is a submersion.
It is a submanifold of $V_p\t V'_{p'}$, so we can restrict $E_p\op
E'_{p'}$, $s_p\op s'_{p'}$ and $\psi_p\t\psi'_{p'}$ to it.

It is easy to verify that coordinate changes between Kuranishi
neighbourhoods in $X$ and $X'$ induce coordinate changes between
neighbourhoods \eq{il2eq1}. So the systems of germs of Kuranishi
neighbourhoods and coordinate changes on $X,X'$ induce such systems on
$X\t_YX'$. This gives a {\it Kuranishi structure\/} on $X\t_YX'$,
making it into a {\it Kuranishi space}. Clearly $\vdim(X\t_YX')=
\vdim X+\vdim X'-\dim Y$, and $X\t_YX'$ is compact if $X,X'$ are
compact.
\label{il2dfn6}
\end{dfn}

\subsection{Tangent bundles and orientations}
\label{il23}

Here is \cite[Def.~5.6]{FuOn}. The equivalent definition in
\cite[Def.~A1.14]{FOOO} involves a choice of {\it good coordinate
system}.

\begin{dfn} Let $X$ be a Kuranishi space. Then $X$ has a germ of
coordinate changes $(\hat\phi_{pq},\phi_{pq},h_{pq})$ between
Kuranishi neighbourhoods $(V_p,E_p,\Ga_p,s_p,\psi_p)$ and $(V_q,E_q,
\Ga_q,s_q,\psi_q)$. We say that $X$ {\it has a tangent bundle\/} if
associated to this germ of coordinate changes $(\hat\phi_{pq},
\phi_{pq},h_{pq})$ we have a germ of $\Ga_q$- and $h_{pq}$-equivariant
isomorphisms of vector bundles over~$V_q$:
\begin{equation}
\chi_{pq}:\frac{\phi_{pq}^*(E_p)}{\hat\phi_{pq}(E_q)}\longra
\frac{\phi_{pq}^*(TV_p)}{(\d\phi_{pq})(TV_q)},
\label{il2eq2}
\end{equation}
where $\hat\phi_{pq}:E_q\ra\phi_{pq}^*(E_p)$ and $\d\phi_{pq}:
TV_q\ra\phi_{pq}^*(TV_p)$ are morphisms of vector bundles over $V_q$.
These must agree on triple overlaps: if $(V_p,E_p,\Ga_p,s_p,\psi_p)$,
$(V_q,E_q,\Ga_q,s_q,\psi_q)$ and $(V_r,E_r,\Ga_r,s_r,\psi_r)$ are
Kuranishi neighbourhoods and $(\hat\phi_{pq},\ab\phi_{pq},\ab h_{pq}),
(\hat\phi_{pr},\phi_{pr},h_{pr}),(\hat\phi_{qr},\phi_{qr},h_{qr})$
coordinate changes between them with $\hat\phi_{pr}=
\hat\phi_{pq}\ci\hat\phi_{qr}$, $\phi_{pr}=\phi_{pq}\ci\phi_{qr}$ and
$h_{pr}=h_{pq}\ci h_{qr}$, then the following diagram of vector
bundles over $V_r$ must commute:
\begin{equation*}
\begin{split}
\xymatrix@C=13pt{0 \ar[r] &
\frac{\phi_{qr}^*(E_q)}{\hat\phi_{qr}(E_r)}
\ar[rr]^{\phi_{qr}^*(\hat\phi_{pq})} \ar[d]_{\chi_{qr}} &&
\frac{\phi_{pr}^*(E_p)}{\hat\phi_{pr}(E_r)}
\ar[rr]^{\!\!\!\!\!\!\!\text{\scriptsize project}}
\ar[d]_{\chi_{pr}} &&
\frac{\phi_{pr}^*(E_p)}{\phi_{qr}^*(\hat\phi_{pq}(E_q))}
\ar[d]_{\phi_{qr}^*(\chi_{pq})} \ar[r] & 0 \\
0 \ar[r] & \frac{\phi_{qr}^*(TV_q)}{(\d\phi_{qr})(TV_r)}
\ar[rr]^{\phi_{qr}^*(\d\phi_{pq})} &&
\frac{\phi_{pr}^*(TV_p)}{(\d\phi_{pr})(TV_r)}
\ar[rr]^{\!\!\!\!\!\!\!\!\!\!\!\!\!\text{\scriptsize project}} &&
\frac{\phi_{pr}^*(TV_p)}{\phi_{qr}^*(\d\phi_{pq})(\phi_{qr})^*(TV_q)}
\ar[r] & 0.}
\end{split}
\end{equation*}
\label{il2dfn7}
\end{dfn}

We can now discuss {\it orientations\/} of Kuranishi spaces.

\begin{dfn} Let $X$ be a Kuranishi space with a tangent bundle. We
say that the Kuranishi structure on $X$ is {\it oriented\/} if
associated to the germ of Kuranishi neighbourhoods $(V_p,E_p,\Ga_p,
s_p,\psi_p)$ on $X$ we are given a germ of orientations of the fibres
of the vector bundles $E_p\op TV_p$ varying continuously over $V_p$.
These must be compatible on overlaps, in the following sense. Let
$(\hat\phi_{pq},\phi_{pq},h_{pq})$ be in the germ of coordinate
changes, and $\chi_{pq}$ be as in~\eq{il2eq2}.

Then if $v\in V_q$ and $(e_q^1,\ldots,e_q^m)$,
$(t_q^1,\ldots,t_q^n)$ are bases of $E_q\vert_v$ and $T_vV_q$ such
that $(e_q^1,\ldots,e_q^m, t_q^1,\ldots,t_q^n)$ is an oriented basis
of $(E_q\op TV_q)\vert_v$, and if
$\bigl(e_p^1,\ldots,e_p^k,\hat\phi_{pq}(e_q^1),\ab
\ldots,\hat\phi_{pq}(e_q^m)\bigr)$ and $\bigl(t_p^1,\ldots,t_p^k,
(\d\phi_{pq})(t_q^1),\ldots,(\d\phi_{pq})(t_q^n)\bigr)$ are bases of
$E_p\vert_{\phi_{pq}(v)}$ and $T_{\phi_{pq}(v)}V_p$ such that
$\chi_{pq}\bigl(e_p^i+\hat\phi_{pq}(E_q\vert_v)\bigr)=t_p^i+
(\d\phi_{pq})(T_vV_q)$ for $i=1,\ldots,k$, then $\bigl(e_p^1,\ldots,
e_p^k,\hat\phi_{pq}(e_q^1),\ldots,\hat\phi_{pq}(e_q^m),t_p^1,\ldots,
t_p^k,(\d\phi_{pq})(t_q^1),\allowbreak \ldots,\allowbreak
(\d\phi_{pq})(t_q^n)\bigr)$ is an oriented basis for~$(E_p\op
TV_p)\vert_{\phi_{pq}(v)}$.
\label{il2dfn8}
\end{dfn}

\subsection{Orientation conventions}
\label{il24}

Suppose $X,X'$ are Kuranishi spaces with tangent bundles and
orientations, $Y$ is an oriented smooth manifold, and $\bs f:X\ra
Y$, $\bs f':X'\ra Y$ are strongly smooth maps. Then by \S\ref{il22}
we have Kuranishi spaces $\pd X$ and $X\t_YX'$. These can also be
given orientations in a natural way. We use the orientation
conventions of Fukaya et al.~\cite[\S 45]{FOOO}.

\begin{conv} First, our conventions for smooth manifolds:
\begin{itemize}
\item[(a)] Let $X$ be an oriented smooth manifold with boundary $\pd
X$. Then we define the orientation on $\pd X$ such that
\begin{equation*}
TX\vert_{\pd X}=\R_{\rm out}\op T(\pd X)
\end{equation*}
is an isomorphism of oriented vector spaces, where $\R_{\rm out}$ is
oriented by an outward-pointing normal vector to~$\pd X$.
\item[(b)] Let $X,X',Y$ be oriented smooth manifolds, and $f:X\ra
Y$, $f':X'\ra Y$ be smooth submersions. Then $\d f:TX\ra f^*(TY)$
and $\d f':TX'\ra (f')^*(TY)$ are surjective maps of vector bundles
over $X,X'$. Choosing Riemannian metrics on $X,X'$ and identifying
the orthogonal complement of $\Ker\d f$ in $TX$ with the image
$f^*(TY)$ of $\d f$, and similarly for $f'$, we have isomorphisms of
vector bundles over $X,X'$:
\begin{equation}
TX\cong \Ker\d f\op f^*(TY) \quad\text{and}\quad TX'\cong
(f')^*(TY)\op\Ker\d f'.
\label{il2eq3}
\end{equation}

Define orientations on the fibres of $\Ker\d f$, $\Ker\d f'$ over
$X,X'$ such that \eq{il2eq3} are isomorphisms of oriented vector
bundles, where $TX,TX'$ are oriented by the orientations on $X,X'$,
and $f^*(TY),(f')^*(TY)$ by the orientation on $Y$. Then we define
the orientation on $X\t_YX'$ so that
\begin{equation}
\begin{split}
T(X\t_YX')&\cong \pi_X^*(\Ker\d f)\op(f\ci\pi_X)^*(TY)
\op\pi_{X'}^*(\Ker\d f')\\
&\cong \pi_X^*(\Ker\d f)\op\pi_{X'}^*(TX')\\
&\cong \pi_X^*(TX)\op\pi_{X'}^*(\Ker\d f')
\end{split}
\label{il2eq4}
\end{equation}
are isomorphisms of oriented vector bundles. Here $\pi_X:X\t_YX'\ra
X$ and $\pi_{X'}:X\t_YX'\ra X'$ are the natural projections, and
$f\ci\pi_X\equiv f'\ci\pi_{X'}$.

Note that the second line of \eq{il2eq4} makes sense if $f$ is a
submersion but $f'$ is only smooth, and the third line makes sense
if $f'$ is a submersion but $f$ is only smooth. Thus, our convention
extends to fibre products $X\t_{f,Yf'}X'$ in which only one of
$f,f'$ is a submersion.
\end{itemize}
Here is how to extend (b) to $X,X'$ Kuranishi spaces:
\begin{itemize}
\item[(c)] Let $X,X'$ be oriented Kuranishi spaces, $Y$ be an
oriented smooth manifold, and $\bs f:X\ra Y,\bs f':X'\ra Y$ be
strong submersions. We take Kuranishi neighbourhoods
$(V_p,E_p,\Ga_p,s_p,\psi_p)$,
$(V'_{p'},E'_{p'},\Ga'_{p'},s'_{p'},\psi'_{p'})$ for $X,X'$,
respectively. First, choose orientations of $V_p$ and $V'_{p'}$, and
we have the orientation of $V_p\t_{f_p,Y,f'_{p'}}V'_{p'}$ by
Convention \ref{il2conv}(b). Secondly, the orientations of $E_p\op
TV_p$ and $E'_{p'}\op TV'_{p'}$ induce the orientation of $(E_p\op
E'_{p'})|_{V_p\t_{f_p,Y,f'_{p'}}V'_{p'}} \op
T_{(p,p')}\bigl(V_p\t_{f_p,Y,f'_{p'}}V'_{p'}\bigr)$. Then we {\it
define\/} an orientation of the Kuranishi neighbourhood \eq{il2eq1}
with the following sign correction term:
\begin{equation*}
(-1)^{\rank E'_{p'}(\dim V_p-\rank E_p-\dim Y)} (E_p\!\op\!
E'_{p'})\vert_{V_p\t_{f_p,Y,f'_{p'}}V'_{p'}}\!\op\!
T_{(p,p')}(V_p\t_{f_p,Y,f'_{p'}}V'_{p'}),
\end{equation*}
where $-1$ means the opposite orientation. This orientation
convention is independent of the choice of Kuranishi neighbourhood.
It extends to only one of $\bs f,\bs f'$ a strong submersion as
in~(b).
\end{itemize}
\label{il2conv}
\end{conv}

If $X$ is a Kuranishi space with tangent bundle and orientation, we
will write $-X$ for the same Kuranishi space with the opposite
orientation. Here is \cite[Lem.~45.3]{FOOO}, except the second line
of \eq{il2eq5}, which is elementary.

\begin{prop} Let\/ $X_1,X_2,\ldots$ be Kuranishi spaces with tangent
bundles and orientations, $Y,Y_1,\ldots$ be oriented smooth
manifolds without boundary, and\/ $\bs f_1:X_1\ra Y,\ldots$ be
strong submersions. Then the following hold, in Kuranishi spaces
with tangent bundles and orientations:
\begin{itemize}
\item[(a)] For\/ $\bs f_1:X_1\ra Y$ and\/ $\bs f_2:X_2\ra Y$ we have
\begin{equation}
\begin{gathered}
\pd(X_1\t_YX_2)=(\pd X_1)\t_YX_2\amalg (-1)^{\vdim X_1+\dim Y}
X_1\t_Y(\pd X_2)\\
\text{and}\quad X_1\t_YX_2=(-1)^{(\vdim X_1-\dim Y)(\vdim X_2-\dim
Y)}X_2\t_YX_1.
\end{gathered}
\label{il2eq5}
\end{equation}
\item[(b)] For\/ $\bs f_1:X_1\ra Y_1$, $\bs f_2:X_2\ra Y_1\t
Y_2$ and\/ $\bs f_3:X_3\ra Y_2$, we have
\begin{equation*}
(X_1\t_{Y_1}X_2)\t_{Y_2}X_3=X_1\t_{Y_1}(X_2\t_{Y_2}X_3).
\end{equation*}
\item[(c)] For\/ $\bs f_1:X_1\ra Y_1\t Y_2$, $\bs f_2:X_2\ra
Y_1$ and\/ $\bs f_3:X_3\ra Y_2$, we have
\begin{equation*}
X_1\t_{Y_1\t Y_2}(X_2\t X_3)=(-1)^{\dim Y_2(\dim Y_1+\vdim X_2)}
(X_1\t_{Y_1}X_2)\t_{Y_2}X_3.
\end{equation*}
\end{itemize}
\label{il2prop1}
\end{prop}

\subsection{Good coordinate systems}
\label{il25}

{\it Good coordinate systems\/} are convenient choices of finite
coverings of $X$ by Kuranishi neighbourhoods, \cite[Def.~6.1]{FuOn},
\cite[Lem.~A1.11]{FOOO}.

\begin{dfn} Let $X$ be a compact Kuranishi space. A {\it good
coordinate system\/} on $X$ consists of a finite indexing set $I$,
an order $<$ on $I$, a family $\bigl\{(V^i,E^i,\Ga^i,s^i,\psi^i):
i\in I\bigr\}$ of Kuranishi neighbourhoods on $X$ with
$X=\bigcup_{i\in I}\Im\psi^i$, and for all $i,j\in I$ with $j<i$ and
$\Im\psi^i\cap\Im\psi^j\ne\emptyset$, a quadruple
$(V^{ij},\hat\phi^{ij},\phi^{ij},h^{ij})$, where $V^{ij}$ is a
$\Ga^j$-invariant open neighbourhood of $(\psi^j)^{-1}(\Im\psi^i)$
in $V^j$, and $(\hat\phi^{ij},\phi^{ij},h^{ij})$ is a coordinate
change from $(V^{ij},E^j\vert_{V^{ij}},\Ga^j,s^j\vert_{V^{ij}},
\psi^j\vert_{V^{ij}})$ to $(V^i,E^i,\Ga^i,s^i,\psi^i)$. Whenever
$i,j,k\in I$ with $k<j<i$ these should satisfy $\hat\phi^{ij}\ci
\hat\phi^{jk}=\hat\phi^{ik}$, $\phi^{ij}\ci\phi^{jk}=\phi^{ik}$ and
$h^{ij}\ci h^{jk}=h^{ik}$ over~$(\phi^{jk})^{-1}(V^{ij})\cap
V^{jk}\cap V^{ik}$.
\label{il2dfn9}
\end{dfn}

Then Fukaya and Ono prove \cite[Lem.~6.3]{FuOn},\cite[Lem.~A1.11]
{FOOO}:

\begin{prop} Let\/ $X$ be a compact Kuranishi space and\/
$\{U_\al:\al\in A\}$ an open cover of\/ $X$. Then there exists a
good coordinate system on\/ $X$ such that for each\/ $i\in I$ we
have $\Im\psi^i\subseteq U_\al$ for some~$\al\in A$.
\label{il2prop2}
\end{prop}

\subsection{Chains and homology}
\label{il26}

Let $Y$ be a smooth manifold. We now explain the complexes we will
use to define the homology of $Y$. We shall work throughout with
{\it singular homology\/} defined using {\it smooth simplicial
chains\/} on $Y$, following Fukaya and Ono \cite{FuOn}. Write $\De_k$
for the $k$-{\it simplex}
\begin{equation}
\bigl\{(x_0,\ldots,x_k)\in\R^{k+1}:x_i\ge 0,\;
x_0+\cdots+x_k=1\bigr\}.
\label{il2eq6}
\end{equation}
The {\it singular chain complex\/} $\bigl(C_*^\rsi(Y;\Q),\pd\bigr)$
of $Y$ has $C_k^\rsi(Y;\Q)$ the $\Q$-vector space with basis smooth
maps $f:\De_k\ra Y$, and boundary operator $\pd:C_k^\rsi(Y;\Q)\ra
C_{k-1}^\rsi(Y;\Q)$ given by
\begin{equation}
\pd:\ts\sum_{a\in A}\rho_a\,f_a\longmapsto \ts\sum_{a\in
A}\sum_{j=0}^k(-1)^j\rho_a(f_a\ci F_j^k),
\label{il2eq7}
\end{equation}
where for $j=0,\ldots,k$ the map $F_j^k:\De_{k-1}\ra\De_k$ is given
by $F_j^k(x_0,\ldots,x_{k-1})=(x_0,\ldots,x_{j-1},0,x_j,\ldots,
x_{k-1})$. The {\it singular homology\/} $H_*^\rsi(Y;\Q)$ of $Y$ is
the homology of~$\bigl(C_*^\rsi(Y;\Q),\pd\bigr)$.

However, following Fukaya et al.\ \cite{FOOO}, when we define
$A_{N,0}$ algebras and $A_\iy$ algebras below we will not use the
full chain complex $\bigl(C_*^\rsi(Y;\Q),\pd\bigr)$, but certain
{\it subcomplexes\/} $(\Q\X,\pd)$. When we do this, we will use the
following conventions:
\begin{itemize}
\setlength{\itemsep}{0pt}
\setlength{\parsep}{0pt}
\item $\X$ is a finite set of smooth maps $f:\De_k\ra Y$, ranging
over different $k=0,1,\ldots$, and allowing $k>\dim Y$. We generally
refer to elements of $\X$ as $f$, taking the domain $\De_k$ of $f$
(that is, the choice of $k=0,1,\ldots$) to be implicit.
\item $\Q\X$ is the graded $\Q$-vector subspace of $C_*^\rsi(Y;\Q)$
with basis $\X$.
\item if $f\in\X$ maps $\De_k\ra Y$, then $f\ci F_j^k\in\X$ for
$j=0,\ldots,k$. Thus $\Q\X$ is closed under $\pd$ by \eq{il2eq7},
and $(\Q\X,\pd)$ is a subcomplex of
$\bigl(C_*^\rsi(Y;\Q),\pd\bigr)$. The inclusion $\Q\X\hookrightarrow
C_*^\rsi(Y;\Q)$ induces a morphism $H_*\bigl((\Q\X,\pd)\bigr)\ra
H_*^\rsi(Y;\Q)$ from the homology of $(\Q\X,\pd)$ to the singular
homology of $Y$. We require $\X$ to be chosen so that this morphism
is an isomorphism.
\item We shall also consider tensor products $\Q\X\ot\La^*_\nov$
with a {\it Novikov ring\/} $\La_\nov^*=\La^0_\nov$ or $\La_\nov$.
Then $(\Q\X\ot\La^*_\nov,\pd)$ is a complex of $\La^*_\nov$-modules.
\end{itemize}

The reason for working with finitely generated subcomplexes
$(\Q\X,\pd)$ is that in the construction of an $A_\iy$ algebra for a
Lagrangian submanifold, when we perturb our moduli spaces
$\oM_{k+1}^\ma(\al,\be,J)$ to make them transverse, just one
perturbation is not enough, we need a different choice of
perturbation for each $k$-tuple $(f_1,\ldots,f_k)$ of chains
$f_1,\ldots,f_k$ in our chain complex $\Q\X$. To keep these choices
under control, we cannot work with the full complex
$C_*^\rsi(Y;\Q)$, but only with finite generated subcomplexes
$\Q\X$, which are constructed together with associated perturbations
of $\oM_{k+1}^\ma(\al,\be,J,f_1,\ldots,f_k)$ for
$f_1,\ldots,f_k\in\X$ using an inductive method.

The following proposition will be an important tool in constructing
such~$\X$.

\begin{prop} Let\/ $Y$ be a compact manifold, possibly with
boundary and corners. Let $\cal W$ be a finite set of smooth maps
$f:\De_k\ra Y$, ranging over different $k=0,1,\ldots$. Then there
exists a finite set $\X$ of smooth maps $f:\De_k\ra Y$, ranging over
different $k=0,1,\ldots$, with the following properties:
\begin{itemize}
\setlength{\itemsep}{0pt}
\setlength{\parsep}{0pt}
\item[{\rm(i)}]${\cal W}\subseteq\X;$
\item[{\rm(ii)}] if\/ $f:\De_k\ra Y$ lies in\/ $\X$ and\/ $k>0$ then\/
$f\ci F_j^k:\De_{k-1}\ra Y$ lies in\/ $\X$ for all\/ $j=0,\ldots,k;$
and
\item[{\rm(iii)}] part\/ {\rm(ii)} implies that\/ $\Q\X$ is
closed under\/ $\pd,$ and a subcomplex of the singular chains\/
$C_*^\rsi(Y;\Q)$. We require that the natural projection\/
$H_*(\Q\X,\pd)\ra H_*^\rsi(Y;\Q)$ should be an isomorphism.
\end{itemize}
\label{il2prop3}
\end{prop}

In \cite{AkJo} the authors will rewrite much of \cite{FOOO} using
the theory of {\it Kuranishi cohomology\/} developed by the second
author \cite{Joyc2,Joyc3}. In this approach there is no need to
perturb moduli spaces $\oM_{k+1}^\ma(\al,\be,J)$ to make them
transverse, and we define our $A_\iy$ structure on the full
Kuranishi cochains $KC^*(Y;\La_\nov^0)$, instead of countably
generated subcomplexes~$(\Q\X,\pd)$.

\subsection{Multisections and virtual chains}
\label{il27}

In many geometric situations, if a moduli space $X$ is singular or
does not have the expected dimension, then one can make a small
perturbation to get a new moduli space $X'$ which is smooth and of
the expected dimension. The Kuranishi structure formalism allows us
to make these perturbations in an abstract way. The basic idea is to
choose a {\it good coordinate system}, as in Definition \ref{il2dfn9},
and then perturb the sections $s^i:V^i\ra E^i$ to smooth
$\ti s^i:V^i\ra E^i$ which are {\it transverse}, that is,
$\d\ti s^i:T_vV^i\ra E^i$ is surjective for each $v\in(\ti s^i)^{-1}
(0)$. Then $(\ti s^i)^{-1}(0)$ is a smooth manifold of dimension
$\vdim X$. The perturbations $\ti s^i,\ti s^j$ must be compatible on
the overlaps $V^{ij}$.

However, it may be impossible to choose $\ti s^i$ both transverse
and $\Ga^i$-equivariant. To deal with this, Fukaya and Ono \cite[\S
3]{FuOn}, \cite[\S A1]{FOOO} introduce {\it multisections}.

\begin{dfn} Let $(V,E,\Ga,s,\psi)$ be a Kuranishi neighbourhood on
some space $X$. For each $n\ge 1$, write $S^nE\ra V$ for the quotient
of the vector bundle $E^n\ra V=E\t\cdots\t E\ra V$ by the symmetric
group $S_n$. That is, the fibre of the bundle $S^nE$ over $v\in V$
is~$(E\vert_v)^n/S_n$.

Define an $n$-{\it multisection\/} $\bs s$ of the orbibundle $E\ra
V$ to be a continuous, $\Ga$-equivariant section of the bundle
$S^nE\ra V$. An $n$-multisection $\bs s$ is called {\it liftable\/}
if there exists $\bs{\ti s}=(s_1,\ldots,s_n):V\ra E^n$ with each
$s_a$ continuous such that $\bs s=\pi\ci\bs{\ti s}$, where
$\pi:E^n\ra S^nE$ is the projection. Note that we do {\it not\/}
require the $s_a$ for $a=1,\ldots,n$ to be $\Ga$-equivariant. A
liftable $n$-multisection $\bs s$ is called {\it smooth\/} if it has
a lift $\bs{\ti s}=(s_1,\ldots,s_n)$ with each $s_a$ smooth, and
{\it transverse\/} if these smooth $s_a$ are transverse, that is,
$\d s_a:T_vV\ra E$ is surjective for each $v\in s_a^{-1}(0)$. When
$V$ has boundary and corners, we also require that the restriction
of each $s_a$ to each codimension $k$ corner of $V$ should be
transverse. This implies that $s_a^{-1}(0)$ is a submanifold of $V$,
of dimension $\dim V-\rank E$, with boundary and corners.

For $n,m\ge 1$, there is an obvious map $E^n\ra E^{nm}$ in which
each $E$ factor of $E^n$ is repeated $m$ times. This induces a map
$S^nE\ra S^{nm}E$. Composing with this maps an $n$-multisection to
an $nm$-multisection. An $n$-multisection $\bs s$ and an
$m$-multisection $\bs s'$ are called {\it equivalent\/} if the
induced $nm$-multisections coincide. A (smooth, or transverse) {\it
multisection\/} $\s$ of $E\ra V$ is defined to be an equivalence
class of (smooth, or transverse) $n$-multisections $\bs s$ over
all~$n$.
\label{il2dfn10}
\end{dfn}

We now sketch the construction of {\it virtual chains} in Fukaya and
Ono \cite[\S 3 \& \S 6]{FuOn}, \cite[\S A1]{FOOO}, without going
into detail. Let $X$ be a compact Kuranishi space with a tangent
bundle and an orientation, which may have boundary and corners, let
$Y$ be an orbifold, and $\bs g:X\ra Y$ a strongly smooth map. By
Proposition \ref{il2prop2} we may choose a {\it good coordinate
system} $\bs I=\bigl(I,<,(V^i,\ldots,\psi^i):i\in I\bigr)$ for $X$,
and smooth maps $g^i:V^i\ra Y$ representing $\bs g$ for $i\in I$,
with $g^i\ci\phi^{ij}\equiv g^j\vert_{V^{ij}}$ when $j<i$ in $I$ and
$\Im\psi^i\cap\Im\psi^j\ne\emptyset$. By induction on $i\in I$ in
the order $<$, for each $i\in I$ Fukaya and Ono choose a sequence
$(\s^i_n)_{n=1}^\iy$ of smooth, transverse multisections on
$(V^i,E^i,\Ga^i,s^i,\psi^i)$, such that $\s^i_n\ra s^i$ in the $C^0$
topology as~$n\ra\iy$.

When $j<i$ in $I$ and $\Im\psi^i\cap\Im\psi^j\ne\emptyset$, the
$(\s^i_n)_{n=1}^\iy$ and $(\s^j_n)_{n=1}^\iy$ satisfy compatibility
conditions: we have $\hat\phi^{ij}\ci{\mathfrak s}^j_n
\equiv\s^i_n\ci\phi^{ij}$ on $V^{ij}$ for all $n=1,2,\ldots$.
Furthermore, since $X$ has a tangent bundle we have isomorphisms
$\chi^{ij}$ over $V^{ij}$ as in \eq{il2eq2}, and Fukaya and Ono use
these and $\hat\phi^{ij}\ci \s^j_n$ to prescribe $\s^i_n$ on an open
neighbourhood of $\phi^{ij}(V^{ij})$ in~$V^i$.

If the multisections $\s^i_n$ were single valued sections of $E^i$,
then as they are transverse $(\s^i_n)^{-1}(0)$ would be a smooth
oriented $\Ga^i$-invariant submanifold of $V^i$ of dimension $\vdim
X$, so $(\s^i_n)^{-1}(0)/\Ga^i$ would be a smooth orbifold. The
compatibility conditions over $V^{ij}$ mean that $\phi^{ij}$ induces
a local diffeomorphism of $(\s^i_n)^{-1}(0)/\Ga^i$ and $(s^j_n)^{-1}
(0)/\Ga^j$ over $V^{ij}/\Ga^j$. Gluing the $(\s^i_n)^{-1}(0)/\Ga^i$
for fixed $n$ and all $i\in I$ together using $\phi^{ij}$ yields a
smooth oriented orbifold $\ti X_n$. When $n\gg0$, so that
$(\s^i_n)^{-1}(0)$ is $C^0$ close to $(s^i)^{-1}(0)$, this $\ti X_n$
would be both compact and Hausdorff, so we would have perturbed $X$ to
a compact, smooth, oriented orbifold $\ti X_n$ of dimension
$k=\vdim X$, which may have boundary and corners.

The smooth maps $g^i:V^i\ra Y$ would glue together to give a smooth
map $\ti g_n:\ti X_n\ra Y$. We would then choose a {\it
triangulation\/} of $\ti X_n$ by smooth singular simplices
$f_a:\De_k\ra\ti X_n$ for $a\in A$, a finite indexing set. The {\it
virtual chain\/} for $(X,\bs g)$ would then be $VC(X,\bs g)=
\sum_{a\in A}\ep_a(\ti g_n\ci f_a)$ in $C_k^\rsi(Y;\Q)$, where
$\ep_a$ is 1 if $f_a$ is orientation-preserving, and $-1$ if $f_a$
is orientation-reversing. If $\pd X=\emptyset$ then $\pd\ti
X_n=\emptyset$, so $\pd VC(X,\bs g)=0$. Then $VC(X,\bs g)$ is called
the {\it virtual cycle\/} of $(X,\bs g)$, and its homology class
$[VC(X,\bs g)]\in H_k^\rsi(Y;\Q)$ is independent of choices of $\bs
I,\s^i_n,n,\ldots$, and is called the {\it virtual class\/}
of~$(X,\bs g)$.

Although the multisections $\s^i_n$ are not in general single valued
sections of $E^i$, we can still follow the method above, with some
adaptations. Represent $\s^i_n$ by a liftable $m$-multisection on
$V^i$ with lift $(s^i_{n,1},\ldots,s^i_{n,m})$. Then each
$(s^i_{n,b})^{-1}(0)$ is an oriented submanifold of $V^i$, not
necessarily $\Ga^i$-invariant. In place of $(\s^i_n)^{-1}(0)$, we
write $\frac{1}{m}\sum_{b=1}^m(s^i_{n,b})^{-1}(0)$, considered as a
$\Q$-linear combination of oriented submanifolds of $V^i$, and this
is then $\Ga^i$-invariant, and essentially independent of the choice
of $m$-multisection and lift $(s^i_{n,1},\ldots,s^i_{n,m})$
representing $\s^i_n$. Here we do not distinguish sheets of
$\frac{1}{m}\sum_{b=1}^m(s^i_{n,b})^{-1}(0)$ that lie on top of each
other locally, but regard them as a single sheet and add up the
multiplicities $\frac{1}{m}$. So we regard
$\bigl(\frac{1}{m}\sum_{b=1}^m(s^i_{n,b})^{-1}(0)\bigr)/\Ga^i$ as a
kind of {\it non-Hausdorff suborbifold of\/ $V^i/\Ga^i$, with
multiplicity in\/}~$\Q$.

With this convention, we can glue the $\bigl(\frac{1}{m}\sum_{b=1}^m
(s^i_{n,b})^{-1}(0)\bigr)/\Ga^i$ for all $i\in I$ using the
$\phi^{ij}$ to get a kind of compact, oriented, non-Hausdorff
orbifold $\ti X_n$ with multiplicity in $\Q$, with a smooth map $\ti
g_n:\ti X_n\ra Y$. Fukaya and Ono then triangulate $\ti X_n$ into
$k$-simplices $f_a:\De_k\ra\ti X_n$, such that on the interior
$f_a(\De_k^\ci)$ of each simplex $\ti X_n$ is Hausdorff and the
multiplicity is a constant $c_a\in\Q$. The {\it virtual chain\/} or
{\it cycle\/} $VC(X,\bs g)$ is then defined to be $\sum_{a\in A}
(\ep_ac_a)(\ti g_n\ci f_a)$ in $C^\rsi_k(Y;\Q)$, using the notation
of~\S\ref{il26}.

{\it Perturbation data\/} is the set of choices needed to construct a
virtual chain.

\begin{dfn} Let $X$ be a compact Kuranishi space with a tangent
bundle and an orientation, $Y$ an orbifold, and $\bs g:X\ra Y$ a
strongly smooth map. A set of {\it perturbation data\/} $\s_X$ for
$(X,\bs g)$ consists of a good coordinate system $\bs I=
\bigl(I,<\nobreak,(V^i,\ldots,\psi^i):i\in I\bigr)$ for $X$, and
smooth maps $g^i:V^i\ra Y$ representing $\bs g$ for $i\in I$, with
$g^i\ci\phi^{ij}\equiv g^j\vert_{V^{ij}}$ when $j<i$ in $I$ and
$\Im\psi^i\cap\Im\psi^j\ne\emptyset$, and smooth, transverse
multisections $\s^i$ on $(V^i,E^i,\Ga^i,s^i,\psi^i)$ for $i\in I$
which are compatible on overlaps $V^{ij}$ and near $\phi^{ij}(V^{ij})$
as above, and such that each $\s^i$ is sufficiently close to $s^i$ in
$C^0$ that the construction of virtual chains above works; in
particular, gluing the $(\s^i)^{-1}(0)/\Ga^i$ for all $i\in I$
together as above should yield a {\it compact\/} oriented
non-Hausdorff manifold $\ti X$ with boundary and corners.

The last item in a set of perturbation data is a choice of
triangulation of $\ti X$ into $k$-simplices $f_a:\De_k\ra\ti X$ for
$a\in A$, where $k=\vdim X$ and $A$ is a finite indexing set, such
that on the interior $f_a(\De_k^\ci)$ of each simplex $\ti X$ is
Hausdorff and the multiplicity is a constant $c_a\in\Q$. We shall
often use $\s_X$, or similar notation, to denote this collection of
data. The {\it virtual chain\/} or {\it cycle\/} $VC(X,\bs g,\s_X)$
constructed using this data $\s_X$ is then defined to be $VC(X,\bs g,
\s_X)=\sum_{a\in A}(\ep_ac_a)(\ti g\ci f_a)$ in $C^\rsi_k(Y;\Q)$,
where $\ep_a$ is 1 if $f_a$ is orientation-preserving, and $-1$ if
$f_a$ is orientation-reversing.
\label{il2dfn11}
\end{dfn}

\begin{rem}{\bf(a)} Perturbation data does not involve a series
$(\s_n^i)_{n=1}^\iy$ for each $(V^i,\ldots,\psi^i)$, but only a
single choice $\s^i$, which we think of as $\s_n^i$ for some fixed
$n\gg 0$. Because of this, we have to require the $\s^i$ to be
`sufficiently close to $s^i$ in $C^0$'. This is rather vague and
unsatisfactory, and will cause problems later; the reason why we
have to introduce $A_{N,0}$ algebras, rather than going straight to
$A_\iy$ algebras, is roughly speaking that we can make only finitely
many choices of $\s^i$ at once and still have these `sufficiently
close' conditions satisfied. This is very inconvenient, but is
central to the approach of Fukaya et al.\ \cite{FOOO}. The authors
give a different approach, avoiding this problem completely,
in~\cite{AkJo}.

\noindent{\bf(b)} When we choose perturbation data $\s_X$ for
$(X,\bs g)$, we usually need $VC(X,\bs g,\s_X)$ to lie in some chain
complex $\Q\X$, as in \S\ref{il26}. That is, we need $\ti g\ci
f_a:\De_k\ra Y$ to lie in $\X$ for all $a\in A$. When this happens
we will say that `the simplices of $VC(X,\bs g,\s_X)$ lie in $\X$'.
Actually, we first choose more-or-less arbitrary perturbations
$\s_X$, and then enlarge $\X$ so that it contains the simplices of
$VC(X,\bs g, \s_X)$. We never try to choose $\s_X$ so that the
simplices of $VC(X,\bs g,\s_X)$ lie in a fixed complex $\X$, as this
would probably be impossible.

\noindent{\bf(c)} Given perturbation data $\s_X$ for $(X,\bs g)$, we
can restrict it to perturbation data $\s_X\vert_{\pd X}$ for $(\pd
X,\bs g\vert_{\pd X})$ in a natural way, and then the virtual chains
satisfy $\pd VC(X,\bs g,\s_X)=VC(\pd X,\bs g\vert_{\pd X},
\s_X\vert_{\pd X})$. Conversely, given perturbation data $\s_{\pd X}$
for $(\pd X,\bs g\vert_{\pd X})$, we often want to choose perturbation
data $\s_X$ for $(X,\bs g)$ with $\s_X\vert_{\pd X}=
\s_{\pd X}$, or at least, we want $\s_X\vert_{\pd X}$ and $\s_{\pd X}$
to be equivalent in some sense that implies that $VC(\pd X,
\bs g\vert_{\pd X},\s_X\vert_{\pd X})=VC(\pd X,\bs g\vert_{\pd X},
\s_{\pd X})$. But there is a problem here, that referred to in (a)
above, as the condition $\s_X\vert_{\pd X}=\s_{\pd X}$ may not be
compatible with the condition that the $\s^i$ in $\s_X$ are
`sufficiently close to $s^i$ in~$C^0$'.

\noindent{\bf(d)} The second author \cite{Joyc2,Joyc3} has developed
{\it Kuranishi homology\/} $KH_*(Y;R)$ and {\it Kuranishi
cohomology\/} $KH^*(Y;R)$ for $Y$ an orbifold and $R$ a
$\Q$-algebra, in which the chains are triples $[X,\bs f,\bs G]$ for
$X$ a compact oriented Kuranishi space, $\bs f:X\ra Y$ a strongly
smooth map or strong submersion, and $\bs G$ some extra
`gauge-fixing data' or `co-gauge-fixing data'. Kuranishi homology
$KH_*(Y;R)$ is isomorphic to singular homology $H_*^\rsi(Y;R)$, and
Kuranishi cohomology $KH^*(Y;R)$ is isomorphic to
compactly-supported cohomology $H^*_{\rm cs}(Y;R)$. Working with
Kuranishi cohomology instead of currents or singular chains gives a
far cleaner approach to virtual chains. In \cite{AkJo} the authors
will rewrite much of \cite{FOOO} using Kuranishi cohomology, which
results in a drastic shortening and technical simplification,
eliminating all the $A_{N,K}$-algebras we will meet below. Parts of
this paper will also be rewritten using Kuranishi cohomology
in~\cite{AkJo}.
\label{il2rem}
\end{rem}

\section{Introduction to $A_\iy$ algebras and $A_{N,K}$ algebras}
\label{il3}

$A_\iy$ algebras were introduced by Stasheff \cite{Stas1,Stas2}. The
following treatment is based on Fukaya et al.\ \cite{FOOO}, and uses
their conventions. Two survey papers by Keller \cite{Kell1,Kell2}
are useful introductions; note that \cite{Kell2} uses the
conventions of \cite{FOOO}, as we do, but \cite{Kell1} has different
conventions on signs and grading. We restrict to $A_\iy$ algebras
over $\Q$, but one can also work over any commutative ring~$R$.

\subsection{(Weak) $A_\iy$ algebras and morphisms}
\label{il31}

Following \cite[\S 7.1]{FOOO}, we define

\begin{dfn} A {\it weak\/ $A_\iy$ algebra} $(A,\m)$ (over $\Q$)
consists of:
\begin{itemize}
\setlength{\itemsep}{0pt}
\setlength{\parsep}{0pt}
\item[(a)] A $\Z$-graded $\Q$-vector space $A=\bigop_{d\in\Z}A^d$; and
\item[(b)] Graded $\Q$-multilinear maps $\m_k:{\buildrel
{\ulcorner\,\,\,\text{$k$ copies } \,\,\,\urcorner} \over
{\vphantom{m}\smash{A\t\cdots\t A}}}\ra A$ for $k=0,1,2,\ldots$, of
degree $+1$. That is, $\m_k$ maps $A^{d_1}\t\cdots\t A^{d_k}\ra
A^{d_1+\cdots+d_k+1}$ for all $d_1,\ldots,d_k\in\Z$. When $k=0$ we
take $\m_0\in A^1$. Write $\m=(\m_k)_{k\ge 0}$.
\end{itemize}
These must satisfy the following condition. Call $a\in A$ {\it pure}
if $a\in A^d\sm\{0\}$ for some $d\in\Z$, and then define the {\it
degree} of $a$ to be $\deg a=d$. Then we require that for all $k\ge
0$ and all pure $a_1,\ldots,a_k$ in $A$ we have
\begin{equation}
\begin{gathered}
\sum_{\begin{subarray}{l}i,k_1,k_2:1\le i\le k_1,\\
k_2\ge 0,\; k_1+k_2=k+1\end{subarray}}
\begin{aligned}[t]
(-1)^{\sum_{l=1}^{i-1}\deg a_l}\m_{k_1}(a_1,\ldots,a_{i-1},\m_{k_2}
(a_i,\ldots,a_{i+k_2-1}), \\
a_{i+k_2}\ldots,a_k)=0.
\end{aligned}
\end{gathered}
\label{il3eq1}
\end{equation}
We call $(A,\m)$ an {\it $A_\iy$ algebra} if it is a weak $A_\iy$
algebra and~$\m_0=0$.
\label{il3dfn1}
\end{dfn}

If $(A,\m)$ is an $A_\iy$ algebra, so that $\m_0=0$, then
\eq{il3eq1} for $k=1$ becomes $\m_1\ci\m_1(a_1)=0$. Thus $\m_1:A\ra
A$ is a graded linear map of degree $+1$ with $\m_1\ci\m_1=0$, so
$(A,\m_1)$ is a complex, and we can form its cohomology $H^*(A)$ by
\begin{equation*}
H^p(A)=\frac{\Ker\m_1:A^p\ra A^{p+1}}{\Im\m_1:A^{p-1}\ra A^p}\,.
\end{equation*}
Then $\m_k$ for $k>1$ induce various operations on $H^*(A)$. For
example, \eq{il3eq1} when $k=2$ yields $\m_2(\m_1(a_1),a_2)+
(-1)^{\deg a_1}\m_2(a_1,\m_1(a_2))+\m_1(\m_2(a_1,a_2))=0$. This
implies that the bilinear product $\bu:H^p(A)\t H^q(A)\ra
H^{p+q+1}(A)$ given by
\begin{equation*}
(a_1+\Im\m_1)\bu(a_2+\Im\m_1)=(-1)^{(\deg a_1+1)\deg
a_2}\m_2(a_1,a_2)+\Im\m_1
\end{equation*}
is well-defined. Then \eq{il3eq1} when $k=3$ implies that $\bu$ is
{\it associative}.

If $(A,\m)$ is only a weak $A_\iy$ algebra, with $\m_0\ne 0$, then
\eq{il3eq1} for $k=1$ yields
\begin{equation*}
\m_1\ci\m_1(a_1)=-\m_2(\m_0,a_1)-(-1)^{\deg a_1}\m_2(a_1,\m_0).
\end{equation*}
So we may no longer have $\m_1\ci\m_1=0$, and we cannot form the
cohomology $H^*(A)$. We regard $\m_0$ as the {\it obstruction} to
$(A,\m_1)$ being a complex.

Equation \eq{il3eq1} can be expressed more naturally using the {\it
bar complex} of~$(A,\m)$.

\begin{dfn} Let $(A,\m)$ be a weak $A_\iy$ algebra. The {\it tensor
coalgebra} $T(A)$ of $A$ is $T(A)=\bigop_{n\ge 0}A^{\ot^n}$, where
we write $A^{\ot^0}=\Q$. It is graded in the obvious way, so that
$T(A)^d=\bigop_{d_1+\cdots+d_n=d}A^{d_1}\ot\cdots\ot A^{d_n}$. It
has a {\it coproduct\/} $\De:T(A)\ra T(A)\ot T(A)$ given by
\begin{equation*}
\De(a_1\ot\cdots\ot a_n)=\ts\sum_{k=0}^n (a_1\ot\cdots\ot a_k)\ot
(a_{k+1}\ot\cdots\ot a_n),
\end{equation*}
taking the $k=0$ and $k=n$ terms to be $1\ot(a_1\ot\cdots\ot a_n)$
and $(a_1\ot\cdots\ot a_n)\ot 1$ respectively. Define a linear map
$\bar\m_k:T(A)\ra T(A)$ for $k\ge 0$ by
\begin{equation*}
\bar\m_k(a_1\ot\cdots\ot a_n)=\sum_{l=1}^{n-k+1}\begin{aligned}[t]
&(-1)^{\deg a_1+\cdots+\deg a_{l-1}} a_1\ot\cdots\ot a_{l-1}\ot\\
&\quad\m_k(a_l,\ldots,a_{l+k-1})\ot a_{l+k}\ot\cdots\ot a_n,
\end{aligned}
\end{equation*}
for all $n\ge 0$ and pure $a_1,\ldots,a_n$ in $A$. In the case
$k=n=0$ we set $\bar\m_0(\la)=\la\m_0\in A^1$ for $\la\in\Q$. Define
$\bar\d=\sum_{k=0}^\iy\bar\m_k$. Then $\bar\d:T(A)\ra T(A)$ is a
graded linear map of degree $+1$, and equation \eq{il3eq1} is
equivalent to $\bar\d\ci\bar\d=0$, so that $\bigl(T(A),\bar\d\bigr)$
is a complex, the {\it bar complex} of $(A,\m)$. Note too that
$\bar\m_k$ and $\bar\d$ are derivations for the coproduct $\De$, so
that $\bigl(T(A),\De,\bar\d\bigr)$ is a {\it differential graded
coalgebra}.
\label{il3dfn2}
\end{dfn}

Here is the notion of morphism of $A_\iy$ algebras.

\begin{dfn} Let $(A,\m)$ and $(B,\n)$ be $A_\iy$ algebras. An
$A_\iy$ {\it morphism} $\f:(A,\m)\ra(B,\n)$ is $\f=(\f_k)_{k\ge 1}$,
where $\f_k:{\buildrel {\ulcorner\,\,\,\text{$k$ copies }
\,\,\,\urcorner} \over {\vphantom{m}\smash{A\t\cdots\t A}}}\ra B$
for $k=1,2,\ldots$ are graded $\Q$-multilinear maps of degree 0,
satisfying
\begin{equation}
\begin{gathered}
\sum_{1\le i<j\le k}
\begin{aligned}[t](-1)^{\sum_{l=1}^{i-1}\deg a_l}
\f_{k-j+i+1}\bigl(a_1,\ldots,a_{i-1},&\\
\m_{j-i}(a_i,\ldots,a_{j-1}),a_j,\ldots,a_k\bigr)&
\end{aligned}\\
=\sum_{0< k_1<k_2<\cdots<k_l=k}
\begin{aligned}[t]
\n_l\bigl(\f_{k_1}(a_1,\ldots,a_{k_1}),
\f_{k_2-k_1}(a_{k_1+1},\ldots,a_{k_2}),&\\
\ldots,\f_{k_l-k_{l-1}}(a_{k_{l-1}+1},\ldots,a_{k_l})\bigr),&
\end{aligned}
\end{gathered}
\label{il3eq2}
\end{equation}
for all $k\ge 0$ and pure $a_1,\ldots,a_k$ in $A$. We can rewrite
\eq{il3eq2} in terms of the bar complexes of $(A,\m)$ and $(B,\n)$:
define $\bar\f:T(A)\ra T(B)$ by
\begin{equation}
\begin{gathered}
\bar\f(a_1\ot\cdots\ot a_n)=\!\!\!\!\sum_{0<
k_1<k_2<\cdots<k_l=n}\!\!\!\!\!\!\!
\begin{aligned}[t]\f_{k_1}(a_1,\ldots,a_{k_1})\ot
\f_{k_2-k_1}(a_{k_1+1},\ldots,a_{k_2})\ot&\\
\cdots\ot \f_{k_l-k_{l-1}}(a_{k_{l-1}+1},\ldots,a_{k_l})&,
\end{aligned}
\end{gathered}
\label{il3eq3}
\end{equation}
for all $n>0$ and pure $a_1,\ldots,a_n$ in $A$. Then \eq{il3eq2} is
equivalent to $\bar\d_B\ci\bar\f=\bar\f\ci\bar\d_A:T(A)\ra T(B)$,
that is, $\bar\f$ is a morphism of bar complexes $\bigl(T(A),\bar
\d_A\bigr)\ra\bigl(T(B),\bar\d_B\bigr)$. It also intertwines the
coproducts $\De_A,\De_B$ on~$T(A),T(B)$.

We call an $A_\iy$ morphism $\f:(A,\m)\ra(B,\n)$ {\it strict\/} if
$\f_k=0$ for $k\ne 1$, an $A_\iy$ {\it isomorphism} if $\f_1:A\ra B$
is an isomorphism of vector spaces, and a {\it strict\/ $A_\iy$
isomorphism} if it is both strict and an $A_\iy$ isomorphism. When
$n=1$, equation \eq{il3eq2} becomes $\f_1\ci\m_1=\n_1\ci\f_1:A\ra
B$. Thus $\f_1$ is a morphism of complexes $(A,\m_1)\ra(B,\n_1)$,
and induces a morphism of cohomology groups $(\f_1)_*:H^*(A)\ra
H^*(B)$. We call $\f$ a {\it weak homotopy equivalence}, or {\it
quasi-isomorphism}, if $(\f_1)_*$ is an isomorphism.

If $(A,\m),(B,\n),(C,{\mathfrak o})$ are $A_\iy$ algebras and
$\f:(A,\m)\ra(B,\n)$, $\g:(B,\n)\ra(C,{\mathfrak o})$ are $A_\iy$
morphisms, the {\it composition} $\g\ci\f:(A,\m) \ra(C,{\mathfrak
o})$ is given by
\begin{equation}
\begin{gathered}
(\g\ci\f)_n(a_1,\ldots,a_n)=\!\!\!\!\sum_{0<
k_1<k_2<\cdots<k_l=n}\!\!\!\!\!\!\!\!\!\!\!\!
\begin{aligned}[t]\g_l\bigl(\f_{k_1}(a_1,\ldots,a_{k_1}),
\f_{k_2-k_1}(a_{k_1+1},\ldots,a_{k_2}),&\\
\ldots,\f_{k_l-k_{l-1}}(a_{k_{l-1}+1},\ldots,a_{k_l})\bigr)&.
\end{aligned}
\end{gathered}
\label{il3eq4}
\end{equation}
On bar complexes this implies that $\overline{(\g\ci\f)}=\bar\g\ci
\bar\f$. Composition is associative.
\label{il3dfn3}
\end{dfn}

This definition of $A_\iy$ morphism also makes sense for {\it
weak\/} $A_\iy$ algebras, allowing $n\ge 0$ and $i\le j$ in
\eq{il3eq2}. In the weak case it would look more natural to take
$\f=(\f_k)_{k\ge 0}$, and include $\f_0$ terms in \eq{il3eq2} and
\eq{il3eq3}. However, both \eq{il3eq2} and \eq{il3eq3} would then
become {\it infinite\/} sums, for instance, \eq{il3eq3} when $n=0$
would be $\f_1(\m_0)=\sum_{l\ge 0}\n_l(\f_0,\ldots,\f_0)$. So we
would need an appropriate notion of convergence of series in $A,B$.
But the definition of weak homotopy equivalence does {\it not\/}
make sense for weak $A_\iy$ algebras, since $H^*(A),H^*(B)$ are not
defined.

\subsection{Homotopy between $A_\iy$ morphisms and algebras}
\label{il32}

Now let $(A,\m),(B,\n)$ be $A_\iy$ algebras, and
$\f,\g:(A,\m)\ra(B,\n)$ be $A_\iy$ morphisms. We will define the
notion of {\it homotopy} $\H$ from $\f$ to $\g$. Our definition is
based on Keller \cite[\S 3.7]{Kell1}. Fukaya et al.\ \cite[\S
15.1--\S 15.2]{FOOO} use a different, more complicated definition,
involving `models of $[0,1]\t B$', but they show in
\cite[Prop.~15.40]{FOOO} that the two definitions yield the same
notion of whether $\f,\g$ are homotopic.

\begin{dfn} Let $(A,\m),(B,\n)$ be $A_\iy$ algebras, and
$\f,\g:(A,\m)\ra(B,\n)$ be $A_\iy$ morphisms. A {\it homotopy} from
$\f$ to $\g$ is $\H=(\H_k)_{k\ge 1}$, where $\H_k:{\buildrel
{\ulcorner\,\,\,\text{$k$ copies } \,\,\,\urcorner} \over
{\vphantom{m}\smash{A\t\cdots\t A}}}\ra B$ for $k=1,2,\ldots$ are
graded $\Q$-multilinear maps of degree $-1$, satisfying
\begin{gather}
\f_n(a_1,\ldots,a_n)-\g_n(a_1,\ldots,a_n)=
\nonumber\\
\sum_{\substack{0< j_1<j_2<\cdots<j_l<\\
k_1<k_2<\cdots<k_m=n}}
\begin{aligned}[t]
&\n_{l+m+1}\bigl(\f_{j_1}(a_1,\ldots,a_{j_1}),
\f_{j_2-j_1}(a_{j_1+1},\ldots,a_{j_2}),\ldots,\\
&\f_{j_l-j_{l-1}}\!(a_{j_{l-1}+1},\ldots,a_{j_l}),
\H_{k_1-j_l}(a_{j_l+1},\ldots,a_{k_1}), \\
&\g_{k_2-k_1}(a_{k_1+1},\ldots,a_{k_2}),\ldots,
\g_{k_m-k_{m-1}}(a_{k_{m-1}+1},\ldots,a_{k_m})\bigr)
\end{aligned}
\label{il3eq5}
\\
+\!\!\!\!\sum_{0\le i<j\le n}\!\!\!\!\!\! (-1)^{\sum_{l=1}^i\deg
a_l} \H_{n-j+i+1}\bigl(a_1,\ldots,a_i,
\m_{j-i}(a_{i+1},\ldots,a_j),a_{j+1},\ldots,a_n\bigr), \nonumber
\end{gather}
for all $n\ge 0$ and pure $a_1,\ldots,a_n$ in $A$. We can rewrite
\eq{il3eq5} in terms of the bar complexes of $(A,\m)$ and $(B,\n)$:
define $\bar\H:T(A)\ra T(B)$ by
\begin{gather*}
\bar\H(a_1\!\ot\!\cdots\!\ot\! a_n)=\!\!\!\!\!\!\!\!\!\!\!\!\!
\sum_{\substack{0< j_1<j_2<\cdots<j_l<\\
k_1<k_2<\cdots<k_m=n}}
\begin{aligned}[t]&\f_{j_1}(a_1,\ldots,a_{j_1})\ot
\f_{j_2-j_1}(a_{j_1+1},\ldots,a_{j_2})\ot\cdots\ot\\
&\f_{j_l-j_{l-1}}\!(a_{j_{l-1}+1},\ldots,a_{j_l})\!\ot\!
\H_{k_1-j_l}(a_{j_l+1},\ldots,a_{k_1})\ot \\
&\!\!\!\!\!\!\!\!\!\!\!\!\!\!\!\!\!\!\!\!\!\!\!\!\!\!\!\!\!\!
\g_{k_2-k_1}(a_{k_1+1},\ldots,a_{k_2})\ot\cdots\ot
\g_{k_m-k_{m-1}}(a_{k_{m-1}+1},\ldots,a_{k_m}),
\end{aligned}
\end{gather*}
for all $n\ge 0$ and pure $a_1,\ldots,a_n$ in $A$. Then $\bar\H$
satisfies $\De_B\ci\bar\H=(\bar\f\ot\bar\H+\bar\H\ot\bar\g)
\ci\De_A$, and \eq{il3eq5} is equivalent to
$\bar\f-\bar\g=\bar\d_B\ci\bar\H+\bar\H\ci\bar\d_A$, so that
$\bar\f$ and $\bar\g$ are homotopic as morphisms of chain complexes
in the usual sense.

$A_\iy$ algebras form a 2-{\it category}, with $A_\iy$ morphisms as
1-morphisms, and homotopies as 2-morphisms. We will sometimes write
a homotopy $\H$ from $\f$ to $\g$ as $\H:\f\Ra\g$, using 2-category
notation. There are various notions of composition between
homotopies and $A_\iy$-morphisms: given $\f,\g,\h:(A,\m)\ra(B,\n)$
and $\H:\f\Ra\g$, ${\mathfrak I}:\g\Ra\h$, we can define ${\mathfrak
I}\ci\H:\f\Ra\h$. Given $\f,\g:(A,\m)\ra(B,\n)$,
$\h:(B,\n)\ra(C,{\mathfrak o})$ and $\H:\f\Ra\g$, we can define
$\h\ci\H:(\h\ci\f)\Ra(\h\ci\g)$. Given $\f:(A,\m)\ra(B,\n)$,
$\g,\h:(B,\n)\ra(C,{\mathfrak o})$ and ${\mathfrak I}:\g\Ra\h$, we
can define ${\mathfrak I}\ci\f:(\g\ci\f)\Ra(\h\ci\f)$. The
definitions, as compositions of maps $\f_k,\g_k,\h_k,\H_k,{\mathfrak
I}_k$, are straightforward. They satisfy the usual 2-category
associativity properties.
\label{il3dfn4}
\end{dfn}

\begin{dfn} Let $(A,\m),(B,\n)$ be $A_\iy$ algebras, and
$\f:(A,\m)\ra(B,\n)$ an $A_\iy$ morphism. A {\it homotopy inverse}
for $\f$ is an $A_\iy$ morphism $\g:(B,\n)\ra(A,\m)$ such that
$\g\ci\f:(A,\m)\ra(A,\m)$ is homotopic to $\id_A:(A,\m)\ra(A,\m)$,
and $\f\ci\g:(B,\n)\ra(B,\n)$ is homotopic to
$\id_B:(B,\n)\ra(B,\n)$. If $\f$ has a homotopy inverse, we call
$\f$ a {\it homotopy equivalence}, and we call $(A,\m),(B,\n)$ {\it
homotopic}.
\label{il3dfn5}
\end{dfn}

The following important theorem is proved by Fukaya et al.\
\cite[Cor.~15.44, Th.~15.45(1)]{FOOO}; see also Keller \cite[\S
3.7]{Kell1}, who cites the thesis of Prout\'e (Paris, 1984).

\begin{thm} Let\/ $(A,\m),(B,\n)$ be $A_\iy$ algebras. Then
\begin{itemize}
\setlength{\itemsep}{0pt}
\setlength{\parsep}{0pt}
\item[{\rm(a)}] Homotopy is an equivalence relation
on $A_\iy$ morphisms $\f:(A,\m)\ra(B,\n)$.
\item[{\rm(b)}] Homotopy is an equivalence relation
on $A_\iy$ algebras.
\item[{\rm(c)}] An $A_\iy$ morphism $\f:(A,\m)\ra(B,\n)$ is a
homotopy equivalence if and only if it is a weak homotopy
equivalence.
\end{itemize}
\label{il3thm1}
\end{thm}

In practice, homotopy is a more useful notion of when two $A_\iy$
algebras are `the same' than either $A_\iy$ isomorphism or strict
$A_\iy$ isomorphism. We are interested in properties of $A_\iy$
algebras which are invariant under homotopy. Constructions of
$A_\iy$ algebras generally depend on some arbitrary choices (such as
the almost complex structure $J$ below), and different choices yield
homotopic but not (strictly) isomorphic $A_\iy$ algebras.

\subsection{Minimal models, and sums over planar trees}
\label{il33}

An $A_\iy$ algebra $(B,\n)$ is called {\it minimal\/} if $\n_1=0$,
so that $H^*(B)=B$. If $(A,\m)$ is an $A_\iy$ algebra, then one can
make $H^*(A)$ into a minimal $A_\iy$ algebra $\bigl(H^*(A),
\n\bigr)$, such that there is an $A_\iy$-morphism $\bs\pi:(A,\m)
\ra\bigl(H^*(A),\n\bigr)$ inducing the identity in cohomology. Thus
$\bigl(H^*(A),\n\bigr)$ is homotopic to $(A,\m)$. We call
$\bigl(H^*(A),\n\bigr)$ a {\it minimal model\/} or {\it canonical
model\/} for $(A,\m)$. It is unique up to $A_\iy$ isomorphism. We
will explain a proof of this using the method of sums over `planar
rooted trees' due to Kontsevich and Soibelman \cite[\S 6.4]{KoSo};
see also Markl \cite{Mark} and Keller~\cite[Th.~2.3]{Kell2}.

\begin{dfn} A {\it planar rooted tree} is a finite, connected,
simply-connected graph $T$ in the plane $\R^2$, whose vertices are
divided into $k+1$ {\it external vertices} numbered $0,1,\ldots,k$,
and at least one {\it internal vertices}. Each external vertex must
be connected to exactly one edge, and the external vertices should
be {\it cyclically ordered}, in the sense that if we embed $T$ into
the unit disc $\{x^2+y^2\le 1\}$ such that $T\cap\{x^2+y^2=1\}$ is
vertices $0,1,\ldots,k$, then the external vertices appear in the
cyclic order $0,1,\ldots,k$ anticlockwise around the circle.

Here when we say $T$ is a {\it graph in the plane}, we mean that $T$
is embedded in $\R^2$ up to continuous deformations. Since $T$ is
simply-connected, such an embedding class of $T$ is equivalent to
prescribing the cyclic order of the edges at each vertex.

We call vertex 0 the {\it root\/} of $T$, and vertices $1,\ldots,k$
the {\it leaves} of $T$. Define a unique {\it orientation\/} on $T$
such that each edge is oriented in the direction of the minimal path
to the root vertex. Then every vertex except the root has exactly
one outgoing edge, and the rest incoming edges. We call an edge the
{\it root edge} if it is connected to the root vertex, a {\it leaf
edge} if it is connected to a leaf vertex, and an {\it internal
edge} if it is connected to no distinguished vertices. (See
Figure~\ref{il3fig1}(a).)
\label{il3dfn6}
\end{dfn}

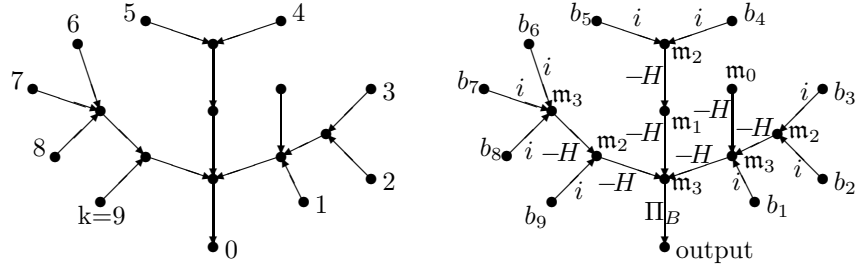
\begin{figure}[htb]
\setlength{\unitlength}{3mm}
\begin{picture}(40,12)(-10,-1)
\put(0,0){\circle*{.5}} \put(0,3){\circle*{.5}}
\put(0,6){\circle*{.5}} \put(0,9){\circle*{.5}}
\put(3,4){\circle*{.5}} \put(-3,4){\circle*{.5}}
\put(3,10){\circle*{.5}} \put(-3,10){\circle*{.5}}
\put(3,7){\circle*{.5}} \put(5,5){\circle*{.5}}
\put(4,2){\circle*{.5}} \put(7,7){\circle*{.5}}
\put(7,3){\circle*{.5}} \put(-5,6){\circle*{.5}}
\put(-5,2){\circle*{.5}} \put(-6,9){\circle*{.5}}
\put(-8,7){\circle*{.5}} \put(-7,4){\circle*{.5}}
\put(0,3){\vector(0,-1){3}} \put(0,6){\vector(0,-1){3}}
\put(0,9){\vector(0,-1){3}} \put(3,4){\vector(-3,-1){3}}
\put(-3,4){\vector(3,-1){3}} \put(3,10){\vector(-3,-1){3}}
\put(-3,10){\vector(3,-1){3}} \put(3,7){\vector(0,-1){3}}
\put(5,5){\vector(-2,-1){2}} \put(4,2){\vector(-1,2){1}}
\put(7,7){\vector(-1,-1){2}} \put(7,3){\vector(-1,1){2}}
\put(-5,6){\vector(1,-1){2}} \put(-5,2){\vector(1,1){2}}
\put(-6,9){\vector(1,-3){1}} \put(-8,7){\vector(3,-1){3}}
\put(-7,4){\vector(1,1){2}} \put(.5,-.5){0} \put(4.5,1.5){1}
\put(7.5,2.5){2} \put(7.5,6.5){3} \put(3.5,10){4} \put(-4,10){5}
\put(-6.5,9.5){6} \put(-9,7){7} \put(-8,4){8} \put(-6,1){k=9}
\put(20,0){\circle*{.5}} \put(20,3){\circle*{.5}}
\put(20,6){\circle*{.5}} \put(20,9){\circle*{.5}}
\put(23,4){\circle*{.5}} \put(17,4){\circle*{.5}}
\put(23,10){\circle*{.5}} \put(17,10){\circle*{.5}}
\put(23,7){\circle*{.5}} \put(25,5){\circle*{.5}}
\put(24,2){\circle*{.5}} \put(27,7){\circle*{.5}}
\put(27,3){\circle*{.5}} \put(15,6){\circle*{.5}}
\put(15,2){\circle*{.5}} \put(14,9){\circle*{.5}}
\put(12,7){\circle*{.5}} \put(13,4){\circle*{.5}}
\put(20,3){\vector(0,-1){3}} \put(20,6){\vector(0,-1){3}}
\put(20,9){\vector(0,-1){3}} \put(23,4){\vector(-3,-1){3}}
\put(17,4){\vector(3,-1){3}} \put(23,10){\vector(-3,-1){3}}
\put(17,10){\vector(3,-1){3}} \put(23,7){\vector(0,-1){3}}
\put(25,5){\vector(-2,-1){2}} \put(24,2){\vector(-1,2){1}}
\put(27,7){\vector(-1,-1){2}} \put(27,3){\vector(-1,1){2}}
\put(15,6){\vector(1,-1){2}} \put(15,2){\vector(1,1){2}}
\put(14,9){\vector(1,-3){1}} \put(12,7){\vector(3,-1){3}}
\put(13,4){\vector(1,1){2}} \put(20.5,-.5){output}
\put(24.5,1.5){$b_1$} \put(27.5,2.5){$b_2$} \put(27.5,6.5){$b_3$}
\put(23.5,10){$b_4$} \put(15.8,10){$b_5$} \put(13.5,9.5){$b_6$}
\put(10.8,7){$b_7$} \put(11.8,4){$b_8$} \put(13.8,1){$b_9$}
\put(20.3,2.5){$\m_3$} \put(20.3,5.2){$\m_1$} \put(20.3,8.3){$\m_2$}
\put(23.5,3.5){$\m_3$} \put(22.7,7.5){$\m_0$} \put(25.5,4.8){$\m_2$}
\put(17,4.4){$\m_2$} \put(15,6.4){$\m_3$} \put(19.1,1.2){$\Pi_B$}
\put(18.2,4.7){$-\!H$} \put(18.2,7.2){$-\!H$} \put(23.1,4.9){$-\!H$}
\put(21.2,5.7){$-\!H$} \put(20.4,3.7){$-\!H$} \put(17,2.45){$-\!H$}
\put(14.5,3.8){$-\!H$} \put(21.3,9.8){$i$} \put(18.6,9.8){$i$}
\put(23,2.5){$i$} \put(25.7,3.1){$i$} \put(26,6.5){$i$}
\put(16,2.1){$i$} \put(13.8,3.9){$i$} \put(14.6,7.6){$i$}
\put(13.4,6.7){$i$}
\end{picture}
\caption{{}\!\!\!(a) A planar rooted tree $T$\quad (b) operators
assigned to it}
\label{il3fig1}
\end{figure}

\begin{dfn} Let $(A,\m)$ be an $A_\iy$ algebra. Then $(A,\m_1)$ is a
complex. Let $B$ be a graded vector subspace of $A$ closed under
$\m_1$, such that the inclusion $i:B\hookra A$ induces an
isomorphism $i_*:H^*(B,\m_1\vert_B)\ra H^*(A,\m_1)$. We will
construct $\n=(\n_k)_{k\ge 1}$ making $(B,\n)$ into an $A_\iy$
algebra homotopic to~$(A,\m)$.

Since $i_*$ is an isomorphism, we can choose a graded vector
subspace $C$ of $A$ such that $C\cap\Ker\m_1=\{0\}$ and $A=B\op
C\op\m_1(C)$. Then $\m_1:C\ra\m_1(C)$ is invertible, so there is a
unique graded linear map $H:A\ra A$ of degree $-1$ with
$H(b)=H(c)=0$ and $H\ci\m_1(c)=c$ for all $b\in B$ and $c\in C$. Let
$\Pi_B:A\ra B$ be the projection, with kernel $C\op\m_1(C)$. Then
$\id_A-\Pi_B=\m_1\ci H+H\ci\m_1$ on~$A$.

For each planar rooted tree $T$ with $k$ leaves, define a graded
multilinear operator $\n_{k,T}:{\buildrel {\ulcorner\,\,\,\text{$k$
copies } \,\,\,\urcorner} \over {\vphantom{m}\smash{B\t\cdots\t
B}}}\ra B$ of degree $+1$, as follows. To define
$\n_{k,T}(b_1,\ldots,b_k)$, assign objects and operators to the
vertices and edges of $T$:
\begin{itemize}
\setlength{\itemsep}{0pt}
\setlength{\parsep}{0pt}
\item assign $b_1,\ldots,b_k$ to the leaf vertices $1,\ldots,k$
respectively.
\item for each internal vertex with 1 outgoing edge and $n$ incoming
edges, assign~$\m_n$.
\item assign $i$ to each leaf edge.
\item assign $\Pi_B$ to the root edge.
\item assign $-H$ to each internal edge.
\end{itemize}
This is illustrated in Figure \ref{il3fig1}(b). Then we define
$\n_{k,T}(b_1,\ldots,b_k)$ to be the composition of all these
objects and morphisms, where we follow the orientations of the
edges, and at each interior vertex with 1 outgoing edge and $n$
incoming edges, we apply $\m_n$ to the $n$ inputs from the $n$
incoming edges in the order counting anticlockwise from the outgoing
edge. In the example of Figure \ref{il3fig1}, this yields
\begin{gather*}
\n_{9,T}(b_1,\ldots,b_9)=\Pi\ci\m_3\bigl(
-H\ci\m_3(i(b_1),-H\ci\m_2(i(b_2),i(b_3)),-H(\m_0)),\\
-H\ci\m_1(-H\ci\m_2(i(b_4),i(b_5))),
-H\ci\m_2(-H\ci\m_3(i(b_6),i(b_7),i(b_8)),i(b_9))\bigr).
\end{gather*}
Note that this includes an $\m_0$ term, and so is zero in the
$A_\iy$ algebra case.

Define $\n_1=\m_1\vert_B$, and for $k\ge 2$ define
$\n_k=\sum_T\n_{k,T}$, where the sum is over all planar rooted trees
$T$ with $k$ leaves, such that {\it every internal vertex has at
least three edges}. (This excludes Figure \ref{il3fig1}. For
filtered $A_\iy$ algebras we will also allow internal vertices with
one or two edges.) This condition implies that $T$ has at most $2k$
vertices and $2k-1$ edges, so there are only finite many such trees
$T$, and $\n_k=\sum_T\n_{k,T}$ is a finite sum.

In a similar way, for each planar rooted tree $T$ with $k$ leaves,
define a graded multilinear operator ${\mathfrak i}_{k,T}:{\buildrel
{\ulcorner\,\,\,\text{$k$ copies } \,\,\,\urcorner} \over
{\vphantom{m}\smash{B\t\cdots\t B}}}\ra A$ of degree 0, as follows.
Assign objects and operators to the vertices and edges of $T$:
\begin{itemize}
\setlength{\itemsep}{0pt}
\setlength{\parsep}{0pt}
\item assign $b_1,\ldots,b_k$ to the leaf vertices $1,\ldots,k$
respectively.
\item for each internal vertex with 1 outgoing edge and $n$ incoming
edges, assign~$\m_n$.
\item assign $i$ to each leaf edge.
\item assign $-H$ to the root edge and to each internal edge.
\end{itemize}
Define ${\mathfrak i}_{k,T}(b_1,\ldots,b_k)$ to be the composition
of all these objects and morphisms. Define ${\mathfrak i}_1:B\ra A$
by ${\mathfrak i}_1=i$, and for $k\ge 2$ define ${\mathfrak
i}_k=\sum_T{\mathfrak i}_{k,T}$, where the sum is over all rooted
planar trees $T$ with $k$ leaves, such that every internal vertex
has at least three edges. Then Markl \cite{Mark} proves:
\label{il3dfn7}
\end{dfn}

\begin{thm} In Definition {\rm\ref{il3dfn7},} $(B,\n)$ is an $A_\iy$
algebra, and\/ ${\mathfrak i}:(B,\n)\ra(A,\m)$ is an $A_\iy$
morphism, and a homotopy equivalence. If we choose $B\cong H^*(A)$
to be a subspace representing $H^*(A),$ so that\/
$\n_1=\m_1\vert_B=0,$ then $(B,\n)$ is a minimal model for~$(A,\m)$.
\label{il3thm2}
\end{thm}

Markl \cite{Mark} also gives much more complicated explicit formulae
for a homotopy inverse ${\mathfrak j}:(A,\m)\ra(B,\n)$ for
$\mathfrak i$ and a homotopy $\H$ from ${\mathfrak i}\ci{\mathfrak
j}$ to $\id_A$. Later we will need a special case of this
construction.

\begin{dfn} Let $(A,\m)$ and $(D,{\mathfrak o})$ be $A_\iy$
algebras, and $\p:(A,\m)\ra(D,{\mathfrak o})$ a strict, surjective
$A_\iy$ morphism which is a weak homotopy equivalence. That is,
$\p_k=0$ for $k\ne 1$, and $\p_1:A\ra D$ is surjective and induces
an isomorphism $(\p_1)_*:H^*(A,\m_1)\ra H^*(D,{\mathfrak o}_1)$. In
Definition \ref{il3dfn7}, choose the subspaces $B,C$ of $A$ such
that $C\op\m_1(C)=\Ker\p_1$, and $\p_1\vert B:B\ra D$ is an
isomorphism. This is possible as $\p_1$ is surjective and $(\p_1)_*$
is an isomorphism.

As $\p$ is a strict $A_\iy$ morphism, we have
$\p_1\ci\m_j={\mathfrak o}_j\ci(\p_1\t\cdots\t\p_1)$ for all
$j=1,2,\ldots$. Since $\Ker\Pi_B=\Ker\p_1$ and $\Im
H\subseteq\Ker\p_1$, this implies that
$\Pi_B\ci\m_j(a_1,\ldots,a_{i-1},-H(a_i),a_{i+1},\ldots,a_j)=0$ for
all $a_1,\ldots,a_j\in A$ and $i=1,\ldots,j$. Applying this to the
root vertex of $T$, we see that $\n_{k,T}=0$ in Definition
\ref{il3dfn7} whenever $T$ has an internal edge. Thus, the only
nonzero $\n_{k,T}$ is the unique $T$ with one internal vertex and
$k$ leaves, and we have $\n_k=\Pi_B\ci\m_k\ci(i\t\cdots\t i)$ for
all $k=1,2,\ldots$. Comparing this with $\p_1\ci\m_k={\mathfrak
o}_k\ci(\p_1\t\cdots\t\p_1)$ and noting that $\p_1\vert B:B\ra D$ is
an isomorphism, we see that $\p_1\vert B:B\ra D$ identifies $\m_k$
and ${\mathfrak o}_k$ for $k=1,2,\ldots$. Hence, $\p_1\vert_B$
induces a {\it strict $A_\iy$ isomorphism}~$(B,\n)\ra(D,{\mathfrak
o})$.

Now define a graded multilinear operator $\q_k:{\buildrel
{\ulcorner\,\,\,\text{$k$ copies } \,\,\,\urcorner} \over
{\vphantom{m}\smash{D\t\cdots\t D}}}\ra A$ of degree 0 by
$\q_k={\mathfrak i}_k\ci\bigl((\p_1\vert B)^{-1}\t\cdots\t(\p_1\vert
B)^{-1}\bigr)$, and write $\q=(\q_k)_{k\ge 1}$. Then Theorem
\ref{il3thm2} implies that $\q:(D,{\mathfrak o})\ra(A,\m)$ is an
$A_\iy$ {\it morphism}, and a {\it homotopy equivalence}. It is easy
to check that $\p\ci\q:(D,{\mathfrak o})\ra(D,{\mathfrak o})$ is the
identity on $(D,{\mathfrak o})$, so $\q$ is a {\it homotopy inverse}
for $\p:(A,\m)\ra(D,{\mathfrak o})$. We have proved:
\label{il3dfn8}
\end{dfn}

\begin{cor} Let\/ $\p:(A,\m)\ra(D,{\mathfrak o})$ be a strict,
surjective $A_\iy$ morphism of\/ $A_\iy$ algebras which is a weak
homotopy equivalence. Then we can construct an explicit homotopy
inverse $\q:(D,{\mathfrak o})\ra(A,\m)$ for $\p$ using sums over
planar trees.
\label{il3cor1}
\end{cor}

\subsection{Novikov rings, and modules over them}
\label{il34}

In defining Lagrangian Floer cohomology, we have to consider sums
involving infinitely many terms, coming from $J$-holomorphic discs
of larger and larger area. To ensure these sums converge, we work
over a ring of formal power series known as a {\it Novikov ring}, as
in Fukaya et al.\ \cite[Def.~6.2]{FOOO}. We consider two kinds, {\it
general Novikov rings} $\La_\nov,\La_\nov^0$ and {\it Calabi--Yau
Novikov rings} $\La_\CY,\La_\CY^0$, to be used in \S\ref{il11},
\S\ref{il13}, and \S\ref{il12}--\S\ref{il13}, respectively.

The reason for having two kinds is this. In $\La_\nov,\La_\nov^0$,
terms $T^\la e^\mu$ keep track of $J$-holomorphic discs in $M$ with
boundary in $L$, area $\la$, and Maslov index $2\mu$. However, if
$M$ is {\it Calabi--Yau} and $L$ is {\it graded\/} then all
$J$-holomorphic curves in $M$ with boundary in $L$ have Maslov index
0, so the $e^\mu$ are unnecessary, and we can use the smaller rings
$\La_\CY,\La_\CY^0$. We restrict to Novikov rings over~$\Q$.

\begin{dfn} Let $T$ and $e$ be formal variables, graded of degree 0
and 2, respectively. Define four {\it universal Novikov rings} (over
$\Q$) by
\begin{align}
\La_\nov&=\bigl\{\ts\sum_{i=0}^\iy a_iT^{\la_i}e^{\mu_i}:
\text{$a_i\in\Q$, $\la_i\in\R$, $\mu_i\in\Z$, $\lim_{i\ra\iy}\la_i
=\iy$}\bigr\},
\label{il3eq6}\\
\La^0_\nov&=\bigl\{\ts\sum_{i=0}^\iy a_iT^{\la_i}e^{\mu_i}:
\text{$a_i\in\Q$, $\la_i\in[0,\iy)$, $\mu_i\in\Z$,
$\lim_{i\ra\iy}\la_i =\iy$}\bigr\},
\label{il3eq7}\\
\La_\CY&=\bigl\{\ts\sum_{i=0}^\iy a_iT^{\la_i}: \text{$a_i\in\Q$,
$\la_i\in\R$, $\lim_{i\ra\iy}\la_i =\iy$}\bigr\},
\label{il3eq8}\\
\La^0_\CY&=\bigl\{\ts\sum_{i=0}^\iy a_iT^{\la_i}: \text{$a_i\in\Q$,
$\la_i\in[0,\iy)$, $\lim_{i\ra\iy}\la_i =\iy$}\bigr\}.
\label{il3eq9}
\end{align}
Then $\La_\nov^0\subset\La_\nov$ and $\La_\CY^0\subset\La_\CY$ are
$\Q$-vector spaces. For brevity we shall write $\La_\nov^*$ to mean
either $\La_\nov^0$ or $\La_\nov$, and $\La_\CY^*$ to mean either
$\La_\CY^0$ or $\La_\CY$. Define multiplications `$\,\cdot\,$' by
$\bigl(\sum_{i=0}^\iy a_iT^{\la_i}e^{\mu_i}\bigr)\cdot
\bigl(\sum_{j=0}^\iy b_jT^{\nu_j}e^{\xi_j}\bigr)=\sum_{i,j=0}^\iy
a_ib_jT^{\la_i+\nu_j}e^{\mu_i+\xi_j}$ on $\La_\nov^*$, and similarly
for $\La^*_\CY$. Here since $\la_i,\nu_j\ra\iy$ as $i,j\ra\iy$, the
sum over $i,j$ can be rewritten as a sum over $k=0,1,\ldots$ such
that $\la_{i_k}+\nu_{j_k}\ra\iy$ as $k\ra\iy$, and so it lies in
$\La^*_\nov$. With these multiplications, $\La_\nov^*,\La_\CY^*$ are
{\it commutative $\Q$-algebras} with identity $1=1T^0e^0$ or~$1T^0$.

The condition that $\lim_{i\ra\iy}\la_i=\iy$ in
\eq{il3eq6}--\eq{il3eq9} is equivalent to saying that for all $C\ge
0$, there are only finitely many $(\la_i,\mu_i)$ or $\la_i$ in the
sums with $\la_i\le C$. We will often write similar conditions this
way. Define {\it filtrations\/} of $\La_\nov^*,\La_\CY^*$ by
\begin{align*}
F^\la\La_\nov^*&=\bigl\{\ts\sum_{i=0}^\iy
a_iT^{\la_i}e^{\mu_i}\in\La_\nov^*:\text{$\la_i\ge\la$ for all
$i=0,1,\ldots$}\bigr\},\\
F^\la\La_\CY^*&=\bigl\{\ts\sum_{i=0}^\iy
a_iT^{\la_i}\in\La_\CY^*:\text{$\la_i\ge\la$ for all
$i=0,1,\ldots$}\bigr\},
\end{align*}
for $\la\in\R$. Then $F^\la\La_\nov^*\subseteq F^\nu\La_\nov^*$ if
$\la\ge\nu$, and $(F^\la\La_\nov^*)\cdot(F^\nu\La_\nov^*)
=F^{\la+\nu}\La_\nov^*$, and $\La_\nov^0=F^0\La_\nov$,
and~$F^\la\La_\nov=T^\la\La_\nov^0$.

These filtrations induce {\it topologies} on $\La_\nov^*,\La_\CY^*$,
and notions of {\it convergence\/} for sequences and series, which
have nothing to do with the topology on $\Q$ or convergence in $\Q$.
An infinite sum $\sum_{k=0}^\iy\al_k$ in $\La_\nov^*$ converges in
$\La_\nov^*$ if and only if for all $\la\in\R$ we have $\al_k\in
F^\la\La_\nov^*$ for all except finitely many~$k=0,1,2,\ldots$.

As $T,e$ are graded of degrees 0,2, we can regard
$\La_\nov,\La_\nov^0$ as {\it graded\/} rings. Write
$\La^{(k)}_\nov,\La_\nov^{0\,(k)}$ for the degree $k$ parts of
$\La_\nov,\La_\nov^0$, for $k\in\Z$. Then
\begin{equation*}
\La_\nov^{(2k)}=\bigl\{\ts\sum_{i=0}^\iy a_iT^{\la_i}e^k:
\text{$a_i\in\Q$, $\la_i\in\R$, $\lim_{i\ra\iy}\la_i=\iy$}\bigr\},
\quad \La^{(2k+1)}_\nov=0,
\end{equation*}
for $k\in\Z$. Note that $\nu\in\La_\nov$ can have nonzero components
$\nu^{(2k)}\in\La_\nov^{(2k)}$ for infinitely many $k\in\Z$, but
$\nu=\sum_{k\in\Z}\nu^{(2k)}$ holds as a convergent sum in
$\La_\nov$. Identifying $e^0=1$ gives $\La_\CY=\La_\nov^{(0)}$
and~$\La^0_\CY=\La_\nov^{0\,(0)}$.

We can also consider {\it modules\/} over $\La_\nov^*$ and
$\La_\CY^*$. In this paper, all modules we consider will be of the
form $V\ot_\Q\La_\nov^*$ and $V\ot_\Q\La_\CY^*$, where
$V=\bigop_{k\in\Z}V^k$ is a finite-dimensional graded $\Q$-vector
space. Then $V\ot\La_\nov^*$ is {\it graded\/} with grading
$(V\ot\La_\nov^*)^l=\bigop_{j+k=l}V^j\ot\La_\nov^{*\,(k)}$, and {\it
filtered\/} with filtration $F^\la(V\ot\La^*_\nov) =V\ot
F^\la\La^*_\nov$ for $\la\in\R$. As $V$ is finite-dimensional, we do
not need to take the {\it completion} $V\hat\ot\La_\nov^*$ with
respect to the filtration of $V\ot\La_\nov^*$, as Fukaya et al.\ do
\cite{FOOO}, since $V\ot\La_\nov^*$ is already complete.
\label{il3dfn9}
\end{dfn}

\subsection{Gapped filtered $A_\iy$ algebras}
\label{il35}

Next we define {\it gapped filtered\/ $A_\iy$ algebras}, following
Fukaya et al.\ \cite[\S 7.2]{FOOO}, and extend the material of
\S\ref{il32}--\S\ref{il33} to them. The rest of the section,
\S\ref{il35}--\S\ref{il37}, can be done either over $\La_\nov^0$ or
$\La_\CY^0$. We shall work over $\La_\nov^0$, as it is more general;
the changes for the $\La_\CY^0$ case are obvious. For instance, in
Definition \ref{il3dfn10}(i) for $\La_\CY^0$ we would take ${\cal
G}\subset[0,\iy)$ closed under addition with $0\in\cal G$ and ${\cal
G}\cap[0,C]$ finite for $C\ge 0$, and write~$\m_k=\sum_{\la\in{\cal
G}}T^\la \m_k^{\smash{\la}}$.

\begin{dfn} A {\it gapped filtered\/ $A_\iy$ algebra}
$(A\ot\La_\nov^0,\m)$ consists of:
\begin{itemize}
\setlength{\itemsep}{0pt}
\setlength{\parsep}{0pt}
\item[(a)] A $\Z$-graded $\Q$-vector space $A=\bigop_{d\in\Z}A^d$,
so that $A\ot\La_\nov^0$ is a graded filtered $\La_\nov^0$-module.
\item[(b)] Graded $\La_\nov^0$-multilinear maps
$\smash{\m_k:{\buildrel {\ulcorner\,\,\,\text{$k$ copies }
\,\,\,\urcorner} \over {\vphantom{m}\smash{(A\ot\La_\nov^0)
\t\cdots\t (A\ot\La_\nov^0)}}}\ra} A\ot\La_\nov^0$ for
$k=0,1,2,\ldots$, of degree $+1$. Write~$\m=(\m_k)_{k\ge 0}$.
\end{itemize}
These must satisfy the following conditions:
\begin{itemize}
\setlength{\itemsep}{0pt}
\setlength{\parsep}{0pt}
\item[(i)] there exists a subset ${\cal G}\subset[0,\iy)\t\Z$, closed
under addition, such that ${\cal G}\cap(\{0\}\t\Z)=\{(0,0)\}$ and
${\cal G}\cap([0,C]\t\Z)$ is finite for any $C\ge 0$, and the maps
$\m_k$ for $k\ge 0$ may be written
$\smash{\m_k=\sum_{(\la,\mu)\in{\cal G}}T^\la
e^\mu\m_k^{\smash{\la,\mu}}}$, for unique $\Q$-multilinear maps
$\m_k^{\smash{\la,\mu}}:{\buildrel {\ulcorner\,\,\,\text{$k$ copies
} \,\,\,\urcorner} \over {\vphantom{m}\smash{A\t\cdots\t A}}}\ra A$
graded of degree $1-2\mu$. When $k=0$, we take
$\m_0\in(A\ot\La_\nov^0)^{(1)}$ and $\m_0^{\smash{\la,\mu}}\in
A^{1-2\mu}$;
\item[(ii)] $\m_0^{\smash{0,0}}=0$, in the notation of (i); and
\item[(iii)] call $a\in A\ot\La_\nov^0$ {\it pure} if $a\in
(A\ot\La_\nov^0)^{(d)}\sm\{0\}$ for some $d\in\Z$, and then define
the {\it degree} of $a$ to be $\deg a=d$. Then we require that for
all $k\ge 0$ and all pure $a_1,\ldots,a_k$ in $A\ot\La_\nov^0$,
equation \eq{il3eq1} holds.
\end{itemize}
There is a unique smallest choice of subset $\cal G$ satisfying (i).
Part (iii) may be rewritten in terms of the $\m_k^{\smash{\la,\mu}}$
as follows: for all $k\ge 0$, all $(\la,\mu)\in{\cal G}$ and all
pure $a_1,\ldots,a_k$ in $A$, we have
\begin{equation}
\begin{gathered}
\sum_{\begin{subarray}{l}i,k_1,k_2,\la_1,\la_2,\mu_1,\mu_2:\;
1\le i\le k_1, k_2\ge 0,\\
 k_1+k_2=k+1,\;\la_1+\la_2=\la,\; \mu_1+\mu_2=\mu\end{subarray}
\!\!\!\!\!\!\!\!\!\!\!\!\!\!\!\!\!\!\!\!\!\!\!\!\!\! }
\!\!\!\!\!\!\!\!\!\!\!\!\!\!\!\!\!\!\!
\begin{aligned}[t]
(-1)^{\sum_{l=1}^{i-1}\deg
a_l}\m_{k_1}^{\la_1,\mu_1}(a_1,\ldots,a_{i-1},\m_{k_2}^{\la_2,\mu_2}
(a_i,\ldots,a_{i+k_2-1}), \\
a_{i+k_2}\ldots,a_k)=0.
\end{aligned}
\end{gathered}
\label{il3eq10}
\end{equation}
Note that a gapped filtered $A_\iy$ algebra $(A\ot\La_\nov^0,\m)$ is
a {\it weak $A_\iy$ algebra\/} in the sense of Definition
\ref{il3dfn1}, with extra structure. Also, if $(\la,\mu)=(0,0)$ then
as ${\cal G}\cap(\{0\}\t\Z)=\{(0,0)\}$, equation \eq{il3eq10}
reduces to
\begin{equation*}
\sum_{\begin{subarray}{l}i,k_1,k_2:1\le i\le k_1,\\
k_2\ge 0,\; k_1+k_2=k+1\end{subarray}}\!\!\!\!
\begin{aligned}[t]
(-1)^{\sum_{l=1}^{i-1}\deg
a_l}\m_{k_1}^{\smash{0,0}}(a_1,\ldots,a_{i-1},\m_{k_2}^{\smash{0,0}}
(a_i,\ldots,a_{i+k_2-1}), \\
a_{i+k_2}\ldots,a_k)=0,
\end{aligned}
\end{equation*}
for all $k\ge 0$ and all pure $a_1,\ldots,a_k$ in $A$. Thus, if
$(A\ot\La_\nov^0,\m)$ is a gapped filtered $A_\iy$ algebra, then
$(A,\m^{\smash{0,0}})$ is an $A_\iy$ {\it algebra}, where
$\m^{\smash{0,0}}=(\m_k^{\smash{0,0}})_{k\ge 0}$. In particular,
$(A,\m_1^{\smash{0,0}})$ is a complex, and we can form its
cohomology $H^*(A,\m_1^{\smash{0,0}})$. Generalizing \S\ref{il33},
we call a gapped filtered $A_\iy$ algebra $(A\ot\La_\nov^0,\m)$ {\it
minimal\/} if~$\m_1^{\smash{0,0}}=0$.

If $A$ is infinite-dimensional then in general we should replace
$A\ot\La_\nov^0$ by the {\it completion\/} $A\hat\ot\La_\nov^0$ of
$A\ot\La_\nov^0$ with respect to the filtration
$F^\la(A\ot\La_\nov^0)$, $\la\ge 0$, as in Fukaya et al.\
\cite{FOOO}. If we do not, then infinite sums in $A\ot\La_\nov^0$
such as $\m_k(a_1,\ldots,a_k)=\sum_{(\la,\mu)\in {\cal G}}T^\la
e^\mu\m_k^{\smash{\la,\mu}}(a_1,\ldots,a_k)$ from (i) need not
converge in $A\ot\La_\nov^0$. But as we work only with
finite-dimensional $A$, for which $A\ot\La_\nov^0$ is already
complete, we shall ignore this point.

A gapped filtered $A_\iy$ algebra $(A\ot\La_\nov^0,\m)$ is called
{\it strict\/} if $\m_0=0$. Then (iii) implies that $\m_1\ci\m_1=0$,
so $(A\ot\La_\nov^0,\m_1)$ is a complex of $\La_\nov^0$-modules, and
we can form its {\it cohomology} $H^*(A\ot\La_\nov^0,\m_1)$, which
is a graded filtered $\La_\nov^0$-module. Also,
$(A\ot\La_\nov,\m_1)$ is a complex of $\La_\nov$-modules, whose
cohomology $H^*(A\ot\La_\nov,\m_1)$ is a graded filtered
$\La_\nov$-module. These are the kinds of cohomology we will use to
define Lagrangian Floer cohomology.
\label{il3dfn10}
\end{dfn}

The term {\it gapped\/} \cite[Def.~7.26]{FOOO} refers to condition
(i) above. This structure arises naturally in $J$-holomorphic curve
problems, and is useful for inductive arguments. We generalize
Definitions \ref{il3dfn3}--\ref{il3dfn5} to the gapped filtered
case.

\begin{dfn} Let $(A\ot\La_\nov^0,\m)$ and
$(B\ot\La_\nov^0,\n)$ be gapped filtered $A_\iy$ algebras. A {\it
gapped filtered\/ $A_\iy$ morphism}
$\f:(A\ot\La_\nov^0,\m)\ra(B\ot\La_\nov^0,\n)$ is $\f=(\f_k)_{k\ge
0}$, where $\f_k:{\buildrel {\ulcorner\,\,\,\text{$k$ copies }
\,\,\,\urcorner} \over {\vphantom{m}\smash{(A\ot\La_\nov^0)\t\cdots
\t(A\ot\La_\nov^0)}}}\ra B\ot\La_\nov^0$ for $k=0,1,\ldots$ are
graded $\La_\nov^0$-multilinear maps of degree 0, satisfying
\begin{itemize}
\setlength{\itemsep}{0pt}
\setlength{\parsep}{0pt}
\item[(i)] there exists a subset ${\cal G}'\subset[0,\iy)\t\Z$, closed
under addition, such that ${\cal G}'\cap(\{0\}\t\Z)=\{(0,0)\}$ and
${\cal G}'\cap([0,C]\t\Z)$ is finite for any $C\ge 0$, and the maps
$\f_k$ for $k\ge 0$ may be written
$\smash{\f_k=\sum_{(\la,\mu)\in{\cal G}'}T^\la
e^\mu\f_k^{\smash{\la,\mu}}}$, for unique $\Q$-multilinear maps
$\f_k^{\smash{\la,\mu}}:{\buildrel{\ulcorner\,\,\,\text{$k$ copies }
\,\,\,\urcorner} \over {\vphantom{m}\smash{A\t\cdots\t A}}}\ra B$
graded of degree $-2\mu$. When $k=0$, we take
$\f_0\in(B\ot\La_\nov^0)^{(0)}$ and $\f_0^{\smash{\la,\mu}}\in
B^{-2\mu}$;
\item[(ii)] $\f_0^{\smash{0,0}}=0$, in the notation of (i); and
\item[(iii)] for all $k\ge 0$ and pure $a_1,\ldots,a_k$ in
$A\ot\La_\nov^0$, we have
\begin{equation}
\begin{gathered}
\sum_{1\le i\le j\le k}
\begin{aligned}[t](-1)^{\sum_{l=1}^{i-1}\deg a_l}
\f_{k-j+i+1}\bigl(a_1,\ldots,a_{i-1},&\\
\m_{j-i}(a_i,\ldots,a_{j-1}),a_j,\ldots,a_k\bigr)&
\end{aligned}\\
=\sum_{0\le k_1\le k_2\le \cdots\le k_l=k}
\begin{aligned}[t]
\n_l\bigl(\f_{k_1}(a_1,\ldots,a_{k_1}),
\f_{k_2-k_1}(a_{k_1+1},\ldots,a_{k_2}),&\\
\ldots,\f_{k_l-k_{l-1}}(a_{k_{l-1}+1},\ldots,a_{k_l})\bigr).&
\end{aligned}
\end{gathered}
\label{il3eq11}
\end{equation}
\end{itemize}
As for \eq{il3eq10}, equation \eq{il3eq11} may be rewritten in terms
of the $\f_k^{\smash{\la,\mu}},\m_k^{\smash{\la,\mu}},
\n_k^{\smash{\la,\mu}}$ as
\begin{equation}
\begin{gathered}
\sum_{\begin{subarray}{l} 1\le i\le j\le k\\
\la_1+\la_2=\la,\; \mu_1+\mu_2=\mu\end{subarray}}
\begin{aligned}[t](-1)^{\sum_{l=1}^{i-1}\deg a_l}
\f_{k-j+i+1}^{\la_1,\mu_1}\bigl(a_1,\ldots,a_{i-1},&\\
\m_{j-i}^{\la_2,\mu_2}(a_i,\ldots,a_{j-1}),a_j,\ldots,a_k\bigr)&
\end{aligned}\\
=\sum_{\begin{subarray}{l} 0\le k_1\le k_2\le \cdots\le k_l=k\\
\la_0+\cdots+\la_l=\la,\; \mu_0+\cdots+\mu_l=\mu\end{subarray}}
\begin{aligned}[t]
\n_l^{\la_0,\mu_0}\bigl(\f_{k_1}^{\la_1,\mu_1}(a_1,\ldots,a_{k_1}),
\f_{k_2-k_1}^{\la_2,\mu_2}(a_{k_1+1},\ldots,a_{k_2}),&\\
\ldots,\f_{k_l-k_{l-1}}^{\la_l,\mu_l}(a_{k_{l-1}+1},\ldots,a_{k_l})\bigr),&
\end{aligned}
\end{gathered}
\label{il3eq12}
\end{equation}
for all $k\ge 0$, pure $a_1,\ldots,a_k$ in $A$, $\la\ge 0$
and~$\mu\in\Z$.

Note the difference between \eq{il3eq2} and \eq{il3eq11}: because we
now allow $\f_0$ to be nonzero, the second line of \eq{il3eq11} is a
sum over $0\le k_1\le k_2\le \cdots\le k_l=n$ rather than over $0<
k_1<k_2<\cdots<k_l=n$. Thus, the second line of \eq{il3eq11} is an
{\it infinite} sum, as for instance it includes the terms
$\n_l(\f_0,\ldots,\f_0,\f_n(a_1,\ldots,a_n))$ for all $l\ge 1$. We
claim that the second line of \eq{il3eq11} is a {\it convergent
sum\/} in the complete filtered $\La_\nov^0$-module
$B\ot\La_\nov^0$, in the sense of \S\ref{il34}. This is more-or-less
equivalent to \eq{il3eq12} being finite sums for all~$\la,\mu$.

To see this, let $\la_0=\min_{(0,0)\ne(\la,n)\in{\cal G}'}\la$,
which is well-defined and positive by (i), unless ${\cal
G}'=\{(0,0)\}$, in which case $\f_0=0$ and the result is trivial.
Then $\f_0\in F^{\la_0}(B\ot\La_\nov^0)$ by (ii). Now for any given
$N\ge 0$, there are only finitely many terms in the second line of
\eq{il3eq11} including fewer that $N$ $\f_0$'s. Thus, there are only
finitely many terms which do not lie in
$F^{N\la_0}(B\ot\La_\nov^0)$. Since $N\la_0\ra\iy$ as $N\ra\iy$,
this implies that for any $\la\in[0,\iy)$, all but finitely many
terms in the second line of \eq{il3eq11} lie in
$F^\la(B\ot\La_\nov^0)$, so it is a convergent sum.

A gapped filtered $A_\iy$ morphism $\f:(A\ot\La_\nov^0,\m)
\ra(B\ot\La_\nov^0,\n)$ is called {\it strict\/} if $\f_k=0$ for
$k\ne 1$, and a {\it gapped filtered\/ $A_\iy$ isomorphism} if
$\f_1:A\ot\La_\nov^0\ra B\ot\La_\nov^0$ is an isomorphism.

If $\f:(A\ot\La_\nov^0,\m)\ra(B\ot\La_\nov^0,\n)$ is a gapped
filtered $A_\iy$ morphism, then
$\f^{\smash{0,0}}:(A,\m^{\smash{0,0}})\ra (B,\n^{\smash{0,0}})$ is
an $A_\iy$ morphism, where
$\f^{\smash{0,0}}=(\f^{\smash{0,0}}_k)_{k\ge 1}$. We call $\f$ a
{\it weak homotopy equivalence} of gapped filtered $A_\iy$ algebras
if $\f^{\smash{0,0}}:(A,\m^{\smash{0,0}})\ra(B,\n^{\smash{0,0}})$ is
a weak homotopy equivalence of $A_\iy$ algebras in the sense of
Definition \ref{il3dfn3}, that is, if $\f^{\smash{0,0}}_1$ induces
an isomorphism~$H^*(A,\m^{\smash{0,0}}_1)\ra
H^*(B,\n^{\smash{0,0}}_1)$.

If $(A\ot\La_\nov^0,\m),(B\ot\La_\nov^0,\n),
(C\ot\La_\nov^0,{\mathfrak o})$ are gapped filtered $A_\iy$ algebras
and $\f:(A\ot\La_\nov^0,\m)\ra(B\ot\La_\nov^0,\n)$,
$\g:(B\ot\La_\nov^0,\n)\ra(C\ot\La_\nov^0,{\mathfrak o})$ are gapped
filtered $A_\iy$ morphisms, the {\it composition}
$\g\ci\f:(A\ot\La_\nov^0,\m)\ra(C\ot\La_\nov^0,{\mathfrak o})$ is
\begin{equation}
\begin{gathered}
(\g\ci\f)_n(a_1,\ldots,a_n)=\!\!\!\!\sum_{0\le k_1\le k_2\le
\cdots\le k_l=n}\!\!\!\!\!\!\!\!\!\!\!\!
\begin{aligned}[t]\g_l\bigl(\f_{k_1}(a_1,\ldots,a_{k_1}),
\f_{k_2-k_1}(a_{k_1+1},\ldots,a_{k_2}),&\\
\ldots,\f_{k_l-k_{l-1}}(a_{k_{l-1}+1},\ldots,a_{k_l})\bigr)&,
\end{aligned}
\end{gathered}
\label{il3eq13}
\end{equation}
which is \eq{il3eq4} but allowing equalities in the sum over $0<
k_1<k_2<\cdots<k_l=n$. As for \eq{il3eq11}, this is an infinite,
convergent sum. Composition is associative.

Let $\f,\g:(A\ot\La_\nov^0,\m)\ra(B\ot\La_\nov^0,\n)$ be gapped
filtered $A_\iy$ morphisms of gapped filtered $A_\iy$ algebras. A
{\it homotopy} from $\f$ to $\g$ is $\H=(\H_k)_{k\ge 0}$, where
$\H_k:{\buildrel {\ulcorner\,\,\,\text{$k$ copies } \,\,\,\urcorner}
\over {\vphantom{m}\smash{(A\ot\La_\nov^0)\t\cdots\t
(A\ot\La_\nov^0)}}}\ra B\ot\La_\nov^0$ for $k=0,1,\ldots$ are graded
$\La_\nov^0$-multilinear maps of degree $-1$, satisfying
\begin{itemize}
\setlength{\itemsep}{0pt}
\setlength{\parsep}{0pt}
\item[(i)] there exists a subset ${\cal G}''\subset[0,\iy)\t\Z$, closed
under addition, such that ${\cal G}''\cap(\{0\}\t\Z)=\{(0,0)\}$ and
${\cal G}''\cap([0,C]\t\Z)$ is finite for any $C\ge 0$, and the maps
$\H_k$ for $k\ge 0$ may be written
$\smash{\H_k=\sum_{(\la,\mu)\in{\cal G}''}T^\la
e^\mu\H_k^{\smash{\la,\mu}}}$, for unique $\Q$-multilinear maps
$\H_k^{\smash{\la,\mu}}:{\buildrel{\ulcorner\,\,\,\text{$k$ copies }
\,\,\,\urcorner} \over {\vphantom{m}\smash{A\t\cdots\t A}}}\ra B$
graded of degree $-1-2\mu$. When $k=0$, we take
$\H_0\in(B\ot\La_\nov^0)^{(-1)}$ and $\H_0^{\smash{\la,\mu}}\in
B^{-1-2\mu}$;
\item[(ii)] $\H_0^{\smash{0,0}}=0$, in the notation of (i); and
\item[(iii)] for all $n\ge 0$ and pure $a_1,\ldots,a_n$ in
$A\ot\La_\nov^0$, we have

\begin{gather}
\f_n(a_1,\ldots,a_n)-\g_n(a_1,\ldots,a_n)=
\nonumber\\
\sum_{\substack{0\le j_1\le j_2\le \cdots\le j_l\le\\
k_1\le k_2\le\cdots\le k_m=n}}
\begin{aligned}[t]
&\n_{l+m+1}\bigl(\f_{j_1}(a_1,\ldots,a_{j_1}),
\f_{j_2-j_1}(a_{j_1+1},\ldots,a_{j_2}),\ldots,\\
&\f_{j_l-j_{l-1}}\!(a_{j_{l-1}+1},\ldots,a_{j_l}),
\H_{k_1-j_l}(a_{j_l+1},\ldots,a_{k_1}), \\
&\g_{k_2-k_1}(a_{k_1+1},\ldots,a_{k_2}),\ldots,
\g_{k_m-k_{m-1}}(a_{k_{m-1}+1},\ldots,a_{k_m})\bigr)
\end{aligned}
\label{il3eq14}
\\
+\!\!\!\!\sum_{0\le i\le j\le n}\!\!\!\!\!\! (-1)^{\sum_{l=1}^i\deg
a_l} \H_{n-j+i+1}\bigl(a_1,\ldots,a_i,
\m_{j-i}(a_{i+1},\ldots,a_j),a_{j+1},\ldots,a_n\bigr), \nonumber
\end{gather}
which is \eq{il3eq5}, but allowing equalities in
$0<j_1<\cdots<k_m=n$ and $0\le i<j\le n$. As for \eq{il3eq11} and
\eq{il3eq13}, \eq{il3eq14} is a convergent infinite sum.
\end{itemize}

Let $\f:(A\ot\La_\nov^0,\m)\ra(B\ot\La_\nov^0,\n)$ be a gapped
filtered $A_\iy$ morphism. A {\it homotopy inverse} for $\f$ is a
gapped filtered $A_\iy$ morphism
$\g:(B\ot\La_\nov^0,\n)\ra(A\ot\La_\nov^0, \m)$ such that
$\g\ci\f:(A\ot\La_\nov^0,\m)\ra(A\ot \La_\nov^0,\m)$ is homotopic to
$\id_A:(A\ot\La_\nov^0,\m)\ra(A\ot\La_\nov^0,\m)$, and
$\f\ci\g:(B\ot\La_\nov^0,\n)\ra(B\ot\La_\nov^0,\n)$ is homotopic to
$\id_B:(B\ot\La_\nov^0,\n)\ra(B\ot\La_\nov^0,\n)$. If $\f$ has a
homotopy inverse, we call $\f$ a {\it homotopy equivalence}, and we
call $(A\ot\La_\nov^0,\m), (B\ot\La_\nov^0,\n)$ {\it homotopic}.
\label{il3dfn11}
\end{dfn}

Here is the analogue of Theorem \ref{il3thm1}, due to Fukaya et
al.~\cite[Th.~15.45(2)]{FOOO}.

\begin{thm} Let\/ $(A\ot\La_\nov^0,\m),(B\ot\La_\nov^0,\n)$
be gapped filtered $A_\iy$ algebras. Then
\begin{itemize}
\setlength{\itemsep}{0pt}
\setlength{\parsep}{0pt}
\item[{\rm(a)}] Homotopy is an equivalence relation
on gapped filtered $A_\iy$ morphisms
$\f:(A\ot\La_\nov^0,\m)\ra(B\ot\La_\nov^0,\n)$.
\item[{\rm(b)}] Homotopy is an equivalence relation
on gapped filtered $A_\iy$ algebras.
\item[{\rm(c)}] A gapped filtered $A_\iy$ morphism
$\f:(A\ot\La_\nov^0,\m)\ra(B\ot\La_\nov^0,\n)$ is a homotopy
equivalence if and only if it is a weak homotopy equivalence.
\end{itemize}
\label{il3thm3}
\end{thm}

We can also generalize the ideas of \S\ref{il33} to the gapped
filtered case. Here are the analogues of Definition \ref{il3dfn7}
and Theorem~\ref{il3thm2}.

\begin{dfn} Let $(A\ot\La_\nov^0,\m)$ be a gapped filtered
$A_\iy$ algebra. Then $\bigl(A,\m_1^{\smash{0,0}}\bigr)$ is a
complex. Let $B$ be a graded vector subspace of $A$ closed under
$\m_1^{\smash{0,0}}$, such that the inclusion $i:B\hookra A$ induces
an isomorphism $i_*:H^*\bigl(B,\m_1^{\smash{0,0}}\vert_B\bigr)\ra
H^*\bigl(A,\m_1^{\smash{0,0}}\bigr)$. We will construct
$\n=(\n_k)_{k\ge 0}$ making $(B\ot\La_\nov^0,\n)$ into a gapped
filtered $A_\iy$ algebra homotopic to~$(A\ot\La_\nov^0,\m)$.

Since $i_*$ is an isomorphism, we can choose a graded vector
subspace $C$ of $A$ such that $C\cap\Ker\m_1^{\smash{0,0}}=\{0\}$
and $A=B\op C\op\m_1^{\smash{0,0}}(C)$. Then
$\m_1^{\smash{0,0}}:C\ra\m_1^{\smash{0,0}}(C)$ is invertible, so
there is a unique graded linear map $H:A\ra A$ of degree $-1$ with
$H(b)=H(c)=0$ and $H\ci\m_1^{\smash{0,0}}(c)=c$ for all $b\in B$ and
$c\in C$. Let $\Pi_B:A\ra B$ be the projection, with kernel
$C\op\m_1^{\smash{0,0}}(C)$. Then $\id_A-\Pi_B=\m_1^{\smash{0,0}}\ci
H+H\ci\m_1^{\smash{0,0}}$ on $A$. Let
$\hat\imath:B\ot\La_\nov^0\hookra A\ot\La_\nov^0$, $\hat
H:A\ot\La_\nov^0\ra A\ot\La_\nov^0$ and $\hat\Pi_B:A\ot\La_\nov^0\ra
B\ot\La_\nov^0$ be the $\La_\nov^0$-linear extensions
of~$i,H,\Pi_B$.

For each planar rooted tree $T$ with $k$ leaves, define a graded
multilinear operator $\n_{k,T}:{\buildrel {\ulcorner\,\,\,\text{$k$
copies } \,\,\,\urcorner} \over {\vphantom{m}\smash{(B\ot
\La_\nov^0)\t\cdots\t(B\ot\La_\nov^0)}}}\ra B\ot\La_\nov^0$ of
degree $+1$, as follows. To define $\n_{k,T}(b_1,\ldots,b_k)$,
assign objects and operators to the vertices and edges of $T$:
\begin{itemize}
\setlength{\itemsep}{0pt}
\setlength{\parsep}{0pt}
\item assign $b_1,\ldots,b_k$ to the leaf vertices $1,\ldots,k$
respectively.
\item for each internal vertex with 1 outgoing edge and $n$ incoming
edges, $n\ne 1$, assign~$\m_n$.
\item for each internal vertex with 1 outgoing edge and 1 incoming
edge, assign~$\m_1-\m_1^{\smash{0,0}}$.
\item assign $\hat\imath$ to each leaf edge.
\item assign $\hat\Pi_B$ to the root edge.
\item assign $-\hat H$ to each internal edge.
\end{itemize}
Let $\n_{k,T}(b_1,\ldots,b_k)$ be the composition of all these
objects and morphisms, as in Definition \ref{il3dfn7}. Define
$\n_k:{\buildrel {\ulcorner\,\,\,\text{$k$ copies } \,\,\,\urcorner}
\over {\vphantom{m}\smash{(B\ot\La_\nov^0)\t\cdots
\t(B\ot\La_\nov^0)}}}\ra B\ot\La_\nov^0$ by
\begin{equation}
\n_k=\begin{cases} \m_1^{\smash{0,0}}+
\sum_T\n_{1,T}, & k=1, \\
\sum_T\n_{k,T}, & k=0,2,3,4,\ldots,
\end{cases}
\label{il3eq15}
\end{equation}
where the sums are over all planar rooted trees $T$ with $k$ leaves.

The sums in \eq{il3eq15} are {\it infinite} sums, since such trees
$T$ can contain arbitrarily large numbers of internal vertices with
1 edge, which are weighted by $\m_0$, or with 2 edges, which are
weighted by $\m_1-\m_1^{\smash{0,0}}$. We claim they are {\it
convergent}. To see this, let $\cal G$ be as in Definition
\ref{il3dfn10}(i), and set $\la_0=\min_{(0,0)\ne(\la,\mu)\in{\cal
G}}\la$. Then $\la_0>0$, provided ${\cal G}\ne\{(0,0)\}$, and
$\m_0\in F^{\la_0}(A\ot\La_\nov^0)$, and $\m_1-\m_1^{\smash{0,0}}:
F^\la(A\ot\La_\nov^0)\ra F^{\la+\la_0}(A\ot\La_\nov^0)$ for all
$\la\in[0,\iy)$. Therefore, if $T$ has $N$ internal vertices with 1
or 2 edges, then $\n_{k,T}$ maps to $F^{N\la_0}(B\ot\La_\nov^0)$. As
there are only finitely many rooted planar trees $T$ with $k$ leaves
and fewer than $N$ internal vertices with 1 or 2 edges, and
$N\la_0\ra\iy$ as $N\ra\iy$, it follows that \eq{il3eq15} is
convergent.

In a similar way, for each planar rooted tree $T$ with $k$ leaves,
define a graded multilinear operator ${\mathfrak i}_{k,T}:{\buildrel
{\ulcorner\,\,\,\text{$k$ copies } \,\,\,\urcorner} \over
{\vphantom{m}\smash{(B\ot\La_\nov^0)\t\cdots\t(B\ot\La_\nov^0)}}}
\ra A\ot\La_\nov^0$ of degree 0, as follows. Assign objects and
operators to the vertices and edges of $T$:
\begin{itemize}
\setlength{\itemsep}{0pt}
\setlength{\parsep}{0pt}
\item assign $b_1,\ldots,b_k$ to the leaf vertices $1,\ldots,k$
respectively.
\item for each internal vertex with 1 outgoing edge and $n$ incoming
edges, $n\ne 1$, assign~$\m_n$.
\item for each internal vertex with 1 outgoing edge and 1 incoming
edge, assign~$\m_1-\m_1^{\smash{0,0}}$.
\item assign $\hat\imath$ to each leaf edge.
\item assign $-\hat H$ to the root edge and to each internal edge.
\end{itemize}
Define ${\mathfrak i}_{k,T}(b_1,\ldots,b_k)$ to be the composition
of all these objects and morphisms. Define ${\mathfrak
i}_k:{\buildrel {\ulcorner\,\,\,\text{$k$ copies } \,\,\,\urcorner}
\over {\vphantom{m}\smash{(B\ot\La_\nov^0)\t\cdots\t
(B\ot\La_\nov^0)}}}\ra A\ot\La_\nov^0$ by
\begin{equation*}
{\mathfrak i}_k=\begin{cases} \hat\imath+
\sum_T{\mathfrak i}_{1,T}, & k=1, \\
\sum_T{\mathfrak i}_{k,T}, & k=0,2,3,4,\ldots,
\end{cases}
\end{equation*}
where the sums are over all planar rooted trees $T$ with $k$ leaves.
As for \eq{il3eq15}, these are convergent infinite sums.
\label{il3dfn12}
\end{dfn}

\begin{thm} In Definition {\rm\ref{il3dfn12},} $(B\ot\La_\nov^0,\n)$
is a gapped filtered\/ $A_\iy$ algebra, and\/ ${\mathfrak
i}:(B\ot\La_\nov^0,\n)\ra(A\ot\La_\nov^0,\m)$ is a gapped filtered\/
$A_\iy$ morphism, and a homotopy equivalence. If we choose $B\cong
H^*(A,\m_1^{\smash{0,0}})$ to be a subspace representing
$H^*(A,\m_1^{\smash{0,0}}),$ so that\/
$\n_1^{\smash{0,0}}=\m_1^{\smash{0,0}}\vert_B=0,$ then
$(B\ot\La_\nov^0,\n)$ is a minimal model for~$(A\ot\La_\nov^0,\m)$.
\label{il3thm4}
\end{thm}

As for Corollary \ref{il3cor1}, we prove:

\begin{cor} Let\/ $\p:(A\ot\La_\nov^0,\m)\ra(D\ot\La_\nov^0,
{\mathfrak o})$ be a strict, surjective gapped filtered\/ $A_\iy$
morphism of gapped filtered\/ $A_\iy$ algebras which is a weak
homotopy equivalence. Then we can construct an explicit homotopy
inverse $\q:(D\ot\La_\nov^0,{\mathfrak o})\ra(A\ot \La_\nov^0,\m)$
for $\p$ using sums over planar trees.
\label{il3cor2}
\end{cor}

\subsection{Bounding cochains}
\label{il36}

As in Definition \ref{il3dfn10}, to define Lagrangian Floer
cohomology we will need {\it strict\/} gapped filtered $A_\iy$
algebras. {\it Bounding cochains} are a method of modifying gapped
filtered $A_\iy$ algebras to make them strict, introduced by Fukaya
et al.\ \cite[\S 5.7, \S 11]{FOOO}.

\begin{dfn} Let $(A\ot\La_\nov^0,\m)$ be a gapped filtered
$A_\iy$ algebra, and suppose $b\in F^\la(A\ot\La_\nov^0)^{(0)}$ for
some $\la>0$. Define graded $\La_\nov^0$-multilinear maps
$\m_k^b:{\buildrel {\ulcorner\,\,\,\text{$k$ copies }
\,\,\,\urcorner} \over {\vphantom{m}\smash{(A\ot\La_\nov^0)\t\cdots
\t(A\ot\La_\nov^0)}}}\ra A\ot\La_\nov^0$ for $k=0,1,2,\ldots$, of
degree $+1$, by
\begin{gather*}
\m_k^b(a_1,\ldots,a_k)=\!\!\sum_{n_0,\ldots,n_k\ge
0}\!\!\!\m_{k+n_0+\cdots+n_k}\begin{aligned}[t]
\bigl({\buildrel{\ulcorner\,\,\text{$n_0$} \,\,\urcorner}
\over{\vphantom{m}\smash{b,\ldots,b}}},a_1, {\buildrel
{\ulcorner\,\,\text{$n_1$}\,\,\urcorner}
\over{\vphantom{m}\smash{b,\ldots,b}}},a_2,{\buildrel
{\ulcorner\,\,\text{$n_2$}\,\,\urcorner}
\over{\vphantom{m}\smash{b,\ldots,b}}}&,\\[-6pt]
\ldots, {\buildrel {\ulcorner\,\,\text{$n_{k-1}$}\,\,\urcorner}
\over{\vphantom{m}\smash{b,\ldots,b}}},a_k,{\buildrel
{\ulcorner\,\,\text{$n_k$} \,\,\urcorner}
\over{\vphantom{m}\smash{b,\ldots,b}}}&\bigr).
\end{aligned}
\end{gather*}
This is an infinite sum, but converges as $b\in
F^\la(A\ot\La_\nov^0)$ for $\la>0$. Write $\m^b=(\m_k^b)_{k\ge 0}$.
We call $b$ a {\it bounding cochain\/} for $(A\ot\La_\nov^0,\m)$ if
$\m_0^b=0$, that is, if
\begin{equation*}
\ts\sum_{k\ge 0}\m_k(b,\ldots,b)=0.
\end{equation*}
This is called the {\it Maurer--Cartan equation}, or {\it
Batalin--Vilkovisky master equation}.\!\!\!\!
\label{il3dfn13}
\end{dfn}

It is then easy to prove \cite[Prop.~11.10]{FOOO}:

\begin{lem} In Definition {\rm\ref{il3dfn13},} $(A\ot\La_\nov^0,\m^b)$
is a gapped filtered $A_\iy$ algebra, which is strict if and only
if\/ $b$ is a bounding cochain. Moreover,
$\f:(A\ot\La_\nov^0,\m^b)\ra(A\ot\La_\nov^0,\m)$ defined by
$\f_0=b,$ $\f_1=\id_{\smash{A\ot\La_\nov^0}}$ and\/ $\f_k=0$ for
$k\ge 2$ is an $A_\iy$ isomorphism. Thus $(A\ot\La_\nov^0,\m^b)$ is
homotopy equivalent to $(A\ot\La_\nov^0,\m)$.
\label{il3lem}
\end{lem}

Thus, if $b$ is a bounding cochain then $(A\ot\La_\nov^0,\m_1^b)$ is
a complex, and we may form its cohomology
$H^*(A\ot\La_\nov^0,\m_1^b)$, which is a $\La_\nov^0$-module. We can
also work over $\La_\nov$ rather than $\La_\nov^0$, so that
$(A\ot\La_\nov,\m_1^b)$ is a complex, with cohomology
$H^*(A\ot\La_\nov,\m_1^b)$.

\subsection{$A_{N,K}$ algebras}
\label{il37}

Stasheff \cite{Stas1,Stas2} introduced $A_K$ {\it algebras} \cite[\S
17.2]{FOOO}, a finite approximation to $A_\iy$ algebras. An $A_K$
algebra $(A,\m)$ is as in Definition \ref{il3dfn1} with $\m_0=0$,
except that $\m=(\m_k)_{k=1}^K$ rather than $(\m_k)_{k=1}^\iy$, and
\eq{il3eq1} holds for $k=1,\ldots,K$ rather than $k=1,\ldots,\iy$.
Similarly, $A_{N,K}$ {\it algebras} \cite[\S 30.6]{FOOO} are a
finite approximation of gapped filtered $A_\iy$ algebras. We omit
the phrase `gapped filtered' used in \cite{FOOO}. Here is the
$A_{N,K}$ analogue of Definitions \ref{il3dfn10} and~\ref{il3dfn11}.

\begin{dfn} Let ${\cal G}\subset[0,\iy)\t\Z$ be closed under addition with
${\cal G}\cap(\{0\}\t\Z)=\{(0,0)\}$ and ${\cal G}\cap([0,C]\t\Z)$
finite for any $C\ge 0$. Define $\nm{\,.\,}:{\cal G}\ra\N$ by
\begin{equation}
\nm{(\la,\mu)}=\max\bigl\{d:(\la,\mu)=\ts\sum_{i=1}^d(\la_i,\mu_i),\;
(0,0)\ne(\la_i,\mu_i)\in{\cal G}\bigr\}+[\la]
\label{il3eq16}
\end{equation}
for $(0,0)\ne(\la,\mu)\in{\cal G}$, where $[\la]$ is the greatest
integer $\le\la$, and $\nm{(0,0)}=0$. (This differs by 1 from
$\nm{(\la,\mu)}$ in~\cite[Def.~30.61]{FOOO}.)

Let $N,K\ge 0$. An $A_{N,K}$ {\it algebra} $(A,{\cal G},\m)$
consists of a $\Z$-graded $\Q$-vector space $A=\bigop_{d\in\Z}A^d$,
$\cal G$ as above, and a family $\m$ of graded $\Q$-multilinear maps
$\m_k^{\smash{\la,\mu}}:{\buildrel {\ulcorner\,\,\,\text{$k$ copies
} \,\,\,\urcorner} \over {\vphantom{m}\smash{A\t\cdots\t A}}}\ra A$
of degree $1-2\mu$ for all $(\la,\mu)\in{\cal G}$ and $k\ge 0$ such
that either (a) $\nm{(\la,\mu)}+k-1<N+K$, or (b)
$\nm{(\la,\mu)}+k-1=N+K$ and $\nm{(\la,\mu)}-1\le N$, satisfying
equation \eq{il3eq10} for all $(\la,\mu)\in{\cal G}$ and $k\ge 0$
such that (a) or (b) hold.

Now suppose $(A,{\cal G},\m)$ and $(B,{\cal G},\n)$ are $A_{N,K}$
algebras. Modifying the first part of Definition \ref{il3dfn11}, an
$A_{N,K}$ {\it morphism} $\f:(A,{\cal G},\m)\ra(B,{\cal G},\n)$
consists of $\Q$-multilinear maps $\f_k^{\smash{\la,\mu}}:
{\buildrel{ \ulcorner\,\,\,\text{$k$ copies } \,\,\,\urcorner} \over
{\vphantom{m}\smash{A\t\cdots\t A}}}\ra B$ graded of degree $-2\mu$
for all $(\la,\mu)\in{\cal G}$ and $k\ge 0$ such that (a) or (b)
hold, with $\f_0^{\smash{0,0}}=0$, satisfying equation \eq{il3eq12}
for all $(\la,\mu)\in{\cal G}$, $k\ge 0$ such that (a) or (b) hold
and pure $a_1,\ldots,a_k\in A$. Note that we use {\it the same}
$\cal G$ for $(A,{\cal G},\m),(B,{\cal G},\n)$ and $\f$, and we
regard $\cal G$ as fixed once and for all. The issue of changing
$\cal G$ will be addressed in the proof of Theorem~\ref{il11thm}.

{\it Composition} of $A_{N,K}$ morphisms is defined in the obvious
way. If $\f:(A,{\cal G},\m)\ra(B,{\cal G},\n)$ is an $A_{N,K}$
morphism then $\f_1^{\smash{0,0}}:A\ra B$ is a well-defined morphism
of complexes $(A,\m_1^{\smash{0,0}})\ra(B,\n_1^{\smash{0,0}})$, and
induces $(\f_1^{\smash{0,0}})_*: H^*(A,\m_1^{\smash{0,0}})\ra
H^*(B,\n_1^{\smash{0,0}})$. We call $\f$ a {\it weak homotopy
equivalence} if $(\f_1^{\smash{0,0}})_*$ is an isomorphism. We can
also define {\it homotopy} $\H:\f\Rightarrow\g$ between
$A_{N,K}$-morphisms $\f,\g:(A,{\cal G},\m)\ra(B,{\cal G},\n)$ by
rewriting \eq{il3eq14} in terms of the $\H_k^{\smash{\la,\mu}}$ and
only requiring it to hold for $(\la,\mu),k$ satisfying (a) or (b).
Thus we define {\it homotopy inverse} and {\it homotopy
equivalence}.
\label{il3dfn14}
\end{dfn}

Here is the analogue of Theorems \ref{il3thm1} and \ref{il3thm3},
\cite[Rem.~30.71]{FOOO}.

\begin{thm} Let\/ $(A,{\cal G},\m),(B,{\cal G},\n)$ be $A_{N,K}$
algebras. Then
\begin{itemize}
\setlength{\itemsep}{0pt}
\setlength{\parsep}{0pt}
\item[{\rm(a)}] Homotopy is an equivalence relation
on $A_{N,K}$ morphisms $\f:(A,{\cal G},\m)\ab\ra(B,{\cal G},\n)$.
\item[{\rm(b)}] Homotopy is an equivalence relation on $A_{N,K}$
algebras.
\item[{\rm(c)}] An $A_{N,K}$ morphism $\f:(A,{\cal G},\m)\ra
(B,{\cal G},\n)$ is a homotopy equivalence if and only if it is a
weak homotopy equivalence.
\end{itemize}
\label{il3thm5}
\end{thm}

For simplicity, in the rest of the paper we will take $K=0$, and
consider only $A_{N,0}$ algebras. These are sufficient for our
purposes, and fixing $K=0$  reduces conditions (a) and (b) of
Definition \ref{il3dfn14} to the single
inequality~$\nm{(\la,\mu)}+k-1\le N$.

If $\bar N\ge N\ge 0$ then any $A_{\bar N,0}$ algebra $(A,{\cal
G},\bar\m)$ induces an $A_{N,0}$ algebra $(A,{\cal G},\m)$ by taking
$\m$ to be the subset of $\bar\m_k^{\smash{\la,\mu}}$ with
$\nm{(\la,\mu)}+k-1\le N$. Similarly, an $A_{\bar N,0}$ morphism
$\bar\f:(A,{\cal G},\bar\m)\ab\ra(B,{\cal G},\bar\n)$ restricts to
an $A_{N,0}$ morphism $\f:(A,{\cal G},\m)\ab\ra(B,{\cal G},\n)$ on
the corresponding $A_{N,0}$ algebras. Conversely, we can ask about
extending $A_{N,0}$ algebras and $A_{N,0}$ morphisms to $A_{\bar
N,0}$ algebras and $A_{\bar N,0}$ morphisms. Our next theorem
follows from Fukaya et al.~\cite[Th.~30.72 \& Lem.~30.128]{FOOO}.

\begin{thm} Let\/ $\f:(A,{\cal G},\m)\ra(B,{\cal G},\n)$ be an
$A_{N,0}$ morphism of\/ $A_{N,0}$ algebras which is a weak homotopy
equivalence. Suppose $\bar N\ge N,$ and $(B,{\cal G},\bar\n)$ is an
$A_{\bar N,0}$ algebra extending $(B,{\cal G},\n)$. Then
\begin{itemize}
\setlength{\itemsep}{0pt}
\setlength{\parsep}{0pt}
\item[{\rm(a)}] there exists an $A_{\bar N,0}$ algebra $(A,{\cal
G},\bar\m)$ extending $(A,{\cal G},\m),$ and an $A_{\bar N,0}$
morphism $\bar\f:(A,{\cal G},\bar\m)\ra(B,{\cal G},\bar\n)$
extending $\f$ which is a weak homotopy equivalence; and
\item[{\rm(b)}] if\/ $(A,{\cal G},\bar\m)$ is an $A_{\bar N,0}$
algebra extending $(A,{\cal G},\m),$ and\/ $\bar\g:(A,{\cal
G},\bar\m)\ra(B,{\cal G},\bar\n)$ is an $A_{\bar N,0}$ morphism
which restricts to an $A_{N,0}$ morphism $\g:(A,{\cal
G},\m)\ra(B,{\cal G},\n)$ which is $A_{N,0}$ homotopic to $\f,$ then
$\f$ extends to an $A_{\bar N,0}$ morphism $\bar\f:(A,{\cal
G},\bar\m)\ra(B,{\cal G},\bar\n)$ which is $A_{\bar N,0}$ homotopic
to~$\bar\g$.
\end{itemize}
\label{il3thm6}
\end{thm}

All of \S\ref{il35}--\S\ref{il37} also works over $\La_\CY^0$ rather
than $\La_\nov^0$, with the obvious changes.

\section{Moduli spaces}
\label{il4}

Next we discuss moduli spaces of isomorphism classes of stable maps
from a genus 0 prestable bordered Riemann surface with immersed
Lagrangian boundary conditions. Most of the arguments are the same
as in the embedded case of Fukaya et al.\ \cite[\S 29]{FOOO} and Liu
\cite{Liu}, but we put some extra data on the boundary of our stable
maps.

\subsection{Definition of the moduli spaces
$\oM_{k+1}^\ma(\al,\be,J)$}
\label{il41}

We first define stable $J$-holomorphic maps from prestable holomorphic
discs with marked points.

\begin{dfn} Let $(M,\om)$ be a compact $2n$-dimensional symplectic
manifold with a compatible almost complex structure $J$, and
$\io:L\ra M$ a compact Lagrangian immersion. Suppose that all the
self-intersection points of the immersion $\io$ are transverse
double self-intersections.

Let $\Si$ be a genus 0 prestable bordered Riemann surface, that is,
$\Si$ is a possibly singular Riemann surface with boundary $\pd\Si$
such that the double $\Si\cup_{\pd\Si}\bar{\Si}$ is a connected and
simply connected compact singular Riemann surface whose only
singularities are nodes. Let $k$ be a non-negative integer, and
choose mutually distinct smooth points $z_0,\ldots,z_k$ on $\pd\Si$,
and write $\vec z=(z_0,\ldots,z_k)$. Let $u:\Si\ra M$ be a
$J$-holomorphic map with $u(\pd\Si)\subset\io(L)$. We call the
triple $(\Si,\vec z,u)$ {\it stable\/} if the automorphism group
$\Aut(\Si,\vec z,u)$ of biholomorphisms $f:\Si\ra\Si$ with $u\ci
f=u$ and $f(z_i)=z_i$ for $i=0,\ldots,k$ is finite. Equivalently,
$(\Si,\vec z,u)$ is stable if for each irreducible component $\Si'$
of $\Si$, $u\vert_{\Si'}$ is not constant, or
\begin{itemize}
\item the number of singular points on $\Si'$ is at least 3 when
$\Si'$ is diffeomorphic to a sphere,
\item the number of marked or singular points on $\pd\Si'$ plus twice
the number of singular points on $\Si'\sm\pd\Si'$ is at least 3 when
$\Si'$ is diffeomorphic to a disc.
\end{itemize}

For $(\Si,\vec z,u)$ as above, we would like to think of the boundary
$\pd\Si$ as a circle, but this is not true if $\Si$ has boundary
nodes. Let ${\cal S}^1=\{z\in\C:\md{z}=1\}$ be a circle with the
counter-clockwise orientation. The boundary $\pd\Si$ has the
orientation induced by the complex structure, and there is a
continuous and orientation-preserving map $l:{\cal S}^1\ra\pd\Si$
unique up to reparameterization such that
\begin{itemize}
\item the inverse image of a singular point of $\pd\Si$ consists of
two points,
\item the inverse image of a smooth point of $\pd\Si$ consists of
one point.
\end{itemize}
Write $\ze_i=l^{-1}(z_i)$, for $i=0,\ldots,k$.
\label{il4dfn1}
\end{dfn}

In the embedded case \cite[\S 2]{FOOO}, one defines moduli spaces
$\oM_{k+1}(\be,J)$ of isomorphism classes $[\Si,\vec z,u]$ of
triples $(\Si,\vec z,u)$. But in our immersed case, we need to keep
track of some extra information. In Definition \ref{il4dfn1}, $u\ci
l$ is a continuous map ${\cal S}^1\ra\io(L)$. We want to know
whether this can be locally lifted to a continuous map $\bar u:{\cal
S}^1\ra L$ with $\io\ci\bar u\equiv u\ci l$. This is only a problem
at the self-intersection points of $\io(L)$. For such a point $p\in
M$ we have $\io^{-1}(p)=\{p_+,p_-\}$, that is, two points $p_+,p_-$
in $L$ map to one point $p$ in $M$, and $\io(L)$ near $p$ in $M$ has
two sheets, the images under $\io$ of disjoint open neighbourhoods
of $p_+$ and~$p_-$.

If $u\ci l(\ze)=p$ for some $\ze\in{\cal S}^1$, it can happen that
$u\ci l$ jumps at $\ze$ between the two sheets of $\io(L)$ near $p$
in $M$, and so $u\ci l$ cannot be lifted to a continuous $\bar u:
{\cal S}^1\ra L$ near $\ze$, since $\bar u$ would have to jump
discontinuously between $p_+$ and $p_-$ at $\ze$. The meaning of the
next definition is that we consider triples $(\Si,\vec z,u)$ in
which $u\ci l$ jumps at $\ze$ between two sheets of $\io(L)$ in this
way if and only if $\ze=\ze_i$ for $i$ in a fixed subset
$I\subseteq\{0,\ldots,k\}$, and that we also prescribe $p=u(\ze_i)$
and the limits $p_+,p_-$ of $u(\ze')$ as $\ze'\ra\ze_i$ in
${\cal S}^1$ from either direction.

\begin{dfn} Let $(M,\om)$ be a compact $2n$-dimensional symplectic
manifold with a compatible almost complex structure $J$, and
$\io:L\ra M$ a compact Lagrangian immersion with only transverse
double self-intersections. Define $R$ to be the set of ordered pairs
$(p_-,p_+)\in L\t L$ such that $p_-\neq p_+$ and $\io(p_-)=\io(p_+)$,
and define an involution $\si:R\ra R$ by $\si(p_-,p_+)=(p_+,p_-)$.

Fix $k\ge0$. Let $I\subset\{0,\ldots,k\}$ be a subset, $\al:I\ra R$
a map, and $\be\in H_2(M,\io(L);\Z)$ a relative homology class.
Consider quintuples $(\Si,\vec z,u,l,\bar{u})$, where $\Si$ is a
genus 0 prestable bordered Riemann surface, and $\vec
z=(z_0,\ldots,z_k)$ for distinct smooth points $z_0,\ldots,z_k$ on
$\pd\Si$, and $u:\Si\ra M$ is a $J$-holomorphic map with
$u(\pd\Si)\subset\io(L)$ and $(\Si,\vec z,u)$ stable, and $l:{\cal
S}^1\ra\pd\Si$ is as in Definition \ref{il4dfn1} with
$\ze_i=l^{-1}(z_i)$ for all $i$, and $\bar{u}:{\cal S}^1\sm
\{\ze_i:i\in I\}\ra L$ is a continuous map, satisfying the following
conditions:
\begin{itemize}
\item $u_*([\Si])=\be\in H_2(M,\io(L);\Z)$, with $[\Si]\in
H_2(\Si,\pd\Si;\Z)$ the fundamental class;
\item $\ze_0,\ldots,\ze_k$ are ordered counter-clockwise
on~${\cal S}^1$;
\item $\io\ci\bar{u}\equiv u\ci l$ on ${\cal S}^1\sm\{\ze_i:i\in I\}$;
and
\item $(\lim_{\th\uparrow 0}\bar{u}(e^{\sqrt{-1}\,\th}\ze_i),
\lim_{\th\downarrow 0}\bar{u}(e^{\sqrt{-1}\,\th}\ze_i))=\al(i)$ in
$R$, for all~$i\in I$.
\end{itemize}

We say that two quintuples $(\Si,\vec z,u,l,\bar{u})$ and
$(\Si',\vec z',u',l',\bar{u}')$ are {\it isomorphic\/} if there
exist a biholomorphic map $\varphi:\Si\ra\Si'$ and an
orientation-preserving homeomorphism $\bar{\varphi}:{\cal S}^1\ra
{\cal S}^1$ such that
\begin{itemize}
\item $u'\ci\varphi=u$, and $\varphi(z_i)=z'_i$ for $i=0,\ldots,k$,
\item $\varphi\ci l=l'\ci\bar{\varphi}$, and
$\bar{u}'\ci\bar{\varphi}=\bar{u}$ on ${\cal S}^1\sm\{\ze_i:i\in I\}$.
\end{itemize}
Denote by $\oM_{k+1}^\ma(\al,\be,J)$ the set of the isomorphism
classes $[\Si,\vec z,u,l,\bar{u}]$ of such quintuples $(\Si,\vec
z,u,l,\bar{u})$. Then we may define a natural, compact, Hausdorff
topology on $\oM_{k+1}^\ma(\al,\be,J)$ called the $C^\iy$ {\it
topology}, following Fukaya et al.\ \cite[\S 29]{FOOO} and
Liu~\cite[\S 5.2]{Liu}.

Define the {\it evaluation maps\/} $\ev_i:\oM_{k+1}^\ma(\al,\be,J)
\ra L\amalg R$ by
\begin{equation}
\ev_i\bigl([\Si,\vec z,u,l,\bar{u}]\bigr)=\begin{cases}
\bar{u}(\ze_i)\in L, & i\notin I,\\
\al(i)\in R, & i\in I,
\end{cases}
\label{il4eq1}
\end{equation}
for $i=0,\ldots,k$, and $\ev:\oM_{k+1}^\ma(\al,\be,J)\ra L\amalg R$ by
\begin{equation}
\ev\bigl([\Si,\vec z,u,l,\bar{u}]\bigr)=\begin{cases}
\bar u(\zeta_0)\in L, & 0\notin I,\\
\si\ci\al(0)\in R, & 0\in I,
\end{cases}
\label{il4eq2}
\end{equation}
where $\si:R\ra R$ is the involution above. Following Fukaya et al.\
\cite[\S 9 \& \S 29]{FOOO} and Liu \cite{Liu} we may define a {\it
Kuranishi structure\/} on $\oM_{k+1}^\ma(\al,\be,J)$, with boundary
and corners and a tangent bundle, and the continuous maps
$\ev_i,\ev$ extend to {\it strong
submersions\/}~$\bev_i,\bev:\oM_{k+1}^\ma(\al,\be,J)\ra L\amalg R$.

We shall also write
\begin{equation}
\oM_{k+1}^\ma(\be,J)=
\coprod\nolimits_{\begin{subarray}{l}I\subseteq\{0,\ldots,k\},\\
\al:I\ra R\end{subarray}}\oM_{k+1}^\ma(\al,\be,J).
\label{il4eq3}
\end{equation}
Since by \eq{il4eq10} below the virtual dimension of $\oM_{k+1}^\ma
(\al,\be,J)$ depends on $I,\al$, this is technically not a Kuranishi
space, only a disjoint union of Kuranishi spaces of different
dimensions. We define strong submersions
$\bev_i,\bev:\oM_{k+1}^\ma(\be,J)\ra L\amalg R$ to be $\bev_i,\bev$
on each component~$\smash{\oM_{k+1}^\ma (\al,\be,J)}$.
\label{il4dfn2}
\end{dfn}

\subsection{The boundary of $\oM_{k+1}^\ma(\al,\be,J)$}
\label{il42}

Following Fukaya et al.\ \cite[\S 30]{FOOO} we can give an
expression for the {\it boundaries\/} of our moduli spaces. We
postpone discussing the {\it orientations} in \eq{il4eq4}
until~\S\ref{il5}.

\begin{thm} In the situation of Definition \ref{il4dfn2}, there is
an isomorphism of unoriented Kuranishi spaces, using the fibre
product of Definition~\ref{il2dfn6}:
\begin{equation}
\begin{gathered}
\pd\oM^\ma_{k+1}(\al,\be,J)\cong \coprod_{\begin{subarray}{l}
k_1+k_2=k+1,\; 1\le i\le k_1,\; I_1\cup_iI_2=I,\\
\al_1\cup_i\al_2=\al,\; \be_1+\be_2=\be
\end{subarray}\!\!\!\!\!\!\!\!\!\!\!\!\!\!\!\!\!\!\!\!\!\!\!\!}
\begin{aligned}[t]
\oM^\ma_{k_2+1}(\al_2,\be_2,J)\t_{\bev,L\amalg R,\bev_i}&\\
\oM^\ma_{k_1+1}(\al_1,\be_1,J)&
\end{aligned}
\end{gathered}
\label{il4eq4}
\end{equation}
where we define\/ $I_1\cup_iI_2\subseteq\{0,\ldots,k\}$ and\/
$\al_1\cup_i\al_2:I_1\cup_iI_2\ra R$ by
\begin{equation}
\begin{split}
I_1\cup_iI_2=&\{j:j\in I_1,\;j<i\}\cup\{j+i-1:j\in I_2,\;0<j\} \\
&\cup\{j+k_2-1:j\in I_1,\;i<j\}, \\
(\al_1\cup_i\al_2)(j)=&\begin{cases}
\al_1(j), &\text{for $0\le j< i,$}\\
\al_2(j-i+1), &\text{for $1\le j-i+1\le k_2,$}\\
\al_1(j-k_2+1), &\text{for $i<j-k_2+1\le k_1$},
\end{cases}
\end{split}
\label{il4eq5}
\end{equation}
and we also use the same notation for the evaluation maps
$\bev_i:\oM^\ma_{k_1+1}(\al_1,\be_1,J)\ra L\amalg R$
and $\bev:\oM^\ma_{k_2+1}(\al_2,\be_2,J)\ra L\amalg R$.
\label{il4thm1}
\end{thm}

Here \eq{il4eq4} is a fairly straightforward consequence of the
construction of the Kuranishi structure on
$\oM_{k+1}^\ma(\al,\be,J)$, as near the boundary strata of
$\oM_{k+1}^\ma(\al,\be,J)$ the Kuranishi neighbourhoods
$(V_p,\ldots,\psi_p)$ are built from Kuranishi neighbourhoods on
terms in the right hand side of \eq{il4eq4}, using gluing theorems
to desingularize boundary nodes in $\Si$. In \eq{il4eq4} we choose
to write the fibre product as
$\oM_{k_2+1}^\ma(\al_2,\be_2,J)\t_{\bev,L\amalg R,\bev_i}
\oM_{k_1+1}^\ma(\al_1,\be_1,J)$, although it would be more obvious
to write it as $\oM_{k_1+1}^\ma(\al_1,\be_1,J) \t_{\bev_i,L\amalg
R,\bev}\oM_{k_2+1}^\ma(\al_2,\be_2,J)$, following Fukaya et al.\
\cite[Prop.~46.3]{FOOO}. As we will explain in Remark
\ref{il5rem2}(b), because of peculiarities of the immersed case,
when we orient our moduli spaces in \S\ref{il5}, the signs in our
formulae will look simpler and more natural with the fibre product
order in~\eq{il4eq4}.

\subsection{The virtual dimension of $\oM_{k+1}^\ma(\al,\be,J)$}
\label{il43}

We shall compute the {\it virtual dimension\/} of
$\oM_{k+1}^\ma(\al,\be,J)$, modifying Fukaya \cite[Th.~3.2]{Fuka},
who calculates the virtual dimension of moduli spaces of holomorphic
discs with boundary attached to a union $L_0\cup\cdots\cup L_k$ of
transversely intersecting embedded Lagrangians, so that
$L_0\cup\cdots\cup L_k$ is an immersed Lagrangian submanifold with
transverse double self-intersections, and also Fukaya et al.\
\cite[Prop.~12.59]{FOOO}, who perform the same calculation
for~$L_0\cup L_1$.

\begin{dfn} Let
\begin{equation}
Y=\bigl\{(x,y)\in\R^2:\text{either $x\le 0$, $x^2+y^2\le 1$ or $x\ge
0$, $\md{y}\le 1$}\bigr\}.
\label{il4eq6}
\end{equation}
For $(p_-,p_+)\in R$, choose a smooth family
$\la_{(p_-,p_+)}=\{\la_{(p_-,p_+)}(x,y)\}_{(x,y)\in\pd Y}$ of
Lagrangian subspaces of $T_pM$, where $p=\io(p_-)=\io(p_+)$, such that
\begin{equation*}
\la_{(p_-,p_+)}(x,y)=\begin{cases}
\d\io(T_{p_-}L), & \mbox{if }y=1,\\
\d\io(T_{p_+}L), & \mbox{if }y=-1.
\end{cases}
\end{equation*}
If $(p_-,p_+)\in R$ then $\si(p_-,p_+)=(p_+,p_-)\in R$, and we
require $\la_{(p_-,p_+)}$ and $\la_{(p_+,p_-)}$ to be related by
$\la_{(p_+,p_-)}(x,y)\equiv\la_{(p_-,p_+)}(x,-y)$. When $L$ is {\it
oriented}, as it will be from \S\ref{il5} onwards, we take
$\la_{(p_-,p_+)}$ to be a smooth family of oriented Lagrangian
subspaces, which agree with $\d\io(T_{p_\mp}L)$ as oriented
subspaces when~$y=\pm 1$.

Consider the differential operator
\begin{equation}
\bar{\pd}_{\la_{(p_-,p_+)}}\!=\!\frac{\pd}{\pd x}\!+\!
J_p\frac{\pd}{\pd y}:W^{1,q}(Y,\pd Y;T_pM,\la_{(p_-,p_+)})\!\ra
L^q(Y;T_pM\ot\La^{0,1}Y),
\label{il4eq7}
\end{equation}
for $q>2$, where $W^{1,q}(Y,\pd Y;T_pM,\la_{(p_-,p_+)})$ is the
Sobolev space of the $W^{1,q}$-maps $\xi:Y\ra T_pM$ with
$\xi(x,y)\in\la_{(p_-,p_+)}(x,y)$, for $(x,y)\in\pd Y$, and
$L^q(Y;T_pM\ot\La^{0,1}Y)$ is the one of the $L^q$-maps $\xi:Y\ra
T_pM\ot\La^{0,1}Y$. Following \cite[Def.~12.62]{FOOO}, define
\begin{equation}
\eta_{(p_-,p_+)}=\ind\bar{\pd}_{\la_{(p_-,p_+)}},
\label{il4eq8}
\end{equation}
the Fredholm index of \eq{il4eq7}. Since $\la_{(p_+,p_-)}(x,y)\equiv
\la_{(p_-,p_+)}(x,-y)$, it is easy to check that
\begin{equation}
\eta_{(p_-,p_+)}+\eta_{(p_+,p_-)}=n.
\label{il4eq9}
\end{equation}
\label{il4dfn3}
\end{dfn}

Note that $\eta_{(p_-,p_+)}$ depends on the choice of
$\la_{(p_-,p+)}$ up to isotopy. When $\la_{(p_-,p_+)}$ is a family
of oriented Lagrangian subspaces, different choices of
$\la_{(p_-,p+)}$ add an even number to $\eta_{(p_-,p_+)}$. Thus the
only invariant information is whether $\eta_{(p_-,p_+)}$ is even or
odd, which depends on whether the transverse, oriented subspaces
$\d\io(T_{p_-}L)$ and $\d\io(T_{p_+}L)$ intersect positively or
negatively in~$T_pM$.

In \S\ref{il46} we will use this freedom to require that
$\eta_{(p_-,p_+)}\ge 0$ for all $(p_-,p_+)\in R$, and ask that
$\la_{(p_-,p_+)}$ is chosen generically, which ensures that $\Ker
\smash{\bar{\pd}_{\la_{(p_-,p_+)}}}$ has dimension
$\eta_{(p_-,p_+)}$, and
$\Coker\smash{\bar{\pd}_{\la_{(p_-,p_+)}}}=0$. This is not strictly
necessary, but it simplifies the arguments.

There is an important case in which it is natural to fix the
$\eta_{(p_-,p_+)}$, however, to be discussed in \S\ref{il12}.
Suppose that $(M,\om)$ is the symplectic manifold underlying a {\it
Calabi--Yau} manifold, and that $L$ is a {\it graded\/} immersed
Lagrangian submanifold, in the sense of Definition \ref{il12dfn}.
Then we can choose $\la_{(p_-,p_+)}$ to be a family of graded
Lagrangian subspaces of $T_pM$, which agree with $\d\io(T_{p_\mp}L)$
as graded Lagrangian subspaces when $y=\pm 1$. This requirement
determines $\eta_{(p_-,p_+)}$ uniquely in $\Z$, independently of the
choice of $\la_{(p_-,p_+)}$. Also in this case the Maslov index
$\mu_L(\be)$ below is automatically zero, provided the
$\la_{(p_-,p_+)}$ are taken to be graded.

We can now define the {\it Maslov index\/} $\mu_L(\be)$, and compute
the virtual dimension of~$\oM_{k+1}^\ma(\al,\be,J)$.

\begin{dfn} For $[\Si,\vec z,u,l,\bar{u}]\in\oM_{k+1}^\ma(\al,\be,J)$,
we take $\vep>0$ and a continuous map $\psi:{\cal S}^1\ra {\cal S}^1$
such that
\begin{itemize}
\item $\psi:{\cal S}^1\sm\bigcup_{i\in I}\{e^{\sqrt{-1}\th}\ze_i:
\th\in[-\vep,\vep]\}\ra {\cal S}^1\sm\{\ze_i:i\in I\}$ is an
orientation preserving homeomorphism,
\item $\psi(\{e^{\sqrt{-1}\th}\ze_i:\th\in[-\vep,\vep]\})=\ze_i$,
for $i\in I$,
\end{itemize}
and define
\begin{equation*}
A_{\al,\be}(z)=\begin{cases} \d\io(T_{\bar{u}\ci\psi(z)}L), &
\mbox{for } z\in {\cal S}^1\sm\bigcup_{i\in
I}\{e^{\sqrt{-1}\th}\ze_i:\th\in(-\vep,\vep)\},\\
\la_{\al(i)}\ci h_i(z), &\mbox{ for }z\in\{e^{\sqrt{-1}\th}\ze_i:
\th\in(-\vep,\vep)\} \mbox{ with }i\in I,
\end{cases}
\end{equation*}
where $h_i:\{e^{\sqrt{-1}\th}\ze_i:\th\in(-\vep,\vep)\}\ra\pd Y$ is
a diffeomorphism with
\begin{equation*}
\lim_{\th\ra-\vep}h_i(e^{\sqrt{-1}\th}\ze_i)=(\iy,1)
\mbox{ and }\lim_{\th\ra\vep}h_i(e^{\sqrt{-1}\th}\ze_i)=(\iy,-1).
\end{equation*}
The symplectic vector bundle $u^*(TM)$ with $u^*(\om)$ is isomorphic
to the trivial one $\Si\t\C^n\ra\Si$. Denote this trivialization by
$f:u^*(TM)\ra\C^n$, and $f\ci A_{\al,\be}$ is a loop in the
Grassmannian of Lagrangian subspaces in~$\C^n$.

Write $\mu_L(\be)$ for the {\it Maslov index\/} of $f\ci
A_{\al,\be}$, in the sense of Fukaya et al.\ \cite[\S 2.1]{FOOO}.
That is, $\mu_L(\be)\in\Z$ is the contraction of the homology class
of $f\ci A_{\al,\be}$ with a certain class in the 1-cohomology of
the Grassmannian of Lagrangian subspaces in $\C^n$. If $L$ is
oriented, as it will be from \S\ref{il5} onwards, then $\mu_L(\be)$
is even. As above and in \S\ref{il12}, if $(M,\om)$ is Calabi--Yau
and $L$ is {\it graded}, we can define the $\la_{(p_-,p_+)}$ using
graded Lagrangian subspaces, and then $\mu_L(\be)=0$ for all~$\be$.

Now $\mu_L(\be)$ depends on the choices of families
$\la_{(p_-,p_+)}$ for $(p_-,p_+)\in R$ above up to isotopy, and
hence in effect on the $\eta_{(p_-,p_+)}$. We regard these as fixed
once and for all, and suppress the dependence of the Maslov index on
them in our notation. In fact $\mu_L(\be)$ is independent of the
other choices involved, except $\be$, which justifies our writing it
as $\mu_L(\be)$. That is, $\mu_L(\be)$ is independent of
$k,I,\al,[\Si,\vec z,u,l,\bar{u}],\psi,h_i$, and the trivialization
of $\bigl(u^*(TM),u^*(\om)\bigr)$. To see this, note that morally
$\mu_L(\be)=\be\cdot c_1\bigl(M,\io(L)\bigr)$, where $\be\in
H_2\bigl(M,\io(L);\Z\bigr)$ and $c_1\bigl(M,\io(L)\bigr)\in
H^2\bigl(M,\io(L);\Z\bigr)$ is the {\it relative first Chern
class\/} for $\om$ on $\bigl(M,\io(L)\bigr)$. The reason
$\mu_L(\be)$ can be independent of $I,\al$ is that $\be$ partially
determines $I,\al$, enough so that the dependence of $\mu_L(\be)$ on
$I,\al$ is determined by~$\be$.
\label{il4dfn4}
\end{dfn}

The following proposition is a straightforward modification of
Fukaya \cite[Th.~3.2]{Fuka} and Fukaya et al.\
\cite[Prop.~29.1]{FOOO} to the immersed case, following
\cite[Prop.~12.59]{FOOO}. In effect, in constructing
$\psi,A_{\al,\be}$ above we are defining a desingularized moduli
problem, with embedded Lagrangian boundary conditions. The virtual
dimension of this desingularized moduli problem is computed as in
\cite[Prop.~29.1]{FOOO}, and is the right hand side of \eq{il4eq10}
omitting the term $-\sum_{i\in I}\eta_{\al(i)}$. But the effect of
desingularizing by gluing in $\la_{\al(i)}$ at $z_i$ is to increase
the virtual dimension by $\eta_{\al(i)}$, so to recover the virtual
dimension of the original moduli problem we subtract~$\sum_{i\in
I}\eta_{\al(i)}$.

\begin{prop} The virtual dimension of the Kuranishi space\/
$\oM_{k+1}^\ma(\al,\be,J)$ is
\begin{equation}
\vdim\oM_{k+1}^\ma(\al,\be,J)=\mu_L(\be)+k-2+n-\ts\sum_{i\in I}
\eta_{\al(i)}.
\label{il4eq10}
\end{equation}
\label{il4prop1}
\end{prop}

\subsection{The moduli spaces
$\oM_{k+1}^\ma(\al,\be,J,f_1,\ldots,f_k)$}
\label{il44}

Next we add {\it smooth simplicial chains\/} to our moduli spaces.

\begin{dfn} For $i=1,\ldots,k$, let $a_i\ge 0$ and $f_i:\De_{a_i}\ra
L\amalg R$ be a smooth map, where $\De_{a_i}$ is the $a_i$-simplex
of \eq{il2eq6}, so that $f_i\in C_{a_i}^\rsi(L\amalg R)$ is a smooth
simplicial chain. Define the Kuranishi space
$\oM_{k+1}^\ma(\al,\be,J,f_1,\ldots,f_k)$ to be the fibre product
\begin{equation}
\begin{split}
&\oM_{k+1}^\ma(\al,\be,J,f_1,\ldots,f_k)= \\
&\oM_{k+1}^\ma(\al,\be,J)\t_{\bev_1\t\cdots\t\bev_k,(L\amalg R)^k,
f_1\t\cdots\t f_k}(\De_{a_1}\t\cdots\t\De_{a_k}).
\end{split}
\label{il4eq11}
\end{equation}
Here $\bev_i$ maps to $L$ if $i\notin I$ and to $R$ if $i\in I$.
Also, the fibre product is over $1,\ldots,k$ although
$I\subseteq\{0,\ldots,k\}$, so we have to exclude 0. Thus,
\eq{il4eq11} is in effect a fibre product over the manifold
$\prod_{i\in\{1,\ldots,k\}\sm I}L\t\prod_{i\in I\sm\{0\}}R$, which
has dimension $n(k-\md{I\sm\{0\}})$. So we see from \eq{il4eq10} and
Definition \ref{il2dfn6} that
\begin{equation}
\begin{split}
&\vdim\oM_{k+1}^\ma(\al,\be,J,f_1,\ldots,f_k)=\\
&\mu_L(\be)+k-2+n-\ts\sum_{i\in I}\eta_{\al(i)}+\ts\sum_{0\ne
i\notin I}(a_i-n)+\ts\sum_{0\ne i\in I}a_i.
\end{split}
\label{il4eq12}
\end{equation}

Let $f:\De_a\ra L\amalg R$ be a smooth map. Since $f$ is connected,
it must map either to $L$, or to some unique $(p_-,p_+)$ in $R$.
Define the {\it shifted cohomological degree\/} of $f:\De_a\ra
L\amalg R$ to be
\begin{equation}
\deg f=\begin{cases} n-a-1, & f(\De_a)\subseteq L,\\
\eta_{(p_-,p_+)}-a-1, & f(\De_a)=\{(p_-,p_+)\}\subset R.
\end{cases}
\label{il4eq13}
\end{equation}
In effect, we are defining a {\it new grading\/} on the simplicial
chains $C_*^\rsi(L\amalg R;\Q)=C_*^\rsi(L;\Q)\op
\bigop_{(p_-,p_+)\in R}C_*^\rsi(\{(p_-,p_+)\};\Q)$, such that $\deg
C_a^\rsi(L;\Q)=n-a-1$ and $\deg C_a^\rsi(\{(p_-,p_+)\};\Q)=
\eta_{(p_-,p_+)}-a-1$.

Note that our notation differs from that of Fukaya et al.\
\cite{FOOO} in the embedded case. Fukaya et al.\ define the {\it
cohomological degree} of $f:\De_a\ra L$ in $C_a^\rsi(L;\Q)$ to be
$\deg f=n-a$, that is, $\deg f$ is in effect the codimension of
$f(\De_a)$ in $L$. But then they work throughout with the {\it
shifted complex} $C_*^\rsi(L;\Q)[1]$ in which $f$ has grading $\deg'
f=\deg f-1$, as in \cite[\S 7.1]{FOOO}. So our $\deg f$ corresponds
to Fukaya et al.'s shifted degree $\deg'f$, which is why we call it
the {\it shifted\/} cohomological degree.

We prefer this convention as it simplifies many of the dimensions
and signs expressed in terms of $\deg f_1,\ldots,\deg f_n$ below,
and also the shifted complexes $C_*^\rsi(L;\Q)[1],\Q\X[1]$ which are
ubiquitous in \cite{FOOO} are replaced below by unshifted complexes
$C_*^\rsi(L;\Q),\Q\X$, simplifying the notation. We undo the shift
when we define Lagrangian Floer cohomology in \eq{il13eq3}. We will
explain the reason for grading $f:\De_a\ra\{(p_-,p_+)\}$ by $\deg
f=\eta_{(p_-,p_+)}-a-1$ in Definition~\ref{il4dfn9}.

Observe that $\oM_{k+1}^\ma(\al,\be,J,f_1,\ldots,f_k)=\emptyset$
unless $f_i:\De_{a_i}\ra L\amalg R$ maps to $L$ if $i\notin I$, and
to $\al(i)\in R$ if $i\in I$. Then combining \eq{il4eq12} and
\eq{il4eq13} yields
\begin{equation}
\begin{aligned}
&\vdim\oM_{k+1}^\ma(\al,\be,J,f_1,\ldots,f_k)=\\
&\begin{cases}
\mu_L(\be)-2+n-\ts\sum_{i=1}^k\deg f_i, & 0\notin I,\\
\mu_L(\be)-2+n-\ts\sum_{i=1}^k\deg f_i-\eta_{\al(0)}, & 0\in I.
\end{cases}
\end{aligned}
\label{il4eq14}
\end{equation}
This also holds trivially in the other cases, as
then~$\oM_{k+1}^\ma(\al,\be,J,f_1,\ldots,f_k)=\emptyset$.

From \eq{il2eq5} and \eq{il4eq4}, $\pd\oM_{k+1}^\ma
(\al,\be,J,f_1,\ldots,f_k)$ is given without orientations by
\begin{equation}
\begin{aligned}
&\coprod_{i=1}^k\coprod_{j=0}^{a_i}\oM_{k+1}^\ma
(\al,\be,J,f_1,\ldots,f_{i-1},f_i\ci F_j^{a_i},f_{i+1},\ldots,f_k) \\
&\amalg\pd\oM_{k+1}^\ma(\al,\be,J)\t_{\bev_1\t\cdots\t\bev_k,(L\amalg
R)^k, f_1\t\cdots\t f_k}(\De_{a_1}\t\cdots\t\De_{a_k}),
\end{aligned}
\label{il4eq15}
\end{equation}
where $F_j^{a_i}:\De_{a_i-1}\ra\De_{a_i}$ is as in~\S\ref{il26}.

Write $\oM_{k_1+1}^\ma(\al_1,\be_1,J,f_1,\ldots,f_{i-1};f_{i+k_2},
\ldots,f_k)$ for the fibre product
\begin{equation}
\begin{aligned}
&\oM_{k_1+1}^\ma(\al_1,\be_1,J)\t_{\bev_1\t\cdots\t\bev_{i-1}\t
\bev_{i+1}\t\cdots\t\bev_{k_1},(L\amalg R)^{k_1-1},} \\
&{}_{f_1\t\cdots\t f_{i-1}\t f_{i+k_2}\t\cdots\t f_k}
(\De_{a_1}\t\cdots\t \De_{a_{i-1}}\t\De_{a_{i+k_2}}\t\cdots\t
\De_{a_k}),
\end{aligned}
\label{il4eq16}
\end{equation}
where $k_1+k_2=k+1$. Then as for \eq{il4eq14} we calculate that
\begin{equation}
\begin{split}
&\vdim\oM_{k_1+1}^\ma(\al_1,\be_1,J,f_1,\ldots,f_{i-1};
f_{i+k_2},\ldots,f_k)= \\
&\begin{cases} \mu_L(\be_1)-1+n-\ts\sum_{j=1}^{i-1}\deg
f_j-\sum_{j=i+k_2}^k \deg f_j, & \hbox to 0pt{\hss
$0,i\!\notin\!I_1$,}\\
\mu_L(\be_1)-1+n-\ts\sum_{j=1}^{i-1}\deg f_j-\sum_{j=i+k_2}^k \deg
f_j-\eta_{\al_1(0)}, & \hbox to 0pt{\hss $0\!\in\!I_1$,
$i\notin I_1$,}\\
\mu_L(\be_1)-1+n-\ts\sum_{j=1}^{i-1}\deg f_j-\sum_{j=i+k_2}^k \deg
f_j-\eta_{\al_1(i)}, & \hbox to 0pt{\hss $0\!\notin\!I_1$,
$i\in I_1$,}\\
\mu_L(\be_1)\!-\!1\!+\!n-\!\ts\sum_{j=1}^{i-1}\deg
f_j\!-\!\sum_{j=i+k_2}^k\deg f_j-\eta_{\al_1(0)}-\eta_{\al_1(i)},
\,\,\,\,\,\,\,\,\,\,\,\,\,\,\,\,\,\,\,\, &\hbox to 0pt{\hss
$0,i\!\in\!I_1$,}
\end{cases}
\end{split}
\label{il4eq17}
\end{equation}

Combining \eq{il4eq4}, \eq{il4eq15} and \eq{il4eq16} shows that
\begin{equation}
\begin{split}
&\pd\oM_{k+1}^\ma(\al,\be,J,f_1,\ldots,f_k)\cong \\
&\coprod_{i=1}^k\coprod_{j=0}^{a_i}
\oM_{k+1}^\ma(\al,\be,J,f_1,\ldots,f_{i-1},f_i\ci F_j^{a_i},
f_{i+1},\ldots,f_k)\\
&\amalg
\coprod_{\begin{subarray}{l}
k_1+k_2=k+1,\;1\le i\le k_1, \\
I_1\cup_iI_2=I,\;\al_1\cup_i\al_2=\al,\\
\be_1+\be_2=\be
\end{subarray}}
\begin{aligned}[t]
&\oM_{k_2+1}^\ma (\al_2,\be_2,J,f_i,\ldots,f_{i+k_2-1})
\t_{\bev,L\amalg R,\bev_i}\\
&\qquad \oM_{k_1+1}^\ma (\al_1,\be_1,J,f_1,\ldots,f_{i-1};
f_{i+k_2},\ldots,f_k),
\end{aligned}
\end{split}
\label{il4eq18}
\end{equation}
in unoriented Kuranishi spaces.

As for \eq{il4eq3}, we shall also write
\begin{equation}
\oM_{k+1}^\ma(\be,J,f_1,\ldots,f_k)=\coprod\nolimits_{
\begin{subarray}{l}I\subseteq\{0,\ldots,k\},\\ \al:I\ra
R\end{subarray}}\oM_{k+1}^\ma(\al,\be,J,f_1,\ldots,f_k).
\label{il4eq19}
\end{equation}
Again, this is a disjoint union of Kuranishi spaces of different
dimensions. We define a strong submersion
$\bev:\oM_{k+1}^\ma(\be,J,f_1,\ldots,f_k)\ra L\amalg R$ to be $\bev$
on each component~$\oM_{k+1}^\ma (\al,\be,J,f_1,\ldots,f_k)$.
\label{il4dfn5}
\end{dfn}

\subsection{Adding families of almost complex structures}
\label{il45}

We can generalize all the material above to {\it smooth families\/}
of almost complex structures $J_t$ for $t\in{\cal T}$, with $\cal T$
a smooth manifold. We will need this in \S\ref{il8}--\S\ref{il9}
with ${\cal T}=[0,1]$, and in \S\ref{il10} with ${\cal T}$ a
semicircle $S$ and a triangle~$T$.

\begin{dfn} Suppose $(M,\om)$ is a compact $2n$-dimensional
symplectic manifold, $\cal T$ an oriented smooth manifold, which may
be noncompact and may have boundary and corners, and $J_t$ for
$t\in{\cal T}$ a smooth family of almost complex structures on $M$
compatible with $\om$. Let $\io:L\ra M$ be a compact Lagrangian
immersion. Suppose that all the self-intersection points of the
immersion $\io$ are transverse double self-intersections.

Generalizing Definition \ref{il4dfn2} and using the same notation,
define $\oM_{k+1}^\ma(\al,\be,J_t:t\in{\cal T})$ to be the set of
$\bigl(t,[\Si,\vec z,u,l,\bar{u}]\bigr)$ for $t\in{\cal T}$ and
$[\Si,\vec z,u,l,\bar{u}]\in\oM_{k+1}^\ma(\al,\be,J_t)$. Define
$\pi_{\cal T}:\oM_{k+1}^\ma(\al,\be,J_t:t\in{\cal T})\ra{\cal T}$ by
$\pi_{\cal T}:\bigl(t,[\Si,\vec z,u,l,\bar{u}]\bigr)\mapsto t$ and
$\ev_i,\ev:\oM_{k+1}^\ma(\al,\be,J_t:t\in{\cal T})\!\ra\!L\!\amalg\!
R$ by $\ev_i,\ev:\bigl(t,[\Si,\vec
z,u,l,\bar{u}]\bigr)\!\mapsto\!\ev_i,\ev \bigl([\Si,\vec
z,u,l,\bar{u}]\bigr)$.

As for the case of $\oM_{k+1}^\ma(\al,\be,J)$ in \S\ref{il41}, we
may define a natural, Hausdorff topology on $\oM_{k+1}^\ma
(\al,\be,J_t:t\in{\cal T})$ called the $C^\iy$ {\it topology}, such
that $\pi_{\cal T},\ev_i,\ev$ are continuous. If $\cal T$ is compact
then $\oM_{k+1}^\ma (\al,\be,J_t:t\in{\cal T})$ is compact.

We can then define a {\it Kuranishi structure} on $\oM_{k+1}^\ma
(\al,\be,J_t:t\in{\cal T})$, with boundary and corners and a tangent
bundle, and $\pi_{\cal T},\ev_i,\ev$ extend to {\it strong
submersions\/} $\bs\pi_{\cal T},\bev_i,\bev$. For each $t'\in{\cal
T}$ there is an isomorphism of Kuranishi spaces
\begin{equation}
\oM_{k+1}^\ma(\al,\be,J_{t'})\cong \{t'\}\t_{\io,{\cal
T},\bs\pi_{\cal T}}\oM_{k+1}^\ma(\al,\be,J_t:t\in{\cal T}),
\label{il4eq20}
\end{equation}
where $\io:\{t'\}\ra{\cal T}$ is the inclusion, and the right hand
side is a fibre product of Kuranishi spaces, which is well-defined
as $\bs\pi_{\cal T}$ is a strong submersion.

There is one subtle point here: the Kuranishi structures on each
side depend on choices made during the constructions, and
\eq{il4eq20} holds provided the choices made in defining the
Kuranishi structures on $\oM_{k+1}^\ma(\al,\be,J_{t'})$ and
$\oM_{k+1}^\ma(\al,\be,J_t:t\in{\cal T})$ are compatible. If ${\cal
T}=[0,1]$ then for any allowed choices of Kuranishi structures on
$\oM_{k+1}^\ma(\al,\be,J_0)$ and $\oM_{k+1}^\ma(\al,\be,J_1)$, we
can choose the Kuranishi structure on $\oM_{k+1}^\ma(\al,\be,J_t:
t\in{\cal T})$ so that \eq{il4eq20} holds when $t'=0,1$. We will
usually suppress this issue of needing to make compatible choices of
Kuranishi structures.
\label{il4dfn6}
\end{dfn}

Here are the generalizations of Theorem \ref{il4thm1} and
Proposition~\ref{il4prop1}.

\begin{thm} In the situation of Definition \ref{il4dfn6}, there is
an isomorphism of unoriented Kuranishi spaces:
\begin{gather}
\begin{split}
&\pd\oM^\ma_{k+1}(\al,\be,J_t:t\in{\cal T})\cong
\oM^\ma_{k+1}(\al,\be,J_t:t\in\pd{\cal T})\,\amalg\\
&\coprod_{\begin{subarray}{l}
k_1+k_2=k+1,\; 1\le i\le k_1,\\
I_1\cup_iI_2=I,\;\al_1\cup_i\al_2=\al,\\
\be_1+\be_2=\be
\end{subarray}}\!\!\!\!\!
\begin{aligned}[t]
\oM^\ma_{k_2+1}(\al_2,\be_2,J_t:t\in{\cal T})\t_{\bs\pi_{\cal
T}\t\bev,{\cal T}\t(L\amalg R),\bs\pi_{\cal T}\t\bev_i}&\\
\oM^\ma_{k_1+1}(\al_1,\be_1,J_t:t\in{\cal T}), \qquad \text{and\/}&
\end{aligned}
\end{split}
\label{il4eq21}\\
\vdim\oM_{k+1}^\ma(\al,\be,J_t:t\in{\cal T})=\mu_L(\be)+
k-2+n-\ts\sum_{i\in I}\eta_{\al(i)}+\dim{\cal T}.
\label{il4eq22}
\end{gather}
\label{il4thm2}
\end{thm}

We can also add smooth simplicial chains, following Definition
\ref{il4dfn5}. The obvious way to do this is to start with
$f_i:\De_{a_i}\ra L\amalg R$ for $i=1,\ldots,k$, and take the fibre
product $\oM^\ma_{k+1}(\al,\be,J_t:t\in{\cal T})\t_{\bev_1\t\cdots\t
\bev_k,(L\amalg R)^k,f_1\t\cdots\t f_k}(\De_{a_1}\t\cdots\t
\De_{a_k})$ as in \eq{il4eq11}. But for our later purposes we need
to do something different: we use simplicial chains on ${\cal
T}\t(L\amalg R)$, so that $f_i$ maps $\De_{a_i}\ra{\cal T}\t(L\amalg
R)$, and then we define $\oM_{k+1}^\ma(\al,\be,J_t:t\in{\cal
T},f_1,\ldots,f_k)$ by a fibre product over $({\cal T}\t(L\amalg
R))^k$. Thus, roughly speaking we want to write
\begin{equation}
\begin{split}
\oM_{k+1}^\ma(\al,&\be,J_t:t\in{\cal T},f_1,\ldots,f_k)=
\oM_{k+1}^\ma(\al,\be,J_t:t\in{\cal T})\\
&\t_{(\bs\pi_{\cal T}\t\bev_1)\t\cdots\t(\bs\pi_{\cal
T}\t\bev_k),({\cal T}\t(L\amalg R))^k, f_1\t\cdots\t
f_k}(\De_{a_1}\t\cdots\t\De_{a_k}).
\end{split}
\label{il4eq23}
\end{equation}

However, there is a problem with \eq{il4eq23}. Although
$\bs\pi_{\cal T}\t\bev_1\t\cdots\t\bev_k:\oM_{k+1}^\ma
(\al,\ab\be,J_t:t\in{\cal T})\ra{\cal T}\t (L\amalg R)^k$ is a
strong submersion, if $\dim{\cal T}>0$ and $k>1$ then $(\bs\pi_{\cal
T}\t\bev_1)\t\cdots\t(\bs\pi_{\cal T}\t\bev_k):\oM_{k+1}^\ma
(\al,\be,J_t:t\in{\cal T})\ra({\cal T}\t(L\amalg R))^k$ is {\it
not\/} a strong submersion, as it does not locally map onto ${\cal
T}^k$, but only onto the diagonal $\bigl\{(t,\ldots,t)\in{\cal
T}^k:t\in{\cal T}\bigr\}$. Since $f_1\t\cdots\t f_k$ may also not be
a strong submersion, the fibre product in \eq{il4eq23} is not
well-defined.

We fix this by including an extra factor in the fibre product, which
modifies the Kuranishi structures and makes the strongly smooth maps
into strong submersions. The same problem holds for the moduli
spaces $\M_{k+1}^\ma(M',L',\{J_{1,s}\}_s:\be;{\rm twp}(x);\vec{\cal
P})$ in Fukaya et al.\ \cite[\S 19.2]{FOOO}, but appears to the
authors to have been overlooked.

\begin{dfn} First suppose for simplicity that ${\cal T}$ is of dimension
$m$ and embedded in $\R^m$. For $k\ge 0$, define a new Kuranishi
structure $\ka^m_k$ on $\R^m$ by the global Kuranishi neighbourhood
$(V^m_k,E^m_k,s^m_k,\psi^m_k)$, where $V^m_k\!=\!(\R^m)^{k+1}$, and
$E^m_k\!=\!(\R^m)^{k+1}\t(\R^m)^k$, the trivial vector bundle over
$V^m_k$ with fibre $(\R^m)^k$. Define $s^m_k:V^m_k\!\ra\!E^m_k$ by
$s^m_k:(\bs v_0,\ldots,\bs v_k)\!\mapsto\!\bigl((\bs v_0,\ldots,\bs
v_k),(\bs v_1-\bs v_0,\ldots,\bs v_k-\bs v_0)\bigr)$, for $\bs
v_0,\ldots,\bs v_k\in\R^m$. Then $(s^m_k)^{-1}(0)\!=\!\bigl\{(\bs
v,\ldots,\bs v)\in(\R^m)^{k+1}:\bs v\in\R^m\bigr\}$. Define
$\psi^m_k:(s^m_k)^{-1}(0)\!\ra\!\R^m$ by $\psi^m_k:(\bs v,\ldots,\bs
v)\!\mapsto\!\bs v$. Define $\pi_i:V^m_k\!\ra\!\R^m$ for
$i\!=\!0,\ldots,k$ by $\pi_i:(\bs v_0,\ldots,\bs v_k)\!\mapsto\!\bs
v_i$. Then $\pi_i$ represents a strongly smooth map
$\bs\pi_i:(\R^m,\ka^m_k)\!\ra\!\R^m$, with
$\bs\pi_0\!\t\!\cdots\!\t\!\bs\pi_k:(\R^m,\ka^m_k)\!\ra\!(\R^m)^{k+1}$
a strong submersion.

Now for $i=1,\ldots,k$, let $a_i\ge 0$ and $f_i:\De_{a_i}\ra {\cal
T}\t(L\amalg R)$ be a smooth map. Define the Kuranishi space
\begin{align}
\oM_{k+1}^\ma(\al,&\be,J_t:t\!\in\!{\cal T},f_1,\ldots,f_k)\!=\!
\bigl((\R^m,\ka^m_k)\t_{\bs\pi_0,\R^m,\bs\pi_{\cal T}}
\oM_{k+1}^\ma(\al,\be,J_t:t\!\in\!{\cal T})\bigr)
\nonumber\\
&\t_{(\bs\pi_1\t\bev_1)\t\cdots\t(\bs\pi_k\t\bev_k),({\cal
T}\t(L\amalg R))^k, f_1\t\cdots\t
f_k}(\De_{a_1}\t\cdots\t\De_{a_k}).
\label{il4eq24}
\end{align}
Unlike \eq{il4eq23}, this is well-defined, as $\bs\pi_0$ and
$(\bs\pi_1\t\bev_1)\!\t\!\cdots\!\t\!(\bs\pi_k\t\bev_k)$ are strong
submersions. Also, one can show the Kuranishi structure of
$(\R^m,\ka^m_k)$ is unchanged by diffeomorphisms of $\R^m$. Thus, by
composing the embedding ${\cal T}\hookra\R^m$ with a diffeomorphism
of $\R^m$, we see that the Kuranishi structure of
$\oM_{k+1}^\ma(\al,\be,J_t:t\!\in\!{\cal T},f_1,\ldots,f_k)$ is
locally independent of the choice of embedding of $\cal T$ in
$\R^m$. In fact, since the Kuranishi structure depends only locally
on ${\cal T}\hookra\R^m$, and any $\cal T$ can be locally embedded
in $\R^m$, the Kuranishi structure of $\oM_{k+1}^\ma(\al,
\be,J_t:t\!\in\!{\cal T},f_1,\ldots,f_k)$ is well-defined even if
$\cal T$ cannot be globally embedded in~$\R^m$.

As for \eq{il4eq12}, but using \eq{il4eq22}, \eq{il4eq24} and
$\vdim(\R^m,\ka_k^m)=m$, we see that
\begin{equation}
\begin{split}
&\vdim\oM_{k+1}^\ma(\al,\be,J_t:t\in{\cal T},f_1,\ldots,f_k)=
(1-k)\dim{\cal T}+\\
&\mu_L(\be)+k-2+n-\ts\sum_{i\in I}\eta_{\al(i)}+\ts\sum_{0\ne
i\notin I}(a_i-n)+\ts\sum_{0\ne i\in I}a_i.
\end{split}
\label{il4eq25}
\end{equation}

As in \S\ref{il44}, it is convenient to rewrite this using a notion
of {\it shifted cohomological degree}. Let $f:\De_a\ra {\cal
T}\t(L\amalg R)$ be a smooth map. Generalizing \eq{il4eq13}, define
\begin{equation}
\deg f=\begin{cases} \dim{\cal T}+n-a-1, & f(\De_a)\subseteq
{\cal T}\t L,\\
\dim{\cal T}\!+\!\eta_{(p_-,p_+)}\!-\!a\!-\!1, &
f(\De_a)\!\subseteq\!{\cal T}\!\t\!\{(p_-,p_+)\},\; (p_-,p_+)\in R.
\end{cases}
\label{il4eq26}
\end{equation}
Then combining \eq{il4eq25} and \eq{il4eq26} yields a generalization
of~\eq{il4eq14}:
\begin{equation}
\begin{aligned}
&\vdim\oM_{k+1}^\ma(\al,\be,J_t:\in{\cal T},f_1,\ldots,f_k)=\\
&\begin{cases}
\mu_L(\be)-2+\dim{\cal T}+n-\ts\sum_{i=1}^k\deg f_i, & 0\notin I,\\
\mu_L(\be)-2+\dim{\cal T}+n-\ts\sum_{i=1}^k\deg f_i-\eta_{\al(0)}, &
0\in I.
\end{cases}
\end{aligned}
\label{il4eq27}
\end{equation}
This illustrates something we will see in \S\ref{il55}, that to
generalize from one complex structure $J$ to a family $J_t:t\in{\cal
T}$, in dimensions or signs we usually change $n$ to $\dim{\cal
T}+n$, and make no other changes.

Write $\oM_{k_1+1}^\ma(\al_1,\be_1,J_t:t\in{\cal T},f_1,\ldots,
f_{i-1};f_{i+k_2}, \ldots,f_k)$ for the fibre product
\begin{equation}
\begin{aligned}
&\bigl((\R^m,\ka^m_{k_1})\t_{\bs\pi_0,\R^m,\bs\pi_{\cal T}}
\oM_{k_1+1}^\ma(\al_1,\be_1,J_t:t\in{\cal
T})\bigr)\\
&\quad\t_{(\bs\pi_1\t\bev_1)\t\cdots(\bs\pi_{i-1}\t\bev_{i-1})
\t(\bs\pi_{i+1}\t\bev_{i+1})\t\cdots\t(\bs\pi_{k_1}\t\bev_{k_1}),
({\cal T}\t (L\amalg R))^{k_1-1},} \\
&\quad{}_{f_1\t\cdots\t f_{i-1}\t f_{i+k_2}\t\cdots\t f_k}
(\De_{a_1}\!\t\!\cdots\!\t\!
\De_{a_{i-1}}\!\t\!\De_{a_{i+k_2}}\!\t\!\cdots\!\t\!\De_{a_k}).
\end{aligned}
\label{il4eq28}
\end{equation}
Its virtual dimension is given by the sum of \eq{il4eq17} with
$\dim{\cal T}$. As for \eq{il4eq18} but using \eq{il4eq21}, and
requiring $k>0$, we find that
\begin{equation}
\begin{split}
&\pd\oM_{k+1}^\ma(\al,\be,J_t:t\in{\cal T},f_1,\ldots,f_k)\cong \\
&\coprod_{i=1}^k\coprod_{j=0}^{a_i}
\oM_{k+1}^\ma(\al,\be,J_t:t\in{\cal T},f_1,\ldots,f_{i-1},f_i\ci
F_j^{a_i},
f_{i+1},\ldots,f_k)\\
&\amalg \!\!\!\coprod_{\begin{subarray}{l}
k_1+k_2=k+1,\;1\le i\le k_1, \\
I_1\cup_iI_2=I,\;\al_1\cup_i\al_2=\al,\\
\be_1+\be_2=\be
\end{subarray}\!\!\!\!\!\!\!\!\!\!\!\!\!\!\!\!\!\!\!\!\!\!\!\!\!}
\!\!\!\!\!
\begin{aligned}[t]
\oM_{k_2+1}^\ma (\al_2,\be_2,J_t:t\!\in\!{\cal
T},f_i,\ldots,f_{i+k_2-1})\!\t_{\bs\pi_0\t\bev,
{\cal T}\t(L\amalg R),\bs\pi_i\t\bev_i}&\\
\oM_{k_1+1}^\ma (\al_1,\be_1,J_t:t\in{\cal T},f_1,\ldots,f_{i-1};
f_{i+k_2},\ldots,f_k)&,
\end{aligned}
\end{split}
\label{il4eq29}
\end{equation}
in unoriented Kuranishi spaces. Here from \eq{il4eq24}, the first
line of \eq{il4eq29} involves a fibre product with $(\R^m,\ka^m_k)$,
but the third line involves fibre products with $(\R^m,\ka^m_{k_1})$
and $(\R^m,\ka^m_{k_2})$. To match these up, we construct an
explicit isomorphism of Kuranishi spaces~$(\R^m,\ka^m_k)\cong(\R^m,
\ka^m_{k_2})\t_{\bs\pi_0,\R^m,\bs\pi_i}(\R^m,\ka^m_{k_1})$.

Note that unlike \eq{il4eq21}, as $k>0$, there are no special
contributions to \eq{il4eq29} from the boundary $\pd{\cal T}$. As
for \eq{il4eq19}, we shall also write
\begin{equation}
\oM_{k+1}^\ma(\be,J_t:t\in{\cal T},f_1,\ldots,f_k)=
\!\!\!\coprod_{I\subseteq\{0,\ldots,k\},\; \al:I\ra
R\!\!\!\!\!\!\!\!\!\!\!\!\!\!\!\!\!\!\!\!\!\!\!\!\!}\!\!
\oM_{k+1}^\ma(\al,\be,J_t:t\in{\cal T},f_1,\ldots,f_k).
\label{il4eq30}
\end{equation}
Again, this is a disjoint union of Kuranishi spaces of different
dimensions.
\label{il4dfn7}
\end{dfn}

\begin{rem} In \S\ref{il8} and \S\ref{il10} the following question
will be important. Suppose $\cal T$ has boundary $\pd{\cal T}$, and
for each $i=1,\ldots,k$ we have smooth $f_i:\De_{a_i}\ra{\cal
T}\t(L\amalg R)$ such that for some $b_i=0,\ldots,a_i$, the boundary
map $g_i=f_i\ci F_{b_i}^{a_i}:\De_{a_i-1}\ra{\cal T}\t(L\amalg R)$
maps to $\pd{\cal T}\t(L\amalg R)$, and that $f_i$ maps
$\De_{a_i}\sm F_{b_i}^{a_i}(\De_{a_i-1})$ to ${\cal T}^\ci\t(L\amalg
R)$, where ${\cal T}^\ci$ is the interior of $\cal T$. Then, what is
the relation between $\oM_{k+1}^\ma(\al,\be,J_t:t\in{\cal
T},f_1,\ldots,f_k)$ and~$\oM_{k+1}^\ma(\al,\be,J_t:t\in\pd{\cal
T},g_1,\ldots,g_k)$?

The answer is complicated, because if we locally embed ${\cal
T}\hookra\R^m$ such that $\pd {\cal T}\hookra\R^{m-1}$, then the
definition \eq{il4eq24} of $\oM_{k+1}^\ma(\al,\be,J_t:t\in{\cal
T},f_1,\ldots,f_k)$ involves $(\R^m,\ka^m_k)$, but for
$\oM_{k+1}^\ma(\al,\be,J_t:t\in\pd{\cal T},g_1,\ldots,g_k)$ it
involves $(\R^{m-1},\ka^{m-1}_k)$. To give a satisfactory relation
we need to impose an extra {\it transversality condition} for
$f_1,\ldots,f_k$ over~$\pd {\cal T}$:

\begin{cond} Assume that\/ $\pi_{\cal T}\ci f_i:\De_{a_i}\ra{\cal T}$
is transverse to $\pd{\cal T}$ along $F_{b_i}^{a_i}(\De_{a_i-1})$
for each\/ $i=1,\ldots,k$. That is, for each\/ $p\in
F_{b_i}^{a_i}(\De_{a_i-1})$ we require that\/~$\d(\pi_{\cal T}\ci
f_i)(T_p\De_{a_i})+T_{\pi_{\cal T}\ci f_i(p)}(\pd{\cal
T})=T_{\pi_{\cal T}\ci f_i(p)}{\cal T}$.
\label{il4cond}
\end{cond}

Supposing that ${\cal T}$ is embedded in $\R^m$ such that $\pd {\cal
T}$ is embedded in $\R^{m-1}\subset\R^m$ locally, and using
Condition \ref{il4cond}, we have isomorphisms of Kuranishi spaces
\begin{gather}
\pd{\cal T}\t_{i,\R^m,\bs\pi_0}\bigl(
(\R^m,\ka_k^m)\t_{\bs\pi_1\t\cdots\t\bs\pi_k,(\R^m)^k,(\pi_{\cal
T}\ci f_1)\t\cdots\t(\pi_{\cal T}\ci f_k})\bigl(\De_{a_1}\t\cdots\t
\De_{a_k}\bigr)\cong
\nonumber\\
\bigl((\R^{m-1},\ka_k^{m-1})\t_{\bs\pi_1\t\cdots\t\bs\pi_k,
(\R^{m-1})^k,(\pi_{\pd{\cal T}}\ci g_1)\t\cdots\t(\pi_{\pd{\cal
T}}\ci g_k)}\bigl(\De_{a_1-1}\t\cdots\t\De_{a_k-1}\bigr)\bigr)
\nonumber\\
\t\bigl[\{0\}\t_{i,\R,\bs\pi_0}\bigl((\R,\ka_k^1)\t_{\bs\pi_1\t
\cdots\t\bs\pi_k,\R^k,i}[0,\iy)^k\bigr)\bigr],
\label{il4eq31}\\
(\R^m,\ka_k^m)\t_{\bs\pi_1\t\cdots\t\bs\pi_k,(\R^m)^k,(\pi_{\cal
T}\ci f_1)\t\cdots\t(\pi_{\cal T}\ci f_{j-1})\t(\pi_{\cal T}\ci
g_j)\t(\pi_{\cal T}\ci f_{j+1})\t\cdots\t(\pi_{\cal T}\ci f_k)}
\nonumber\\
\bigl(\De_{a_1}\t\cdots\t
\De_{a_{j-1}}\t\De_{a_j-1}\t\De_{a_{j+1}}\t\cdots\t\De_{a_k}\bigr)\cong
\nonumber\\
\bigl((\R^{m-1},\ka_k^{m-1})\t_{\bs\pi_1\t\cdots\t\bs\pi_k,(\R^{m-1})^k,
(\pi_{\pd{\cal T}}\ci g_1)\t\cdots\t(\pi_{\pd{\cal T}}\ci g_k)}
\bigl(\De_{a_1-1}\t\cdots\t\De_{a_k-1}\bigr)\bigr)
\nonumber\\
\t\bigl[(\R,\ka_k^1)\t_{\bs\pi_1\t\cdots\t
\bs\pi_k,\R^k,i}[0,\iy)^{j-1}\t\{0\}\t[0,\iy)^{k-j}\bigr],
\label{il4eq32}
\end{gather}
for $j=1,\ldots,k$, where $i$ denotes inclusion maps. To prove
\eq{il4eq31} and \eq{il4eq32}, we use the isomorphism
$(\R^m,\ka^m_k)\cong(\R^{m-1},\ka^{m-1}_k)\t (\R,\ka^1_k)$ and the
isomorphism $\De_{a_j}\cong\De_{a_j-1}\t[0,\iy)$ near
$F_{b_j}^{a_j}(\De_{a_j-1})$. Condition \ref{il4cond} ensures that
the factor $[0,\iy)$ in $\De_{a_j}\cong\De_{a_j-1}\t[0,\iy)$ locally
submerses to the factor $\R$ in~$\R^m\cong\R^{m-1}\t\R$.

Equations \eq{il4eq24}, \eq{il4eq31}, \eq{il4eq32} and properties of
fibre products yield isomorphisms
\begin{align}
\begin{split}
&\pd{\cal T}\t_{i,{\cal
T},\bs\pi_0}\oM_{k+1}^\ma(\al,\be,J_t:t\in{\cal T},
f_1,\ldots,f_k)\\
&\quad\cong \oM_{k+1}^\ma(\al,\be,J_t:t\in\pd{\cal T},g_1,\ldots,g_k)\\
&\qquad\t\bigl[\{0\}\t_{i,\R,\bs\pi_0}\bigl((\R,\ka_k^1)\t_{\bs\pi_1\t\cdots\t
\bs\pi_k,\R^k,i}[0,\iy)^k\bigr)\bigr],
\end{split}
\label{il4eq33}\\
\begin{split}
&\oM_{k+1}^\ma(\be,J_t:t\in{\cal T},
f_1,\ldots,f_{j-1},g_j,f_{j+1},\ldots,f_k)\\
&\quad\cong
\oM_{k+1}^\ma(\al,\be,J_t:t\in\pd{\cal T},g_1,\ldots,g_k)\\
&\qquad\t\bigl[(\R,\ka_k^1)\t_{\bs\pi_1\t\cdots\t
\bs\pi_k,\R^k,i}[0,\iy)^{j-1}\t\{0\}\t[0,\iy)^{k-j}\bigr],
\end{split}
\label{il4eq34}
\end{align}
for all $j=1,\ldots,k$. These are the relations we seek between
$\oM_{k+1}^\ma(\al,\be,J_t:t\in{\cal T},f_1,\ldots,f_k)$
and~$\oM_{k+1}^\ma(\al,\be,J_t:t\in\pd{\cal T},g_1,\ldots,g_k)$.

Note that the third lines of \eq{il4eq33} and \eq{il4eq34} are each
a point $\{0\}$ with an unusual Kuranishi structure, of virtual
dimension 0. Since the Kuranishi maps of these Kuranishi structures
are already transverse, when we choose {\it perturbation data} as in
\S\ref{il27}, they do not need to be perturbed. Hence, from
\eq{il4eq34}, a choice of perturbation data for $\oM_{k+1}^\ma(\al,
\be,J_t:t\in\pd{\cal T},g_1,\ldots,g_k)$ determines perturbation
data for $\oM_{k+1}^\ma(\be,J_t:t\in{\cal T},
f_1,\ldots,f_{j-1},g_j,f_{j+1},\ldots,f_k)$, which have the same
virtual chains. This will enable us to relate $A_{N,K}$ algebras of
singular chains on ${\cal T}\t(L\amalg R)$ to $A_{N,K}$ algebras of
singular chains on $\pd{\cal T}\t(L\amalg R)$
in~\S\ref{il8}--\S\ref{il10}.
\label{il4rem}
\end{rem}

\subsection{Modified moduli spaces $\tM_{k+1}^\ma(\al,\be,J)$}
\label{il46}

We will see in \S\ref{il5} that defining and computing with
orientations on the moduli spaces $\oM_{k+1}^\ma(\al,\be,J)$ is
rather complicated. This is mostly to do with the r\^ole of the
operators $\bar{\pd}_{\la_{(p_-,p_+)}}$. We will now define
modified, noncompact spaces $\tM_{k+1}^\ma(\al,\be,J)$ whose
dimensions and orientations behave in a simpler, more natural way.
To compute the sign in some orientation problem for the
$\oM_{k+1}^\ma(\cdots)$, it is usually simpler to first work out the
answer for the $\tM_{k+1}^\ma(\cdots)$. Also, the
$\tM_{k+1}^\ma(\cdots)$ provide geometric explanations for the
notions of {\it grading\/} and {\it shifted cohomological degree\/}
introduced in~\S\ref{il44}.

\begin{dfn} In Definition \ref{il4dfn3}, suppose that the
families $\la_{(p_-,p_+)}$ for all $(p_-,p_+)$ in $R$ have been
chosen such that $\eta_{(p_-,p_+)}\ge 0$, and $\la_{(p_-,p_+)}$ is
generic. This genericity implies that $\dim\Ker\bar{\pd}_{
\la_{(p_-,p_+)}}$ and $\dim\Coker\bar{\pd}_{\la_{(p_-, p_+)}}$ are
both as small as possible, so $\dim\Ker
\smash{\bar{\pd}_{\la_{(p_-,p_+)}}}=\eta_{(p_-,p_+)}$ and
$\dim\Coker\smash{\bar{\pd}_{ \la_{(p_-,p_+)}}}=0$, since
$\eta_{(p_-,p_+)}=\ind\smash{\bar{\pd}_{\la_{(p_-,p_+)}}}\ge 0$.

Consider the linear map $\ev_{(-1,0)}:\Ker\bar{\pd}_{
\la_{(p_-,p_+)}}\ra\la_{(p_-,p_+)}(-1,0)$ mapping
$\ev_{(-1,0)}:\xi\mapsto \xi(-1,0)$. We have
$\dim\Ker\bar{\pd}_{\la_{(p_-,p_+)}}\le
n=\dim\la_{(p_-,p_+)}(-1,0)$, since $0\le\eta_{(p_-,p_+)}\le n$ by
\eq{il4eq9}. Thus, genericness implies that $\ev_{(-1,0)}$ is
injective, so $\ev_{(-1,0)}\bigl(\Ker\bar{\pd}_{
\la_{(p_-,p_+)}}\bigr)$ is a vector subspace of
$\la_{(p_-,p_+)}(-1,0)$ of dimension~$\eta_{(p_-,p_+)}$.

As $\la_{(p_+,p_-)}(x,y)\equiv\la_{(p_-,p_+)}(x,-y)$ we have
$\la_{(p_+,p_-)}(-1,0)=\la_{(p_-,p_+)}(-1,0)$. Hence
$\ev_{(-1,0)}\bigl(\Ker\bar{\pd}_{ \la_{(p_-,p_+)}}\bigr)$ and
$\ev_{(-1,0)}\bigl(\Ker\bar{\pd}_{ \la_{(p_+,p_-)}}\bigr)$ are
subspaces of $\la_{(p_-,p_+)}\ab(-1,0)\ab\cong\R^n$, of dimensions
$\eta_{(p_-,p_+)}$ and $\eta_{(p_+,p_-)}=n-\eta_{(p_-,p_+)}$. By
genericness they intersect transversely, so that
\begin{equation}
\la_{(p_-,p_+)}(-1,0)=
\ev_{(-1,0)}\bigl(\Ker\bar{\pd}_{\la_{(p_-,p_+)}}\bigr)\op
\ev_{(-1,0)}\bigl(\Ker\bar{\pd}_{\la_{(p_+,p_-)}}\bigr).
\label{il4eq35}
\end{equation}
In \S\ref{il5} we will choose orientations for the
$\Ker\bar{\pd}_{\la_{(p_-,p_+)}}$, and so we can ask whether or not
\eq{il4eq35} holds in oriented vector spaces.

In the situation of Definition \ref{il4dfn2}, define
\begin{equation}
\tM_{k+1}^\ma(\al,\be,J)=\oM_{k+1}^\ma(\al,\be,J)\t\ts\prod_{i\in
I}\Ker\bar{\pd}_{\la_{\al(i)}}.
\label{il4eq36}
\end{equation}
We write elements of $\tM_{k+1}^\ma(\al,\be,J)$ as $\bigl([\Si,\vec
z,u,l,\bar{u}],\xi_i:i\in I\bigr)$, for $[\Si,\vec
z,u,l,\bar{u}]\in\oM_{k+1}^\ma (\al,\be,J)$ and
$\xi_i\in\Ker\bar{\pd}_{ \la_{\al(i)}}$. When for computing
orientations we need to regard \eq{il4eq36} as an {\it ordered\/}
product, since $I\subseteq\{0,\ldots,k\}$ we regard the product
$\prod_{i\in I}$ as occurring in the natural order $\le$ on $I$. We
interpret $\tM_{k+1}^\ma(\al,\be,J)$ as a {\it Kuranishi space},
since the $\Ker\bar{\pd}_{\la_{\al(i)}}$ are manifolds of dimension
$\eta_{\al(i)}$ and $\oM_{k+1}^\ma(\al,\be,J)$ is a Kuranishi space
from \S\ref{il41}. Equation \eq{il4eq10} implies the simpler
equation
\begin{equation}
\vdim\tM_{k+1}^\ma(\al,\be,J)=\mu_L(\be)+k-2+n.
\label{il4eq37}
\end{equation}
This is {\it independent of\/} $\al$, and agrees with the embedded
case~\cite[Prop.~29.1]{FOOO}.

Define $\ti R=\coprod_{(p_-,p_+)\in R}\bigl(\{(p_-,p_+)\}\t
\la_{(p_-,p_+)}(-1,0)\bigr)$. Then $\ti R$ is an $n$-manifold, as
each $\la_{(p_-,p_+)}(-1,0)\cong\R^n$. Thus $L\amalg\ti R$ is an
$n$-manifold. It is nicer to work with than $L\amalg R$, the
disjoint union of an $n$-manifold and a 0-manifold. Define {\it
modified evaluation maps\/} $\widetilde{\ev}_i:
\tM_{k+1}^\ma(\al,\be,J) \ra L\amalg\ti R$ by
\begin{equation}
\widetilde{\ev}_i\bigl([\Si,\vec z,u,l,\bar{u}],\xi_i:i\in
I\bigr)=\begin{cases}
\bar{u}(\ze_i)\in L, & i\notin I,\\
\bigl(\al(i),\ev_{(-1,0)}(\xi_i)\bigr)\in\ti R, & i\in I,
\end{cases}
\label{il4eq38}
\end{equation}
for $i=0,\ldots,k$, and $\widetilde{\ev}:\tM_{k+1}^\ma(\al,\be,J)\ra
L\amalg \ti R$ by
\begin{equation}
\widetilde{\ev}\bigl([\Si,\vec z,u,l,\bar{u}],\xi_i:i\in
I\bigr)=\begin{cases}
\bar{u}(\zeta_0)\in L, & 0\notin I,\\
\bigl(\si\ci\al(0),\ev_{(-1,0)}(\xi_0)\bigr)\in\ti R, & 0\in I.
\end{cases}
\label{il4eq39}
\end{equation}
As for $\ev_i,\ev$, these extend to {\it strongly smooth maps}
$\btev_i,\btev:\tM_{k+1}^\ma(\al,\be,J) \ra L\amalg\ti R$ at the
Kuranishi space level. They are not strong submersions, since the
maps $\ev_{(-1,0)}:\Ker\bar{\pd}_{\la_{(p_-,p_+)}}
\ra\la_{(p_-,p_+)}(-1,0)$ are not submersions, but this will not
matter in the fibre products we take in \eq{il4eq40} and elsewhere
below, because of the transverseness of the subspaces
in~\eq{il4eq35}.

We can now generalize \eq{il4eq4} to an isomorphism of unoriented
Kuranishi spaces:
\begin{equation}
\begin{gathered}
\pd\tM_{k+1}^\ma(\al,\be,J)\cong
\coprod_{\begin{subarray}{l}
k_1+k_2=k+1,\;1\le i\le k_1,\; I_1\cup_iI_2=I,\\
\al_1\cup_i\al_2=\al,\;\be_1+\be_2=\be \end{subarray}
\!\!\!\!\!\!\!\!\!\!\!\!\!\!\!\!\!\!\!\!\!\!\!\!\!\!\!\!}
\begin{aligned}[t]
\tM_{k_2+1}^\ma(\al_2,\be_2,J)\t_{\btev,L\amalg \ti R,\btev_i}&\\
\tM_{k_1+1}^\ma(\al_1,\be_1,J).&
\end{aligned}
\end{gathered}
\label{il4eq40}
\end{equation}
Note that if $i\in I_1$ and $0\notin I_2$, or if $i\notin I_1$ and
$0\in I_2$, then the fibre products in \eq{il4eq4} and \eq{il4eq40}
are empty, since one side maps to $L$, and the other to $R$ or $\ti
R$. Thus, to deduce \eq{il4eq40} from \eq{il4eq4}, for fixed
$i,\ldots,\be_2$ we may divide into the two cases (a) $i\notin I_1$
and $0\notin I_2$, and (b) $i\in I_1$ and $0\in I_2$.

In case (a), the right hand sides of \eq{il4eq4} and \eq{il4eq40}
are both fibre products over $L$, and to see they are isomorphic we
have to give an isomorphism between the extra factors $\prod_{j\in
I}\Ker\bar{\pd}_{\la_{\al(j)}}$ from $\tM_{k+1}^\ma(\al,\be,J)$ on
the left, and $\prod_{j\in I_1}\Ker\bar{\pd}_{
\la_{\al_1(j)}}\t\prod_{j\in I_2}\Ker\bar{\pd}_{\la_{\al_2(j)}}$
from $\tM_{k_1+1}^\ma(\al_1,\be_1,J)$ and $\tM_{k_2+1}^\ma(\al_2,
\be_2,J)$ on the right. In this case, \eq{il4eq5} defines an
isomorphism between $I$ and $I_1\amalg I_2$ which identifies $\al$
and $\al_1\amalg\al_2$, which induces an isomorphism between
$\prod_{j\in I}\Ker\bar{\pd}_{\la_{\al(j)}}$ and $\prod_{j\in
I_1}\Ker\bar{\pd}_{\la_{\al_1(j)}}\t\prod_{j\in I_2}\Ker
\bar{\pd}_{\la_{\al_2(j)}}$ from~$\tM_{k_1+1}^\ma(\al_1,\be_1,J)$.

In case (b), equation \eq{il4eq4} is a fibre product over $R$, and
equation \eq{il4eq40} a fibre product over $\ti R$. By
\eq{il4eq1}--\eq{il4eq2} and \eq{il4eq38}--\eq{il4eq39}, both can
only be nonempty if $\al_1(i)=\si\ci\al_2(0)$, so we suppose this.
Set $\al_1(i)=(p_-,p_+)$ in $R$, so that $\al_2(0)=(p_+,p_-)$, and
let $p=\io(p_-)=\io(p_+)$. Then the term in \eq{il4eq4} is a fibre
product over the point $\{(p_-,p_+)\}$, that is, it is just a
product. The term in \eq{il4eq40} is a fibre product over the
Lagrangian subspace $\la_{(p_-,p_+)}(-1,0)$ in $T_pM$, and $\btev_i$
maps the factor $\Ker\bar{\pd}_{\la_{\al_1(i)}}$ from
$\tM_{k_1+1}^\ma(\al_1,\be_1,J)$ to $\la_{(p_-,p_+)}(-1,0)$ by
$\ev_{(-1,0)}$, and $\btev$ maps the factor
$\Ker\bar{\pd}_{\la_{\al_2(0)}}$ from $\tM_{k_2+1}^\ma(\al_2,
\be_2,J)$ to $\la_{(p_-,p_+)}(-1,0)$ by~$\ev_{(-1,0)}$.

Since \eq{il4eq35} is a direct sum, and $\ev_{(-1,0)}$ are
embeddings, the fibre product of these two factors over
$\la_{(p_-,p_+)}(-1,0)$ is just a point. The remaining extra factors
$\prod_{i\ne j\in I_1}\Ker\bar{\pd}_{\la_{\al_1(j)}}\t\prod_{0\ne
j\in I_2}\Ker\bar{\pd}_{\la_{\al_2(j)}}$ from
$\tM_{k_1+1}^\ma(\al_1,\be_1,J),\tM_{k_2+1}^\ma(\al_2, \be_2,J)$ are
identified with $\prod_{j\in I}\Ker\bar{\pd}_{\la_{\al(j)}}$ from
$\tM_{k+1}^\ma(\al,\be,J)$ using \eq{il4eq5} as in case (a). This
proves \eq{il4eq40}. Note that \eq{il4eq40} is a fibre product over
the $n$-manifold $L\amalg\ti R$. This makes it easier to work with
than \eq{il4eq4}, which is a fibre product over the disjoint union
of an $n$-manifold $L$, and a 0-manifold $R$.

As for \eq{il4eq3} and \eq{il4eq19}, we shall write
\begin{equation}
\tM_{k+1}^\ma(\be,J)=\coprod\nolimits_{\begin{subarray}{l}I\subseteq
\{0,\ldots,k\},\\ \al:I\ra R\end{subarray}}\tM_{k+1}^\ma
(\al,\be,J).
\label{il4eq41}
\end{equation}
Since by \eq{il4eq37} the virtual dimension of $\tM_{k+1}^\ma
(\al,\be,J)$ is independent of $I,\al$, this is a Kuranishi space,
possibly noncompact because of the vector space $\prod_{i\in
I}\Ker\bar\pd_{\la_{\al(i)}}$, of virtual dimension \eq{il4eq37},
another illustration of how the $\tM_{k+1}^\ma (\al,\be,J)$ are
better behaved that the $\oM_{k+1}^\ma (\al,\be,J)$. We define
strong smooth maps $\btev_i,\btev:\tM_{k+1}^\ma(\be,J)\ra L\amalg\ti
R$ to be $\btev_i,\btev$ on each~$\smash{\tM_{k+1}^\ma
(\al,\be,J)}$.
\label{il4dfn8}
\end{dfn}

\subsection{The moduli spaces
$\tM_{k+1}^\ma(\al,\be,J,f_1,\ldots,f_k)$}
\label{il47}

We can also define modified versions $\tM_{k+1}^\ma
(\al,\be,J,f_1,\ldots,f_k),\tM_{k_1+1}^\ma(\al_1,\ab \be_1,\ab J,\ab
f_1,\ldots,f_{i-1}; f_{i+k_2},\ldots,f_k)$ of the moduli spaces of
\S\ref{il44}, in a similar way to~\S\ref{il46}.

\begin{dfn} In the situation of Definition \ref{il4dfn5}, define
\begin{gather}
\tM_{k+1}^\ma(\al,\be,J,f_1,\ldots,f_k)\!=\!\begin{cases}
\oM_{k+1}^\ma(\al,\be,J, f_1,\ldots,f_k), & 0\!\notin\!I,\\
\oM_{k+1}^\ma(\al,\be,J,f_1,\ldots,f_k)\!\t\!\Ker
\bar{\pd}_{\la_{\al(0)}},\!\! & 0\!\in\!I,\end{cases}
\label{il4eq42}\\
\tM_{k_1+1}^\ma(\al_1,\ab \be_1,\ab J,\ab f_1,\ldots,f_{i-1};
f_{i+k_2},\ldots,f_k)=\qquad\qquad\qquad\qquad{}
\label{il4eq43}\\
\begin{cases}
\oM_{k_1+1}^\ma(\al_1,\ab \be_1,\ab J,\ab f_1,\ldots,f_{i-1};
f_{i+k_2},\ldots,f_k), & \hbox to 0pt{\hss $0,i\notin I_1,$}\\
\oM_{k_1+1}^\ma(\al_1,\ab \be_1,\ab J,\ab f_1,\ldots,f_{i-1};
f_{i+k_2},\ldots,f_k)\t \Ker\bar{\pd}_{\la_{\al_1(0)}},& \hbox to
0pt{\hss $0\in I_1$, $i\notin I_1$,}\\
\oM_{k_1+1}^\ma(\al_1,\ab \be_1,\ab J,\ab f_1,\ldots,f_{i-1};
f_{i+k_2},\ldots,f_k)\t\Ker\bar{\pd}_{\la_{\al_1(i)}},& \hbox to
0pt{\hss $0\notin I_1$, $i\in I_1$,}\\
\oM_{k_1+1}^\ma(\al_1,\ab \be_1,\ab J,\ab f_1,\ldots,f_{i-1};
f_{i+k_2},\ldots,f_k)\!\t\!
\Ker\bar{\pd}_{\la_{\al_1(0)}}\!\t\!\Ker\bar{\pd}_{\la_{\al_1(i)}},
\,\,\,\,\,\,\,\,\,\,\,\,\,\,\,\,\, & \hbox to 0pt{\hss
$0,i\!\in\!I_1$.}
\end{cases}
\nonumber
\end{gather}
Then equations \eq{il4eq14} and \eq{il4eq17} imply the simpler
formulae
\begin{align}
&\vdim\tM_{k+1}^\ma(\al,\be,J,f_1,\ldots,f_k)=
\mu_L(\be)-2+n-\ts\sum_{i=1}^k\deg f_i,
\label{il4eq44}\\
&\vdim\tM_{k_1+1}^\ma(\al_1,\be_1,J,f_1,\ldots,f_{i-1};
f_{i+k_2},\ldots,f_k)=
\nonumber\\
&\qquad\qquad\mu_L(\be_1)-1+n-\ts\sum_{j=1}^{i-1}\deg
f_j-\sum_{j=i+k_2}^k \deg f_j.
\nonumber
\end{align}

Suppose now that $f_i:\De_{a_i}\ra L\amalg R$ maps to $L$ if
$i\notin I$, and to $\al(i)\in R$ if $i\in I$. As above, if this
does not hold then $\oM_{k+1}^\ma(\al,\be,J,f_1,\ldots,f_k)=
\emptyset=\tM_{k+1}^\ma(\al,\be,J,f_1,\ldots,f_k)$. Then
\eq{il4eq42} is equivalent to the fibre product
\begin{equation}
\begin{split}
&\tM_{k+1}^\ma(\al,\be,J,f_1,\ldots,f_k)=
\tM_{k+1}^\ma(\al,\be,J)\\
&\;\t_{\prod_{i=1}^k\btev_i,(L\amalg\ti
R)^k,\prod_{i=1}^k\bigl\{\begin{smallmatrix}f_i, & i\notin I\\
f_i\t\ev_{(-1,0)}, & i\in I\end{smallmatrix}\bigr\}}
\left.\prod_{i=1}^k
\begin{cases}\De_{a_i}, & \!i\notin I\\
\De_{a_i}\t\Ker\bar{\pd}_{\la_{\si\ci\al(i)}}, & \!i\in I\end{cases}
\right\}.
\end{split}
\label{il4eq45}
\end{equation}

The difference between \eq{il4eq42} and \eq{il4eq45} is that in
\eq{il4eq45} we have extra factors $\prod_{0\ne i\in
I}\Ker\bar{\pd}_{\la_{\al(i)}}$ in $\tM_{k+1}^\ma(\al,\be,J)$ (we
exclude 0 because of the factor $\Ker\bar{\pd}_{\la_{\al(0)}}$ in
\eq{il4eq42}) and $\prod_{0\ne i\in
I}\Ker\bar{\pd}_{\la_{\si\ci\al(i)}}$ in $\prod_{0\ne i\in
I}\De_{a_i}\t\Ker\bar{\pd}_{\la_{\si\ci\al(i)}}$. However, we are
taking a fibre product over $(L\amalg\ti R)^k$ rather than $(L\amalg
R)^k$. The effect of this is that for each $0\ne i\in I$, in
\eq{il4eq45} we take the fibre product $\Ker\bar{\pd}_{\la_{\al(i)}}
\t_{\ev_{(-1,0)},\la_{\al(i)}(-1,0),\ev_{(-1,0)}}
\Ker\bar{\pd}_{\la_{\si\ci\al(i)}}$, which is just a point by
\eq{il4eq35} and injectivity of the $\ev_{(-1,0)}$. Thus
\eq{il4eq42} and \eq{il4eq45} differ only by the product with
$\md{I\sm\{0\}}$ points, so they are equivalent.

Similarly, using \eq{il4eq16} we find \eq{il4eq43} is equivalent to
the fibre product
\begin{equation}
\begin{split}
& \tM_{k_1+1}^\ma(\al_1,\ab \be_1,\ab J,\ab f_1,\ldots,f_{i-1};
f_{i+k_2},\ldots,f_k)=\tM_{k_1+1}^\ma(\al_1,\be_1,J)\\
&\t_{\btev_1\t\cdots\t\btev_{i-1}\t
\btev_{i+1}\t\cdots\t\btev_{k_1},(L\amalg\ti R)^{k_1-1},
\prod_{\begin{subarray}{l} j=1,\ldots,k:\\
j<i\;\text{or}\; j\ge i+k_2\end{subarray}}
\bigl\{\begin{smallmatrix}f_j, & j\notin I\\
f_j\t\ev_{(-1,0)}, & j\in I\end{smallmatrix}\bigr\}}\\
&\prod_{\begin{subarray}{l} j=1,\ldots,k:\\
j<i\;\text{or}\; j\ge i+k_2\end{subarray}}
\left.\begin{cases}\De_{a_j}, & \!j\notin I\\
\De_{a_j}\t\Ker\bar{\pd}_{\la_{\si\ci\al(j)}}, & \!j\in I\end{cases}
\right\}.
\end{split}
\label{il4eq46}
\end{equation}

Combining \eq{il4eq40}, \eq{il4eq45} and \eq{il4eq46} we find that
by analogy with \eq{il4eq18} we have
\begin{equation}
\begin{aligned}
&\pd\tM_{k+1}^\ma(\al,\be,J,f_1,\ldots,f_k)\cong \\
&\coprod_{i=1}^k\coprod_{j=0}^{a_i}
\tM_{k+1}^\ma(\al,\be,J,f_1,\ldots,f_{i-1},f_i\ci F_j^{a_i},
f_{i+1},\ldots,f_k)\\
&\amalg
\coprod_{\begin{subarray}{l}
k_1+k_2=k+1,\;1\le i\le k_1,\\
I_1\cup_iI_2=I,\;\al_1\cup_i\al_2=\al, \\
\be_1+\be_2=\be\end{subarray}}
\begin{aligned}[t]
&\tM_{k_2+1}^\ma (\al_2,\be_2,J,f_i,\ldots,f_{i+k_2-1})
\t_{\btev,L\amalg\ti R,\btev_i}\\
&\quad\tM_{k_1+1}^\ma
(\al_1,\be_1,J,f_1,\ldots,f_{i-1};f_{i+k_2},\ldots,f_k),
\end{aligned}
\end{aligned}
\label{il4eq47}
\end{equation}
in unoriented Kuranishi spaces.

As for \eq{il4eq3}, \eq{il4eq19} and \eq{il4eq41} we shall also
write
\begin{equation*}
\tM_{k+1}^\ma(\be,J,f_1,\ldots,f_k)=\coprod\nolimits_{
\begin{subarray}{l}I\subseteq\{0,\ldots,k\},\\ \al:I\ra
R\end{subarray}}\tM_{k+1}^\ma(\al,\be,J,f_1,\ldots,f_k).
\end{equation*}
This is a Kuranishi space, of virtual dimension \eq{il4eq44}, which
may be noncompact because of the vector space
$\Ker\bar\pd_{\la_{\al(0)}}$ in \eq{il4eq42}. We define
$\btev:\tM_{k+1}^\ma(\be,J,f_1,\ldots,f_k)\ab\ra L\amalg\ti R$ to be
$\btev$ on each~$\tM_{k+1}^\ma (\al,\be,J,f_1,\ldots,f_k)$.

We can now explain the notion of {\it shifted cohomological
degree\/} in Definition \ref{il4dfn5}, and the grading it induces on
$C_*^\rsi(L\amalg R;\Q)$. Suppose $f:\De_a\ra L\amalg R$ is smooth.
By \eq{il4eq13}, if $f$ maps to $L$ then $\deg f=n-a-1$, which is
the (virtual) codimension of $f(\De_a)$ in $L$ minus one. But if $f$
maps to $(p_-,p_+)$ in $R$ then $\deg f=\eta_{(p_-,p_+)}-a-1$. Here
is a good way to understand this. Morally, we want to lift $f$ to a
map $\ti f$ to the $n$-manifold $L\amalg\ti R$. Since $f$ maps to
$\{(p_-,p_+)\}\subset R$, the lift $\ti f$ should map to
$\{(p_-,p_+)\}\t\la_{(p_-,p_+)}(-1,0)\subset\ti R$. But the domain
of $f$ should not be $\De_a$. Motivated by \eq{il4eq45}, we see that
$f:\De_a\ra\{(p_-,p_+)\}\subset R$ should lift to
\begin{equation*}
\ti f=f\t\ev_{(-1,0)}:\De_a\t\Ker\bar{\pd}_{\la_{(p_+,p_-)}}
\longra\{(p_-,p_+)\}\t\la_{(p_-,p_+)}(-1,0)\subset\ti R.
\end{equation*}

This is not a chain in $C_*^\rsi(L\amalg\ti R;\Q)$, as $\De_a\t
\Ker\bar{\pd}_{\la_{\si(p_-,p_+)}}$ is not a simplex. But it does
justify the change in degree in \eq{il4eq13}. We have
$\dim\bigl(\De_a\t\Ker\bar{\pd}_{\la_{(p_+,p_-)}}\bigr)=
a+\eta_{(p_+,p_-)}=a+n-\eta_{(p_-,p_+)}$ by \eq{il4eq9}. Thus, the
(virtual) codimension of $\ti f\bigl(\De_a\t\Ker
\bar{\pd}_{\la_{(p_+,p_-)}}\bigr)$ in $\ti R$ minus one is $n-(a+
n-\eta_{(p_-,p_+)})-1=\eta_{(p_-,p_+)}-a-1=\deg f$. Hence, when we
lift to modified moduli spaces in this way, the shifted
cohomological degree $\deg f$ is the genuine shifted cohomological
degree of the `chain' $\ti f$ in~$L\amalg\ti R$.
\label{il4dfn9}
\end{dfn}

We could also easily define modified versions of the moduli spaces
of \S\ref{il45} for families of complex structures, but we will not,
as we only need the modified spaces for motivation anyway.

\section{Orientations}
\label{il5}

We now define orientations on the Kuranishi spaces defined in
\S\ref{il4}, and prove formulae for their boundaries in oriented
Kuranishi spaces, so computing the appropriate signs in \eq{il4eq4},
\eq{il4eq18}, \eq{il4eq40}, and \eq{il4eq47}.

\subsection{Orientations on $\tM_{k+1}^\ma(\al,\be,J)$}
\label{il51}

Fukaya et al.\ \cite[Def.~44.2]{FOOO} define {\it relative spin
structures\/} on $L$. We adapt their definition to the immersed
case.

\begin{dfn} Let $\io:L\ra M$ be an immersed submanifold with
transverse self-intersections in $M$. Fix triangulations of $L$ and
$M$ compatible under $\io$. This can be done by first triangulating
the self-intersection of $\io(L)$ in $M$, then extending this to a
triangulation of $\io(L)$ which pulls back to one of $L$, and then
extending the triangulation of $\io(L)$ to one of $M$. A {\it
relative spin structure\/} for $\io:L\ra M$ consists of an
orientation on $L$; a class ${\rm st}\in H^2(M;\Z_2)$ such that
$\io^*({\rm st})=w_2(L)\in H^2(L;\Z_2)$, the second Stiefel--Whitney
class of $L$; an oriented vector bundle $V$ on the 3-skeleton
$M_{[3]}$ of $M$ with $w_2(V)={\rm st}$; and a spin structure
on~$(TL\op\io^*(V))\vert_{L_{[2]}}$.

Here $L_{[2]}$ is the 2-skeleton of $L$, and as $w_2(V
\vert_{L_{[2]}})=\io^*({\rm st})\vert_{L_{[2]}}=
w_2(L)\vert_{L_{[2]}}$ we have $w_2\bigl((TL\op
\io^*(V))\vert_{L_{[2]}}\bigr)=0$, so $(TL\op\io^*(V))
\vert_{L_{[2]}}$ admits a spin structure. If $L$ is spin then
$w_2(L)=0$, so we can take ${\rm st}=0$ and $V=0$ and the spin
structure on $TL\vert_{L_{[2]}}$ to be the restriction of that on
$TL$. Hence, an orientation and spin structure on $L$ induce a
relative spin structure for~$\io:L\ra M$.
\label{il5def1}
\end{dfn}

We first construct orientations on the modified spaces
$\tM_{k+1}^\ma(\al,\be,J)$ of~\S\ref{il46}.

\begin{thm} Let\/ $(M,\om)$ be a compact symplectic manifold with
compatible almost complex structure\/ $J,$ and\/ $\io:L\ra M$ a
compact Lagrangian immersion with only transverse double
self-intersections. Then choices of a relative spin structure for\/
$\io:L\ra M,$ and of\/ $\la_{(p_-,p_+)}$ for\/ $(p_-,p_+)\in R$ as
in\/ \S{\rm\ref{il43},} determine orientations on the modified
Kuranishi spaces\/ $\tM_{k+1}^\ma(\al,\be,J)$ of\/ \S{\rm\ref{il46}}
for all\/~$k,\al,\be$.
\label{il5thm1}
\end{thm}

\begin{proof} Let $[\Si,\vec z,u,l,\bar{u}]\in\oM_{k+1}^\ma
(\al,\be,J)$, so that $\vec z=(z_0,\ldots,z_k)$ with
$z_0,\ldots,z_k$ distinct smooth points of $\pd\Si$. For each $i\in
I$ we choose a small open neighbourhood $U_i$ of $z_i$ in $\Si$,
such that $U_i\cap U_j=\emptyset$ if $i\neq j\in I$, and $z_j\notin
U_i$ if $i\ne j\notin I$, and $U_i\sm\{z_i\}$ is biholomorphic to
$(-\iy,0)\t[-1,1]$, where $z_i$ corresponds to $-\iy$. We identify
$U_i$ with $\{-\iy\}\cup(-\iy,0)\t[-1,1]$, and define
\begin{equation*}
U_i^r=\{-\iy\}\cup(-\iy,-r)\t[-1,1]\subset U_i,
\end{equation*}
for $r>0$. For $i\in I$ we also define
\begin{equation*}
Y_i^r=\bigl\{(x,y)\in\R^2:\text{either $x\le 0$, $x^2+y^2\le 1$ or
$0\le x\le r$, $\md{y}\le 1$}\bigl\}\subset Y,
\end{equation*}
where $Y$ is as in \eq{il4eq6}, and we set $y_i=(-1,0)\in Y_i^r$.
For $j\notin I$ we define $y_j=z_j\in\Si\sm\bigcup_{i\in I}U_i$.
Glue $\Si\sm\bigcup_{i\in I}U_i^r$ and $\bigcup_{i\in I}Y_i^r$ by
identifying $\{-r\}\t[-1,1]\subset U_i$ with $\{r\}\t[-1,1]\subset
Y_i^r$ to make $(\Si^r,y_0,\ldots,y_k)$, which is diffeomorphic to
$(\Si,z_0,\ldots,z_k)$. Consider the linearized Cauchy--Riemann
operator
\begin{equation*}
\begin{split}
D_u\bar{\pd}&:W^{1,q}\bigl(\Si\sm\{z_i:i\in I\},\pd\Si\sm\{z_i:i\in
I\};u^*(TM),u^*(\d\io(TL)\bigr) \\
&\longra L^q\bigl(\Si\sm\{z_i:i\in
I\};u^*(TM)\ot\La^{0,1}(\Si\sm\{z_i:i\in I\})\bigr),
\end{split}
\end{equation*}
for $q>2$, and define the virtual vector space
\begin{equation*}
\Ind D_u\bar{\pd}=\Ker D_u\bar{\pd}\ominus\Coker D_u\bar{\pd}.
\end{equation*}
Here for a Fredholm operator $P$, we will write $\ind P=\dim\Ker
P-\dim\Coker P$ in $\Z$, and $\Ind P=\Ker P\ominus\Coker P$ as a
virtual vector space.

By a suitable partition of unity, we define differential operators
\begin{equation*}
D_{u,\la_{\al}}:W^{1,q}(\Si^r,\pd\Si^r;E_u,F_{u,\la_{\al}})\longra
L^q(\Si^r;E_u\ot\La^{0,1}\Si^r),
\end{equation*}
for $q>2$ and large $r$, whose restrictions to $\Si\sm\bigcup_{i\in
I}U_i^{r-1}$ and $Y_i^{r-1}$ coincide with $D_u\bar{\pd}$ and
$\bar{\pd}_{\la_{\al(i)}}$, respectively, and we define the virtual
vector space
\begin{equation*}
\Ind D_{u,\la_{\al}}=\Ker D_{u,\la_{\al}}\ominus\Coker
D_{u,\la_{\al}}.
\end{equation*}

Here $E_u\ra\Si^r$ is a complex vector bundle agreeing with
$u^*(TM)$ on $\Si\sm\bigcup_{i\in I}U_i^r$, and is trivial with
fibre $T_{p_i}M$ on $Y_i^r$ for $i\in I$, where $\al(i)=(p_-,p_+)\in
R$ with $\io(p_-)=\io(p_+)=p_i$. Also $F_{u,\la_{\al}}$ is a real
vector subbundle of $E_u\vert_{\pd\Si^r}$ which agrees with
$\d\io(TL)$ on $\pd\Si\sm\bigcup_{i\in I}U_i^r$, and with
$\la_{\al(i)}$ on $\pd Y_i^r$ for $i\in I$, except near
$\{-r\}\t[-1,1]$ where we interpolate between these two values. The
notation $\la_\al$ in $D_{u,\la_{\al}}$ and $F_{u,\la_{\al}}$
denotes that these depend on the choice of $\la_{\al(i)}$ for all
$i\in I$, where $\al(i)=(p_-,p_+)\in R$ and $\la_{(p_-,p_+)}$ is as
in \S\ref{il43}. Then, by a gluing theorem for large $r$, we have an
isomorphism of virtual vector spaces
\begin{equation}
\Ind D_u\bar{\pd}\op\ts\bigop_{i\in I}\Ker\bar{\pd}_{\la_{\al(i)}}
\cong\Ind D_{u,\la_{\al}},
\label{il5eq1}
\end{equation}
since $\Coker\bar{\pd}_{\la_{\al(i)}}=0$ as in \S\ref{il46}. Really
this holds in the limit~$r\ra\iy$.

The {\it virtual tangent bundle\/} of $\oM_{k+1}^\ma(\al,\be,J)$ is
\begin{equation*}
\bigcup_{[\Si,\vec z,u,l,\bar{u}]\in\oM_{k+1}^\ma(\al,\be,J)} (\Ind
D_u\bar{\pd}\op T_{[\Si,\vec z]}\oM_{k+1}^\ma)
\longra\oM_{k+1}^\ma(\al,\be,J),
\end{equation*}
where $\oM_{k+1}^\ma$ is the moduli space of isomorphism classes of
genus 0 prestable bordered Riemann surfaces with $k+1$ distinct
smooth boundary marked points ordered counter-clockwise. Combining
this with \eq{il4eq36} and using \eq{il5eq1} shows that in the limit
$r\ra\iy$, the {\it virtual tangent bundle\/} of
$\tM_{k+1}^\ma(\al,\be,J)$ is
\begin{equation}
\bigcup_{(\xi_i:i\in I,[\Si,\vec z,u,l,\bar{u}])\in
\tM_{k+1}^\ma(\al,\be,J)}(\Ind D_{u,\la_{\al}}\op T_{[\Si,\vec
z]}\oM_{k+1}^\ma)\longra\tM_{k+1}^\ma(\al,\be,J).
\label{il5eq2}
\end{equation}

Since $\oM_{k+1}^\ma$ is oriented \cite[\S 2.2]{FOOO}, \cite[\S
4.5]{Liu}, the factor $T_{[\Si,\vec z]}\oM_{k+1}^\ma$ in \eq{il5eq2}
is oriented. As in the embedded case \cite[\S 44]{FOOO}, a relative
spin structure for $\io:L\ra M$ canonically determines a homotopy
type of trivializations of $F_{u,\la_{\al}}$, which gives an
orientation of $\Ind D_{u,\la_{\al}}$. This is obtained by gluing in
$\la_{\al(i)}$ at $z_i$ for $i\in I$, and so also depends on the
choice of $\la_{(p_-,p_+)}$ for $(p_-,p_+)$ in $R$. Combining these
two gives an orientation for the virtual tangent bundle \eq{il5eq2},
and hence an orientation on the Kuranishi
space~$\tM_{k+1}^\ma(\al,\be,J)$.
\end{proof}

With these orientations, we compute the signs in~\eq{il4eq40}.

\begin{thm} In the situation of\/ \S{\rm\ref{il46},} with the
orientations for\/ $\tM_{k+1}^\ma(\al,\be,J)$ in Theorem\/
{\rm\ref{il5thm1}} and using the conventions of\/ \S{\rm\ref{il24},}
the orientations of\/ $\pd\tM_{k+1}^\ma(\al,\ab\be,\ab J)$ and\/
$\tM_{k_1+1}^\ma (\al_1,\be_1,J)\t_{\btev_i,L\amalg\ti
R,\btev}\tM_{k_2+1}^\ma (\al_2,\ab\be_2,\ab J)$ in\/ \eq{il4eq40}
differ by a factor\/ $(-1)^{(k_1-i)(k_2-1)+(n+k_1)},$ so that in
oriented Kuranishi spaces we have
\begin{equation}
\begin{gathered}
\pd\tM_{k+1}^\ma(\al,\be,J)\cong
\coprod_{\!\!\!\!\!\!\!\!\begin{subarray}{l}
k_1+k_2=k+1,\;1\le i\le k_1, I_1\cup_iI_2=I,\\
\al_1\cup_i\al_2=\al,\;
\be_1+\be_2=\be\end{subarray}\!\!\!\!\!\!\!\!\!\!\!\!\!\!}
\begin{aligned}[t]
(-1)^{n+i+ik_2}\tM_{k_2+1}^\ma(\al_2,\be_2,J)&\\
\t_{\btev,L\amalg\ti R,\btev_i}\tM_{k_1+1}^\ma(\al_1,&\be_1,J).
\end{aligned}
\end{gathered}
\label{il5eq3}
\end{equation}
\label{il5thm2}
\end{thm}

\begin{proof} Suppose $[\Si,\vec z,u,l,\bar{u}]$ in
$\pd\oM_{k+1}^\ma(\al,\be,J)$ is identified in \eq{il4eq4} with a
point in $\oM_{k_2+1}^\ma(\al_2,\be_2,J)\t_{\bev,L\amalg R,\bev_i}
\oM_{k_1+1}^\ma(\al_1,\be_1,J)$ represented by $[\Si_1,\vec
z_1,u_1,l_1,\bar{u}_1]\in\ab\oM_{k_1+1}^\ma(\al_1,\be_1,J)$ and
$[\Si_2,\vec z_2,u_2,l_2,\bar{u}_2]\in
\oM_{k_2+1}^\ma\ab(\al_2,\ab\be_2,J)$. Then $\vec z_1\!=\!(z_0^1,
\ldots,z_{k_1}^1)$ and $\vec z_2=(z_0^2,\ldots,z_{k_2}^2)$, and as
the point lies in the fibre product we have $u_1(z_i^1)= u_2(z_0^2)$
in $\io(L)$, and either $i\notin I_1$ and $0\notin I_2$, or $i\in
I_1$ and $0\in I_2$ and $\al_1(i)=\si\ci\al_2(0)$ in $R$, noting the
differing definitions of $\ev_i$ and $\ev$ in \eq{il4eq1} and
\eq{il4eq2}. From these we construct $(\Si^r,y_0,\ldots,y_k),
(\Si_1^r,y_0^1, \ldots,y_{k_1}^1),(\Si^r_2,y_0^2,\ldots,y_{k_2}^2)$,
and smoothed operators $D_{u,\la_{\al}},D_{u_1,\la_{\al_1}},
D_{u_2,\la_{\al_2}}$ upon them, as in the proof of Theorem
\ref{il5thm1}. The following lemma is then proved as in Fukaya et
al.~\cite[Lem.~46.4]{FOOO}:

\begin{lem} We have an isomorphism of oriented virtual vector
spaces
\begin{equation}
\Ind D_{u,\la_{\al}}\cong\Ind D_{u_2,\la_{\al_2}}
\t_{\bar{\ev},T(L\amalg \ti R),\bar{\ev}_i}\Ind D_{u_1,\la_{\al_1}},
\label{il5eq4}
\end{equation}
where, for $\xi_1\in
W^{1,q}(\Si_1^r,\pd\Si_1^r;E_{u_1},F_{u_1,\la_{\al_1}})$,
\begin{equation}
\bar{\ev}_i(\xi_1)=\begin{cases} \xi_1(y_i^1)\in T_{u_1(z_i^1)}L, &
i\notin I_1,\\
\xi_1(y_i^1) \in \la_{\al_1(i)}(-1,0), & i\in I_1,
\end{cases}
\label{il5eq5}
\end{equation}
and, for $\xi_2\in W^{1,q}(\Si_2^r,\pd\Si_2^r;E_{u_2},
F_{u_2,\la_{\al_2}})$, we define
\begin{equation}
\bar{\ev}(\xi_2)=\begin{cases} \xi_2(y_0^2)\in T_{u_2(z_0^2)}L,
& 0\notin I_2,\\
\xi_2(y_0^2) \in \la_{\si\ci\al_2(0)}(-1,0), & 0\in I_2.
\end{cases}
\label{il5eq6}
\end{equation}
\label{il5lem1}
\end{lem}

Since $L$ is oriented and $\la_{(p_-,p_+)}$ is compatible with
orientations, $\mu_L(\be)$ is even. Thus we obtain the following
corollary, proved as in Fukaya et al.\ \cite[Prop.~46.2]{FOOO}. For
reasons to be explained in Remark \ref{il5rem2}(b), we have reversed
the order of their fibre product, as for \eq{il4eq4} in
\S\ref{il43}, so the sign in \eq{il5eq3} is not the same as that in
\cite[Prop.~46.2]{FOOO}; the difference can be computed using the
second line of~\eq{il2eq5}.

\begin{cor} We have isomorphisms of oriented virtual vector spaces
\begin{equation}
\begin{split}
&\Ind D_{u,\la_{\al}}\op T_{[\Si,z_0,\ldots,z_k]}\pd\oM_{k+1}^\ma
\cong \\
& (-1)^{n+i+ik_2}(\Ind D_{u_2,\la_{\al_2}}\op
T_{[\Si_2,z_0^2,\ldots,z_{k_2}^2]}\oM_{k_2+1}^\ma)\\
& \t_{\bar{\ev},T(L\amalg \ti R),\bar{\ev}_i}(\Ind D_{u_1,\la_{\al_1}}
\op T_{[\Si_1,z_0^1,\ldots,z_{k_1}^1]}\oM_{k_1+1}^\ma).
\end{split}
\label{il5eq7}
\end{equation}
\label{il5cor1}
\end{cor}

By \eq{il5eq2}, in the limit $r\ra\iy$ the three terms in
\eq{il5eq7} become the virtual tangent bundles of
$\pd\tM_{k+1}^\ma(\al,\be,J),\tM_{k_1+1}^\ma(\al_1,\be_1,J)$ and
$\tM_{k_2+1}^\ma(\al_2,\be_2,J)$. By comparing
\eq{il4eq38}--\eq{il4eq39} and \eq{il5eq5}--\eq{il5eq6} we see that
in the limit $r\ra\iy$, the fibre product
`$\cdots\t_{\bar{\ev},T(L\amalg \tilde R),\bar{\ev}_i}\cdots$' in
\eq{il5eq7} becomes that induced on virtual tangent bundles by the
fibre product `$\cdots\t_{\btev,L\amalg\ti R,\btev_i}\cdots$' in
\eq{il4eq40} and \eq{il5eq3}. Taking the limit $r\ra\iy$, equation
\eq{il5eq7} now implies the oriented virtual tangent bundle version
of \eq{il5eq3}, so Theorem \ref{il5thm2} follows from this
and~\eq{il4eq40}.
\end{proof}

\subsection{Orientations on $\tM_{k+1}^\ma(\al,\be,J,f_1,\ldots,f_k)$}
\label{il52}

Next we orient the spaces of~\S\ref{il47}.

\begin{dfn} In the situation of \S\ref{il43}, choose orientations
$o_{(p_-,p_+)}$ on the vector spaces $\Ker\bar{\pd}_{
\smash{\la_{(p_-,p_+)}}}$ for all $(p_-,p_+)$ in $R$. In equation
\eq{il4eq35}, $\la_{(p_-,p_+)}(-1,0)$ is an oriented Lagrangian
subspace of $T_pM$, and the maps $\ev_{(-1,0)}$ are injective, so
our orientations $o_{(p_\mp,p_\pm)}$ on
$\Ker\bar{\pd}_{\la_{(p_\mp,p_\pm)}}$ induce orientations on
$\ev_{(-1,0)}\bigl(\Ker\bar{\pd}_{\la_{(p_\mp, p_\pm)}}\bigr)$.
Thus, all three vector spaces in \eq{il4eq35} are oriented. Define
$\ep_{(p_-,p_+)}=1$ if \eq{il4eq35} is true in oriented vector
spaces, and $\ep_{(p_-,p_+)}=-1$ otherwise, for all~$(p_-,p_+)\in
R$.

The subspaces on the r.h.s.\ of \eq{il4eq35} have dimensions
$\eta_{(p_-,p_+)}$ and $n-\eta_{(p_-,p_+)}$, so exchanging them
changes orientations by a factor $(-1)^{\eta_{(p_-,p_+)}
(n-\eta_{(p_-,p_+)})}$. Thus
\begin{equation}
\ep_{(p_-,p_+)}\ep_{(p_+,p_-)}=(-1)^{\eta_{(p_-,p_+)}
(n-\eta_{(p_-,p_+)})}.
\label{il5eq8}
\end{equation}
If $n$ is {\it odd\/} then one of
$\eta_{(p_-,p_+)},n-\eta_{(p_-,p_+)}$ is even, so \eq{il5eq8} gives
$\ep_{(p_-,p_+)}\ep_{(p_+,p_-)}=1$. In this case, we can choose the
orientations on the $\Ker\bar{\pd}_{\la_{(p_-,p_+)}}$ so that
$\ep_{(p_-,p_+)}=1$ for all $(p_-,p_+)$ in $R$, which simplifies
some formulae below. But if $n$ is {\it even\/} and some
$\eta_{(p_-,p_+)}$ is odd then \eq{il5eq8} gives $\ep_{(p_-,p_+)}
\ep_{(p_+,p_-)}=-1$, so we cannot choose the orientations on the
$\Ker\bar{\pd}_{\la_{(p_-,p_+)}}$ to make all~$\ep_{(p_-,p_+)}=1$.

We work in the situation of Definition \ref{il4dfn9} with
orientations on $\tM_{k+1}^\ma(\al,\be,J)$ from Theorem
\ref{il5thm1}, and $o_{(p_-,p_+)}$ on $\Ker\bar{\pd}_{
\la_{(p_-,p_+)}}$. Define an orientation on $\tM_{k+1}^\ma
(\al,\ab\be,\ab J,\ab f_1,\ab\ldots,f_k)$ by the fibre product of
oriented Kuranishi spaces:
\begin{equation}
\begin{split}
&\tM_{k+1}^\ma(\al,\be,J,f_1,\ldots,f_k)=
(-1)^{(n+1)\sum_{l=1}^k(k-l)(\deg f_l+1)}\\
&\tM_{k+1}^\ma(\al,\be,J)
\t_{\prod_{i=1}^k\btev_i,(L\amalg\ti R)^k,
\prod_{i=1}^k\bigl\{\begin{smallmatrix}f_i, & i\notin I,\\
f_i\t\ev_{(-1,0)}, & i\in I\end{smallmatrix}\bigr\}}\\
&\left.\prod_{i=1}^k
\begin{cases}\De_{a_i}, & i\notin I,\\
\De_{a_i}\t\Ker\bar{\pd}_{\la_{\si\ci\al(i)}}, & i\in I\end{cases}
\right\},
\end{split}
\label{il5eq9}
\end{equation}
which is \eq{il4eq45} with a choice of sign taken from Fukaya et
al.\ \cite[Def.~47.1]{FOOO}. Roughly speaking, the sign
$(-1)^{(n+1)\sum_{l=1}^k(k-l)(\deg f_l+1)}$ is chosen so that in the
$A_\iy$ algebra we will construct later, $\m_k(f_1,\ldots,f_k)$ is a
virtual chain for the oriented Kuranishi space
$\tM_{k+1}^\ma(\al,\be,J,f_1,\ldots,f_k)$. But we will actually
define $\m_k(f_1,\ldots,f_k)$ using the $\oM_{k+1}^\ma(\al,\be,J,
f_1,\ldots,f_k)$, and the calculations in this section are just
motivation for the complicated choice of orientation on
$\oM_{k+1}^\ma(\al,\be,J, f_1,\ldots,f_k)$ in~\S\ref{il54}.

Similarly, define an orientation on $\tM_{k_1+1}^\ma(\al_1,\ab
\be_1,\ab J,\ab f_1,\ldots,f_{i-1}; f_{i+k_2},\ldots,f_k)$ by
\begin{gather}
\tM_{k_1+1}^\ma(\al_1,\ab \be_1,\ab J,\ab f_1,\ldots,f_{i-1};
f_{i+k_2},\ldots,f_k)=(-1)^{n\sum_{l=1}^{i-1}(\deg f_l+1)}
\label{il5eq10}\\
(-\!1)^{(n+1)\sum_{l=1}^{i-1}(k-k_2+1-l)(\deg f_l+1)}
(-\!1)^{(n+1)\sum_{l=i+k_2}^k(k-l)(\deg f_l+1)}
\tM_{k_1+1}^\ma(\al_1,\be_1,J)
\nonumber\\
\t_{\btev_1\t\cdots\t\btev_{i-1}\t
\btev_{i+1}\t\cdots\t\btev_{k_1},(L\amalg\ti R)^{k_1-1},
\prod_{\begin{subarray}{l} j=1,\ldots,k:\\
j<i\;\text{or}\; j\ge i+k_2\end{subarray}}
\bigl\{\begin{smallmatrix}f_j, & j\notin I\\
f_j\t\ev_{(-1,0)}, & j\in I\end{smallmatrix}\bigr\}}
\nonumber\\
\prod_{\begin{subarray}{l} j=1,\ldots,k:\\
j<i\;\text{or}\; j\ge i+k_2\end{subarray}}
\left.\begin{cases}\De_{a_j}, & \!j\notin I\\
\De_{a_j}\t\Ker\bar{\pd}_{\la_{\si\ci\al(j)}}, & \!j\in
I\end{cases}\right\},
\nonumber
\end{gather}
which is \eq{il4eq46} with a sign inserted, chosen to achieve a
simple form for the signs in \eq{il5eq11} and \eq{il5eq12} below.
\label{il5dfn2}
\end{dfn}

We can now add orientations to equation~\eq{il4eq47}.

\begin{thm} In the situation of Definition {\rm\ref{il4dfn9},} with
the orientations of Definition {\rm\ref{il5dfn2},} in oriented
Kuranishi spaces we have
\begin{equation}
\begin{aligned}
&\pd\tM_{k+1}^\ma(\al,\be,J,f_1,\ldots,f_k)\cong \\
&\coprod_{i=1}^k\coprod_{j=0}^{a_i}\,\,
\begin{aligned}[t]
&(-1)^{j+1+\sum_{l=1}^{i-1}\deg f_l}\\
&\quad\tM_{k+1}^\ma(\al,\be,J,f_1,\ldots,f_{i-1},f_i\ci F_j^{a_i},
f_{i+1},\ldots,f_k)
\end{aligned}
\\
&\amalg
\coprod_{\begin{subarray}{l}
k_1+k_2=k+1,\;1\le i\le k_1,\\
I_1\cup_iI_2=I,\;\al_1\cup_i\al_2=\al,\\
\be_1+\be_2=\be
\end{subarray}}
\begin{aligned}[t]
&(-1)^{n+\bigl(1+\sum_{l=1}^{i-1}\deg f_l\bigr)
\bigl(1+\sum_{l=i}^{i+k_2-1}\deg f_l\bigr)}\\
&\quad\tM_{k_2+1}^\ma (\al_2,\be_2,J,f_i,\ldots,f_{i+k_2-1})
\t_{\btev,L\amalg\ti R,\btev_i}\\
&\quad\tM_{k_1+1}^\ma(\al_1,\be_1,J,f_1,\ldots,f_{i-1};
f_{i+k_2},\ldots,f_k).
\end{aligned}
\end{aligned}
\label{il5eq11}
\end{equation}
Also, if\/ $f:\De_a\ra L\amalg R$ is smooth then in oriented
Kuranishi spaces we have
\begin{gather}
\left.\begin{cases}\De_a, & \!f(\De_a)\subset L\\
\De_a\t\Ker\bar{\pd}_{\la_{(p_+,p_-)}}, &
\!f(\De_a)=\{(p_-,p_+)\}\subset R
\end{cases}\right\}
\label{il5eq12}\\
\t_{\bigl\{\begin{smallmatrix}f, & f(\De_a)\subset L\\
f\t\ev_{(-1,0)}, & f(\De_a)\subset R\end{smallmatrix}\bigr\},
L\amalg\ti R,\btev_i} \tM_{k_1+1}^\ma
(\al_1,\be_1,J,f_1,\ldots,f_{i-1};f_{i+k_2},\ldots,f_k)
\nonumber\\
=(-1)^{(1+\deg f)\bigl(1+\sum_{j=1}^{i-1}\deg f_j\bigr)}\,
\tM_{k_1+1}^\ma(\al_1,\be_1,J,f_1,\ldots,f_{i-1},
f,f_{i+k_2},\ldots,f_k).
\nonumber
\end{gather}
\label{il5thm3}
\end{thm}

Here \eq{il5eq11} is proved by a sign calculation using equations
\eq{il4eq47} and \eq{il5eq9}--\eq{il5eq10}, Proposition
\ref{il2prop1}, Theorem \ref{il5thm2}, and the formula
$\pd\De_{a_i}=\sum_{j=0}^{a_i}(-1)^jF_j^{a_i}(\De_{a_i-1})$ in
oriented manifolds with corners, in the notation of \S\ref{il26},
and \eq{il5eq12} follows in a similar way from equations
\eq{il5eq9}--\eq{il5eq10} and Proposition~\ref{il2prop1}.

\subsection{Orientations on $\oM_{k+1}^\ma(\al,\be,J)$.}
\label{il53}

\begin{dfn} Choose a relative spin structure for $\io:L\ra M$, so
that Theorem \ref{il5thm1} gives orientations on the modified moduli
spaces $\tM_{k+1}^\ma(\al,\be,J)$. Inserting signs in \eq{il4eq36},
define the orientation on $\oM_{k+1}^\ma(\al,\be,J)$ to be that for
which
\begin{equation}
\begin{split}
\tM_{k+1}^\ma(\al,\be,J)=\,&\ts\prod_{0\ne j\in I}\ep_{\al(j)}\,
(-1)^{\sum_{0\ne j\in I}\eta_{\al(j)}\bigl[k-j+\sum_{l\in
I:l>j}\eta_{\al(l)}\bigr]}\\
&\oM_{k+1}^\ma(\al,\be,J)\t\ts \prod_{i\in
I}\Ker\bar{\pd}_{\la_{\al(i)}}
\end{split}
\label{il5eq13}
\end{equation}
holds as a product of oriented Kuranishi spaces. This orientation on
$\oM_{k+1}^\ma(\al,\be,J)$ depends on the choices of a relative spin
structure for $\io:L\ra M$, and the $\la_{(p_-,p_+)}$ in
\S\ref{il42}, and the orientations $o_{(p_-,p_+)}$ for the
$\Ker\bar{\pd}_{\la_{(p_-,p_+)}}$ in \S\ref{il52}. The complicated
choice of sign in \eq{il5eq13} will be explained in Remark
\ref{il5rem2}(c). One thing it does is achieve a fairly simple form
for the sign in \eq{il5eq14} below.
\label{il5dfn3}
\end{dfn}

We compute the orientations in Theorem \ref{il4thm1}. The theorem
will be important in \cite{AkJo}, where we do not use moduli
spaces~$\oM_{k+1}^\ma(\al,\be,J,f_1,\ldots,f_k)$.

\begin{thm} Using the orientations of Definition {\rm\ref{il5dfn3},}
the isomorphism\/ \eq{il4eq4} in oriented Kuranishi spaces becomes:
\begin{gather}
\pd\oM_{k+1}^\ma(\al,\be,J)\cong
\coprod_{\begin{subarray}{l}
k_1+k_2=k+1,\;1\le i\le k_1,\\
I_1\cup_iI_2=I,\;\al_1\cup_i\al_2=\al,\\
\be_1+\be_2=\be\end{subarray}}\!\!\!
\begin{aligned}[t]
\ze_1\,\oM_{k_2+1}^\ma(\al_2,\be_2,J)\t_{\bev,L\amalg
R,\bev_i}&\\
\oM_{k_1+1}^\ma(\al_1,\be_1,J),&
\end{aligned}
\label{il5eq14}\\
\text{where}\quad \ze_1=(-1)^{n+\bigl(i+\sum_{j\in
I:0<j<i}\eta_{\al(j)}\bigr) \bigl(1+k_2+\sum_{l\in I:i\le
l<i+k_2}\eta_{\al(l)}\bigr)}
\label{il5eq15}
\end{gather}
if\/ $i\notin I_1$\/ and $0\notin I_2,$ and
\begin{equation}
\ze_1=(-1)^{n+\bigl(i+\sum_{j\in I:0<j<i}\eta_{\al(j)}\bigr)
\bigl(\eta_{\al_1(i)}+1+k_2+\sum_{l\in I:i\le
l<i+k_2}\eta_{\al(l)}\bigr)}
\label{il5eq16}
\end{equation}
if\/ $i\in I_1,$ $0\in I_2,$ and\/ $\al_2(0)=\si\ci\al_1(i)$. Note
that in the cases not covered by \eq{il5eq15} and\/ \eq{il5eq16} we
have\/ $\oM_{k_1+1}^\ma(\al_1,\be_1,J)\ab \t_{\bev_i,L\amalg
R,\bev}\oM_{k_2+1}^\ma(\al_2,\be_2,J)=\emptyset,$ so we do not need
to define\/~$\ze_1$.
\label{il5thm4}
\end{thm}

\begin{proof} Substitute \eq{il5eq13}, as an isomorphism of oriented
Kuranishi spaces, into \eq{il5eq3} three times for $k,\al,\be$ and
$k_1,\al_1,\be_1$ and $k_2,\al_2,\be_2$. This yields
\begin{gather}
\pd\oM_{k+1}^\ma(\al,\be,J)\t\ts\prod\limits_{j\in I}
\Ker\bar{\pd}_{\la_{\al(j)}}=\displaystyle\coprod
\nolimits_{\begin{subarray}{l}
k_1+k_2=k+1,\;1\le i\le k_1,\; I_1\cup_iI_2=I,\\
\al_1\cup_i\al_2=\al,\; \be_1+\be_2=\be\end{subarray}}
\,(-1)^{n+i+ik_2}
\nonumber\\
\ts\prod\limits_{0\ne j\in I}\ep_{\al(j)} \prod\limits_{0\ne j\in
I_1}\ep_{\al_1(j)}\prod\limits_{0\ne j\in I_2}\ep_{\al_2(j)}
(-1)^{\sum_{0\ne j\in I}\eta_{\al(j)}[k-j+\sum_{l\in
I:l>j}\eta_{\al(l)}]}
\nonumber\\
(-1)^{\sum_{0\ne j\in I_1}\eta_{\al_1(j)}[k_1\!-\!j\!+\!\sum_{l\in
I_1:l>j}\eta_{\al_1(l)}]} (-1)^{\sum_{0\!\ne\!j\!\in\!I_2}
\eta_{\al_2(j)}[k_2\!-\!j\!+\!\sum_{l\in
I_2:l>j}\eta_{\al_2(l)}]}
\label{il5eq17}\\
\bigl(\oM_{k_2+1}^\ma(\al_2,\be_2, J)\t\ts\prod_{j\in
I_2}\Ker\bar{\pd}_{\la_{\al_2(j)}}\bigr)\t_{\btev,L\amalg\ti
R,\btev_i}
\nonumber\\
{}\qquad \bigl(\oM_{k_1+1}^\ma(\al_1,\be_1,J) \!\t\!\ts\prod_{j\in
I_1}\Ker\bar{\pd}_{\la_{\al_1(j)}}\bigr). \nonumber
\end{gather}
The left hand side is $\pd\tM_{k+1}^\ma(\al,\be,J)$. Fix
$i,\ldots,\be_2$ in \eq{il5eq17}, and first consider the case
$i\notin I_1$ and $0\notin I_2$. Then we have
\begin{equation}
\begin{split}
& \bigl(\oM_{k_2+1}^\ma (\al_2,\be_2,J)\t\ts\prod_{j\in I_2}\Ker
\bar{\pd}_{\la_{\al_2(j)}}\bigr)\\
&\qquad\t_{\btev,L\amalg\ti R,\btev_i}
\bigl(\oM_{k_1+1}^\ma(\al_1,\be_1,J)\t\ts\prod_{j\in I_1}\Ker
\bar{\pd}_{\la_{\al_1(j)}}\bigr)\\
&=(-1)^{k_1\sum_{l\in I_2}\eta_{\al_2(l)}}
\oM_{k_2+1}^\ma(\al_2,\be_2,J)\t_{\bev,L\amalg R,\bev_i}\\
&\qquad \oM_{k_1+1}^\ma(\al_1,\be_1,J)\t\bigl(\ts\prod_{j\in
I_1}\Ker \bar{\pd}_{\la_{\al_1(j)}}\bigr)\t \bigl(\ts\prod_{j\in
I_2}\Ker \bar{\pd}_{\la_{\al_2(j)}}\bigr)\\
&=(-1)^{k_1\sum_{l\in I_2}\eta_{\al_2(l)}} (-1)^{(\sum_{j\in
I_1:j>i}\eta_{\al_1(j)})(\sum_{l\in I_2}\eta_{\al_2(l)})}\\
&\qquad  \oM_{k_2+1}^\ma(\al_2,\be_2,J)\t_{\bev,L\amalg
R,\bev_i}\oM_{k_1+1}^\ma(\al_1,\be_1,J) \t\bigl(\ts\prod_{j\in
I}\Ker\bar{\pd}_{\la_{\al(j)}}\bigr).
\end{split}
\label{il5eq18}
\end{equation}

Here in the first step we pull the factors $\prod_{j\in I_1}\Ker
\bar{\pd}_{\la_{\al_1(j)}}$ and $\prod_{j\in I_2}\Ker
\bar{\pd}_{\la_{\al_2(j)}}$, which are not involved in the fibre
product, out to the right. Since $\prod_{j\in I_1}\Ker
\bar{\pd}_{\la_{\al_1(j)}}$ is already on the right, it causes no
sign changes. Pulling $\prod_{j\in I_2}\Ker\bar{\pd}_{
\la_{\al_2(j)}}$ through the fibre product with $L\amalg\ti R$ and
then through $\tM_{k_2+1}^\ma(\al_2,\be_2,J)$ changes orientations
by a factor $(-1)^{\dim(\prod_{j\in I_2}\Ker
\bar{\pd}_{\la_{\al_2(j)}})(\dim L\amalg\ti R +
\dim\tM_{k_1+1}^\ma(\al_1,\be_1,J))}$. Using \eq{il4eq37} to compute
$\dim\tM_{k_1+1}^\ma(\al_1,\be_1,J)$ and omitting even terms $2$,
$2n$ and $\mu_L(\be_1)$ in $\dim L\amalg\ti R +
\dim\tM_{k_1+1}^\ma(\al_1,\be_1,J)$ gives the sign on the third line
of \eq{il5eq18}. In the fifth and sixth lines we reorder
$\bigl(\prod_{j\in I_1}\Ker\bar{\pd}_{\la_{\al_1(j)}}\bigr)\t
\bigl(\prod_{j\in I_2}\Ker \bar{\pd}_{\la_{\al_2(j)}}\bigr)$ to
obtain $\bigl(\prod_{j\in I}\Ker\bar{\pd}_{\la_{\al(j)}}\bigr)$. By
\eq{il4eq5}, this means swapping over factors $\prod_{j\in
I_1:j>i}\Ker\bar{\pd}_{\la_{\al_1(j)}}$ and $\prod_{l\in I_2}\Ker
\bar{\pd}_{\la_{\al_2(l)}}$, and so contributes the sign
$(-1)^{(\sum_{j\in I_1:j>i}\eta_{\al_1(j)})(\sum_{l\in
I_2}\eta_{\al_2(l)})}$ in the fifth line. Combining signs in
\eq{il5eq17} and \eq{il5eq18} we obtain \eq{il5eq15}, proving the
theorem in the case $i\notin I_1$ and $0\notin I_2$. The second case
is similar.
\end{proof}

\begin{rem} If we reverse the order of the fibre product in
\eq{il5eq14} using Proposition \ref{il2prop1}(a) and \eq{il4eq10},
noting that the fibre product is over $L$ with $\dim L=n$ in the
case $i\notin I_1$, $0\notin I_2$, and over $R$ with $\dim R=0$ in
the case $i\in I_1$, $0\in I_2$, we obtain
\begin{equation}
\begin{gathered}
\pd\oM_{k+1}^\ma(\al,\be,J)\cong\coprod_{\begin{subarray}{l}
k_1+k_2=k+1,\;1\le i\le k_1,\; I_1\cup_iI_2=I,\\
\al_1\cup_i\al_2=\al,\; \be_1+\be_2=\be\end{subarray}
\!\!\!\!\!\!\!\!\!\!\!\!\!\!\!\!\!\!\!\!\!\!\!\!}
\begin{aligned}[t]
\ze_2\,\oM_{k_1+1}^\ma(\al_1,\be_1,J)\t_{\bev_i,L\amalg R,\bev}&\\
\oM_{k_2+1}^\ma(\al_2,\be_2,J)&
\end{aligned}
\end{gathered}
\label{il5eq19}
\end{equation}
in oriented Kuranishi spaces, where
\begin{equation*}
\begin{split}
\ze_2=\,&(-1)^{n+i+\sum_{j\in
I:0<j<i}\eta_{\al(j)}}\\
&\cdot
\begin{cases}
(-1)^{\bigl(k_2+\sum_{j\in I:i\le j<i+k_2}\eta_{\al(j)}\bigr)
\bigl(k_1+i+\sum_{l\in I:i+k_2\le l\le k}\eta_{\al(l)}\bigr)},
& \mbox{if }0\notin I, \\
(-1)^{\bigl(k_2+\sum_{j\in I:i\le j<i+k_2}\eta_{\al(j)}\bigr)
\bigl(k_1+i+\eta_{\al(0)}+
\sum_{l\in I:i+k_2\le l\le k}\eta_{\al(l)}\bigr)},
& \mbox{if }0\in I,
\end{cases}
\end{split}
\end{equation*}
if $i\notin I_1$ and $0\notin I_2$, and
\begin{equation*}
\begin{split}
&\ze_2=(-1)^{n+i+\sum_{j\in I:0<j<i}\eta_{\al(j)}}\\
&\cdot
\begin{cases}
(-1)^{\bigl(\eta_{\al_1(i)}+k_2+\sum_{j\in I:i\le
j<i+k_2}\eta_{\al(j)}\bigr)\bigl(\eta_{\al_2(0)}+ k_1+i+\sum_{l\in
I:i+k_2\le l\le k}\eta_{\al(l)}\bigr)}, &\mbox{if }0\notin I, \\
(-1)^{\bigl(\eta_{\al_1(i)}+k_2+\sum_{j\in I:i\le
j<i+k_2}\eta_{\al(j)}\bigr)\bigl(\eta_{\al_2(0)}+ k_1+i+\sum_{l\in
I:i+k_2\le l\le k}\eta_{\al(l)}\bigr)}, &\mbox{if }0\in I,
\end{cases}
\end{split}
\end{equation*}
if $i\in I_1$, $0\in I_2$, and $\al_2(0)=\si\ci\al_1(i)$. In the
embedded case, when $I=\emptyset$, the sign $\ze_2$ reduces to
$(-1)^{n+i+k_2(k_1+i)}$, which agrees with that calculated by Fukaya
et al.\ in \cite[Prop.~46.2 \& Rem.~46.3]{FOOO} when~$i=1$.
\label{il5rem1}
\end{rem}

\subsection{Orientations on $\oM_{k+1}^\ma(\al,\be,J,f_1,\ldots,f_k)$}
\label{il54}

\begin{dfn} We work in the situation of Definitions \ref{il4dfn5}
and \ref{il4dfn9} with the orientations on the $\tM_{k+1}^\ma(\al,
\be,J,f_1,\ldots,f_k)$ from Definition \ref{il5dfn2}, and
$o_{(p_-,p_+)}$ on $\Ker\bar{\pd}_{\smash{\la_{(p_-,p_+)}}}$ from
\S\ref{il52}. Define $\oM_{k+1}^\ma(\al,\ab\be,\ab J,
f_1,\ldots,f_k)$ to have the unique orientation such that
\begin{equation}
\tM_{k+1}^\ma(\al,\be,J,f_1,\ldots,f_k)\!=\!\begin{cases}
\oM_{k+1}^\ma(\al,\be,J, f_1,\ldots,f_k), & 0\!\notin\!I,\\
\oM_{k+1}^\ma(\al,\be,J,f_1,\ldots,f_k)\!\t\!\Ker
\bar{\pd}_{\la_{\al(0)}},\!\! & 0\!\in\!I,\end{cases}
\label{il5eq20}
\end{equation}
holds, in oriented Kuranishi spaces. This is just \eq{il4eq42}, with
no extra sign added. Similarly, adding signs to \eq{il4eq43}, let
$\oM_{k_1+1}^\ma(\al_1,\ab \be_1,\ab J,\ab f_1,\ldots,f_{i-1};
f_{i+k_2},\ldots,f_k)$ have the unique orientation for which in
oriented Kuranishi spaces we have
\begin{equation}
\begin{split}
&\tM_{k_1+1}^\ma(\al_1,\ab \be_1,\ab J,\ab f_1,\ldots,f_{i-1};
f_{i+k_2},\ldots,f_k)= \\
&\text{\begin{footnotesize}$\displaystyle
\begin{cases}
\oM_{k_1+1}^\ma(\al_1,\be_1,J,f_1,\ldots,f_{i-1};
f_{i+k_2},\ldots,f_k), & \hbox to 0pt{\hss $0,i\notin I_1,$}\\
\oM_{k_1+1}^\ma(\al_1,\be_1,J,f_1,\ldots,f_{i-1};
f_{i+k_2},\ldots,f_k)\t \Ker\bar{\pd}_{\la_{\al_1(0)}},& \hbox to
0pt{\hss $0\in I_1$, $i\notin I_1$,}\\
\begin{aligned}
&(-1)^{\eta_{\al_1(i)}\bigl(1+\sum_{j=1}^{i-1}\deg
f_j+\sum_{j=i+k_2}^k\deg f_j\bigr)}\ep_{\al_1(i)}\\
&\oM_{k_1+1}^\ma(\al_1,\be_1,J,f_1,\ldots,f_{i-1};
f_{i+k_2},\ldots,f_k)\t\Ker\bar{\pd}_{\la_{\al_1(i)}},
\end{aligned}
& \hbox to 0pt{\hss $0\notin I_1$, $i\in I_1$,}\\
\begin{aligned}
&(-1)^{\eta_{\al_1(i)}\bigl(1+\sum_{j=1}^{i-1}\deg
f_j+\sum_{j=i+k_2}^k\deg f_j\bigr)}\ep_{\al_1(i)}\\
&\oM_{k_1+1}^\ma(\al_1,\be_1,J,f_1,\ldots,f_{i-1};
f_{i+k_2},\ldots,f_k)\!\t\!
\Ker\bar{\pd}_{\la_{\al_1(0)}}\!\t\!\Ker\bar{\pd}_{\la_{\al_1(i)}},
\,\,\,\,\,\,\,\,\,\,\,\,\,\,
\end{aligned}
& \hbox to 0pt{\hss $0,i\!\in\!I_1$.}
\end{cases}
$\end{footnotesize}}
\end{split}
\label{il5eq21}
\end{equation}
Reordering the factors using \eq{il2eq5}, \eq{il4eq17} and
\eq{il5eq8} gives
\begin{gather}
\tM_{k_1+1}^\ma(\al_1,\ab \be_1,\ab J,\ab f_1,\ldots,f_{i-1};
f_{i+k_2},\ldots,f_k)=
\label{il5eq22}\\
\text{\begin{footnotesize}$\displaystyle
\begin{cases} \oM_{k_1+1}^\ma(\al_1,\be_1,J,f_1,\ldots,f_{i-1};
f_{i+k_2},\ldots,f_k), & \hbox to 0pt{\hss $0,i\notin I_1,$}\\
\oM_{k_1+1}^\ma(\al_1,\be_1,J,f_1,\ldots,f_{i-1};
f_{i+k_2},\ldots,f_k)\t \Ker\bar{\pd}_{\la_{\al_1(0)}},& \hbox to
0pt{\hss $0\in I_1$, $i\notin I_1$,}\\
\ep_{\si\ci\al_1(i)}\Ker\bar{\pd}_{\la_{\al_1(i)}}\t
\oM_{k_1+1}^\ma(\al_1,\be_1,J,f_1,\ldots,f_{i-1};
f_{i+k_2},\ldots,f_k),& \hbox to
0pt{\hss $0\notin I_1$, $i\in I_1$,}\\
\ep_{\si\ci\al_1(i)}\Ker\bar{\pd}_{\la_{\al_1(i)}}\!\t\!
\oM_{k_1+1}^\ma(\al_1,\be_1,J,f_1,\ldots,f_{i-1};
f_{i+k_2},\ldots,f_k)\!\t\! \Ker\bar{\pd}_{\la_{\al_1(0)}},
\,\,\,\,\,\,\,\,\,\,\,\,\,\,\,\, & \hbox to 0pt{\hss
$0,i\!\in\!I_1$,}
\end{cases}
$\end{footnotesize}}
\nonumber
\end{gather}
\label{il5dfn4}
\end{dfn}

Combining equations \eq{il5eq9}, \eq{il5eq13} and \eq{il5eq20} and
calculating using Proposition \ref{il2prop1}, the definition of
$\ep_{(p_-,p_+)}$ and \eq{il5eq8} to determine the signs, we prove
that:

\begin{thm} An alternative way to define the orientations in
Definition {\rm\ref{il5dfn4},} in terms of the orientation on\/
$\oM_{k+1}^\ma(\al,\be,J)$ given in Definition {\rm\ref{il5dfn3},}
is that
\begin{equation}
\begin{split}
&\oM_{k+1}^\ma(\al,\be,J,f_1,\ldots,f_k)=\\
&\ze_3\,\oM_{k+1}^\ma(\al,\be,J)\t_{\bev_1\t\cdots\t\bev_k,(L\amalg
R)^k, f_1\t\cdots\t f_k}(\De_{a_1}\t\cdots\t\De_{a_k})
\end{split}
\label{il5eq23}
\end{equation}
in oriented Kuranishi spaces, which is\/ \eq{il4eq11} with signs
inserted, where
\begin{gather}
\ze_3\!=\!(-1)^{\sum\limits_{0\ne i\in
I}\!(n\!-\!\eta_{\al(i)})\bigl[\sum\limits_{j=1\!}^i\!(\deg
f_j+1)-\!\!\!\!\sum\limits_{j\in I:0<j\le
i}\!\!\!\!\eta_{\al(j)}\bigr]}(-1)^{(n\!+\!1)\bigl[\sum\limits_{i=1\!}^k
\!(k\!-\!i)(\deg f_i\!+\!1)-\!\!\sum\limits_{0\ne i\in
I}\!\!(k-i)\eta_{\al(i)}\bigr]}
\nonumber\\
\cdot\begin{cases} 1, & 0\notin I, \\
(-1)^{\eta_{\al(0)}\bigl[\sum_{i=1}^k(\deg f_i+1)-\sum_{0\ne i\in
I}\eta_{\al(i)}\bigr]}, & 0\in I.
\end{cases}
\label{il5eq24}
\end{gather}
\label{il5thm5}
\end{thm}

The sign $\ze_3$ in \eq{il5eq24} will be important in determining
the right definition for our $A_\iy$ algebras in \cite{AkJo}. We can
now prove an analogue of Theorem \ref{il5thm3}. Note that the signs
in equations \eq{il5eq25}--\eq{il5eq26} are exactly the same as
those in equations~\eq{il5eq11}--\eq{il5eq12}.

\begin{thm} In the situation of Definition {\rm\ref{il4dfn9},} with
the orientations of Definition {\rm\ref{il5dfn2},} in oriented
Kuranishi spaces we have
\begin{equation}
\begin{aligned}
&\pd\oM_{k+1}^\ma(\al,\be,J,f_1,\ldots,f_k)\cong \\
&\coprod_{i=1}^k\coprod_{j=0}^{a_i}\,\,
\begin{aligned}[t]
&(-1)^{j+1+\sum_{l=1}^{i-1}\deg f_l}\\
&\quad\oM_{k+1}^\ma(\al,\be,J,f_1,\ldots,f_{i-1},f_i\ci F_j^{a_i},
f_{i+1},\ldots,f_k)
\end{aligned}
\\
&\amalg
\coprod_{\begin{subarray}{l}
k_1+k_2=k+1,\;1\le i\le k_1,\\
I_1\cup_iI_2=I,\;\al_1\cup_i\al_2=\al,\\
\be_1+\be_2=\be\end{subarray}}
\begin{aligned}[t]
&(-1)^{n+\bigl(1+\sum_{l=1}^{i-1}\deg f_l\bigr)
\bigl(1+\sum_{l=i}^{i+k_2-1}\deg f_l\bigr)}\\
&\quad\oM_{k_2+1}^\ma (\al_2,\be_2,J,f_i,\ldots,f_{i+k_2-1})
\t_{\bev,L\amalg R,\bev_i}\\
&\quad\oM_{k_1+1}^\ma(\al_1,\be_1,J,f_1,\ldots,f_{i-1};
f_{i+k_2},\ldots,f_k).
\end{aligned}
\end{aligned}
\label{il5eq25}
\end{equation}
Also, if\/
$f:\De_a\ra L\amalg R$ is smooth then in oriented Kuranishi spaces
we have
\begin{gather}
\De_a\t_{f, L\amalg R,\bev_i}\oM_{k_1+1}^\ma
(\al_1,\be_1,J,f_1,\ldots,f_{i-1};f_{i+k_2},\ldots,f_k)
\label{il5eq26}\\
=(-1)^{(1+\deg f)\bigl(1+\sum_{j=1}^{i-1}\deg f_j\bigr)}\,
\oM_{k_1+1}^\ma(\al_1,\be_1,J,f_1,\ldots,f_{i-1},
f,f_{i+k_2},\ldots,f_k). \nonumber
\end{gather}
\label{il5thm6}
\end{thm}

\begin{proof} To prove \eq{il5eq25}, we substitute \eq{il5eq20} and
\eq{il5eq22} into \eq{il5eq11}. We must consider separately the
cases $0\notin I$ and $0\in I$ in \eq{il5eq20}. As $0\in I$ if and
only if $0\in I_1$ by \eq{il4eq4} and \eq{il4eq5}, these determine
whether or not $0\in I_1$, but for each $i,\ldots,\be_2$ in
\eq{il5eq25} we must still consider separately the cases $i\notin
I_1$ and $i\in I_1$ in \eq{il5eq22}, so there are four cases to
consider. We explain the most complicated case $0\in I$ and $0,i\in
I_1$. Then substituting \eq{il5eq20} and \eq{il5eq22} into
\eq{il5eq11} yields in oriented Kuranishi spaces
\begin{align*}
\pd\bigl(\oM_{k+1}^\ma(\al,&\be,J,f_1,\ldots,f_k)\t\Ker
\bar{\pd}_{\la_{\al(0)}}\bigr)=\\
&\bigl(\pd(\oM_{k+1}^\ma(\al,\be,J,f_1,\ldots,f_k)\bigr)\t\Ker
\bar{\pd}_{\la_{\al(0)}}
\end{align*}
on the left hand side, using Proposition \ref{il2prop1}(a) and $\pd
\bigl(\Ker\bar{\pd}_{\la_{\al(0)}}\bigr)=\emptyset$, and
\begin{equation*}
(-1)^{j\!+\!1\!+\!\sum_{l=1}^{i-1}\deg f_l}
\oM_{k\!+\!1}^\ma(\al,\be,J,f_1,\ldots,f_{i-1},f_i\ci F_j^{a_i},
f_{i\!+\!1},\ldots,f_k)\!\t\!\Ker \bar{\pd}_{\la_{\al(0)}}
\end{equation*}
for the first term on the right hand side, for each $j$, and
\begin{equation*}
\begin{split}
&(-1)^{n+\bigl(1+\sum_{l=1}^{i-1}\deg f_l\bigr)
\bigl(1+\sum_{l=i}^{i+k_2-1}\deg f_l\bigr)} \\
&\quad
\bigl(\oM_{k_2+1}^\ma (\al_2,\be_2,J,f_i,\ldots,f_{i+k_2-1}) \t\Ker
\bar{\pd}_{\la_{\al_2(0)}}\bigr)
\t_{\btev,L\amalg\ti R,\btev_i}\\
&\quad
\bigl(\ep_{\si\ci\al_1(i)}\Ker\bar{\pd}_{\la_{\al_1(i)}}\!\t\!
\oM_{k_1+1}^\ma(\al_1,\be_1,J,f_1,\ldots,f_{i-1};
f_{i+k_2},\ldots,f_k)\!\t\!\Ker\bar{\pd}_{\la_{\al_1(0)}}\bigr)\\
&= (-1)^{n+\bigl(1+\sum_{l=1}^{i-1}\deg f_l\bigr)
\bigl(1+\sum_{l=i}^{i+k_2-1}\deg f_l\bigr)}
\oM_{k_2+1}^\ma(\al_2,\be_2,J,f_i,\ldots,f_{i+k_2-1})\t\\
&\quad
\bigl(\ep_{\si\ci\al_1(i)}\Ker\bar{\pd}_{\la_{\al_2(0)}}
\t_{\la_{\al_2(0)}(-1,0)}\Ker\bar{\pd}_{\la_{\al_1(i)}}\bigr)\t\\
&\quad
\oM_{k_1+1}^\ma(\al_1,\be_1,J,f_1,\ldots,f_{i-1};
f_{i+k_2},\ldots,f_k)\t\Ker\bar{\pd}_{\la_{\al_1(0)}}\\
&=(-1)^{n+\bigl(1+\sum_{l=1}^{i-1}\deg f_l\bigr)
\bigl(1+\sum_{l=i}^{i+k_2-1}\deg f_l\bigr)}\\
&\quad
\bigl(\oM_{k_2+1}^\ma (\al_2,\be_2,J,f_i,\ldots,f_{i+k_2-1})
\t_{\bev,L\amalg R,\bev_i}\\
&\quad\oM_{k_1+1}^\ma(\al_1,\be_1,J,f_1,\ldots,f_{i-1};
f_{i+k_2},\ldots,f_k)\bigr)\t\Ker\bar{\pd}_{\la_{\al_1(0)}}
\end{split}
\end{equation*}
for the final term on the right hand side, for fixed
$i,\ldots,\be_2$ with $i\in I_1$. Here we use the fact that
$\Ker\bar{\pd}_{\la_{\al_2(0)}}\t_{\la_{\al_2(0)}(-1,0)}
\Ker\bar{\pd}_{\la_{\al_1(i)}}$ is a point with sign
$\ep_{\al_2(0)}$, and as $\al_2(0)=\si\ci\al_1(i)$ this cancels with
$\ep_{\si\ci\al_1(i)}$, so that the fifth line is just a point with
sign 1. In the last two lines, the fibre product over $L\amalg R$ is
actually a fibre product over the point $\al_1(i)$ in $R$, so it is
a product, as in the fifth to seventh lines.

The last three equations are the oriented products of the
corresponding terms in \eq{il5eq25} with
$\Ker\bar{\pd}_{\la_{\al_1(0)}}$. This proves \eq{il5eq25} in the
case $0\in I$ and $0,i\in I_1$. The other cases follow by similar
but simpler arguments. To prove \eq{il5eq26} we substitute
\eq{il5eq20} and \eq{il5eq22} into \eq{il5eq12}, and use the same
method.
\end{proof}

\begin{rem}{\bf(a)} Theorem \ref{il5thm6} is the main result of this
section. It is important that the signs in \eq{il5eq25} and
\eq{il5eq26} depend only on $n,i,j,k_2$ and the {\it shifted
cohomological degrees\/} $\deg f_j,\deg f$. In particular, they do
not involve the $\ep_{\al(j)},\eta_{\al(j)}$ or $a_j$. Because of
this, in the rest of the paper we will be able to write all our
signs in terms of $\deg f_j,\deg f$, without any correction factors
involving $\ep_{\al(j)},\eta_{\al(j)},a_j$. This was one aim of the
careful definition of orientations above.

Theorem \ref{il5thm6} is an analogue in the immersed case of Fukaya
et al.\ \cite[Prop.~48.1]{FOOO}; roughly speaking, if we substitute
\eq{il5eq26} into \eq{il5eq25}, then we get \cite[Prop.~48.1]{FOOO},
with the same signs, noting that our definition of $\deg f_i$
differs by 1 from that of \cite{FOOO}. Since our signs are
compatible with those of Fukaya et al.\ \cite{FOOO}, we can follow
their proof to construct an $A_\iy$ algebra, and there will be no
new orientation issues, provided we grade our complexes using
shifted cohomological degrees in~\eq{il4eq13}.
\smallskip

\noindent{\bf(b)} In equations \eq{il4eq4}, \eq{il4eq18},
\eq{il4eq40}, \eq{il4eq47}, \eq{il5eq3}, \eq{il5eq11}, \eq{il5eq14},
\eq{il5eq17} and \eq{il5eq25} above, we chose to order the fibre
products as $\oM_{k_2+1}^\ma(\al_2,\ldots)\t_{\ldots}
\oM_{k_1+1}^\ma(\al_1,\ldots)$ rather than the other way round; this
order was reversed in \eq{il5eq19}. Fukaya et al.\ adopt the
opposite order to us, in \cite[Prop.~46.2]{FOOO} for instance. We
can now explain why we chose this order for our fibre products.
Using \eq{il2eq5}, \eq{il4eq14} and \eq{il4eq17} we may rewrite
\eq{il5eq25} and \eq{il5eq26} with the other fibre product order,
which yields:
\begin{gather}
\pd\oM_{k+1}^\ma(\al,\be,J,f_1,\ldots,f_k)=
\label{il5eq27}\\
\coprod_{i=1}^k\coprod_{j=0}^{a_i}\,\,
\begin{aligned}[t]
&(-1)^{j+1+\sum_{l=1}^{i-1}\deg f_l}\\
&\oM_{k+1}^\ma(\al,\be,J,f_1,\ldots,f_{i-1},f_i\ci F_j^{a_i},
f_{i+1},\ldots,f_k)
\end{aligned}
\nonumber\allowdisplaybreaks\\
\amalg \!\!\!\coprod_{\begin{subarray}{l}
k_1+k_2=k+1,\;1\le i\le k_1,\\
I_1\cup_iI_2=I,\;\al_1\cup_i\al_2=\al,\\
\be_1+\be_2=\be\end{subarray}}
\begin{aligned}[t]
&(-1)^{n+1+i+ik_2+\sum_{l=1}^{i-1}\deg f_l+
\sum_{l=i}^{i+k_2-1}\deg f_l\sum_{l=i+k_2}^k\deg f_l}\cdot\\
&\oM_{k_1+1}^\ma(\al_1,\be_1,J,f_1,\ldots,f_{i-1};
f_{i+k_2},\ldots,f_k)\t_{\bev_i,L\amalg R,\bev}\\
&\qquad \oM_{k_2+1}^\ma
(\al_2,\be_2,J,f_i,\ldots,f_{i+k_2-1})\,\cdot
\end{aligned}
\nonumber\allowdisplaybreaks
\\
\text{\begin{footnotesize}$\displaystyle\left.\begin{cases}
1, & \hbox to 0pt{\hss $0,i\!\notin\!I_1$,}\\
(-1)^{\eta_{\al_1(0)}\sum_{l=i}^{i+k_2-1}\deg f_l},
 & \hbox to 0pt{\hss $0\!\in\!I_1$,
$i\!\notin\!I_1$,}\\
(-1)^{\eta_{\al_1(i)}\bigl[1+\!\sum\limits_{l=1}^{i-1}\!\deg
f_l\!+\!\sum\limits_{l=i+k_2}^k\!\deg f_l\bigr]}
(-1)^{\eta_{\si\ci\al_1(i)}\bigl[\eta_{\al_1(i)}\!+\!
\sum\limits_{l=i}^{i+k_2-1}\!\deg f_l\bigr]},
& \hbox to 0pt{\hss $0\!\notin\!I_1$,
$i\!\in\!I_1$,}\\
(-1)^{\eta_{\al_1(i)}\bigl[1\!+\!\!\sum\limits_{l=1}^{i-1}\!\!\deg
f_l+\!\!\sum\limits_{l=i+k_2}^k\!\!\deg f_l\bigr]}
(-1)^{(\eta_{\al_1(0)}\!+\!\eta_{\si\ci\al_1(i)})
\bigl[\eta_{\al_1(i)}\!+\!\!
\sum\limits_{l=i}^{i+k_2-1}\!\!\deg f_l\bigr]},
\,\,\,\,\,\,\,\,\,\,\,\,\,\,\,\,\,\, & \hbox to 0pt{\hss
$0,i\!\in\!I_1$,}
\end{cases}\right\}\!,$\end{footnotesize}}
\nonumber
\allowdisplaybreaks
\\
\oM_{k_1+1}^\ma
(\al_1,\be_1,J,f_1,\ldots,f_{i-1};f_{i+k_2},\ldots,f_k)
\t_{\bev_i,L\amalg R,f}\De_a
\label{il5eq28}\\
=(-1)^{(\deg f+1)\sum_{l=i+k_2}^k\deg f_l}\,
\oM_{k_1+1}^\ma(\al_1,\be_1,J,f_1,\ldots,f_{i-1},
f,f_{i+k_2},\ldots,f_k)\,\cdot
\nonumber\\
\text{\begin{footnotesize}$\displaystyle \left.\begin{cases} 1, &
\hbox to 0pt{\hss $0,i\!\notin\!I_1$,}\\
(-1)^{\eta_{\al_1(0)}(\deg f+1)}, & \hbox to 0pt{\hss $0\!\in\!I_1$,
$i\!\notin\!I_1$,}\\
(-1)^{\eta_{\al_1(i)}\bigl[1+\sum_{l=1}^{i-1}\deg
f_l+\sum_{l=i+k_2}^k\deg f_l\bigr]} (-1)^{\eta_{\si\ci\al_1(i)}(\deg
f+1+\eta_{\al_1(i)})},
&\hbox to 0pt{\hss $0\!\notin\!I_1$, $i\!\in\!I_1$,}\\
(-1)^{\eta_{\al_1(i)}\bigl[1\!+\!\sum_{l=1}^{i-1}\!\deg
f_l+\!\sum_{l=i+k_2}^k\!\deg f_l\bigr]}
(-1)^{(\eta_{\al_1(0)}+\eta_{\si\ci\al_1(i)})(\deg
f+1+\eta_{\al_1(i)})}, \,\,\,\,\,\,\,\,\,\,\,\,\,\,\,\,\,\,\,\,\, &
\hbox to 0pt{\hss $0,i\!\in\!I_1$,}
\end{cases}\right\}\!.
$\end{footnotesize}} \nonumber
\end{gather}

Observe that equations \eq{il5eq27}--\eq{il5eq28} have complicated
extra sign terms involving
$\eta_{\al_1(0)},\eta_{\al_1(i)},\eta_{\si\ci\al_1(i)}$, so they are
not simply written in terms of $n,i,j,k_2$ and $\deg f_j,\deg f$, as
\eq{il5eq25}--\eq{il5eq26} were. Thus we prefer the fibre product
order in \eq{il5eq25}--\eq{il5eq26}. One might guess that by
changing the signs in \eq{il5eq20} and \eq{il5eq21}, altering the
orientations of $\oM_{k+1}^\ma(\al,\be,J,f_1,\ldots,f_k)$,
$\oM_{k_1+1}^\ma(\al_1,\be_1,J,f_1,\ldots, f_{i-1};\ab
f_{i+k_2},\ab\ldots,\ab f_k)$, one could eliminate the troublesome
terms in \eq{il5eq27} and \eq{il5eq28}, to get signs depending only
on $n,i,j,k_1,k_2,\deg f_j,\deg f$. However, calculations by the
authors indicate that this is impossible, at least with the
orientation conventions of~\S\ref{il24}.
\smallskip

\noindent{\bf(c)} We defined the orientation on
$\oM_{k+1}(\al,\be,J)$ in \S\ref{il53} by \eq{il5eq13}, which
includes a complicated choice of sign. We chose this particular sign
by requiring that if $a_i=n$ for $i\notin I$ and $a_i=0$ for $i\in
I$, so that $\deg f_i=-1$ for $i\notin I$ and $\deg
f_i=\eta_{\al(i)}-1$ for $i\in I$, then the sign $\ze_3$ in
\eq{il5eq23} and \eq{il5eq24} should be 1. The sign in \eq{il5eq13}
was then determined as in the proof of Theorem \ref{il5thm6}. The
motivation for this choice is that we have found natural
orientations for the moduli spaces
$\oM_{k+1}^\ma(\al,\be,J,f_1,\ldots,f_k)$, with good properties
under boundaries as in Theorem \ref{il5thm6}. Now we have
\begin{equation*}
\begin{split}
&\oM_{k+1}^\ma(\al,\be,J)=\\
&\oM_{k+1}^\ma(\al,\be,J)\t_{\prod_{i=1}^k\bev_i,(L\amalg
R)^k,\prod_{i=1}^k\bigl\{\begin{smallmatrix}\id_L, & i\notin I\\
\id_{\{\al(i)\}}, & i\in I\end{smallmatrix}\bigr\}}
\smash{\left.\prod_{i=1}^k
\begin{cases}L, & \!i\notin I\\
\{\al(i)\}, & \!i\in I\end{cases} \right\}.}
\end{split}
\end{equation*}
This is like the fibre product \eq{il4eq11} defining
$\oM_{k+1}^\ma(\al,\be,J,f_1,\ldots,f_k)$, but replacing
$f_i:\De_{a_i}\!\ra\!L\amalg R$ by $\id_L:L\!\ra\!L\amalg R$ for
$i\notin I$ and $\id_{\{\al(i)\}}:\{\al(i)\}\!\ra\!L\amalg R$ for
$i\in I$. Thus we can think of $\oM_{k+1}^\ma(\al,\be,J)$ as a
generalization of $\oM_{k+1}^\ma(\al,\be,J,f_1,\ab\ldots,\ab f_k)$
in which $a_i=n=\dim L$ for $i\notin I$ and $a_i=0=\dim\{\al(i)\}$
for $i\in I$, and so we should arrange to get $\ze_3=1$ in
\eq{il5eq24} in this case.
\label{il5rem2}
\end{rem}

The orientations on $\oM_{k+1}^\ma(\al,\be,J,f_1,\ldots,f_k)$ depend
on the choice of paths $\la_{(p_-,p_+)}$ in \S\ref{il43} and
orientations $o_{(p_-,p_+)}$ on $\Ker\bar\pd_{
\smash{\la_{(p_-,p_+)}}}$ in \S\ref{il52}, for $(p_-,p_+)\in R$.
Suppose we change to alternative choices $\ti\la_{(p_-,p_+)},\ti
o_{(p_-,p_+)}$. Note that changing $\la_{(p_-,p_+)}$ to
$\ti\la_{(p_-,p_+)}$ alters the index $\eta_{(p_-,p_+)}$ in
\eq{il4eq8} to $\ti\eta_{(p_-,p_+)}$, and this changes the shifted
cohomological degree $\deg f$ in~\eq{il4eq13}.

As $\la_{(p_-,p_+)},\ti\la_{(p_-,p_+)}$ are paths in {\it
oriented\/} Lagrangian spaces, $\eta_{(p_-,p_+)},
\ti\eta_{(p_-,p_+)}$ differ by an even number, so we may write
$\ti\eta_{(p_-,p_+)}=\eta_{(p_-,p_+)}+2d_{(p_-,p_+)}$ for
$d_{(p_-,p_+)}\in\Z$. So degrees in \eq{il4eq13} change by $\deg
f\mapsto\deg f+2d_{(p_-,p_+)}$ if $f:\De_a\ra\{(p_-,p_+)\}$. Since
the changes in $\eta_{(p_-,p_+)},\deg f$ are even, all the signs
above, such as those in \eq{il5eq25} and \eq{il5eq26}, are
unchanged. Here is how changing to alternative choices
$\ti\la_{(p_-,p_+)},\ti o_{(p_-,p_+)}$ affects the orientations
on~$\oM_{k+1}^\ma(\al,\be,J,f_1,\ldots,f_k)$.

\begin{prop} In the situation above, suppose that for all\/
$(p_-,p_+)\in R$ we replace the paths $\la_{(p_-,p_+)}$ in
{\rm\S\ref{il43}} and orientations $o_{(p_-,p_+)}$ on
$\Ker\bar\pd_{\smash{\la_{(p_-,p_+)}}}$ in {\rm\S\ref{il52}} by
alternative choices $\ti\la_{(p_-,p_+)},\ti o_{(p_-,p_+)},$ so
that\/ $\eta_{(p_-,p_+)}$ is replaced by $\ti\eta_{(p_-,p_+)},$ but
we make no other changes. Then for all\/ $(p_-,p_+)\in R$ there
exist\/ $\xi_{(p_-,p_+)}=\pm 1$ depending only on $\la_{(p_-,p_+)},
o_{(p_-,p_+)},\ti\la_{(p_-,p_+)},\ti o_{(p_-,p_+)},$ such that for
all\/ $k,\al,\be,f_1,\ldots,f_k$ the orientation on
$\oM_{k+1}^\ma(\al,\be,J,f_1,\ldots,f_k)$ changes by a factor
\begin{equation}
\prod_{0\ne i\in I}\xi_{\si\ci\al(i)}\cdot\begin{cases}
\xi_{\al(0)}, & 0\in I,\\ 1, & 0\notin I.\end{cases}
\label{il5eq29}
\end{equation}
\label{il5prop}
\end{prop}

\begin{proof} When we change only the orientations $o_{(p_-,p_+)}$,
so that $\ti\la_{(p_-,p_+)}=\la_{(p_-,p_+)}$, using \eq{il5eq9},
\eq{il5eq20} and the fact that the orientation of
$\tM_{k+1}^\ma(\al,\be,J)$ is independent of the $o_{(p_-,p_+)}$, we
see that changing from $o_{(p_-,p_+)}$ to $\ti o_{(p_-,p_+)}$ for
all $(p_-,p_+)\in R$ changes the orientation of
$\oM_{k+1}^\ma(\al,\be,J,\ab f_1,\ab\ldots,\ab f_k)$ by a factor
\eq{il5eq29}, with the $\xi_{(p_-,p_+)}$ determined by~$\ti
o_{(p_-,p_+)}=\xi_{(p_-,p_+)}o_{(p_-,p_+)}$. For the general case,
we must also consider how the {\it virtual tangent bundle} of
$\tM_{k+1}(\al,\be,J)$ in \S\ref{il51} changes when we replace
$\smash{\la_{(p_-,p_+)}}$ by~$\smash{\ti\la_{(p_-,p_+)}}$.

In fact the virtual tangent bundle changes by direct sum with
$\bigop_{i\in I}V_{\al(i)}$, where
$V_{(p_-,p_+)}=\Ind\bar\pd_{(p_-,p_+)}$ for $(p_-,p_+)\in R$ are
oriented virtual vector spaces, and $\bar\pd_{(p_-,p_+)}$ is an
elliptic operator on the disc $D=\{(x,y)\in\R^2:x^2+y^2\le 1\}$ with
boundary conditions $\la_{(p_-,p_+)}(x,y)$ on the semicircle $x\le
0$ and $\ti\la_{(p_-,p_+)}(-x,y)$ on the semicircle $x\ge 0$. There
is an isomorphism $\xi_{(p_-,p_+)}V_{(p_-,p_+)}\cong\Ker
\bar\pd_{\ti\la_{\smash{(p_-,p_+)}}}
\ominus\Ker\bar\pd_{\la_{\smash{(p_-,p_+)}}}$, where
$\Ker\bar\pd_{\la_{\smash{(p_-, p_+)}}},
\Ker\bar\pd_{\ti\la_{(p_-,p_+)}}$ have orientations
$o_{(p_-,p_+)},\ti o_{(p_-,p_+)}$ and $\xi_{(p_-,p_+)}=\pm 1$, and
the proposition holds with these~$\xi_{(p_-,p_+)}$.
\end{proof}

\subsection{Adding families of almost complex structures}
\label{il55}

We can generalize the material above to the moduli spaces with
smooth families of almost complex structures in \S\ref{il45}. First
we explain how to orient the moduli spaces $\oM_{k+1}^\ma
(\al,\be,J_t:t\in{\cal T})$ of Definition \ref{il4dfn6},
generalizing Definition~\ref{il5dfn3}.

\begin{dfn} We work in the situation of \S\ref{il45}, with $M,L$,
and $J_t:t\in{\cal T}$, with the additional assumptions of
\S\ref{il51}--\S\ref{il53}, that is, that we have chosen a relative
spin structure for $\io:L\ra M$, and orientations for
$\Ker\bar\pd_{\la_{(p_+,p_-)}}$ for all $(p_+,p_-)\in R$. We also
suppose that $\cal T$ is oriented. At a point $\bigl(t,[\Si,\vec
z,u,l,\bar{u}]\bigr)$ of $\oM_{k+1}^\ma (\al,\be,J_t:t\in{\cal T})$,
we have an isomorphism of virtual vector spaces
\begin{equation}
T_{(t,[\Si,\vec z,u,l,\bar{u}])}\oM_{k+1}^\ma(\al,\be,J_t:t\in{\cal
T})\cong T_t{\cal T}\op T_{[\Si,\vec z,u,l,\bar{u}]}
\oM_{k+1}^\ma(\al,\be,J_t).
\label{il5eq30}
\end{equation}
In Definition \ref{il5dfn3} we constructed an orientation on
$\oM_{k+1}^\ma(\al,\be,J_t)$, and hence on $T_{[\Si,\vec
z,u,l,\bar{u}]}\oM_{k+1}^\ma(\al,\be,J_t)$. As $\cal T$ is oriented,
$T_t{\cal T}$ is oriented. Define $\oM_{k+1}^\ma(\al,\be,J_t:
t\in{\cal T})$ to have the orientation such that \eq{il5eq30} holds
in oriented virtual vector spaces.

A special case of this which is useful for computing signs in
formulae is to take $J_t=J$ for some almost complex structure $J$
and all $t\in{\cal T}$. Then
\begin{equation}
\oM_{k+1}^\ma(\al,\be,J_t:t\in{\cal T})\cong {\cal
T}\t\oM_{k+1}^\ma(\al,\be,J)
\label{il5eq31}
\end{equation}
holds in oriented Kuranishi spaces.
\label{il5dfn5}
\end{dfn}

Here is the analogue of Theorem \ref{il5thm4}. We can prove it by
the same method; alternatively, we can take $J_t=J$ for $t\in{\cal
T}$, so that \eq{il5eq31} holds, and then deduce the signs in
\eq{il5eq32} from Proposition \ref{il2prop1} and~\eq{il5eq14}.

\begin{thm} Using the orientations of Definition \ref{il5dfn5},
the isomorphism\/ \eq{il4eq21} in oriented Kuranishi spaces becomes:
\begin{gather}
\begin{split}
&\pd\oM^\ma_{k+1}(\al,\be,J_t:t\in{\cal T})\cong
\oM^\ma_{k+1}(\al,\be,J_t:t\in\pd{\cal T})\\
&\coprod_{\begin{subarray}{l}
k_1+k_2=k+1,\; 1\le i\le k_1,\\
I_1\cup_iI_2=I,\;\al_1\cup_i\al_2=\al,\\
\be_1+\be_2=\be
\end{subarray}}\!\!\!\!\!
\begin{aligned}[t]
\ze_4\,\oM^\ma_{k_2+1}(\al_2,\be_2,J_t:t\in{\cal T})\t_{\bs\pi_{\cal
T}\t\bev,{\cal T}\t(L\amalg R),\bs\pi_{\cal T}\t\bev_i}&\\
\oM^\ma_{k_1+1}(\al_1,\be_1,J_t:t\in{\cal T}), &
\end{aligned}
\end{split}
\label{il5eq32}\\
\text{where}\quad \ze_4=(-1)^{\dim{\cal T}+n+\bigl(i+\sum_{j\in
I:0<j<i}\eta_{\al(j)}\bigr) \bigl(1+k_2+\sum_{l\in I:i\le
l<i+k_2}\eta_{\al(l)}\bigr)}
\nonumber
\end{gather}
if\/ $i\notin I_1$\/ and $0\notin I_2,$ and
\begin{equation*}
\ze_4=(-1)^{\dim{\cal T}+n+\bigl(i+\sum_{j\in
I:0<j<i}\eta_{\al(j)}\bigr) \bigl(\eta_{\al_1(i)}+1+k_2+\sum_{l\in
I:i\le l<i+k_2}\eta_{\al(l)}\bigr)}
\end{equation*}
if\/ $i\in I_1,$ $0\in I_2,$ and\/ $\al_2(0)=\si\ci\al_1(i)$.
\label{il5thm7}
\end{thm}

Next we add simplicial chains, and orient the moduli spaces
$\oM_{k+1}^\ma(\al,\be,J_t:t\in{\cal T},f_1,\ldots,f_k)$ of
Definition~\ref{il4dfn7}.

\begin{dfn} In the situation of Definition \ref{il5dfn5}, for
$i=1,\ldots,k$, let $a_i\ge 0$ and $f_i:\De_{a_i}\ra {\cal
T}\t(L\amalg R)$ be a smooth map, as in Definition \ref{il4dfn7}.
Since we have not defined modified moduli spaces
$\tM_{k+1}^\ma(\al,\be,J_t:t\in{\cal T},f_1,\ldots,f_k)$, we cannot
define an orientation on $\tM_{k+1}^\ma(\al,\be,J_t:t\in{\cal
T},f_1,\ldots,f_k)$ following \eq{il5eq20}. Instead, we will take
the analogue of Theorem \ref{il5thm5} to be our definition.
Inserting signs in \eq{il4eq24} motivated by
\eq{il5eq23}--\eq{il5eq24}, define
$\tM_{k+1}^\ma(\al,\be,J_t:t\in{\cal T},f_1,\ldots,f_k)$ to have the
orientation given in oriented Kuranishi spaces by
\begin{align}
\oM_{k+1}^\ma(\al,&\be,J_t:t\!\in\!{\cal T},f_1,\ldots,f_k)\!=\!
\ze_5\bigl((\R^m,\ka^m_k)\t_{\bs\pi_0,\R^m,\bs\pi_{\cal T}}
\oM_{k+1}^\ma(\al,\be,J_t:t\!\in\!{\cal T})\bigr)
\nonumber\\
&\t_{(\bs\pi_1\t\bev_1)\t\cdots\t(\bs\pi_k\t\bev_k),({\cal
T}\t(L\amalg R))^k, f_1\t\cdots\t
f_k}(\De_{a_1}\t\cdots\t\De_{a_k}),
\label{il5eq33}
\end{align}
where $(\R^m,\ka^m_k)$ and $\R^m$ have their natural orientations,
the orientation of $\oM_{k+1}^\ma(\al,\ab\be,\ab J_t:t\in{\cal T})$
is as in Definition \ref{il5dfn5}, and
\begin{equation}
\begin{split}
\ze_5=&(-1)^{\sum\limits_{0\ne i\in
I}(n-\eta_{\al(i)})\bigl[\sum\limits_{j=1}^i(\deg
f_j+1)-\sum\limits_{j\in I:0<j\le
i}\eta_{\al(j)}\bigr]}\\
&(-1)^{(\dim{\cal T}+n+1)\bigl[\sum\limits_{i=1}^k (k-i)(\deg
f_i+1)-\sum\limits_{0\ne i\in I}(k-i)\eta_{\al(i)}\bigr]} \\
&\cdot\begin{cases} 1, & 0\notin I, \\
(-1)^{\eta_{\al(0)}\bigl[\sum_{i=1}^k(\deg f_i+1)-\sum_{0\ne i\in
I}\eta_{\al(i)}\bigr]}, & 0\in I,
\end{cases}
\end{split}
\label{il5eq34}
\end{equation}
where the degrees $\deg f_i$ are as in \eq{il4eq26}. Similarly, we
orient $\oM_{k_1+1}^\ma(\al_1,\be_1,J_t:t\in{\cal T},f_1,\ldots,\ab
f_{i-1};\ab f_{i+k_2},\ab\ldots,f_k)$ by inserting signs in
\eq{il4eq28}. We will not write this sign down explicitly, but we
choose it to satisfy \eq{il5eq36} below.
\label{il5dfn6}
\end{dfn}

Calculation using equations \eq{il5eq32}--\eq{il5eq34} and
Proposition \ref{il2prop1} then yields an analogue of
Theorem~\ref{il5thm6}:

\begin{thm} In the situation of Definition {\rm\ref{il4dfn7},} with the
orientations of Definition {\rm\ref{il5dfn6}} and degrees in
{\rm\eq{il4eq26},} for $k>0$ in oriented Kuranishi spaces we have
\begin{equation}
\begin{split}
&\pd\oM_{k+1}^\ma(\al,\be,J_t:t\in{\cal T},f_1,\ldots,f_k)\cong \\
&\coprod_{i=1}^k\coprod_{j=0}^{a_i}\begin{aligned}[t]
&(-1)^{j+1+\sum_{l=1}^{i-1}\deg f_l}\\
&\oM_{k+1}^\ma(\al,\be,J_t:t\in{\cal T},f_1,\ldots,f_{i-1},f_i\ci
F_j^{a_i}, f_{i+1},\ldots,f_k)
\end{aligned}\\
&\amalg \!\!\!\coprod_{\begin{subarray}{l}
k_1+k_2=k+1,\;1\le i\le k_1, \\
I_1\cup_iI_2=I,\;\al_1\cup_i\al_2=\al,\\
\be_1+\be_2=\be
\end{subarray}}
\begin{aligned}[t]
&(-1)^{\dim{\cal T}+n+\bigl(1+\sum_{l=1}^{i-1}\deg f_l\bigr)
\bigl(1+\sum_{l=i}^{i+k_2-1}\deg f_l\bigr)}\\
&\oM_{k_2+1}^\ma (\al_2,\be_2,J_t:t\!\in\!{\cal
T},f_i,\ldots,f_{i+k_2-1})\\
&\t_{\bs\pi_{\cal T}\t\bev,
{\cal T}\t(L\amalg R),\bs\pi_{\cal T}\t\bev_i}\\
&\oM_{k_1+1}^\ma (\al_1,\be_1,J_t:t\in{\cal T},f_1,\ldots,f_{i-1};
f_{i+k_2},\ldots,f_k).
\end{aligned}
\end{split}
\label{il5eq35}
\end{equation}
When $k=0$ this holds with an extra term
$\oM^\ma_1(\al,\be,J_t:t\in\pd{\cal T})$ on the right hand side, as
in~\eq{il5eq32}.

Also, if\/ $f:\De_a\ra{\cal T}\t(L\amalg R)$ is smooth, in oriented
Kuranishi spaces we have
\begin{gather}
\De_a\t_{f,{\cal T}\t(L\amalg R),\bs\pi_{\cal
T}\t\bev_i}\oM_{k_1+1}^\ma (\al_1,\be_1,J_t:t\!\in\!{\cal
T},f_1,\ldots,f_{i-1};f_{i+k_2},\ldots,f_k)
\label{il5eq36}\\
=\!(-\!1)^{(\deg f\!+\!1)\bigl(1\!+\!\sum_{j=1}^{i-1}\deg f_j\bigr)}
\oM_{k_1+1}^\ma(\al_1,\be_1,J_t:t\!\in\!{\cal T},
f_1,\ldots,f_{i-1},f,f_{i+k_2},\ldots,f_k). \nonumber
\end{gather}
\label{il5thm8}
\end{thm}

\section{Perturbation data and virtual chains}
\label{il6}

We shall now choose {\it perturbation data\/}
$\s_{\be,J,f_1,\ldots,f_k}$ for families of moduli spaces
$\oM_{k+1}^\ma(\be,J,f_1,\ldots,f_k)$ in \eq{il4eq19}, as in
\S\ref{il27}, which are compatible at the boundaries with choices
made for the boundary strata appearing in \eq{il4eq18}. Technically
$\oM_{k+1}^\ma(\be,J,f_1,\ldots,f_k)$ may not be a Kuranishi space,
as the components $\oM_{k+1}^\ma(\al,\ab\be,\ab J,f_1,\ldots,f_k)$
in \eq{il4eq19} may have different virtual dimensions. By
perturbation data for $\oM_{k+1}^\ma(\be,J,f_1,\ldots,f_k)$ we mean
perturbation data for $\oM_{k+1}^\ma(\al,\be,J,f_1,\ldots,f_k)$ for
all $I,\al$ in \eq{il4eq19}, in the obvious way.

Our goal is to define $A_{N,0}$ algebras and gapped filtered $A_\iy$
algebras, which are filtered using $\cal G\subset[0,\iy)\t\Z$. It is
convenient to introduce $\cal G$ at this point. Choose ${\cal
G}\subset[0,\iy)\t\Z$ to satisfy the conditions:
\begin{itemize}
\setlength{\itemsep}{0pt}
\setlength{\parsep}{0pt}
\item[(i)] $\cal G$ is closed under addition with ${\cal
G}\cap(\{0\}\t\Z)=\{(0,0)\}$, and ${\cal G}\cap([0,C]\t\Z)$ is
finite for any $C\ge 0$; and
\item[(ii)] If $\be\in H_2(M,\io(L);\Z)$, and
$\oM_1^\ma(\be,J)\ne\emptyset$ then $\bigl([\om]\cdot\be,
\ha\mu_L(\be)\bigr)\in{\cal G}$.
\end{itemize}
Here (i) is as in \S\ref{il35} and \S\ref{il37}. If we define ${\cal
G}_J$ to be the smallest subset of $[0,\iy)\t\Z$ containing
$\bigl([\om]\cdot\be,\ha\mu_L(\be)\bigr)$ for all $\be\in
H_2(M,\io(L);\Z)$ with $\oM_1^\ma(\be,J)\ne\emptyset$ and closed
under addition, then ${\cal G}_J\cap(\{0\}\t\Z)=\{(0,0)\}$ is
immediate as $[\om]\cdot\be=0$, $\oM_1^\ma(\be,J)\ne\emptyset$ imply
$\be=0$, and ${\cal G}_J\cap([0,C]\t\Z)$ finite for any $C\ge 0$
follows from compactness for the family of stable $J$-holomorphic
curves with area~$\le C$.

Thus there exists at least one subset $\cal G$ satisfying (i),(ii).
However, we do not want to fix ${\cal G}={\cal G}_J$, since in
\S\ref{il8}--\S\ref{il10} we will vary the complex structure $J$,
and we will want $\cal G$ to be independent of $J$. So for the
moment we take $\cal G$ satisfying (i),(ii) to be given. If $\be\in
H_2(M,\io(L);\Z)$ and $\oM_{k+1}^\ma(\be,J) \ne\emptyset$ for any
$k\ge 0$ then $\bigl([\om]\cdot\be, \ha\mu_L(\be)\bigr)\in{\cal G}$.
Write $\nm{\be}=\bnm{([\om]\cdot\be,\ha\mu_L(\be))}$, using the
notation of \eq{il3eq16}. Then $\nm{\be}\ge 0$, and if
$\be=\be_1+\be_2$ for $\be_1,\be_2\in H_2(M,\io(L);\Z)$ with
$\oM_{k_1+1}^\ma(\be_1,J),\ab\oM_{k_2+1}^\ma(\be_2,J)\ne\emptyset$
then $\nm{\be}\ge\nm{\be_1}+\nm{\be_2}$. With this notation we
prove:

\begin{thm} For a given $N\in\N,$ there are
$\X_0\subset\cdots\subset\X_N$ and\/ $\{\s_{\be,J,f_1,\ldots,f_k}\}$
which satisfy the following conditions:
\begin{itemize}
\item[$(N1)$] $\X_0,\ldots,\X_N$ are finite sets of smooth
simplicial chains\/ $f:\De_a\ra L\amalg R$ such that
\begin{itemize}
\item[{\rm(a)}] if\/ $f:\De_a\ra L\amalg R$ lies in\/ $\X_i$ and\/
$a>0$ then\/ $f\ci F_j^a:\De_{a-1}\ra L\amalg R$ lies in\/ $\X_i$
for all\/ $j=0,\ldots,a$, using the notation of\/ {\rm\S\ref{il26};}
and
\item[{\rm(b)}] part\/ {\rm(a)} implies that\/ $\Q\X_i$ is
closed under\/ $\pd,$ and a subcomplex of the singular chains\/
$C_*^\rsi(L\amalg R;\Q)$. We require that the natural projection\/
$H_*(\Q\X_i,\pd)\ra H_*^\rsi(L\amalg R;\Q)$ should be an
isomorphism.
\end{itemize}
\item[$(N2)$] For all\/ $k\ge 0$, $f_1\in\X_{i_1},\ldots,f_k\in
\X_{i_k}$ and $\be\in H_2(M,\io(L);\Z)$ with
$i_1+\cdots+i_k+\nm{\be}+k-1\le N$ and
$\oM_{k+1}^\ma(\be,J,f_1,\ldots,f_k)\ne\emptyset,$
$\s_{\be,J,f_1,\ldots,f_k}$ is perturbation data for
$\bigl(\oM_{k+1}^\ma(\be,J,f_1,\ldots,f_k),\bev\bigr)$ in the sense
of\/ {\rm\S\ref{il27},} and all the simplices of\/ $VC\bigl(
\oM_{k+1}^\ma(\be,J,f_1,\ldots,f_k),\bev,\s_{\be,J,f_1,\ldots,f_k}
\bigr)$ lie in\/ $\X_{i_1+\cdots+i_k+\nm{\be}+k-1}$. At the
boundary\/ $\pd\oM_{k+1}^\ma(\be,J,f_1,\ldots,f_k),$ given by the
union of\/ \eq{il4eq18} over all\/ $I,\al,$ this
$\s_{\be,J,f_1,\ldots,f_k}$ must be compatible with:
\begin{itemize}
\item[{\rm(i)}] the choices of\/ $\s_{\be,J,f_1,\ldots,f_{i-1},f_i\ci
F_j^{a_i},f_{i+1},\ldots,f_k}$ for the term\/
$\oM_{k+1}^\ma(\al,\ab\be,\ab J,\ab f_1,\ldots,f_{i-1},f_i\ci
F_j^{a_i}, f_{i+1},\ldots,f_k)$ in\/~{\rm\eq{il4eq18};}
\item[{\rm(ii)}] the choices of\/ $\s_{\be_2,J,f_i,\ldots,
f_{i+k_2-1}}$ for the term $\oM_{k_2+1}^\ma (\al_2,\be_2,J,f_i,\ab
\ldots,\ab f_{i+k_2-1})$ in {\rm\eq{il4eq18};} and
\item[{\rm(iii)}] for each\/ $g:\De_a\ra L\amalg R$ appearing in $VC
\bigl(\oM_{k_2+1}^\ma(\be_2,J,f_i,\ab\ldots,\ab
f_{i\!+\!k_2\!-\!1}),\ab\bev,\s_{\be_2,J,f_i,\ldots,
f_{i+k_2-1}}\bigr),$ the choices of\/
$\s_{\be_1,J,f_1,\ldots,f_{i-1},g,f_{i+k_2},\ldots,f_k}$ for the
term\/ $\oM_{k_1\!+\!1}^\ma(\al_1,\be_1,J,f_1,\ldots,f_{i\!-\!1};
f_{i\!+\!k_2},\ldots,f_k)$ in\/ \eq{il4eq18} combined with\/ $VC
\bigl(\oM_{k_2+1}^\ma(\be_2,J,f_i,\ldots,f_{i+k_2-1}),\bev,
\s_{\be_2,J,f_i,\ldots,f_{i+k_2-1}}\bigr)$.
\end{itemize}

This boundary compatibility implies that, for $f_1:\De_{a_1}\ra
L\amalg R$ in $\X_{i_1},\ldots,f_k:\De_{a_k}\ra L\amalg R$ in
$\X_{i_k}$ as above, we have
\begin{gather}
\pd VC\bigl(\oM_{k+1}^\ma(\be,J,f_1,\ldots,f_k),\bev,\s_{\be,J,f_1,
\ldots,f_k}\bigr)=
\label{il6eq1} \\
\sum_{i=1}^k\sum_{j=0}^{a_i}\,\,
\begin{aligned}[t]
(-1)^{j+1+\sum_{l=1}^{i-1}\deg f_l}
VC\bigl(&\oM_{k+1}^\ma(\be,J,f_1,\ldots,f_{i-1}, f_i\ci
F_j^{a_i},f_{i+1},\\
&\ldots,f_k),\bev,\s_{\be,J,f_1,\ldots, f_{i-1},f_i\ci
F_j^{a_i},f_{i+1},\ldots,f_k}\bigr)
\end{aligned}
\nonumber\\
+\!\!\!\!\sum_{\begin{subarray}{l}
k_1+k_2=k+1,\\
1\le i\le k_1,\\
\be_1+\be_2=\be\end{subarray}}
\begin{aligned}[t]
&(-1)^{n+1+\sum_{l=1}^{i-1}\deg f_l}
VC\bigl(\oM_{k_1+1}^\ma(\be_1,J,f_1,\ldots,f_{i-1},\\
&VC(\oM_{k_2+1}^\ma(\be_2,J,f_i,\ldots, f_{i+k_2-1}),\bev,
\s_{\be_2,J,f_i,\ldots,f_{i+k_2-1}}), f_{i+k_2},\\
&{}\qquad\ldots,f_k),\bev,\s_{\substack{
\be_1,J,f_1,\ldots,f_{i-1},VC(\oM_{k_2+1}^\ma(\be_2,J,f_i,\ldots,
f_{i+k_2-1}),\\
\bev,\s_{\be_2,J,f_i,\ldots,f_{i+k_2-1}}),
f_{i+k_2},\ldots,f_k}}\,\bigr).
\end{aligned}
\nonumber
\end{gather}

Here if we have\/ $VC(\oM_{k_2+1}^\ma(\be_2,J,f_i,\ldots,
f_{i+k_2-1}),\bev,
\s_{\be_2,J,f_i,\ldots,f_{i+k_2-1}})\ab=\ab\sum_{a\in A}\si_a\,g_a$
for $\si_a\in\Q$ and\/ $g_a$ in
$\X_{i_i+\cdots+i_{i+k_2-1}+\nm{\be_2}+k_2-1},$ the final term
$VC(\oM_{k_1+1}^\ma(\ldots,VC(\ldots),\ldots),
\bev,\s_{\ldots,VC(\ldots),\ldots})$ in \eq{il6eq1} is short for
\begin{equation}
\sum_{a\in A}\begin{aligned}[t]\si_a\,
VC\bigl(&\oM_{k_1+1}^\ma(\be_1,J,f_1,\ldots,f_{i-1},g_a,f_{i+k_2},
\ldots,f_k),\\
&\bev,\s_{\be_1,J,f_1,\ldots,f_{i-1},g_a,f_{i+k_2},\ldots,f_k}\bigr).
\end{aligned}
\label{il6eq2}
\end{equation}
\end{itemize}
\label{il6thm}
\end{thm}

\begin{proof} Our proof is based on Fukaya et al.\
\cite[\S 30.5]{FOOO}. It involves a quadruple induction, an outer
induction over $g=0,\ldots,N$ in which we choose $\X_0,\ldots,\X_N$,
and an inner triple induction over $(j,k,l)$ during the construction
of~$\X_{g+1}$.

For the first step $g=0$ of the outer induction, let
$(\nm{\be},k)=(1,0)$. Since $\oM_1^\ma(\be,J)$ has no boundary,
(i)--(iii) are trivial. Choose arbitrary (but `small', in a sense
discussed below) perturbation data $\s_{\be,J}$ for
$\bigl(\oM_1^\ma(\be,J),\bev\bigr)$ for all $\be\in
H_2(M,\io(L);\Z)$ with $\nm{\be}=1$ and
$\oM_1^\ma(\be,J)\ne\emptyset$. There are only finitely many such
$\be$, and we can choose such $\s_{\be,J}$ as in \S\ref{il27}. The
virtual cycles $VC\bigl(\oM_1^\ma(\be,\ab
J),\ab\bev,\s_{\be,J}\bigr)$ for all such $\be$ involve only
finitely many simplices $f:\De_a\ra L\amalg R$. We must choose
$\X_0$ to contain all these simplices, and to satisfy (a),(b) in
$(N1)$ above. This is possible by Proposition~\ref{il2prop3}.

For the inductive step, suppose that we have chosen
$\X_0\subset\cdots\subset\X_g$ and $\{\s_{\be,J,f_1,\ldots,f_k}\}$,
which satisfy $(N1)$ and $(N2)$ with $N=g$. We shall construct
$\X_{g+1}$ and further choices of $\s_{\be,J,f_1,\ldots,f_k}$
satisfying $(N1)$ and $(N2)$ with $N=g+1$. These choices are not
independent of each other, but have to be made in a certain order.
Consider triples of integers $(j,k,l)$ such that $j\ge 0$, $k\ge 1$,
$(j,k)\ab\ne\ab(0,1)$, $j+k\ab\le\ab g+2$ and $0\le l\le nk$, where
$n=\dim L$. There are only finitely many such triples.

Define a total order $\le$ on such triples $(j,k,l)$ by
$(j_1,k_1,l_1)\le (j_2,k_2,l_2)$ if either
\begin{itemize}
\setlength{\itemsep}{0pt}
\setlength{\parsep}{0pt}
\item[$(*_1)$] $j_1+k_1<j_2+k_2$; or
\item[$(*_2)$] $j_1+k_1=j_2+k_2$ and $j_1<j_2$; or
\item[$(*_3)$] $j_1+k_1=j_2+k_2$ and $j_1=j_2$ and $l_1\le l_2$.
\end{itemize}
In a triple induction on $(j,k,l)$, at step $(j,k,l)$ we consider
all possible choices of $\be\in H_2(M,\io(L);\Z)$ with $\nm{\be}=j$
and $i_1,\ldots,i_k\ge 0$ with $i_1+\cdots+i_k+j+k-1=g+1$ and
$f_1\in\X_{i_1},f_2\in\X_{i_2},\ldots,f_k\in\X_{i_k}$ with
$f_i:\De_{a_i}\ra L\amalg R$, where $a_1+\cdots+a_k=l$, and such
that $\M^\ma_{k+1}(\be,J,f_1,\ldots,f_k)\ne\emptyset$. There are
only finitely many possibilities for such
$\be,i_1,\ldots,i_k,f_1,\ldots,f_k$. We will choose perturbation
data $\s_{\be,J,f_1,\ldots,f_k}$ on such choices in the order $\le$
on triples~$(j,k,l)$.

The important thing about this way of organizing our choices is that
for given $\be,i_1,\ldots,i_k,f_1,\ldots,f_k$ at step $(j,k,l)$, the
compatibilities (i)--(iii) on $\s_{\be,J,f_1,\ldots,f_k}$ in $(N2)$
depend only on $\s_{\be',J,f'_1,\ldots,f'_{k'}}$ which were either
chosen with $\X_{g'}$ for $g'\le g$, or were chosen during this step
$g+1$, but for some $(j',k',l')$ with $(j',k',l')<(j,k,l)$. So the
boundary conditions on $\s_{\be,J,f_1,\ldots,f_k}$ always depend on
choices we have already made, not on choices we have yet to make.

To see this, note that at step $(g,j,k,l)$, part (i) involves
choices made at step $(g,j,k,l-1)$, part (ii) choices at step
$(g',j',k',l')$ for $g'\le g$, $j'\le j$, $k'\le k$ and $l'$
arbitrary, but with either $j'<j$ (if $\be_1\ne 0)$ or $k'<k$ (if
$\be_1=0$), and part (iii) choices at step $(g',j',k',l')$ for
$g'\le g$, $j'\le j$, $k'\le k+1$, and $l'$ arbitrary, but with
either $j'<j$ (if $\be_2\ne 0)$ or $k'<k$ (if $\be_2=0$); this
allows $(j',k')=(j-1,k+1)$. Here we use the fact that
$\M^\ma_{k+1}(0,J,f_1,\ldots,f_k)=\emptyset$ unless $k\ge 2$. In
each case $(j',k',l')<(j,k,l)$ by $(*_1)$--$(*_3)$ above.

So, at step $(j,k,l)$ we must choose perturbation data
$\s_{\be,J,f_1, \ldots,f_k}$ for $\bigl(\oM_{k+1}^\ma(\be,\ab J,\ab
f_1,\ab\ldots,f_k),\bev \bigr)$ for the finitely many possibilities
for $\be,i_1,\ldots,i_k,f_1,\ldots,f_k$ above, satisfying the
compatibilities (i)--(iii) above with previous choices, which should
be `small' in the sense discussed below. Essentially, (i)--(iii)
prescribe $\s_{\be,J,f_1,\ldots,f_k}$ over
$\pd\oM_{k+1}^\ma(\be,J,f_1,\ldots,f_k)$, and we have to extend
these prescribed values over the interior
of~$\oM_{k+1}^\ma(\be,J,f_1,\ldots,f_k)$.

Note that because of the definition of boundaries of Kuranishi
spaces in \S\ref{il22}, regarded as subspaces of
$\oM_{k+1}^\ma(\be,J,f_1,\ldots,f_k)$, the disjoint components of
\eq{il4eq16} do actually intersect in
$\oM_{k+1}^\ma(\be,J,f_1,\ldots,f_k)$, in the codimension 2 corners
of $\oM_{k+1}^\ma(\be,J,f_1,\ldots,f_k)$ which lift to
$\pd^2\oM_{k+1}^\ma(\be,J,f_1,\ldots,f_k)$. But by induction we find
that (i)--(iii) prescribe consistent values for
$\s_{\be,J,f_1,\ldots,f_k}$ on these codimension 2 corners, since
the boundary values $\s_{\be,J,f_1,\ldots,f_{i-1},f_i\ci F_j^{a_i},
f_{i+1},\ldots,f_k}$, $\s_{\be_2,J,f_i,\ldots, f_{i+k_2-1}}$,
$\s_{\be_1,J,f_1,\ldots,f_{i-1},f,f_{i+k_2},\ldots,f_k}$ appearing
in (i)--(iii) themselves satisfy (i)--(iii).

Therefore, the discussion of \S\ref{il27} shows that we can choose
perturbation data $\s_{\be,J,f_1,\ldots,f_k}$ satisfying boundary
compatibilities (i)--(iii), but with one caveat. In Definition
\ref{il2dfn11} and Remark \ref{il2rem}(a) we said that a set of
perturbation data $\s_X$ for a Kuranishi space involves a finite
cover of $X$ by Kuranishi neighbourhoods $(V^i,E^i,\Ga^i,s^i,
\psi^i)$ and smooth, transverse multisections $\s^i$ on
$(V^i,\ldots,\psi^i)$ such that {\it each\/ $\s^i$ is sufficiently
close to $s^i$ in\/} $C^0$. Here the definition of `sufficiently
close' is rather vague; it has to do with ensuring that the
perturbed Kuranishi spaces remain compact.

Now it is conceivable that conditions (i)--(iii) on
$\s_{\be,J,f_1,\ldots,f_k}$ might be incompatible with this
requirement that the multisections $\s^i$ be `sufficiently close' to
$s^i$ in $C^0$. That is, in effect (i)--(iii) prescribe $\s^i$ over
$\pd V^i$, and if these prescribed values are not `sufficiently
close' to $s^i\vert_{\pd V^i}$ in $C^0$, then we cannot choose
$\s^i$ on $V^i$ `sufficiently close' to $s^i$ in $C^0$ with these
values on $s^i$. In this case, we could not choose
$\s_{\be,J,f_1,\ldots,f_k}$ satisfying all the necessary conditions.

A version of this problem is described in \cite[\S 30.3]{FOOO}. The
solution adopted by Fukaya et al.\ \cite[\S 30.2--\S 30.3]{FOOO},
which we follow, is that at every step we should choose the
perturbations $\s_{\be,J,f_1,\ldots,f_k}$ to be `small', by which we
mean that the $\s^i$ should be sufficiently close to $s^i$ in $C^0$
that not only does the construction of $VC\bigl(
\oM_{k+1}^\ma(\be,J,f_1,\ldots,f_k),\bev,\s_{\be,J,f_1,\ldots,f_k}
\bigr)$ work, as in Definition \ref{il2dfn11}, but also, for all the
conditions (i)--(iii) involving $\s_{\be,J,f_1,\ldots,f_k}$ at later
inductive steps in the proof, the prescribed values for $\s^i$ on
$\pd V^i$ should be sufficiently close to $s^i\vert_{V^i}$ that the
later constructions of $VC(\cdots)$ also work. We will discuss this
in Remark~\ref{il6rem}.

Thus following this method, at each step $(j,k,l)$ in our triple
induction, we can choose perturbation data
$\s_{\be,J,f_1,\ldots,f_k}$ satisfying (i)--(iii) for all the
finitely many choices of $\be,f_1,\ldots,f_k$ required. This
completes the inner induction on $(j,k,l)$. To finish the outer
inductive step on $g$, it remains to choose $\X_{g+1}$. The
conditions on $\X_{g+1}$ are that it should contain $\X_g$, and that
it should contain the simplices of $VC\bigl(\oM_{k+1}^\ma
(\be,J,f_1,\ldots,f_k), \ab \bev,\s_{\be,J,f_1,\ldots,f_k}\bigr)$
for all the finitely many $\be,f_1,\ldots,f_k$ we have just
considered over all $(j,k,l)$, which is a finite set of simplices,
and that it should satisfy (a),(b). This is possible by Proposition
\ref{il2prop3}. So we can choose $\X_{g+1}$ satisfying all the
conditions. This completes the inductive step for~$g=0,\ldots,N$.

We have now constructed $\X_0\subset\cdots\subset\X_N$ and
$\{\s_{\be,J,f_1,\ldots,f_k}\}$ satisfying $(N1)$ and $(N2)$. It
remains only to prove \eq{il6eq1}. Essentially, this is equation
\eq{il5eq25}, summed over all $\al$, perturbed using the
$\s_{\be,J,f_1,\ldots,f_k}$, and regarded as an equation in virtual
cycles in $C_*^\rsi(L\amalg R;\Q)$ rather than in oriented Kuranishi
spaces. However, since we have not chosen perturbation data for
$\oM_{k_1+1}^\ma(\al_1,\be_1,J,f_1,\ldots,
f_{i-1};f_{i+k_2},\ldots,f_k)$, we have to treat the final term of
\eq{il5eq25} differently. We perturb
$\oM_{k_2+1}^\ma(\al_2,\ab\be_2,J,\ab f_i,\ab\ldots,\ab
f_{i+k_2-1})$ in the fifth line of \eq{il5eq25}, summed over all
$\al_2$, using $\s_{\be_2,J,f_i,\ldots,f_{i+k_2-1}}$, so that it
becomes a virtual cycle
\begin{equation}
VC(\oM_{k_2+1}^\ma(\be_2,J,f_i,\ldots,
f_{i+k_2-1}),\bev,\s_{\be_2,J,f_i,\ldots,f_{i+k_2-1}})=\ts\sum_{a\in
A}\si_a\,g_a.
\label{il6eq3}
\end{equation}

Then in the fibre product in the fifth and sixth lines of
\eq{il5eq25}, we substitute each $g_a$ which is part of the
perturbed $\oM_{k_2+1}^\ma(\al_2,\ab\be_2,J,\ab f_i,\ab\ldots,\ab
f_{i+k_2-1})$ into $\oM_{k_1+1}^\ma(\al_1,\ab\be_1,\ab J,\ab
f_1,\ldots, f_{i-1};f_{i+k_2},\ldots,f_k)$. This gives
$(-1)^{\cdots}\oM_{k_1+1}^\ma(\al_1,\be_1,J,f_1,\ldots, f_{i-1},\ab
g_a,\ab f_{i+k_2},\ldots,f_k)$ by \eq{il5eq26}, and we perturb this
using $\s_{\be_1,J,f_1,\ldots,f_{i-1},g_a,f_{i+k_2},\ldots,f_k}$,
and take its virtual cycle. Considering (i)--(iii) above, we see
that modifying \eq{il5eq25} in this way to give an equation in
virtual cycles is valid, because it corresponds exactly to the
conditions (i)--(iii) on $\s_{\be,J,f_1, \ldots,f_k}$, which equate
to boundary conditions on $VC\bigl(\oM_{k+1}^\ma(\be,J,f_1,\ldots,
f_k),\bev,\s_{\be,J,f_1,\ldots,f_k}\bigr)$.

Thus \eq{il6eq1} follows from \eq{il5eq25} summed over all $\al$ and
\eq{il5eq26}, except for the signs in \eq{il6eq1}, which we have not
yet computed. The sign on the second line of \eq{il6eq1} is the sign
on the second line of \eq{il5eq25}. The sign on the second line of
\eq{il6eq1} is the combination of the sign on the fourth line on
\eq{il5eq25}, and the sign in \eq{il5eq26} when we substitute $g_a$
into $\oM_{k_1+1}^\ma(\al_1,\be_1,J,f_1,\ldots,f_{i-1};f_{i+k_2},
\ldots,f_k)$. To calculate this, we need to know $\deg g_a$ for the
$g_a$ in \eq{il6eq3}. We have $g_a:\De_b\ra L\amalg R$, where
$b=\vdim\oM_{k_2+1}^\ma(\be_2,J,f_i,\ldots,f_{i+k_2-1})$, which is
given by \eq{il4eq14}. Then $\deg g_a$ is given in terms of $b$ by
\eq{il4eq13}. Both equations are divided into cases $0\notin I_2$
and $0\in I_2$, and involve $\eta_{\al_2(0)}$ if $0\in I_2$. But
combining them, these contributions cancel out, so in every case we
have
\begin{equation}
\deg g_a=1-\mu_L(\be_2)+\ts\sum_{l=i}^{i+k_2-1}\deg f_l.
\label{il6eq4}
\end{equation}

Therefore the overall sign in the fourth line of \eq{il6eq1} should
be
\begin{equation}
\begin{split}
&(-1)^{n+\bigl(1+\sum_{l=1}^{i-1}\deg f_l\bigr)
\bigl(1+\sum_{l=i}^{i+k_2-1}\deg f_l\bigr)}\cdot\\
&(-1)^{\bigl(2-\mu_L(\be_2)+\sum_{l=i}^{i+k_2-1}\deg f_l\bigr)
\bigl(1+\sum_{l=1}^{i-1}\deg f_l\bigr)},
\end{split}
\label{il6eq5}
\end{equation}
where the first line comes from the fourth line of \eq{il5eq25}, and
the second line from \eq{il5eq26}, with $g_a$ in place of $f$ and
\eq{il6eq4} in place of $\deg f$. Noting that $2-\mu_L(\be_2)$ is
even, \eq{il6eq5} simplifies to $(-1)^{n+1+\sum_{l=1}^{i-1}\deg
f_l}$. This completes the proof.
\end{proof}

\begin{rem} In Theorem \ref{il6thm}, we had to fix a {\it
finite\/} $N\ge 0$, and then choose $\X_0,\ldots,\X_N$ and
$\{\s_{\be,J,f_1,\ldots,f_k}\}$. The conditions on
$\s_{\be,J,f_1,\ldots,f_k}$ with $f_1\in\X_{i_1},\ldots,$\ab $f_k
\in\X_{i_k}$ and $i_1+\cdots+i_k+\nm{\be}+k-1\le g$ for $g\le N$
really do depend not just on the $\X_1,\ldots,\X_g$, {\it but also
on the choice of\/} $N$, because we had to choose
$\s_{\be,J,f_1,\ldots,f_k}$ to be `small', that is, the
multisections $\s^i$ must be sufficiently close to $s^i$ in $C^0$,
and this notion of `sufficiently close' depends not just on
$\oM_{k+1}^\ma(\be,J,f_1,\ldots,f_k)$, but on all the other fibre
products involving $\s_{\be,J,f_1,\ldots,f_k}$ in their boundary
conditions in the later inductive steps~$g+1,g+2,\ldots,N$.

Because of this, we cannot prove the theorem for $N=\iy$, that is,
we cannot get an infinite sequence $\X_0\subset\X_1\subset\cdots$
and an infinite set of choices of perturbation data
$\{\s_{\be,J,f_1,\ldots,f_k}\}$ satisfying $(N1)$, $(N2)$. Taking
the limit $N\ra\iy$ does not work, since the $\X_g$ for $g<N$ and
$\s_{\be,J,f_1,\ldots,f_k}$ in the theorem depend on $N$.

One way to explain this is to note that by imposing a fixed upper
limit $N$ for $i_1+\cdots+i_k+\nm{\be}+k-1$, each choice of
perturbation data $\s_{\be,J,f_1,\ldots,f_k}$ has to satisfy only
finitely many smallness conditions, of the form `$\s^i$ is
sufficiently close to $s^i$'. But if we allowed $N=\iy$ then
$\s_{\be,J,f_1,\ldots,f_k}$ would have to satisfy {\it infinitely
many smallness conditions}. While we can always satisfy finitely
many smallness conditions, infinitely many might force $\s^i=s^0$,
which would contradict $\s^i$ being transverse. As an analogy,
consider finding $\ep\in\R$ satisfying $\ep\ne 0$ (a transversality
condition) and $\md{\ep}\le 1/n$ for $n=1,2,\ldots$ (infinitely many
smallness conditions).

This point is important, because the necessity to restrict to finite
$N$ is responsible for a large part of the length and technical
complexity of Fukaya et al.\ \cite{FOOO}. If Theorem \ref{il6thm}
held with $N=\iy$ then one could immediately construct an $A_\iy$
algebra, just using geometry. But instead, we have to consider
$A_{N,0}$ algebras, and obtain the $A_\iy$ algebra from them as a
kind of limit, using algebraic techniques. In \cite{AkJo} we will
reformulate Lagrangian Floer cohomology using the theory of {\it
Kuranishi cohomology\/} given in \cite{Joyc2,Joyc3}. There the
problem above disappears because we do not perturb our moduli
spaces, so we construct $A_\iy$ algebras geometrically, without
passing through $A_{N,K}$ algebras.
\label{il6rem}
\end{rem}

\section{$A_{N,0}$ algebras from immersed Lagrangian submanifolds}
\label{il7}

\begin{dfn} Let $\cal G$ be as in \S\ref{il6}, and $\nm{\,.\,}:{\cal
G}\ra\N$ be as in \eq{il3eq16}. For a given $N\in\N$, let
$\X_0\subset\cdots\subset\X_N$ and $\{\s_{\be,J,f_1,\ldots,f_k}\}$
be as in Theorem \ref{il6thm}. Suppose $k\ge 0$, $(\la,\mu)\in{\cal
G}$, and $i_1,\ldots,i_k=0,\ldots,N$ with
$i_1+\cdots+i_k+\nm{(\la,\mu)}+k-1\le N$. Define a $\Q$-multilinear
map $\m_{k,\geo}^{\la,\mu}:\Q\X_{i_1}\t\cdots\t\Q\X_{i_k}
\ra\Q\X_{i_1+\cdots+i_k+\nm{(\la,\mu)}+k-1}$ by
\begin{equation}
\begin{split}
&\m_{1,\geo}^{\smash{0,0}}(f_1)=(-1)^n\pd f_1,\\
&\m_{k,\geo}^{\la,\mu}(f_1,\ldots,f_k)=\!\!\!\!
\sum_{\begin{subarray}{l}
\be\in H_2(M,\io(L);\Z):\\
[\om]\cdot\be=\la,\;\> \mu_L(\be)=2\mu,\\
\oM_{k+1}^\ma(\be,J,f_1,\ldots,f_k)\ne\emptyset\end{subarray}}
\!\!\!\!\!\!\!\!\!\!\!\!\!\!
\begin{aligned}[t]VC\bigl(\oM_{k+1}^\ma(\be,J,f_1,\ldots,f_k)&,\\
\bev,\s_{\be,J,f_1,\ldots,f_k}\bigr)&,\\
(k,\la,\mu)\ne&(1,0,0).
\end{aligned}
\end{split}
\label{il7eq1}
\end{equation}
Combining \eq{il4eq13}, \eq{il4eq14} and $\mu_L(\be)=2\mu$ shows
that the shifted cohomological degree in \eq{il7eq1} is
\begin{equation}
\deg VC\bigl(\oM_{k+1}^\ma(\be,J,f_1,\ldots,f_k),\bev,
\s_{\be,J,f_1,\ldots,f_k}\bigr)=1-2\mu+\ts\sum_{i=1}^k\deg f_i.
\label{il7eq2}
\end{equation}
Thus $\m_{k,\geo}^{\la,\mu}:\Q\X_{i_1}\t\cdots\t\Q\X_{i_k}
\ra\Q\X_{i_1+\cdots+i_k+\nm{(\la,\mu)}+k-1}$ has degree~$1-2\mu$.
\label{il7dfn1}
\end{dfn}

Then \eq{il6eq1} immediately implies the following:

\begin{prop} For $k\in\N,$ $(\la,\mu)\in{\cal G}$ and\/
$f_1\in\X_{i_1},\ldots,f_k\in\X_{i_k}$ with
$i_1+\cdots+i_k+\nm{(\la,\mu)}+k-1\le N,$ we have
\begin{equation}
\begin{gathered}
\sum_{\begin{subarray}{l} k_1+k_2=k+1,\; 1\le i\le k_1,\; k_2\ge 0,\\
(\la_1,\mu_1),(\la_2,\mu_2)\in{\cal G},\;
(\la_1,\mu_1)+(\la_2,\mu_2)=(\la,\mu)\end{subarray}
\!\!\!\!\!\!\!\!\!\!\!\!\!\!\!\!\!\!\!\!\!\!\!\!\!\!\!\!
\!\!\!\!\!\!\!\!\!\!\!\!\!\!\!\!\!\!\!\!\!\!\!\!\!\!\!\!}
\!\!\!\!\!\!\!\!\!\!\!
\begin{aligned}[t]
(-1)^{\sum_{l=1}^{i-1}\deg f_l}
\m_{k_1,\geo}^{\la_1,\mu_1}\bigl(f_1,\ldots,f_{i-1},
\m_{k_2,\geo}^{\la_2,\mu_2}(f_i,\ldots,f_{i+k_2-1}), \\
f_{i+k_2}\ldots,f_k\bigr)=0.
\end{aligned}
\end{gathered}
\label{il7eq3}
\end{equation}
\label{il7prop}
\end{prop}

Equation \eq{il7eq3} is just \eq{il3eq10} for the
$\m_{k,\geo}^{\la,\mu}$. Thus, the data
$\Q\X_0\subset\cdots\subset\Q\X_N$ and $\m_{k,\geo}^{\la,\mu}$ are a
finite approximation to a gapped filtered $A_\iy$ algebras, as for
$A_{N,K}$ algebras in \S\ref{il37}. But it is not an $A_{N,K}$
algebra, as the conditions for when $\m_{k,\geo}^{\la,\mu}(f_1,
\ldots,f_k)$ is defined are different. We can apply purely algebraic
methods from Fukaya et al.\ \cite[\S 23.4, \S 30.7]{FOOO} to define
$A_{N,0}$ algebras. We use the method of sums over planar trees from
\S\ref{il33}, based on the construction of $\n$ in
Definition~\ref{il3dfn7}.

\begin{dfn} For a given $N\in\N$, we take $N'=N(N+2)$. Let
$\X_0\subset\cdots\subset\X_{N'}$ and $\{s_{\be,J,f_1,\ldots,f_k}\}$
be as in Theorem \ref{il6thm} with $N'$ in place of $N$. Since the
homologies of $(\Q\X_N,\pd),(\Q\X_{N'},\pd)$ are isomorphic, we can
find some linear subspace $A\subset\Q\X_{N'}$ such that
$\Q\X_{N'}=\Q\X_N\op A\op\pd A$ and $\pd:A\ra\pd A$ is an
isomorphism. Let $\Pi:\Q\X_{N'}\ra\Q\X_N$ for the projection, and
define linear $H:\Q\X_{N'}\ra \Q\X_{N'}$ by
\begin{equation}
H(x)=\begin{cases} 0, & \text{for $x\in\Q\X_N\op A$,}\\
\pd^{-1}x, & \text{for $x\in\pd A$.}\end{cases}
\label{il7eq4}
\end{equation}
Then $\id-\Pi=\pd H+H\pd$, as in \S\ref{il34} with $\m_1=\pd$.

Suppose $k\ge 0$ and $(\la,\mu)\in{\cal G}$ with
$\nm{(\la,\mu)}+k-1\le N$. Let $T$ be a rooted planar tree with $k$
leaves, and $(\bs\la,\bs\mu)$ be a family of $(\la_v,\mu_v)\in{\cal
G}$ for each internal vertex $v$ of $T$, such that
$\sum_v(\la_v,\mu_v)=(\la,\mu)$, and $(\la_v,\mu_v)=(0,0)$ implies
that $v$ has at least 2 incoming edges. We shall define a graded
multilinear operator $\m_{k,T}^{(\bs\la,\bs\mu)}:{\buildrel
{\ulcorner\,\,\,\text{$k$ copies } \,\,\,\urcorner} \over
{\vphantom{m}\smash{\Q\X_N\t\cdots\t \Q\X_N}}}\ra \Q\X_N$ of degree
$-2\mu+1$. Let $f_1,\ldots,f_k\in\Q\X_N$. Assign objects and
operators to the vertices and edges of $T$:
\begin{itemize}
\setlength{\itemsep}{0pt}
\setlength{\parsep}{0pt}
\item assign $f_1,\ldots,f_k$ to the leaf vertices $1,\ldots,k$
respectively.
\item for each internal vertex $v$ with 1 outgoing edge and $n$ incoming
edges, assign $\m_{n,\geo}^{\la_v,\mu_v}$. (Here by assumption
$(\la_v,\nu_v)=(0,0)$ implies $n\ge 2$, so we never assign the
special case $\m_{1,\geo}^{\smash{0,0}}$ in \eq{il7eq1}.)
\item assign $\id$ to each leaf edge.
\item assign $\Pi$ to the root edge.
\item assign $(-1)^{n+1}H$ to each internal edge.
\end{itemize}
Then define $\m_{k,T}^{(\bs\la,\bs\mu)}(f_1,\ldots,f_k)$ to be the
composition of all these objects and morphisms, as in \S\ref{il33}.
Define a $\Q$-multilinear map $\m_k^{\la,\mu}: {\buildrel
{\ulcorner\,\,\,\text{$k$ copies } \,\,\,\urcorner} \over
{\vphantom{m}\smash{\Q\X_N\t\cdots\t \Q\X_N}}}\ra\Q\X_N$ graded of
degree $1-2\mu$ by
$\m_1^{\smash{0,0}}=\m_{1,\geo}^{\smash{0,0}}=(-1)^n\pd$ and
$\m_k^{\la,\mu}=\sum_{T,(\bs\la,\bs\mu)}\m_{k,T}^{(\bs\la,\bs\mu)}$
for $(k,\la,\mu)\ne(1,0,0)$, where the sum is over all
$T,(\bs\la,\bs\mu)$ as above.
\label{il7dfn2}
\end{dfn}

We can now associate an $A_{N,0}$ algebra to $L$. It depends on the
choices of almost complex structure $J$, perturbation data
$\s_{\be,J,f_1,\ldots,f_k}$, and $N,N',\X_N,\X_{N'},A$ above.

\begin{thm}{\bf(a)} In Definition {\rm\ref{il7dfn2},} the
$\m_k^{\la,\mu}$ satisfy equation \eq{il3eq10} for all\/ $k\ge 0$
and\/ $(\la,\mu)\in{\cal G}$ with\/ $\nm{(\la,\mu)}+k-1\le N$. Thus
$(\Q\X_N,{\cal G},\m)$ is an $A_{N,0}$ algebra in the sense of
Definition {\rm\ref{il3dfn14},} where $\m=\bigl(\m_k^{\la,\mu}:k\ge
0,$ $(\la,\mu)\in{\cal G},$
$\nm{(\la,\mu)}\!+\!k\!-\!1\!\le\!N\bigr),$ and\/ $\Q\X_N$ is graded
by shifted cohomological degree in~\eq{il4eq13}.
\smallskip

\noindent{\bf(b)} If\/ $f_1\in\X_{i_1},\ldots,f_k\in\X_{i_k}$ with\/
$i_1+\cdots+i_k+\nm{(\la,\mu)}+k-1\le N$ then
\begin{equation}
\m_k^{\la,\mu}(f_1,\ldots,f_k)=\m_{k,\geo}^{\la,\mu}(f_1,\ldots,f_k).
\label{il7eq5}
\end{equation}
\label{il7thm}
\end{thm}

\begin{proof} The proof of part (a) follows the first parts of those
of Theorems \ref{il3thm2} and \ref{il3thm4}, as in \cite[\S
30.7]{FOOO}. For (b), suppose $f_1\in\X_{i_1},\ldots,f_k\in\X_{i_k}$
with $i_1+\cdots+i_k+ \nm{(\la,\mu)}+k-1\le N$ and
$T,(\bs\la,\bs\mu)$ are as in Definition \ref{il7dfn2}, where $T$
has at least one internal edge. Then
$\m_{k,T}^{(\bs\la,\bs\mu)}(f_1,\ldots,f_k)$ includes an expression
$-H\ci\m_{n,\geo}^{\la_v,\mu_v}(f_{a+1},\ldots,f_{a+n})$, where $-H$
comes from an internal edge of $T$, and
$\m_{n,\geo}^{\la_v,\mu_v}(f_{a+1},\ldots,f_{a+n})$ lies in
$\X_{i_{a+1}+\cdots+i_{a+n}+\nm{(\la_v,\mu_v)}+n-1}$, and so in
$\X_N$, as $n\le k$ and $\nm{(\la_v,\mu_v)}\le\nm{(\la,\mu)}$. But
$H=0$ on $\X_N$, so $-H\ci\m_{n,\geo}^{\la_v,\mu_v}
(f_{a+1},\ldots,f_{a+n})=0$,
and~$\m_{k,T}^{(\bs\la,\bs\mu)}(f_1,\ldots,f_k)=0$.

Therefore $\m_{k,T}^{(\bs\la,\bs\mu)}(f_1,\ldots,f_k)=0$ if $T$ has
an internal edge, so for $(k,\la,\mu)\ne(1,0,0)$ the only nonzero
contribution to $\m_k^{\la,\mu}(f_1,\ldots,f_k)$ comes from the
unique planar tree $T$ with one internal vertex and $k$ leaves,
which gives $\Pi\ci\m_{k,\geo}^{\la,\mu}(f_1,\ldots,f_k)$. But
$\m_{k,\geo}^{\la,\mu}(f_1,\ldots,f_k)\in\X_N$, so $\Pi$ acts as the
identity on it, proving \eq{il7eq5}. When $(k,\la,\mu)=(1,0,0)$,
equation \eq{il7eq5} holds by definition.
\end{proof}

\section{Choosing perturbation data for
$\oM_{k+1}^\ma(\be,J_t:t\in[0,1],f_1,\ldots,f_k)$}
\label{il8}

The $A_{N,0}$ algebras of \S\ref{il7} depended on a choice of almost
complex structure $J$. In \S\ref{il9} we will show that for two
choices $J_0,J_1$ for $J$, the resulting $A_{N,0}$ algebras are
homotopy equivalent. We do this by choosing a smooth 1-parameter
family $J_t:t\in[0,1]$ of almost complex structures interpolating
between $J_0$ and $J_1$, and using the moduli spaces
$\oM_{k+1}^\ma(\be,J_t:t\in[0,1],f_1,\ldots,f_k)$.

In this section we generalize Theorem \ref{il6thm} to choose
perturbation data for the $\oM_{k+1}^\ma(\be,J_t:t\in[0,1],
f_1,\ldots,f_k)$. Choose ${\cal G}\subset[0,\iy)\t\Z$ to satisfy the
conditions:
\begin{itemize}
\setlength{\itemsep}{0pt}
\setlength{\parsep}{0pt}
\item[(i)] $\cal G$ is closed under addition with ${\cal
G}\cap(\{0\}\t\Z)=\{(0,0)\}$, and ${\cal G}\cap([0,C]\t\Z)$ is
finite for any $C\ge 0$; and
\item[(ii)] If $\be\in H_2(M,\io(L);\Z)$, and
$\oM_1^\ma(\be,J_t:t\in[0,1])\ne\emptyset$ then
$\bigl([\om]\cdot\be, \ha\mu_L(\be)\bigr)\in{\cal G}$.
\end{itemize}
As for ${\cal G}_J$ in \S\ref{il6}, there exists a unique smallest
subset ${\cal G}_{J_t:t\in[0,1]}$ satisfying (i),(ii), but we do not
necessarily take ${\cal G}={\cal G}_{J_t:t\in[0,1]}$. Write
$\nm{\be}=\bnm{([\om]\cdot\be,\ha\mu_L(\be))}$, using \eq{il3eq16}
for this $\cal G$. With this notation we prove:

\begin{thm} Let\/ $(M,\om)$ be a compact\/ $2n$-dimensional
symplectic manifold and $\io:L\ra M$ be a compact Lagrangian
immersion with only transverse double self-intersections. Suppose
$J_t$ for $t\in[0,1]$ is a smooth family of almost complex
structures on $M$ compatible with\/ $\om$. Define compact Kuranishi
spaces $\oM_{k+1}^\ma(\al,\be,J_t:t\in[0,1],f_1,\ldots,f_k)$ as in
{\rm\S\ref{il45}} with\/ ${\cal T}=[0,1],$ and orient them as
in~{\rm\S\ref{il55}}.

Then for a given $N\in\N,$ there are
$\X_0^{\sst[0,1]}\subset\cdots\subset\X_N^{\sst[0,1]}$ and\/
$\{\s_{\be,J_t:t\in[0,1],f_1,\ldots,f_k}\}$ which satisfy the
following conditions:
\begin{itemize}
\item[$(N1)$] $\X^{\sst[0,1]}_0,\ldots,\X^{\sst[0,1]}_N$ are finite
sets of smooth simplicial chains\/ $f:\De_a\ra [0,1]\t(L\amalg R)$
such that
\begin{itemize}
\item[{\rm(a)}] There is a decomposition $\X^{\sst[0,1]}_i=
\X^{\sst 0}_i\amalg\X^{\sst(0,1)}_i\amalg\X^{\sst 1}_i$ for
$i=0,\ldots,N,$ where $f\in\X^{\sst 0}_i$ implies $f(\De_a)\subseteq
\{0\}\t(L\amalg R),$ and\/ $f\in\X^{\sst 1}_i$ implies
$f(\De_a)\subseteq\{1\}\t(L\amalg R),$ and\/ $f\in\X^{\sst(0,1)}_i$
implies $f(\De_a^\ci)\subseteq (0,1)\t(L\amalg R),$ where
$\De^\ci_a$ is the interior of\/ $\De_a,$ and $\pi_{[0,1]}\ci
f:\De_a\ra[0,1]$ is a submersion near $(\pi_{[0,1]}\ci
f)^{-1}(\{0,1\})$. {\rm(}This is equivalent to Condition
{\rm\ref{il4cond}.)} We shall sometimes regard\/
$\X^{\sst(0,1)}_{\smash{i}}$ as singular chains on $[0,1]\t(L\amalg
R)$ \begin{bfseries}relative to\end{bfseries} $\{0,1\}\t(L\amalg
R),$ that is, we project to~$C^\rsi_*\bigl([0,1]\t(L\amalg
R),\{0,1\}\t(L\amalg R);\Q\bigr)$.
\item[{\rm(b)}] if\/ $f:\De_a\ra [0,1]\t(L\amalg R)$ lies in\/
$\X^{\sst[0,1]}_i$ and\/ $a>0$ then\/ $f\ci F_b^a:\De_{a-1}\ra
[0,1]\t(L\amalg R)$ lies in\/ $\X^{\sst[0,1]}_i$ for all\/
$b=0,\ldots,a,$ using the notation of\/ {\rm\S\ref{il26}}. If\/
$g:\De_{a-1}\ra [0,1]\t(L\amalg R)$ lies in $\X^{\sst 0}_i$ or
$\X^{\sst 1}_i$ then $g=f\ci F_b^a$ for some $f:\De_a\ra
[0,1]\t(L\amalg R)$ in $\X^{\sst(0,1)}_i$ and\/ $b=0,\ldots,a$. If\/
$f:\De_a\ra [0,1]\t(L\amalg R)$ in $\X^{\sst(0,1)}_i$ then $f\ci
F_b^a$ lies in $\X^{\sst 0}_i$ or $\X^{\sst 1}_i$ for at most one
$b=0,\ldots,a$.
\item[{\rm(c)}] {\rm(a),(b)} imply that\/ $\Q\X^{\sst 0}_i,
\Q\X^{\sst 1}_i$ and\/ $\Q\X^{\sst(0,1)}_i$ are subcomplexes of the
(relative) singular chains\/
$C_*^\rsi\bigl(\{0\}\!\t\!(L\!\amalg\!R);\Q\bigr),$
$C_*^\rsi\bigl(\{1\}\!\t\!(L\!\amalg\!R);\Q\bigr)$ and\/
$C^\rsi_*\bigl([0,1]\!\t\!(L\!\amalg\!R),\{0,1\}\!\t\!(L\amalg
R);\Q\bigr)$ respectively. We require that the corresponding three
natural projections should be isomorphisms:
\begin{equation}
\begin{gathered}
H_*(\Q\X^{\sst 0}_i,\pd^{\sst 0})\,{\buildrel\cong\over\longra}\,
H_*^\rsi(L\amalg R;\Q),\quad H_*(\Q\X^{\sst
1}_i,\pd^{\sst 1})\,{\buildrel\cong\over\longra}\,H_*^\rsi(L\amalg R;\Q),\\
H_*\bigl(\Q\X^{\sst(0,1)}_i,\pd^{\sst(0,1)}\bigr)\,{\buildrel\cong\over
\longra}\, H_*^\rsi\bigl([0,1]\t(L\amalg R),\{0,1\}\t(L\amalg
R);\Q\bigr),
\end{gathered}
\label{il8eq1}
\end{equation}
identifying $\{0\}\t(L\amalg R)$ and\/ $\{1\}\t(L\amalg R)$
with\/~$L\amalg R$.
\end{itemize}
\item[$(N2)$] For all\/ $k\ge 0,$ $f_1\in\X^{\sst[0,1]}_{i_1},\ldots,f_k\in
\X^{\sst[0,1]}_{i_k}$ and\/ $\be\in H_2(M,\io(L);\Z)$ with\/
$i_1+\cdots+i_k+\nm{\be}+k-1\le N$ and
$\oM_{k+1}^\ma(\be,J_t:t\in[0,1],f_1,\ldots,f_k)\ne\emptyset,$
$\s_{\be,J_t:t\in[0,1],f_1,\ldots,f_k}$ is perturbation data for
$\bigl(\oM_{k+1}^\ma(\be,J_t:t\in[0,1],f_1,\ab\ldots,\ab
f_k),\ab\bs\pi_{[0,1]}\t\bev\bigr),$ and all the simplices of\/
$VC\bigl(\oM_{k+1}^\ma(\be,J_t:t\in[0,1],f_1,\ab\ldots,\ab
f_k),\ab\bs\pi_{[0,1]} \t\bev, \s_{\be,J_t:t\in[0,1],f_1,\ldots,f_k}
\bigr)$ lie in\/ $\X^{\sst[0,1]}_{i_1+\cdots+i_k+\nm{\be}+k-1}$. At
the boundary\/ $\pd\oM_{k+1}^\ma(\be,J_t:t\in[0,1],f_1,\ldots,f_k),$
given by the union of\/ \eq{il4eq29} over all\/ $I,\al,$ this
$\s_{\be,J_t:t\in[0,1],f_1,\ldots,f_k}$ must be compatible with:
\begin{itemize}
\item[{\rm(i)}] the choices of\/ $\s_{\be,J_t:t\in[0,1],f_1,\ldots,
f_{i-1},f_i\ci F_j^{a_i},f_{i+1},\ldots,f_k}$ for the term\/
$\oM_{k+1}^\ma\ab(\al,\ab\be,\ab J_t:t\in[0,1],\ab
f_1,\ldots,f_{i-1},f_i\ci F_j^{a_i}, f_{i+1},\ldots,f_k)$
in\/~{\rm\eq{il4eq29};}
\item[{\rm(ii)}] the choices of\/ $\s_{\be_2,J_t:t\in[0,1],f_i,\ldots,
f_{i+k_2-1}}$ for the term $\oM_{k_2+1}^\ma
(\al_2,\be_2,J_t:t\in[0,1],f_i,\ab \ldots,\ab f_{i+k_2-1})$ in
{\rm\eq{il4eq29};} and
\item[{\rm(iii)}] for each\/ $g:\De_a\ra [0,1]\t(L\amalg R)$
appearing in $VC\bigl(\oM_{k_2+1}^\ma(\be_2,J_t:t\in[0,1],f_i,
\ab\ldots,\ab f_{i\!+\!k_2\!-\!1}),\ab\bs\pi_{[0,1]}\t\bev,
\s_{\be_2,J_t:t\in[0,1], f_i,\ldots, f_{i+k_2-1}}\bigr),$ the
choices of\/ $\s_{\be_1,J_t:t\in[0,1],f_1,\ldots,f_{i-1},g,
f_{i+k_2},\ldots,f_k}$ for the term\/
$\oM_{k_1\!+\!1}^\ma(\al_1,\be_1,J_t:t\in[0,1],f_1,\ldots,f_{i\!-\!1};
f_{i\!+\!k_2},\ldots,f_k)$ in\/ \eq{il4eq29} combined with\/ $VC
\bigl(\oM_{k_2+1}^\ma(\be_2,\ab
J_t:t\in[0,1],f_i,\ldots,f_{i+k_2-1}),\bs\pi_{[0,1]}\t\bev,
\s_{\be_2,J_t:t\in[0,1],f_i,\ldots,f_{i+k_2-1}}\bigr)$.
\end{itemize}

This boundary compatibility implies that, for $f_1:\De_{a_1}\ra
[0,1]\t(L\amalg R)$ in $\X^{\sst[0,1]}_{i_1},\ldots,f_k:\De_{a_k}\ra
L\amalg R$ in $\X^{\sst[0,1]}_{i_k}$ as above, when $k>0$ we have
\begin{gather}
\pd VC\bigl(\oM_{k+1}^\ma(\be,J_t:t\in[0,1],f_1,\ldots,f_k),
\bs\pi_{[0,1]} \t\bev,\s_{\be,J_t:t\in[0,1],f_1, \ldots,f_k}\bigr)=
\label{il8eq2} \\
\sum_{i=1}^k\sum_{j=0}^{a_i}\!
\begin{aligned}[t]
(-1)^{j+1+\sum_{l=1}^{i-1}\deg f_l}
VC\bigl(\oM_{k+1}^\ma(\be,J_t:t\!\in\![0,1],f_1,\ldots,f_{i-1},
f_i\!\ci\!F_j^{a_i},f_{i+1}&,\\
\ldots,f_k),\bs\pi_{[0,1]}\!\t\!\bev,\s_{\be,J_t:t\in[0,1],f_1,\ldots,
f_{i-1},f_i\ci F_j^{a_i},f_{i+1},\ldots,f_k}\bigr)&
\end{aligned}
\nonumber\\
+\!\!\!\!\sum_{\begin{subarray}{l}
k_1+k_2=k+1,\\
1\le i\le k_1,\\
\be_1+\be_2=\be\end{subarray}}
\begin{aligned}[t]
&(-1)^{n+\sum_{l=1}^{i-1}\deg f_l}
VC\bigl(\oM_{k_1+1}^\ma(\be_1,J_t:t\in[0,1],f_1,\ldots,f_{i-1},
\\
&VC(\oM_{k_2+1}^\ma(\be_2,J_t:t\in[0,1],f_i,\ldots,
f_{i+k_2-1}),\bs\pi_{[0,1]}\t\bev,\\
&\s_{\be_2,J_t:t\in[0,1],f_i,\ldots,f_{i+k_2-1}}), f_{i+k_2},
\ldots,f_k),\bs\pi_{[0,1]}\t\bev,\\
&\s_{\substack{
\be_1,J_t:t\in[0,1],f_1,\ldots,f_{i-1},VC(\oM_{k_2+1}^\ma
(\be_2,J_t:t\in[0,1],f_i,\ldots,f_{i+k_2-1}),\\
\bs\pi_{[0,1]}\t\bev,\s_{\be_2,J_t:t\in[0,1],f_i,\ldots,f_{i+k_2-1}}),
f_{i+k_2},\ldots,f_k}}\,\bigr),
\end{aligned}
\nonumber
\end{gather}
using notation $VC(\oM_{k_1+1}^\ma(\ldots,VC(\ldots),\ldots),
\bev,\s_{\ldots,VC(\ldots),\ldots})$ as in \eq{il6eq2}.

When $k=0$ equation \eq{il8eq2} holds with the addition of an extra
term supported on $\{0,1\}\t(L\amalg R)$ corresponding to
$\oM^\ma_1(\be,J_t:t\in\{0,1\}),$ as in Theorem {\rm\ref{il5thm8}}.
Thus, if we project to relative chains
$C^\rsi_*\bigl([0,1]\!\t\!(L\!\amalg\!R),\{0,1\}\!\t\!(L\amalg
R);\Q\bigr),$ then \eq{il8eq2} holds for all\/~$k\ge 0$.
\item[$(N3)$] As well as the boundary compatibilities
$(N2)${\rm(i)--(iii),} we can impose compatibilities at the boundary
$\{0,1\}\t(L\amalg R)$ of\/ $[0,1]\t(L\amalg R),$ as follows.
Suppose $g_1\in\X^{\sst 0}_{i_1},\ldots,g_k\in\X^{\sst 0}_{i_k},$
where $g_j:\De_{a_j-1}\ra \{0\}\t(L\amalg R)$. We shall also abuse
notation by regarding $g_j$ as mapping $\De_{a_j-1}\ra L\amalg R$.
Then $(N1)${\rm(b)} implies that there exist\/
$f_1\in\X^{\sst(0,1)}_{i_1},\ldots,f_k\in\X^{\sst(0,1)}_{i_k}$ and\/
$b_1,\ldots,b_k$ such that\/ $b_j=0,\ldots,a_j$ and\/ $g_j=f_j\ci
F_{b_j}^{a_j}$ for $j=1,\ldots,k$. Then using the notation of Remark
\ref{il4rem} and inserting signs in {\rm\eq{il4eq33}--\eq{il4eq34},}
there are natural isomorphisms of oriented Kuranishi spaces:
\begin{align}
\begin{split}
\{0\}&\t_{i,[0,1],\bs\pi_{[0,1]}}\oM_{k+1}^\ma(\be,J_t:t\in[0,1],
f_1,\ldots,f_k)\\
&\cong (-1)^{k+\sum_{j=1}^kb_j}\oM_{k+1}^\ma(\be,J_0,g_1,
\ldots,g_k)\\
&\quad\t\bigl[\{0\}\t_{i,\R,\bs\pi_0}\bigl((\R,\ka_k^1)\t_{\bs\pi_1\t\cdots\t
\bs\pi_k,\R^k,i}[0,\iy)^k\bigr)\bigr],
\end{split}
\label{il8eq3}\\
\begin{split}
&\oM_{k+1}^\ma(\be,J_t:t\in[0,1],
f_1,\ldots,f_{j-1},g_j,f_{j+1},\ldots,f_k)\\
&\cong \pm\oM_{k+1}^\ma(\be,J_0,g_1, \ldots,g_k)\\
&\qquad\t\bigl[(\R,\ka_k^1)\t_{\bs\pi_1\t\cdots\t
\bs\pi_k,\R^k,i}[0,\iy)^{j-1}\t\{0\}\t[0,\iy)^{k-j}\bigr],
\label{il8eq4}
\end{split}
\end{align}
for all\/ $j=1,\ldots,k,$ where $i:\{0\}\ra[0,1]$ is the inclusion.
The analogues for $g_1\in\X^{\sst 1}_{i_1},\ldots,g_k\in \X^{\sst
1}_{i_k}$ and\/ $i:\{1\}\ra[0,1]$ are\goodbreak
\begin{align}
\begin{split}
\{1\}&\t_{i,[0,1],\bs\pi_{[0,1]}}\oM_{k+1}^\ma(\be,J_t:t\in[0,1],
f_1,\ldots,f_k)\\
&\cong (-1)^{\sum_{j=1}^kb_j}\oM_{k+1}^\ma(\be,J_1,g_1,
\ldots,g_k)\\
&\quad\t\bigl[\{0\}\t_{i,\R,\bs\pi_0}\bigl((\R,\ka_k^1)\t_{\bs\pi_1\t\cdots\t
\bs\pi_k,\R^k,i}[0,\iy)^k\bigr)\bigr],
\end{split}
\label{il8eq5}\\
\begin{split}
&\oM_{k+1}^\ma(\be,J_t:t\in[0,1],
f_1,\ldots,f_{j-1},g_j,f_{j+1},\ldots,f_k)\\
&\cong \pm\oM_{k+1}^\ma(\be,J_1,g_1, \ldots,g_k)\\
&\qquad\t\bigl[(\R,\ka_k^1)\t_{\bs\pi_1\t\cdots\t
\bs\pi_k,\R^k,i}[0,\iy)^{j-1}\t\{0\}\t[0,\iy)^{k-j}\bigr].
\label{il8eq6}
\end{split}
\end{align}

Suppose $\ti\X{}^{\sst 0}_0\subset\cdots\subset\ti\X{}^{\sst 0}_N$
and\/ $\{\s^{\sst 0}_{\be,J_0,f_1,\ldots,f_k}\}$ are possible
choices in Theorem \ref{il6thm} with\/ $J_0$ in place of\/ $J,$ and
that\/ $\ti\X{}^{\sst 1}_0\subset\cdots\subset\ti\X{}^{\sst 1}_N$
and\/ $\{\s^{\sst 1}_{\be,J_1,f_1,\ldots,f_k}\}$ are possible
choices in Theorem \ref{il6thm} with\/ $J_1$ in place of\/ $J$. Then
we can choose $\X_0^{\sst[0,1]}\subset\cdots\subset\X_N^{\sst[0,1]}$
and\/ $\{\s_{\be,J_t:t\in[0,1],f_1,\ldots,f_k}\}$ above such that
\begin{itemize}
\item[{\rm(a)}] $\X{}^{\sst 0}_i=\ti\X{}^{\sst 0}_i$ and\/ $\X{}^{\sst
1}_i=\ti\X{}^{\sst 1}_i$ for $i=1,\ldots,N$.
\item[{\rm(b)}] For all\/ $g_1\in\ti\X{}^{\sst 0}_{i_1},\ldots,
g_k\in\ti\X{}^{\sst 0}_{i_k},$ and all choices of\/
$f_1,\ldots,f_k,b_1,\ldots,b_k$ and $j$ above, the perturbation data
$\s_{\be,J_t:t\in[0,1],f_1,\ldots,f_{j-1},g_j,f_{j+1},\ldots,f_k}$
for $\bigl(\oM_{k+1}^\ma(\be,J_t:t\in[0,1],f_1,\ab\ldots,f_{j-1},
g_j,f_{j+1},\ldots,\ab f_k),\ab\bs\pi_{[0,1]}\t\bev\bigr)$ over
$\{0\}\t(L\amalg R)$ is identified with the perturbation data
$\s^{\sst 0}_{\be,J_0,g_1,\ldots,g_k}$ for
$\bigl(\oM_{k+1}^\ma(\be,\ab J_0,\ab g_1,\ldots,g_k),\bev\bigr)$
over $L\amalg R$ under the isomorphism {\it\eq{il8eq4}} and the
identification $\{0\}\t(L\amalg R)\cong L\amalg R,$ noting that\/
$\bigl[\{0\}\t_{i,\R,\bs\pi_0}\bigl((\R,\ka_k^1)\t_{\bs\pi_1\t\cdots\t
\bs\pi_k,\R^k,i}[0,\iy)^k\bigr)\bigr]$ is a single point whose
Kuranishi structure has transverse Kuranishi map, so it needs no
perturbation, and perturbation data $\s^{\sst
0}_{\be,J_0,g_1,\ldots,g_k}$ induces perturbation data
$\s_{\be,J_t:t\in[0,1],f_1,\ldots,f_{j-1},g_j,f_{j+1},\ldots, f_k}$
with the same virtual chain, up to sign.
\item[{\rm(c)}] The analogue of\/ {\rm(b)} holds for
$g_1\in\ti\X{}^{\sst 1}_{i_1},\ldots,g_k\in\ti\X{}^{\sst 1}_{i_k}$
and\/~$J_1$.
\end{itemize}
\end{itemize}
\label{il8thm}
\end{thm}

\begin{proof} Most of the proof is a straightforward generalization
of that of Theorem \ref{il6thm}, so we will just comment on the
differences. As in $(N3)$, we suppose some choices $\ti\X{}^{\sst
0}_0\subset\cdots\subset\ti\X{}^{\sst 0}_N,$ $\{\s^{\sst
0}_{\be,J_0,f_1,\ldots,f_k}\}$ and $\ti\X{}^{\sst
1}_0\subset\cdots\subset\ti\X{}^{\sst 1}_N,$ $\{\s^{\sst
1}_{\be,J_1,f_1,\ldots,f_k}\}$ are given for the outcomes of Theorem
\ref{il6thm} with $J_0,J_1$ in place of $J$. Then $(N3)$(a)
determines $\X^{\sst 0}_0,\ldots,\X^{\sst 0}_N$ and $\X^{\sst
1}_0,\ldots,\X^{\sst 1}_N$, and in the inductive proof we are only
free to choose $\X^{\sst(0,1)}_0,\ldots,\X^{\sst(0,1)}_N$. Also,
$(N3)$(b),(c) determine $\s_{\be,J_t:t\in[0,1],f_1,\ldots,f_k}$ if
any $f_j$ lies in $\X^{\sst 0}_{i_j}$ or $\X^{\sst 1}_{i_j}$, so in
the inductive proof we are only free to choose
$\s_{\be,J_t:t\in[0,1],f_1,\ldots,f_k}$ when $f_j\in\X^{\sst(0,1)
}_{i_j}$ for all~$j=1,\ldots,k$.

As in Theorem \ref{il6thm}, we perform a quadruple induction in
which we choose $\X^{\sst(0,1)}_0\subset\cdots\subset
\X^{\sst(0,1)}_N$ and $\s_{\be,J_t:t\in[0,1],f_1,\ldots,f_k}$ for
all $k\ge 0$, $f_1\in\X^{\sst(0,1)}_{i_1},\ldots,f_k\in
\X^{\sst(0,1)}_{i_k}$ and $\be\in H_2(M,\io(L);\Z)$ with
$i_1+\cdots+i_k+\nm{\be}+k-1\le N$ and
$\oM_{k+1}^\ma(\be,J_t:t\in[0,1],f_1,\ldots,f_k)\ne\emptyset$. At
the point when we choose $\s_{\be,J_t:t\in[0,1],f_1,\ldots,f_k}$, we
have already chosen perturbation data for every components in
$\pd\oM_{k+1}^\ma(\be,J_t:t\in[0,1],f_1,\ldots,f_k)$, which are
consistent on corners of codimension 2 and higher, and we must
extend these choices over the interior of $\oM_{k+1}^\ma(\be,
J_t:t\in[0,1],f_1,\ldots,f_k)$. In this proof, for the components of
$\pd\oM_{k+1}^\ma(\be, J_t:t\in[0,1],f_1,\ldots,f_k)$ lying over
$t=0$ or $t=1$ in $[0,1]$ the choice of perturbation data is given
by some $\{\s^{\sst 0}_{\be,J_0,g_1,\ldots,g_l}\}$ or $\{\s^{\sst
1}_{\be,J_1,g_1,\ldots,g_l}\}$, but in Theorem \ref{il6thm}, all the
boundary choices were made at previous steps in the quadruple
induction.

Since each $f_j$ maps $\De_{a_j}^\ci\ra (0,1)\t(L\amalg R)$, it is
immediate that the interior of each simplex in
$VC\bigl(\oM_{k+1}^\ma(\be,J_t:t\in[0,1],f_1,\ldots,f_k),\bs\pi_{[0,1]}
\t\bev,\s_{\be,J_t:t\in[0,1],f_1, \ldots,f_k}\bigr)$ maps to
$(0,1)\t(L\amalg R)$, and so satisfies the conditions in $(N1)$(a)
to lie in $\X^{\sst(0,1)}_i$. Thus, in the final step in the outer
induction when we have to choose $\X^{\sst(0,1)}_{g+1}$, there will
be a finite set $\cal W$ of smooth simplices $f:\De_a\ra
[0,1]\t(L\amalg R)$ with $f(\De_a^\ci)\subseteq (0,1)\t(L\amalg R)$
that are the new simplices introduced in virtual cycles for
$\oM_{k+1}^\ma(\be,J_t:t\in[0,1],f_1,\ldots,f_k)$ in this step, and
we must choose $\X^{\sst(0,1)}_{g+1}$ with ${\cal
W}\subseteq\X^{\sst(0,1)}_{g+1}$ and
$\X^{\sst(0,1)}_g\subseteq\X^{\sst(0,1)}_{g+1}$ to satisfy
$(N1)$(a)--(c). This is possible by a relative version of
Proposition \ref{il2prop3}, given the properties of $\X^{\sst
0}_0,\ldots,\X^{\sst 0}_N$ and $\X^{\sst 1}_0,\ldots,\X^{\sst 1}_N$
in Theorem \ref{il6thm}, and the fact that any face of $f:\De_a\ra
[0,1]\t(L\amalg R)$ either lies in $\X^{\sst 0}_{g+1}$ or $\X^{\sst
1}_{g+1}$, or its interior maps to~$(0,1)\t(L\amalg R)$.

In equation \eq{il8eq2}, the sign on the fourth line is
$(-1)^{n+\sum_{l=1}^{i-1}\deg f_l}$, rather than
$(-1)^{n+1+\sum_{l=1}^{i-1}\deg f_l}$ in the fourth line of
\eq{il6eq1}, because of the factor $(-1)^{\dim{\cal T}}=-1$ in the
fourth line of \eq{il5eq35}, which does not occur in the
corresponding equation \eq{il5eq25} used to deduce \eq{il6eq1}. The
extra term in \eq{il8eq2} when $k=0$ supported on $\{0,1\}\t(L\amalg
R)$ comes from the extra $\pd{\cal T}$ term in Theorem \ref{il5thm8}
when~$k=0$.

It remains only to justify the isomorphisms
\eq{il8eq3}--\eq{il8eq6}. These are given in unoriented Kuranishi
spaces in \eq{il4eq33}--\eq{il4eq34}, and we do not specify signs in
\eq{il8eq4} and \eq{il8eq6}, so we only have to compute the signs in
\eq{il8eq3} and \eq{il8eq5}. This is done by going through the proof
of \eq{il4eq33} inserting orientations. The signs
$(-1)^{k+\sum_{j=1}^kb_j}$ and $(-1)^{\sum_{j=1}^kb_j}$ come from
the isomorphisms of oriented manifolds
\begin{align}
\begin{split}
\{0\}^k&\t_{i,[0,1]^k,(\pi_{[0,1]}\ci f_1)\t\cdots\t(\pi_{[0,1]}\ci
f_k)}\bigl(\De_{a_1}\t\cdots\t\De_{a_k}\bigr)\\
&\cong (-1)^{k+\sum_{j=1}^kb_j} \De_{a_1-1}\t\cdots\t\De_{a_k-1},
\end{split}
\label{il8eq7}\\
\begin{split}
\{1\}^k&\t_{i,[0,1]^k,(\pi_{[0,1]}\ci f_1)\t\cdots\t(\pi_{[0,1]}\ci
f_k)}\bigl(\De_{a_1}\t\cdots\t\De_{a_k}\bigr)\\
&\cong(-1)^{\sum_{j=1}^kb_j}\De_{a_1-1}\t\cdots\t\De_{a_k-1},
\end{split}
\label{il8eq8}
\end{align}
for \eq{il8eq3} and \eq{il8eq5} respectively. Here the factors
$(-1)^{\sum_{j=1}^kb_j}$ arise since $F_{b_j}^{a_j}:\De_{a_j-1}\ra
\pd\De_{a_j}$ multiplies orientations by $(-1)^{b_j}$, and the extra
$(-1)^k$ in \eq{il8eq7} is the coefficient $-1$ of $\{0\}$ in
$\pd[0,1]=-\{0\}\amalg\{1\}$, raised to the power~$k$.
\end{proof}

In fact it is not difficult to extend Theorem \ref{il8thm} from a
family $J_t:t\in[0,1]$ to a general family $J_t:t\in{\cal T}$ for
${\cal T}$ a compact manifold with boundary and corners, and we will
use this extension in \S\ref{il10} when $\cal T$ is a closed
semicircle or triangle in $\R^2$. But the statement of this
generalization is even more complex, with special treatment for the
codimension $k$ corners of $\cal T$ for $k=0,1,\ldots,\dim{\cal T}$,
and the analogue of $(N3)$ referring recursively to the outcome of
Theorem \ref{il8thm} with $\pd{\cal T}$ in place of $\cal T$, rather
than just to the outcome of Theorem \ref{il6thm}. For simplicity, it
seemed better just to state the result for~${\cal T}=[0,1]$.

\section{$A_{N,0}$ morphisms from $J_0$ to $J_1$ $A_{N,0}$ algebras}
\label{il9}

We work in the situation of \S\ref{il8}, with $J_t$ for $t\in[0,1]$
a smooth family of almost complex structures on $M$ compatible with
$\om$. We begin by constructing an $A_{N,0}$ algebra of relative
chains $C^\rsi_*\bigl([0,1]\t(L\amalg R),\{0,1\}\t(L\amalg
R);\Q\bigr)$ depending on the whole family $J_t:t\in[0,1]$. Here are
the analogues of Definitions \ref{il7dfn1} and \ref{il7dfn2},
Proposition \ref{il7prop} and Theorem~\ref{il7thm}.

\begin{dfn} Let $\cal G$ be as in \S\ref{il8}, and $\nm{\,.\,}:{\cal
G}\ra\N$ be as in \eq{il3eq16}. For a given $N\in\N$, let
$\X^{\sst[0,1]}_i=\X^{\sst 0}_i\amalg\X^{\sst(0,1)}_i\amalg\X^{\sst
1}_i$ for $i=0,\ldots,N,$ and $\{\s_{\be,J_t:t\in[0,1],f_1,\ldots,
f_k}\}$ be as in Theorem \ref{il8thm}. Write $\Pi^{\sst(0,1)}:
\Q\X^{\sst[0,1]}_i\ra\Q\X^{\sst(0,1)}_i$ for the projection, with
kernel $\Q\X^{\sst 0}_i\op\Q\X^{\sst 1}_i$. Suppose $k\ge 0$,
$(\la,\mu)\in{\cal G}$, and $i_1,\ldots,i_k=0,\ldots,N$ with
$i_1+\cdots+i_k+\nm{(\la,\mu)}+k-1\le N$. Generalizing \eq{il7eq1},
define a $\Q$-multilinear map $\m_{k,\geo}^{{\sst(0,1)}\la,\mu}:
\Q\X_{i_1}^{\sst(0,1)}\t
\cdots\t\Q\X_{i_k}^{\sst(0,1)}\ra\Q\X_{i_1+\cdots+i_k+\nm{(\la,
\mu)}+k-1}^{\sst(0,1)}$ of degree $1-2\mu$ by
\begin{equation}
\begin{split}
&\m_{1,\geo}^{{\sst(0,1)}0,0}(f_1)=\Pi^{\sst(0,1)}\bigl[(-1)^{n+1}\pd
f_1\bigr]=(-1)^{n+1}\pd^{\sst(0,1)}f_1,\\
&\m_{k,\geo}^{{\sst(0,1)}\la,\mu}(f_1,\ldots,f_k)=
\!\!\!\!\!\!\!\!\! \sum_{\begin{subarray}{l}
\be\in H_2(M,\io(L);\Z):\\
[\om]\cdot\be=\la,\;\> \mu_L(\be)=2\mu,\\
\oM_{k+1}^\ma(\be,J_t:t\in[0,1],f_1,\ldots,f_k)\ne\emptyset\end{subarray}
\!\!\!\!\!\!\!\!\!\!\!\!\!\!\!\!\!\!\!\!\!\!\!\!\!\!\!\!\!\!\!\!\!\!\!\!\!}
\!\!\!\!\!\!\!\!\!\!\!\!
\begin{aligned}[t]\Pi^{\sst(0,1)}\bigl[VC\bigl(\oM_{k+1}^\ma
(\be,J_t:t\!\in\![0,1],f_1,\ldots,f_k)&,\\
\bev,\s_{\be,J_t:t\in[0,1],f_1,\ldots,f_k}\bigr)\bigr]&,\\
(k,\la,\mu)\ne(1,0,0)&.
\end{aligned}
\end{split}
\label{il9eq1}
\end{equation}
\label{il9dfn1}
\end{dfn}

Now applying $\Pi^{\sst(0,1)}:\Q\X^{\sst[0,1]}_i
\ra\Q\X^{\sst(0,1)}_i$ is equivalent to projecting to {\it relative}
singular chains to $C^\rsi_*\bigl([0,1]\t(L\amalg R),\{0,1\}\t
(L\amalg R);\Q\bigr)$, so we can regard $\Q\X^{\sst(0,1)}_i$ as a
space of relative chains. As in Theorem \ref{il8thm}$(N2)$, equation
\eq{il8eq2} holds in relative chains for all $k\ge 0$. Note the two
sign differences compared to \S\ref{il6}--\S\ref{il7}: the signs in
the fourth lines of \eq{il6eq1} and \eq{il8eq2} differ by $-1$, and
the signs on the first lines of \eq{il7eq1} and \eq{il9eq1} differ
by $-1$. Both of these are really $(-1)^{\dim{\cal T}}$, where
${\cal T}=[0,1]$. In proving \eq{il9eq3} below, these two sign
differences cancel out, so that the signs in \eq{il7eq3} and
\eq{il9eq3} are the same. Thus as for Proposition \ref{il7prop} we
deduce:

\begin{prop} For $k\in\N,$ $(\la,\mu)\in{\cal G}$ and\/
$f_1\in\X^{\sst(0,1)}_{i_1},\ldots,f_k\in\X^{\sst(0,1)}_{i_k}$ with
$i_1+\cdots+i_k+\nm{(\la,\mu)}+k-1\le N,$ we have
\begin{equation}
\begin{gathered}
\sum_{\begin{subarray}{l} k_1+k_2=k+1,\; 1\le i\le k_1,\; k_2\ge 0,\\
(\la_1,\mu_1),(\la_2,\mu_2)\in{\cal G},\\
(\la_1,\mu_1)+(\la_2,\mu_2)=(\la,\mu)\end{subarray}
\!\!\!\!\!\!\!\!\!\!\!\!\!\!\!\!\!\!\!\!\!\!\!\!\!\!\!\!\!\!\!\!\!\!\!}
\!\!\!\!\!\!\!\!\!\!\!
\begin{aligned}[t]
(-1)^{\sum_{l=1}^{i-1}\deg f_l}
\m_{k_1,\geo}^{{\sst(0,1)}\la_1,\mu_1}\bigl(f_1,\ldots,f_{i-1},
\m_{k_2,\geo}^{{\sst(0,1)}\la_2,\mu_2}(f_i,\ldots,f_{i+k_2-1}), \\
f_{i+k_2}\ldots,f_k\bigr)=0.
\end{aligned}
\end{gathered}
\label{il9eq3}
\end{equation}
\label{il9prop}
\end{prop}

\begin{dfn} For a given $N\in\N$, we take $N'=N(N+2)$. Let
$\X^{\sst[0,1]}_i=\X^{\sst 0}_i\amalg\X^{\sst(0,1)}_i\amalg\X^{\sst
1}_i$ for $i=0,\ldots,N'$ and $\{s_{\be,J_t:t\in[0,1],f_1,
\ldots,f_k}\}$ be as in Theorem \ref{il8thm} with $N'$ in place of
$N$. Since by \eq{il8eq1} the homologies of $(\Q\X^{\sst(0,1)}_N,
\pd^{\sst(0,1)}),(\Q\X^{\sst(0,1)}_{N'},\pd^{\sst(0,1)})$ are
isomorphic, we can find some linear subspace
$A^{\sst(0,1)}\subset\Q\X^{\sst(0,1)}_{N'}$ such that
\begin{itemize}
\item $\Q\X^{\sst(0,1)}_{N'}=\Q\X^{\sst(0,1)}_N\op
A^{\sst(0,1)}\op\pd^{\sst(0,1)} A^{\sst(0,1)}$; and
\item $\pd^{\sst(0,1)}:A^{\sst(0,1)}\ra\pd^{\sst(0,1)} A^{\sst(0,1)}$ is an isomorphism.
\end{itemize}
Later we will take $A^{\sst(0,1)}$ compatible with choices of $A$ in
Definition \ref{il7dfn2} for $J_0,J_1$. Define a linear map
$H^{\sst(0,1)}:\Q\X^{\sst(0,1)}_{N'}\ra\Q\X^{\sst(0,1)}_{N'}$ by
\begin{equation}
H^{\sst(0,1)}(x)=\begin{cases}
0, & \text{for $x\in \Q\X_N\op A^{\sst(0,1)}$,}\\
(\pd^{\sst(0,1)})^{-1}x, & \text{for $x\in\pd^{\sst(0,1)}
A^{\sst(0,1)}$.}\end{cases}
\label{il9eq4}
\end{equation}
Write $\Pi:\Q\X^{\sst(0,1)}_{N'}\ra\Q\X^{\sst(0,1)}_N$ for the
projection. Then $\id-\Pi=\pd^{\sst(0,1)}
H^{\sst(0,1)}+H^{\sst(0,1)}\pd^{\sst(0,1)}$.

Suppose $k\ge 0$ and $(\la,\mu)\in{\cal G}$ with
$\nm{(\la,\mu)}+k-1\le N$. Let $T$ be a rooted planar tree with $k$
leaves, and $(\bs\la,\bs\mu)$ be a family of $(\la_v,\mu_v)\in{\cal
G}$ for each internal vertex $v$ of $T$, such that
$\sum_v(\la_v,\mu_v)=(\la,\mu)$, and $(\la_v,\mu_v)=(0,0)$ implies
that $v$ has at least 2 incoming edges. We shall define a graded
multilinear operator $\m_{k,T}^{{\sst(0,1)}(\bs\la,\bs\mu)}:
{\buildrel {\ulcorner\,\,\,\text{$k$ copies } \,\,\,\urcorner} \over
{\vphantom{m}\smash{\Q\X^{\sst(0,1)}_N\t\cdots\t
\Q\X^{\sst(0,1)}_N}}}\ra\Q\X^{\sst(0,1)}_N$ of degree $1-2\mu$. Let
$f_1,\ldots,f_k\in\Q\X_N^{\sst(0,1)}$. Assign objects and operators
to the vertices and edges of $T$:
\begin{itemize}
\setlength{\itemsep}{0pt}
\setlength{\parsep}{0pt}
\item assign $f_1,\ldots,f_k$ to the leaf vertices $1,\ldots,k$
respectively.
\item for each internal vertex $v$ with 1 outgoing edge and $n$ incoming
edges, assign $\m_{n,\geo}^{{\sst(0,1)}\la_v,\mu_v}$.
\item assign $\id$ to each leaf edge.
\item assign $\Pi$ to the root edge.
\item assign $(-1)^nH^{\sst(0,1)}$ to each internal edge.
\end{itemize}
Then define $\m_{\smash{k,T}}^{{\sst(0,1)}(\bs\la,\bs\mu)}(f_1,
\ldots,f_k)$ to be the composition of all these. Define a
$\Q$-multilinear map $\m_k^{{\sst(0,1)}\la,\mu}:{\buildrel
{\ulcorner\,\,\,\text{$k$ copies } \,\,\,\urcorner} \over
{\vphantom{m}\smash{\Q\X^{\sst
(0,1)}_N\t\cdots\t\Q\X^{\sst(0,1)}_N}}}\ra\Q\X^{\sst(0,1)}_N$ graded
of degree $1-2\mu$ by $\m_1^{{\sst(0,1)}0,0}=
\m_{1,\geo}^{{\sst(0,1)}0,0}=(-1)^{n+1}\pd^{\sst(0,1)}$ and
$\m_k^{{\sst(0,1)}\la,\mu}=\sum_{T,(\bs\la,\bs\mu)}\m_{k,T}^{{\sst(0,1)}
(\bs\la,\bs\mu)}$ for~$(k,\la,\mu)\ne(1,0,0)$.
\label{il9dfn2}
\end{dfn}

\begin{thm}{\bf(a)} In Definition {\rm\ref{il9dfn2},} the
$\m_k^{{\sst(0,1)}\la,\mu}$ satisfy equation \eq{il3eq10} for all\/
$k\ge 0$ and\/ $(\la,\mu)\in{\cal G}$ with\/ $\nm{(\la,\mu)}+k-1\le
N$. Thus $(\Q\X^{\sst(0,1)}_N,{\cal G},\m^{\sst(0,1)})$ is an
$A_{N,0}$ algebra in the sense of Definition {\rm\ref{il3dfn14},}
where $\m^{\sst(0,1)}=\bigl(\m_k^{{\sst(0,1)}\la,\mu}:k\ge 0,$
$(\la,\mu)\in{\cal G},$ $\nm{(\la,\mu)}\!+\!k\!-\!1\!\le\!N\bigr),$
and\/ $\Q\X_N^{\sst(0,1)}$ is graded by shifted cohomological degree
in~\eq{il4eq26}.\!\!\!
\smallskip

\noindent{\bf(b)} If\/
$f_1\in\X_{i_1}^{\sst(0,1)},\ldots,f_k\in\X_{i_k}^{\sst(0,1)}$
with\/ $i_1+\cdots+i_k+\nm{(\la,\mu)}+k-1\le N$ then
\begin{equation}
\m_k^{{\sst(0,1)}\la,\mu}(f_1,\ldots,f_k)=
\m_{k,\geo}^{{\sst(0,1)}\la,\mu}(f_1,\ldots,f_k).
\label{il9eq5}
\end{equation}
\label{il9thm1}
\end{thm}

Now let $(\Q\X^{\sst 0}_N,{\cal G},\m^{\sst 0})$ and $(\Q\X^{\sst
1}_N,{\cal G},\m^{\sst 1})$ be $A_{N,0}$ algebras constructed in
Theorem \ref{il7thm} with $J=J_0$ and $J=J_1$. We shall construct
strict, surjective $A_{N,0}$ morphisms $\p^{\sst 0}:
(\Q\X^{\sst(0,1)}_N,{\cal G},\m^{\sst(0,1)})\ra (\Q\X^{\sst
0}_N,{\cal G},\m^{\sst 0})$ and $\p^{\sst 1}:
(\Q\X^{\sst(0,1)}_N,{\cal G},\m^{\sst(0,1)})\ra (\Q\X^{\sst
1}_N,{\cal G},\m^{\sst 1})$, and show they are homotopy
equivalences.

\begin{dfn} Let $J_0,J_1$ be complex structures on $M$ compatible
with $\om$, and $J_t:t\in[0,1]$ a smooth family of complex
structures on $M$ compatible with $\om$ interpolating between them.
Fix once and for all $N\in\N$, $N'=N(N+2)$ and ${\cal
G}\subset[0,\iy)\t\Z$ satisfying conditions (i),(ii) of \S\ref{il8}.
This implies that $\cal G$ satisfies conditions (i),(ii) of
\S\ref{il6} for $J=J_0$ and~$J=J_1$.

With these $N,N',{\cal G}$, suppose $\X{}^{\sst
0}_0\subset\cdots\subset\X{}^{\sst 0}_{N'}$, $\{\s^{\sst
0}_{\be,J_0,f_1,\ldots,f_k}\}$ are possible choices in Theorem
\ref{il6thm} with $J_0,N'$ in place of $J,N$, and $\X{}^{\sst
1}_0\subset\cdots\subset\X{}^{\sst 1}_{N'}$, $\{\s^{\sst
1}_{\be,J_1,f_1,\ldots,f_k}\}$ possible choices in Theorem
\ref{il6thm} with $J_1,N'$. Let $(\Q\X^{\sst 0}_N,{\cal G},\m^{\sst
0})$ and $(\Q\X^{\sst 1}_N,{\cal G},\m^{\sst 1})$ be possible
$A_{N,0}$ algebras constructed in Theorem \ref{il7thm} from this
data for each of $J_0,J_1$. As in Definition \ref{il7dfn2}, this
involves additional choices of subspace $A$ and corresponding
operator $H$, which we write as $A^{\sst 0},H^{\sst 0}$ and $A^{\sst
1},H^{\sst 1}$ respectively.

Suppose $\X^{\sst[0,1]}_i=\X^{\sst 0}_i\amalg\X^{\sst(0,1)}_i\amalg
\X^{\sst 1}_i$ for $i=0,\ldots,N'$ and $\{s_{\be,J_t:t\in[0,1],f_1,
\ldots,f_k}\}$ are possible choices in Theorem \ref{il8thm} with
$N'$ in place of $N$, and compatible in $(N3)$ with the above
choices of $\X{}^{\sst 0}_0\subset\cdots\subset\X{}^{\sst 0}_{N'}$,
$\{\s^{\sst 0}_{\be,J_0,f_1,\ldots,f_k}\}$ and $\X{}^{\sst
1}_0\subset\cdots\subset\X{}^{\sst 1}_{N'}$, $\{\s^{\sst
1}_{\be,J_1,f_1,\ldots,f_k}\}$, dropping the distinction between
$\ti\X{}^{\sst 0}_i,\ti\X{}^{\sst 1}_i$ and $\X{}^{\sst
0}_i,\X{}^{\sst 1}_i$. Let $(\Q\X^{\sst(0,1)}_N,{\cal
G},\m^{\sst(0,1)})$ be a possible $A_{N,0}$ algebra constructed in
Theorem \ref{il9thm1} from this data. This involves an additional
choice of $A^{\sst(0,1)}$, yielding $H^{\sst(0,1)}$. We will shortly
require $A^{\sst(0,1)},H^{\sst(0,1)}$ to be compatible with $A^{\sst
0},H^{\sst 0}$ and $A^{\sst 1},H^{\sst 1}$.

Write $\pd^{\sst 0},\pd^{\sst 1},\pd^{\sst[0,1]},\pd^{\sst(0,1)}$
for the boundary operators on $\Q\X^{\sst 0}_i,\Q\X^{\sst
1}_i,\Q\X^{\sst[0,1]}_i,\Q\X^{\sst(0,1)}_i$ respectively, where we
regard $\pd^{\sst(0,1)}:\Q\X^{\sst(0,1)}_i\ra\Q\X^{\sst(0,1)}_i$ as
acting on $\Q\X^{\sst(0,1)}_i$ as a subspace of the relative chains
$C^\rsi_*\bigl([0,1]\t(L\amalg R),\{0,1\}\t(L\amalg R);\Q\bigr)$.
But we will also regard $\Q\X^{\sst(0,1)}_i$ as a subspace of
$\Q\X^{\sst[0,1]}_i$, so that $\pd^{\sst[0,1]}$
maps~$\Q\X^{\sst(0,1)}_i\ra\Q\X^{\sst[0,1]}_i$.

Define linear maps $P^{\sst 0}:\Q\X^{\sst(0,1)}_i\ra\Q\X^{\sst 0}_i$
and $P^{\sst 1}:\Q\X^{\sst(0,1)}_i\ra\Q\X^{\sst 1}_i$ for
$i=0,\ldots,N'$ by $P^{\sst 0}=-\Pi^{\sst 0}\ci\pd^{\sst[0,1]}$ and
$P^{\sst 1}=\Pi^{\sst 1}\ci\pd^{\sst[0,1]}$, where $\Pi^{\sst
0},\Pi^{\sst 1}:\Q\X^{\sst[0,1]}_i\ra\Q\X^{\sst 0}_i,\Q\X^{\sst
1}_i$ are the projections coming from the decomposition
$\X^{\sst[0,1]}_i=\X^{\sst 0}_i\amalg\X^{\sst(0,1)}_i\amalg\X^{\sst
1}_i$. Observe that although $\pd^{\sst[0,1]}$ reduces dimension of
singular chains by one, $\Q\X^{\sst 0}_i,\Q\X^{\sst 1}_i$ are graded
by $\deg f$ in \eq{il4eq13}, but $\Q\X^{\sst(0,1)}_i$ is graded by
$\deg f$ in \eq{il4eq26} with $\dim{\cal T}=1$. Therefore $P^{\sst
0},P^{\sst 1}$ are actually {\it graded of degree zero}.

Considering the components of $\pd^{\sst[0,1]}:\Q\X^{\sst[0,1]}_i
\ra\Q\X^{\sst[0,1]}_i$ in the splitting
$\Q\X^{\sst[0,1]}_i=\Q\X^{\sst 0}_i\op\Q\X^{\sst(0,1)}_i
\op\Q\X^{\sst 1}_i$, we see that $\pd^{\sst[0,1]}=\pd^{\sst
0}+\pd^{\sst(0,1)}+\pd^{\sst 1}-P^{\sst 0}+P^{\sst 1}$. Since
$(\pd^{\sst[0,1]})^2=0$, taking components of $(\pd^{\sst[0,1]})^2$
mapping from $\Q\X^{\sst(0,1)}_i$ to $\Q\X^{\sst 0}_i,\Q\X^{\sst
1}_i$ shows that
\begin{equation}
P^{\sst 0}\ci\pd^{\sst(0,1)}+\pd^{\sst 0}\ci P^{\sst 0}=0
\quad\text{and}\quad P^{\sst 1}\ci\pd^{\sst(0,1)}+\pd^{\sst 1}\ci
P^{\sst 1}=0.
\label{il9eq6}
\end{equation}
Thus $P^{\sst 0},P^{\sst 1}$ are morphisms of complexes
$(\Q\X^{\sst(0,1)}_i,\pd^{\sst(0,1)}),(\Q\X^{\sst 0}_i,\pd^{\sst
0}),(\Q\X^{\sst 1}_i,\pd^{\sst 1})$, and induce maps $P^{\sst
0}_*,P^{\sst 1}_*$ on cohomology. But by assumption \eq{il8eq1} are
isomorphisms. Under these, $P^{\sst 0}_*$ corresponds (though not
with gradings) to the natural map
\begin{equation*}
H_*^\rsi\bigl([0,1]\t(L\amalg R),\{0,1\}\t(L\amalg
R);\Q\bigr)\longra H_{*-1}^\rsi\bigl(\{0\}\t(L\amalg R);\Q\bigr).
\end{equation*}
Since this is an isomorphism, $P^{\sst 0}_*$ and similarly $P^{\sst
1}_*$ are isomorphisms.

Theorem \ref{il8thm}$(N1)$(b) implies that if $g\in\X_i^{\sst 0}$
then there exists $f\in\X_i^{\sst(0,1)}$ with $\Pi^{\sst 0}(f)=\pm
g$, and also $\Pi^{\sst 1}(f)=0$. Similarly, if $g\in\X_i^{\sst 1}$
there exists $f\in\X_i^{\sst(0,1)}$ with $\Pi^{\sst 1}(f)=\pm g$,
and also $\Pi^{\sst 0}(f)=0$. Therefore $\Pi^{\sst 0}\op\Pi^{\sst
1}:\Q\X_i^{\sst(0,1)}\ra\Q\X_i^{\sst 0}\op\Q\X_i^{\sst 1}$ is
surjective. Combining this with \eq{il9eq6}, one can show that in
Definition \ref{il9dfn2}, one can choose $A^{\sst(0,1)}$ so that
$\Pi^{\sst 0}(A^{\sst(0,1)})=A^{\sst 0}$ and $\Pi^{\sst 1}(A^{\sst
(0,1)})=A^{\sst 1}$. Combining this with \eq{il7eq4}, \eq{il9eq4}
and \eq{il9eq6}, we see that
\begin{equation}
P^{\sst 0}\ci H^{\sst(0,1)}+H^{\sst 0}\ci P^{\sst 0}=0
\quad\text{and}\quad P^{\sst 1}\ci H^{\sst(0,1)}+H^{\sst 1}\ci
P^{\sst 1}=0.
\label{il9eq7}
\end{equation}

Now define $\p^{{\sst 0}0,0}_1:\Q\X^{\sst(0,1)}_N\ra\Q\X^{\sst 0}_N$
by $\p^{{\sst 0}0,0}_1=P^{\sst 0}$, and for all $k\ge 0$ and
$(\la,\mu)\in{\cal G}$ with $\nm{(\la,\mu)}+k-1\le N$ and
$(k,\la,\mu)\ne(1,0,0)$, define $\p^{\smash{{\sst
0}\la,\mu}}_k:{\buildrel{ \ulcorner\,\,\,\text{$k$ copies }
\,\,\,\urcorner}\over{\vphantom{m}\smash{\Q\X^{\sst(0,1)}_N
\t\cdots\t\Q\X^{\sst(0,1)}_N}}}\ra\Q\X^{\sst 0}_N$ by
$\p^{\smash{{\sst 0}\la,\mu}}_k=0$. Write $\p^{\sst 0}=\bigl(
\p^{\smash{{\sst 0}\la,\mu}}_k:k\ge 0$, $(\la,\mu)\in{\cal G}$,
$\nm{(\la,\mu)}+k-1\le N\bigr)$. Similarly, define $\p^{\sst
1}=\bigl(\p^{\smash{{\sst 1}\la,\mu}}_k:k\ge 0$, $(\la,\mu)\in{\cal
G}$, $\nm{(\la,\mu)}+k-1\le N\bigr)$ by $\p^{{\sst 1}0,0}_1=P^{\sst
1}$ and $\p^{\smash{{\sst 1}\la,\mu}}_k=0$
for~$(k,\la,\mu)\ne(1,0,0)$.
\label{il9dfn3}
\end{dfn}

\begin{thm} In Definition {\rm\ref{il9dfn3},} $\p^{\sst 0}:
(\Q\X^{\sst(0,1)}_N,{\cal G},\m^{\sst(0,1)})\ra(\Q\X^{\sst
0}_N,{\cal G},\m^{\sst 0})$ and\/ $\p^{\sst 1}:
(\Q\X^{\sst(0,1)}_N,{\cal G},\m^{\sst(0,1)})\ra (\Q\X^{\sst
1}_N,{\cal G},\m^{\sst 1})$ are strict, surjective $A_{N,0}$
morphisms, and weak homotopy equivalences.
\label{il9thm2}
\end{thm}

\begin{proof} Combining \eq{il7eq1}, \eq{il9eq1} and \eq{il9eq6},
and noting the difference in signs $(-1)^n$, $(-1)^{n+1}$ in the
first lines of \eq{il7eq1} and \eq{il9eq1} gives
\begin{equation}
\m_{1,\geo}^{{\sst 0}\,0,0}\ci P^{\sst 0}=P^{\sst
0}\ci\m_{1,\geo}^{{\sst(0,1)}\,0,0}\quad\text{and}\quad
\m_{1,\geo}^{{\sst 1}\,0,0}\ci P^{\sst 1}=P^{\sst
1}\ci\m_{1,\geo}^{{\sst(0,1)}\,0,0}.
\label{il9eq8}
\end{equation}
We shall prove the analogue of \eq{il9eq8} for $\m_{k,\geo}^{{\sst
0}\,\la,\mu},\m_{k,\geo}^{{\sst 1}\,\la,\mu}$,
$(k,\la,\mu)\ne(1,0,0)$, using equations \eq{il8eq3} and
\eq{il8eq5}. To do this we relate $P^{\sst 0},P^{\sst 1}$ to the
fibre products $\{0\}\t_{i,[0,1],\ldots}\cdots$,
$\{1\}\t_{i,[0,1],\ldots}\cdots$ used in \eq{il8eq3}
and~\eq{il8eq5}.

Suppose $f:\De_a\ra [0,1]\t(L\amalg R)$ lies in $\X^{\sst(0,1)}_N$
for $a>0$. Then $\pd^{\sst[0,1]}f=\sum_{b=0}^a(-1)^bf\ci F_b^a$. By
Theorem \ref{il8thm}$(N1)$(b), $f\ci F_b^a\in\X^{\sst 0}_N$ for at
most one $b=0,\ldots,a$. Suppose $f\ci F_b^a\in\X^{\sst 0}_N$. Then
$P^{\sst 0}(f)=-\Pi^{\sst 0}(\pd^{\sst[0,1]}f)=(-1)^{1+b}f\ci
F_b^a$. But as in the proof of Theorem \ref{il8thm}, we have $\{0\}
\t_{i,[0,1],\pi_1\ci f}\De_a\cong (-1)^{1+b}\De_{a-1}$, and the
restriction of $f$ to this $\De_{a-1}$ is $f\ci F_b^a$. Thus it is
natural to identify $P^{\sst 0}(f)$ with $(\{0\}\t_{i,[0,1],\pi_1\ci
f}\De_a,f\ci\pi_{\De_a})$, as signed singular simplices. This is
also valid if $f\ci F_b^a\notin\X^{\sst 0}_N$ for any
$b=0,\ldots,a$, since then $P^{\sst 0}(f)=0$ and
$\{0\}\t_{i,[0,1],\pi_1\ci f}\De_a=\emptyset$.

Therefore $P^{\sst 0}:\Q\X^{\sst(0,1)}_N\ra\Q\X^{\sst 0}_N$ is
essentially equivalent, with signs, to the fibre product
$\{0\}\t_{i,[0,1],\ldots}\cdots$, that is, $P^{\sst 0}$ takes
$f:\De_a\ra[0,1]\t(L\amalg R)$ to $f\ci\pi_{\De_a}:\{0\}
\t_{i,[0,1],\pi_1\ci f}\De_a\ra[0,1]\t(L\amalg R)$. In the same way,
$P^{\sst 1}:\Q\X^{\sst(0,1)}_N\ra\Q\X^{\sst 1}_N$ is essentially
equivalent, with signs, to the fibre product
$\{1\}\t_{i,[0,1],\ldots}\cdots$.

Now suppose as in $(N3)$ that $g_1\in\X^{\sst 0}_N,\ldots,
g_k\in\X^{\sst 0}_N$ and $f_1\in\X^{\sst(0,1)}_N,\ldots,
f_k\in\X^{\sst (0,1)}_N$ with $f_j:\De_{a_j}\ra[0,1]\t(L\amalg R)$,
$g_j:\De_{a_j-1}\ra\{0\}\t(L\amalg R)$ and $g_j=f_j\ci
F_{b_j}^{a_j}$ for $b_j=0,\ldots,a_j$ and $j=1,\ldots,k$. Then
$P^{\sst 0}(f_j)=(-1)^{1+b_j}g_j$, as above. Let $k\ge 0$ and
$(\la,\mu)\in{\cal G}$ with $\nm{(\la,\mu)}+k-1\le N$ and
$(k,\la,\mu)\ne(1,0,0)$. Then
\begin{align}
P^{\sst 0}\!\ci\!\m_{k,\geo}^{{\sst(0,1)}\la,\mu}&(f_1,\ldots,f_k)
\!= \!\!\!\!\!\!\!\!\!\!\!
\sum_{\begin{subarray}{l}
\be\in H_2(M,\io(L);\Z):\; [\om]\cdot\be=\la,\;\> \mu_L(\be)=2\mu,\\
\oM_{k+1}^\ma(\be,J_t:t\in[0,1],f_1,\ldots,f_k)\ne\emptyset\end{subarray}
\!\!\!\!\!\!\!\!\!\!\!\!\!\!\!\!\!\!\!\!\!\!\!\!\!\!\!\!\!\!\!\!\!\!\!\!\!
\!\!\!\!\!\!\!\!\!\!\!\!\!\!\!\!\!\!\!\!\!\!}
\!\!\!\!\!\!\!\!\!\!\!\!\!
\begin{aligned}[t]P^{\sst 0}\ci\Pi^{\sst(0,1)}\bigl[VC\bigl(\oM_{k+1}^\ma
(\be,J_t:t\!\in\![0,1],f_1,\ldots,f_k)&,\\
\bev,\s_{\be,J_t:t\in[0,1],f_1,\ldots,f_k}\bigr)\bigr]&\\
\end{aligned}
\nonumber\\
&\!\!\!\!\!\!\!\!\!\!\!\!\!\!\!\!=(-1)^{k+\sum_{j=1}^kb_j}
\!\!\!\!\!\!\!\!\!
\sum_{\begin{subarray}{l}
\be\in H_2(M,\io(L);\Z):[\om]\cdot\be=\la,\;\> \mu_L(\be)=2\mu,\\
\oM_{k+1}^\ma(\be,J_0,g_1,\ldots,g_k)\ne\emptyset\end{subarray}
\!\!\!\!\!\!\!\!\!\!\!\!\!\!\!\!\!\!\!\!\!\!\!\!\!\!\!\!\!
\!\!\!\!\!\!\!\!\!\!\!\!\!\!\!\!} \!\!\!\!\!\!\!\!
VC\bigl(\oM_{k+1}^\ma(\be,J_0,g_1,\ldots,g_k), \bev,\s^{\sst
0}_{\be,J_0,g_1,\ldots,g_k}\bigr)
\label{il9eq9}\\
&\!\!\!\!\!\!\!\!\!\!\!\!\!\!\!\!=(-1)^{k+\sum_{j=1}^kb_j}
\m_{k,\geo}^{{\sst 0}\la,\mu}(g_1,\ldots,g_k)=\m_{k,\geo}^{{\sst
0}\la,\mu}\bigl(P^{\sst 0}(f_1),\ldots,P^{\sst 0}(f_k)\bigr),
\nonumber
\end{align}
using \eq{il9eq1}, \eq{il7eq1} in the first and third steps, and
$P^{\sst 0}(f_j)\!=\!(-1)^{1+b_j}g_j$ in the fourth.

In the second, most difficult step of \eq{il9eq9} we use the
essential equivalence of $P^{\sst 0}$ with the fibre product
$\{0\}\t_{i,[0,1],\ldots}\cdots$, equations \eq{il8eq3} and
\eq{il8eq4}, and the identification of perturbation data
$\s_{\be,J_t:t\in[0,1],f_1, \ldots,f_{j-1},g_j,f_{j+1},\ldots,f_k}$,
$\s^{\sst 0}_{\be,J_0,g_1,\ldots,g_k}$ under \eq{il8eq4} in Theorem
\ref{il8thm}$(N3)$. The idea here is that because of this
compatibility of perturbation data, the two operations of taking
fibre product $\{0\}\t_{i,[0,1],\ldots}\cdots$ (basically $P^{\sst
0}$), and taking virtual chains using perturbation data, commute
when applied to $\oM_{k+1}^\ma(\be,J_t:t\in[0,1],f_1,\ldots,f_k)$.
That is, we can take virtual chains first and then apply $P^{\sst
0}$, giving the r.h.s.\ of the first line of \eq{il9eq9}. Or we can
apply $\{0\}\t_{i,[0,1],\ldots}\cdots$ first, giving
$(-1)^{k+\sum_{j=1}^kb_j}\oM_{k+1}^\ma(\be,J_0,g_1,\ldots,g_k)$ by
\eq{il8eq3}, and then take virtual chains, giving the second line of
\eq{il9eq9}. Since the two operations commute, the two expressions
are equal.

Now let $T$ be a rooted planar tree with $k$ leaves, and
$(\bs\la,\bs\mu)$ be as in Definition \ref{il9dfn2}. Then
Definitions \ref{il7dfn2} and \ref{il9dfn2} define $\m_{k,T}^{{\sst
0}(\bs\la,\bs\mu)}: {\buildrel {\ulcorner\,\,\,\text{$k$ copies }
\,\,\,\urcorner} \over {\vphantom{m}\smash{\Q\X^{\sst 0}_N\t\cdots\t
\Q\X^{\sst 0}_N}}}\ra\Q\X^{\sst(0,1)}_N$ and
$\m_{k,T}^{{\sst(0,1)}(\bs\la,\bs\mu)}: {\buildrel
{\ulcorner\,\,\,\text{$k$ copies } \,\,\,\urcorner} \over
{\vphantom{m}\smash{\Q\X^{\sst(0,1)}_N\t\cdots\t
\Q\X^{\sst(0,1)}_N}}}\ra\Q\X^{\sst(0,1)}_N$, where $\m_{k,T}^{{\sst
0}(\bs\la,\bs\mu)}$ assigns $\m_{n,\geo}^{{\sst 0}\la_v,\mu_v}$ to
an internal vertex and $(-1)^{n+1}H^{\sst 0}$ to an internal edge,
and $\m_{k,T}^{{\sst(0,1)}(\bs\la,\bs\mu)}$ assigns
$\m_{n,\geo}^{{\sst (0,1)}\la_v,\mu_v}$ to an internal vertex and
$(-1)^nH^{\sst(0,1)}$ to an internal edge.

Equation \eq{il9eq9} gives $P^{\sst 0}\ci\m_{n,\geo}^{{\sst(0,1)}
\la_v,\mu_v}=\m_{n,\geo}^{{\sst 0}\la_v,\mu_v}\ci(P^{\sst 0}\t
\cdots\t P^{\sst 0})$, and \eq{il9eq7} implies that $P^{\sst 0}\ci
(-1)^nH^{\sst(0,1)}=(-1)^{n+1}H^{\sst 0}\ci P^{\sst 0}$. Combining
these we see that $P^{\sst 0}\ci\m_{k,T}^{{\sst
(0,1)}(\bs\la,\bs\mu)}=\m_{k,T}^{{\sst 0}(\bs\la,\bs\mu)}\ci(P^{\sst
0}\t \cdots\t P^{\sst 0})$. Summing this over $T,(\bs\la,\bs\mu)$
and using \eq{il9eq8} now shows that $P^{\sst
0}\ci\m_k^{\smash{{\sst(0,1)}\la,\mu}}=\m_k^{\smash{{\sst
0}\la,\mu}}\ci(P^{\sst 0}\t \cdots\t P^{\sst 0})$ for all $k\ge 0$
and $(\la,\mu)\in{\cal G}$ with $\nm{(\la,\mu)}+k-1\le N$. This and
the definition of $\p^{\sst 0}$ imply that $\p^{\sst 0}$ is a {\it
strict\/ $A_{N,0}$ morphism}, as we have to prove.

From Definition \ref{il9dfn2}, $P^{\sst 0}:\Q\X^{\sst(0,1)}_N\ra
\Q\X^{\sst 0}_N$ is surjective, and $P^{\sst
0}_*:H^*\bigl(\Q\X^{\sst(0,1)}_N,\ab\pd^{\sst(0,1)}\bigr)\ab\ra
H^*\bigl(\Q\X^{\sst 0}_N,\pd^{\sst 0}\bigr)$ is an isomorphism. As
$\p_1^{{\sst 0}0,0}=P^{\sst 0}$, $\m^{{\sst 0}0,0}_1=(-1)^n\pd^{\sst
0}$ and $\m^{{\sst (0,1)}0,0}_1=(-1)^{n+1}\pd^{\sst(0,1)}$, it
follows that $\p^{\sst 0}$ is {\it surjective}, and a {\it weak
homotopy equivalence}, as we have to prove. The proof for $\p^{\sst
1}$ is the same, apart from sign differences $P^{\sst
1}(f_j)=(-1)^{b_j}g_j$ and between \eq{il8eq3} and~\eq{il8eq5}.
\end{proof}

By the $A_{N,0}$ version of Corollary \ref{il3cor2}, we deduce:

\begin{cor} We can construct explicit $A_{N,0}$ morphisms
${\mathfrak i}^{\sst 0}:(\Q\X^{\sst 0}_N,{\cal G},\m^{\sst
0})\ra(\Q\X^{\sst(0,1)}_N,{\cal G},\m^{\sst(0,1)})$ and\/
${\mathfrak i}^{\sst 1}:(\Q\X^{\sst 1}_N,{\cal G},\m^{\sst 1})\ra
(\Q\X^{\sst(0,1)}_N,{\cal G},\m^{\sst(0,1)})$ which are homotopy
inverses for $\p^{\sst 0},\p^{\sst 1}$ respectively, using sums over
planar trees. Hence $\f^{\sst 01}=\p^{\sst 1}\ci{\mathfrak i}^{\sst
0}:(\Q\X^{\sst 0}_N,{\cal G},\m^{\sst 0})\ra (\Q\X^{\sst 1}_N,{\cal
G},\m^{\sst 1})$ is an $A_{N,0}$ morphism and a homotopy
equivalence, with homotopy inverse $\f^{\sst 10}=\p^{\sst
0}\ci{\mathfrak i}^{\sst 1}$.
\label{il9cor}
\end{cor}

This is important, as it shows that the $A_{N,0}$ algebras we
associated to $L$ in \S\ref{il7} are independent of the almost
complex structure $J$ and other choices, up to homotopy equivalence.
We can now compare our proof of this with analogous results in
Fukaya et al.\ \cite[\S 19.1 \& \S 30.9]{FOOO}. In effect, in
\cite[Th.~19.1]{FOOO} Fukaya et al.\ construct a version $\f$ of our
homotopy equivalence $\f^{\sst 01}$ directly, without introducing an
intermediate $A_{N,0}$ algebra $(\Q\X^{\sst(0,1)}_N,{\cal
G},\m^{\sst(0,1)})$ as we do.

Since our ${\mathfrak i}^{\sst 0}$ involves a sum over planar trees,
one would expect their $\f$ also to involve sums over planar trees,
and it does, though this is not made very explicit. In
\cite[Def.~19.8]{FOOO}, Fukaya et al.\ define complicated moduli
spaces $\M_{k+1}^\ma(M',L',\{J_\rho\}_\rho:\be;{\rm top}(\rho))$
which are in effect disjoint unions over planar trees $T$ with $k$
leaves of multiple fibre products over $T$ of Kuranishi spaces,
where to each internal vertex of $T$ we associate
$\oM_{n+1}^\ma(\be_v,J_t:t\in[0,1])$ in our notation, and to each
internal edge of $T$ we associate $\{(s,t)\in[0,1]^2:s\le t\}$. Here
the fibre product `$\{(s,t)\in[0,1]^2:s\le
t\}\t_{\pi_1,[0,1],\ldots}\cdots$' is an analogue of our $H$, an
explicit partial inverse for~$\pd$.

All these sums and products over trees happen at the level of
Kuranishi spaces, not complexes $\Q\X_i$. To extend them to
complexes, Fukaya et al.\ \cite[Prop.~19.14, \S 30.9]{FOOO} choose
perturbation data $\s^{{\rm top}(\rho)}$ for the moduli spaces
$\M_{k+1}^\ma(M',L',\{J_\rho\}_\rho:\be;{\rm top}(\rho))$, and
further chain complexes $\Q\X_i'$, satisfying many compatibility
conditions. This adds an extra layer of complexity to the proof. We
believe our method in \S\ref{il8}--\S\ref{il9} is preferable to that
of \cite{FOOO}, as it is shorter and more transparent.

\section{Homotopies between $A_{N,0}$ morphisms}
\label{il10}

In \S\ref{il7} we constructed $A_{N,0}$ algebras $(\Q\X_N,{\cal
G},\m)$ from $L$ using a choice of almost complex structure $J$, and
in \S\ref{il9}, given two such $A_{N,0}$ algebras $(\Q\X^{\sst
0}_N,{\cal G},\m^{\sst 0}),\ab(\Q\X^{\sst 1}_N,\ab{\cal
G},\ab\m^{\sst 1})$ from $J_0,J_1$, we constructed a homotopy
equivalence $\f^{\sst 01}:(\Q\X^{\sst 0}_N,{\cal G},\m^{\sst 0})\ra
(\Q\X^{\sst 1}_N,{\cal G},\m^{\sst 1})$. We will now show that such
$\f^{\sst 01}$ are {\it unique up to homotopy}, and also that they
{\it form commutative triangles up to homotopy}.

\subsection{Uniqueness of $\f^{\sst 01}$ in Corollary \ref{il9cor}
up to homotopy}
\label{il101}

Let $J_0,J_1$ be complex structures on $M$ compatible with $\om$.
Fix $N\ge 0$, $N'=N(N+2)$ and $\cal G$, which must satisfy some
conditions below, once and for all. Suppose $(\Q\X^{\sst 0}_N,{\cal
G},\m^{\sst 0})$, $(\Q\X^{\sst 1}_N,{\cal G},\m^{\sst 1})$ are
possible outcomes for the $A_{N,0}$ algebra of Theorem \ref{il7thm}
with $J=J_0$ and $J=J_1$ respectively, and $N,N',{\cal G}$ as above.

Suppose $J_t:t\in[0,1]$ and $\hat J_t:t\in[0,1]$ are smooth
1-parameter families of almost complex structures on $M$ compatible
with $\om$ interpolating between $J_0$ and $J_1$, so that $\hat
J_0=J_0$ and $\hat J_1=J_1$. Let $(\Q\X^{\sst(0,1)}_N,{\cal
G},\m^{\sst(0,1)})$, $(\Q\hat\X{}^{\sst(0,1)}_N,{\cal
G},\hat\m^{\sst(0,1)})$ be possible outcomes for the $A_{N,0}$
algebra of Theorem \ref{il9thm1} using $J_t:t\in[0,1]$ and $\hat
J_t:t\in[0,1]$, and $\p^{\sst 0},\p^{\sst 1},{\mathfrak i}^{\sst
0},\f^{\sst 01},\hat\p^{\sst 0},\hat\p^{\sst 1},\hat{\mathfrak
i}^{\sst 0},\hat\f^{\sst 01}$ corresponding outcomes for the
$A_{N,0}$ morphisms $\p^{\sst 0},\p^{\sst 1},{\mathfrak i}^{\sst
0},\f^{\sst 01}$ of Theorem \ref{il9thm2} and
Corollary~\ref{il9cor}.

Then $\f^{\sst 01}=\p^{\sst 1}\ci{\mathfrak i}^{\sst 0}$ and
$\hat\f^{\sst 01}=\hat\p^{\sst 1}\ci\hat{\mathfrak i}^{\sst 0}$ are
both $A_{N,0}$ morphisms $(\Q\X^{\sst 0}_N,{\cal G},\m^{\sst 0})\ra
(\Q\X^{\sst 1}_N,{\cal G},\m^{\sst 1})$. We shall construct a {\it
homotopy} $\H:\f^{\sst 01}\Ra\hat\f^{\sst 01}$. This implies that
the $A_{N,0}$ morphism $\f^{\sst 01}:(\Q\X^{\sst 0}_N,{\cal
G},\m^{\sst 0})\ra(\Q\X^{\sst 1}_N,{\cal G},\m^{\sst 1})$ in
Corollary \ref{il9cor} is {\it independent of choices up to
homotopy}, and thus that the $A_{N,0}$ algebra $(\Q\X_N,{\cal
G},\m)$ in Theorem \ref{il7thm} is independent of $J$ and other
choices {\it up to canonical homotopy equivalence}, rather than just
up to homotopy equivalence.

To construct $\H$ we need to choose a 2-parameter family of almost
complex structures $J_s:s\in S$ interpolating between
$J_t:t\in[0,1]$ and $\hat J_t:t\in[0,1]$. The most obvious way to do
this is, as in Fukaya et al.\ \cite[\S 19.2]{FOOO}, is to take
$S=[0,1]^2$, with boundary conditions $J_{(0,t)}=J_t$,
$J_{(1,t)}=\hat J_t$, $J_{(s,0)}=J_0$ and $J_{(s,1)}=J_1$ for
$s,t\in[0,1]$. But for us there is a better choice: we take $S$ to
be the {\it semicircle}
\begin{equation*}
S=\bigl\{(x,y)\in\R^2:0\le x^2+y^2\le 1,\;\> y\ge 0\bigr\},
\end{equation*}
and $J_{(x,y)}:(x,y)\in S$ a smooth family of almost complex
structures on $M$ compatible with $\om$, with the boundary
conditions
\begin{equation*}
J_{(-1,0)}=J_0,\;\> J_{(1,0)}=J_1,\;\> J_{(2t-1,0)}=J_t,\;\>
J_{(-\cos\pi t,\sin\pi t)}=\hat J_t,\; t\in[0,1].
\end{equation*}

Here we regard $S$ as a 2-manifold with boundary and corners. It has
two corners $(\mp 1,0)$ to which we assign $J_0,J_1$, and two edges,
a straight edge $E$ to which we assign $J_t:t\in[0,1]$, and a
semicircle $\hat E$ to which we assign $\hat J_t:t\in[0,1]$. This is
illustrated in Figure \ref{il10fig1}(a). The semicircle is
preferable because our method will associate an $A_{N,0}$ algebra to
each face, edge and vertex of $S$. Using the square $[0,1]^2$ we
would have to deal with $1+4+4=9$ $A_{N,0}$ algebras, but the
semicircle gives only $1+2+2=5$ $A_{N,0}$ algebras, leading to a
simpler proof.

\begin{figure}[htb]
\centerline{
\begin{footnotesize}
\raisebox{-55pt}{$\begin{xy}
0;<1.7mm,0mm>:
,(0,-1.5)*{J_0}
,(20,-1.5)*{J_1}
,(10,-1.5)*{J_t:t\in[0,1]}
,(10,11.5)*{\hat J_t:t\in[0,1]}
,(10,3.5)*{J_{(x,y)}:(x,y)\in S}
,(0,0)*{\bullet}
,(20,0)*{\bullet}
,{\ellipse_{}}
,\ar@{-}(0,0);(20,0)
\end{xy}$}\qquad
$\displaystyle \xymatrix@!0@R=30pt@C=80pt{
&(\Q\hat\X{}^{\sst(0,1)}_N,{\cal G},\hat\m^{\sst(0,1)})
\ar@/_2pc/[ddl]_{\hat\p^{\sst 1}} \ar@/_-2pc/[ddr]^{\hat\p^{\sst 1}}
\\
& (\Q\X^{\sst S}_N,{\cal G},\m^{\sst S}) \ar[u]^{\hat\p^{\sst(0,1)}}
\ar[d]_{\p^{\sst(0,1)}}
\\
(\Q\X^{\sst 0}_N,{\cal G},\m^{\sst 1})
&
(\Q\X^{\sst(0,1)}_N,{\cal G},\m^{\sst(0,1)})
\ar[l]_{\p^{\sst 0}} \ar[r]^{\p^{\sst 1}}
&
(\Q\X^{\sst 1}_N,{\cal G},\m^{\sst 1})
}$
\end{footnotesize}}
\caption{\!\!\!(a) $J_s:s\in S$\qquad (b) $A_{N,0}$ algebras and
morphisms}
\label{il10fig1}
\end{figure}
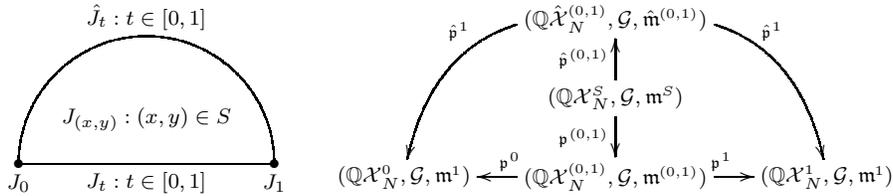

We need the family $J_s:s\in S$ to be compatible with $\cal G$ in
the sense that if $\be\in H_2(M,\io(L);\Z)$ and
$\oM_1^\ma(\be,J_s:s\in S)\ne\emptyset$ then $\bigl([\om]\cdot\be,
\ha\mu_L(\be)\bigr)\in{\cal G}$, generalizing condition (ii) of
\S\ref{il6} and \S\ref{il8}. One way to ensure this is always
possible is to choose $\cal G$ as follows: let $J_t:t\in\cal T$ be a
smooth family of complex structures on $M$ compatible with $\om$,
where $\cal T$ is a compact, connected, simply-connected manifold
with boundary. We think of $J_t:t\in\cal T$ as a {\it large} family,
with $\dim{\cal T}\gg 0$, the set of all almost complex structures
we are interested in. Define ${\cal G}\subset[0,\iy)\t\Z$ to be the
unique smallest subset satisfying the conditions:
\begin{itemize}
\setlength{\itemsep}{0pt}
\setlength{\parsep}{0pt}
\item[(i)] $\cal G$ is closed under addition with ${\cal
G}\cap(\{0\}\t\Z)=\{(0,0)\}$, and ${\cal G}\cap([0,C]\t\Z)$ is
finite for any $C\ge 0$; and
\item[(ii)] If $\be\in H_2(M,\io(L);\Z)$,
$\oM_1^\ma(\be,J_t:t\in{\cal T})\ne\emptyset$ then
$\bigl([\om]\cdot\be, \ha\mu_L(\be)\bigr)\in{\cal G}$.
\end{itemize}
Then $\cal G$ satisfies conditions (i),(ii) in \S\ref{il6} and
\S\ref{il8} and upon $J_s:s\in S$ above provided all the (families
of) complex structures $J$, $J_t:t\in[0,1]$, $J_s:s\in S$ that we
choose lie in $\cal T$. This problem of dependence of $\cal G$ on
$J$ will disappear in \S\ref{il11}, since although we need to
specify a particular $\cal G$ to define an $A_{N,K}$ algebra, we do
not need to specify $\cal G$ to define a gapped filtered $A_\iy$
algebra, there just has to exist some suitable~$\cal G$.

Our next result generalizes the material of \S\ref{il8}--\S\ref{il9}
to our 2-parameter family $J_s:s\in S$. To write the details out in
full would take pages, but the proof involves few new ideas, so we
will just briefly indicate how to modify sections \ref{il8}
and~\ref{il9}.

\begin{thm} In the situation above, generalizing Theorem
{\rm\ref{il9thm1}} we can define an $A_{N,0}$ algebra $(\Q\X^{\sst
S}_N,{\cal G},\m^{\sst S}),$ and generalizing Theorem
{\rm\ref{il9thm2}} we can define strict, surjective $A_{N,0}$
morphisms $\p^{\sst(0,1)}:(\Q\X^{\sst S}_N,{\cal G},\m^{\sst S})\ra
(\Q\X^{\sst(0,1)}_N,{\cal G},\m^{\sst(0,1)})$ and\/
$\hat\p^{\sst(0,1)}:(\Q\X^{\sst S}_N,{\cal G},\m^{\sst S})\ra
(\Q\hat\X{}^{\sst(0,1)}_N,{\cal G},\hat\m^{\sst(0,1)})$ which are
weak homotopy equivalences, such that Figure {\rm\ref{il10fig1}(b)}
is a commutative diagram of\/ $A_{N,0}$ morphisms, that is,
\begin{equation}
\p^{\sst 0}\ci\p^{\sst(0,1)}=\hat\p^{\sst
0}\ci\hat\p^{\sst(0,1)}\quad\text{and}\/\quad \p^{\sst
1}\ci\p^{\sst(0,1)}=\hat\p^{\sst 1}\ci\hat\p^{\sst(0,1)}.
\label{il10eq1}
\end{equation}
\label{il10thm1}
\end{thm}

\begin{proof} Here is how to modify Theorem \ref{il8thm} to the new
$J_s:s\in S$. The conclusion is that for a given $N\in\N,$ there are
$\bar\X{}_0^{\sst S}\subset\cdots \subset\bar\X{}_N^{\sst S}$ and\/
$\{\s_{\be,J_s:s\in S,f_1,\ldots,f_k}\}$ satisfying analogues of
$(N1)$--$(N3)$. In $(N1)$, $\bar\X{}_0^{\sst
S},\ldots,\bar\X{}_N^{\sst S}$ are finite sets of smooth simplicial
chains $f:\De_a\ra S\t(L\amalg R)$ with decompositions
\begin{equation*}
\bar\X{}^{\sst S}_i= \X^{\sst 0}_i \amalg\X^{\sst 1}_i
\amalg\X^{\sst(0,1)}_i \amalg\hat\X{}^{\sst(0,1)}_i \amalg\X^{\sst
S}_i\quad\text{for $i=0,\ldots,N$,}
\end{equation*}
where if $f\in\bar\X_i^{\sst S}$ and $a>0$ then $f\ci
F_b^a\in\bar\X_i^{\sst S}$ for $b=0,\ldots,a$, and
\begin{itemize}
\setlength{\itemsep}{0pt}
\setlength{\parsep}{0pt}
\item $\X^{\sst 0}_i$ consists of $f:\De_a\ra\{(-1,0)\}\t(L\amalg
R)$, and are identified with choices of $\X_i$ in Theorem
\ref{il6thm} with $J=J_1$ under~$L\amalg R\cong\{(-1,0)\}\t(L\amalg
R)$.
\item $\X^{\sst 1}_i$ consists of $f:\De_a\ra\{(1,0)\}\t(L\amalg
R)$, and are identified with choices of $\X_i$ in Theorem
\ref{il6thm} with $J=J_1$ under~$L\amalg R\cong\{(1,0)\}\t(L\amalg
R)$.
\item $\X^{\sst(0,1)}_i$ consists of $f:\De_a\ra E\t(L\amalg
R)$, and are identified with choices of $\X_i^{\sst(0,1)}$ in
Theorem \ref{il8thm} with for $J_t:t\in[0,1]$ under $[0,1]\t(L\amalg
R)\cong E\t(L\amalg R)$ given by $t\mapsto (2t-1,0)$. Also $f$ maps
$\De_a^\ci\ra E^\ci\t(L\amalg R)$ and $\pi_E\ci f$ is a submersion
near $(\pi_E\ci f)^{-1}\bigl(\{(\pm 1,0)\}\bigr)$, as in~$(N1)$(a).
\item $\hat\X{}^{\sst(0,1)}_i$ consists of $f:\De_a\ra\hat E\t(L\amalg
R)$, and are identified with choices of $\X_i^{\sst(0,1)}$ in
Theorem \ref{il8thm} with for $\hat J_t:t\in[0,1]$ under
$[0,1]\t(L\amalg R)\cong\hat E\t(L\amalg R)$ given by $t\mapsto
(-\cos\pi t,\sin\pi t)$. Also $f$ maps $\De_a^\ci\ra\hat
E^\ci\t(L\amalg R)$ and $\pi_{\hat E}\ci f$ is a submersion near
$(\pi_{\hat E}\ci f)^{-1}\bigl(\{(\pm 1,0)\}\bigr)$.
\item $\X^{\sst S}_i$ consists of $f:\De_a\ra S\t(L\amalg
R)$ such that $f$ maps $\De_a^\ci\ra S^\ci\t(L\amalg R)$ and
$\pi_S\ci f$ is transverse to $\pd S$. That is, for each $p\in
\pd\De_a$ with $\pi_S\ci f(p)\in\pd S$, we require that $\d(\pi_S\ci
f)(T_p\De_a)+T_{\pi_S\ci f(p)}(\pd S)=T_{\pi_S\ci f(p)}S$.
Furthermore, if $a>0$ then for each $b=0,\ldots,a$ we have $f\ci
F_b^a\in\X^{\sst(0,1)}_i,\hat\X{}^{\sst(0,1)}_i$ or $\X^{\sst S}_i$,
that is, we do not allow $f\ci F_b^a\in\X^{\sst 0}_i$ or $\X^{\sst
1}_i$. Also, $f\ci F_b^a\in\X^{\sst(0,1)}_i$ for at most one
$b=0,\ldots,a$, and $f\ci F_b^a\in\hat\X{}^{\sst(0,1)}_i$ for at
most one~$b=0,\ldots,a$.
\end{itemize}
Here the submersion and transversality conditions are equivalent to
Condition \ref{il4cond}, so that we can apply Remark~\ref{il4rem}.

We regard $\Q\X^{\sst S}_i$ as a space of relative chains in
$C^\rsi_*\bigl(S\!\t\!(L\!\amalg\!R),\pd S\!\t\!(L\amalg
R);\Q\bigr)$. As for \eq{il8eq1} we require the following maps to be
isomorphisms:
\begin{equation}
\begin{split}
&H_*(\Q\X^{\sst 0}_i,\pd^{\sst 0})\,{\buildrel\cong\over\longra}\,
H_*^\rsi(L\amalg R;\Q),\quad H_*(\Q\X^{\sst
1}_i,\pd^{\sst 1})\,{\buildrel\cong\over\longra}\,H_*^\rsi(L\amalg R;\Q),\\
&H_*\bigl(\Q\X^{\sst(0,1)}_i,\pd^{\sst(0,1)}\bigr)\,{\buildrel\cong\over
\longra}\, H_*^\rsi\bigl(E\t(L\amalg R),\{(\pm 1,0)\}\t(L\amalg
R);\Q\bigr),\\
&H_*\bigl(\Q\hat\X{}^{\sst(0,1)}_i,\hat\pd^{\sst(0,1)}\bigr)\,
{\buildrel\cong\over \longra}\, H_*^\rsi\bigl(\hat E\t(L\amalg
R),\{(\pm 1,0)\}\t(L\amalg
R);\Q\bigr),\\
&H_*\bigl(\Q\X^{\sst S}_i,\pd^{\sst S}\bigr)\,{\buildrel \cong
\over\longra}\, H_*^\rsi\bigl(S\t(L\amalg R),\pd S\t(L\amalg
R);\Q\bigr).
\end{split}
\label{il10eq2}
\end{equation}
In $(N2),(N3)$, for all $k\ge 0$, $f_1\in\bar\X{}^{\sst
S}_{i_1},\ldots, f_k\in\bar\X{}^{\sst S}_{i_k}$ and $\be\in
H_2(M,\io(L);\Z)$ with $i_1+\cdots+i_k+\nm{\be}+k-1\le N$ and
$\oM_{k+1}^\ma(\be,J_s:s\in S,f_1,\ldots,f_k)\ne\emptyset$,
$\s_{\be,J_s:s\in S,f_1,\ldots,f_k}$ is perturbation data for
$\oM_{k+1}^\ma(\be,J_s:s\in S,f_1,\ldots,f_k)$, which should satisfy
compatibilities both over the boundary of
$\oM_{k+1}^\ma(\be,J_s:s\in S,\ab f_1,\ab\ldots,\ab f_k)$, and with
previous choices made in Theorem \ref{il6thm} for $J_0,J_1$ and
Theorem \ref{il8thm} for $J_t:t\in[0,1]$ and~$\hat J_t:t\in[0,1]$.

We modify Definition \ref{il9dfn1} to define $\Q$-multilinear maps
$\m_{k,\geo}^{{\sst S}\la,\mu}: \Q\X_{i_1}^{\sst S}\t
\cdots\t\Q\X_{i_k}^{\sst S}\ra\Q\X_{i_1+\cdots+i_k+\nm{(\la,
\mu)}+k-1}^{\sst S}$ of degree $1-2\mu$ by
\begin{align*}
&\m_{1,\geo}^{{\sst S}0,0}(f_1)=\Pi^{\sst S}\bigl[(-1)^{n+2}\pd
f_1\bigr]=(-1)^{n+2}\pd^{\sst S}f_1,\\
&\m_{k,\geo}^{{\sst S}\la,\mu}(f_1,\ldots,f_k)= \!\!\!\!\!\!\!\!\!
\sum_{\begin{subarray}{l}
\be\in H_2(M,\io(L);\Z):\\
[\om]\cdot\be=\la,\;\> \mu_L(\be)=2\mu,\\
\oM_{k+1}^\ma(\be,J_s:s\in
S,f_1,\ldots,f_k)\ne\emptyset\end{subarray}
\!\!\!\!\!\!\!\!\!\!\!\!\!\!\!\!\!\!\!\!\!\!\!\!\!\!\!\!\!\!\!\!\!\!\!\!\!}
\!\!\!\!\!\!\!\!\!\!\!\!
\begin{aligned}[t]\Pi^{\sst S}\bigl[VC\bigl(\oM_{k+1}^\ma
(\be,J_s:s\!\in\!S,f_1,\ldots,f_k)&,\\
\bev,\s_{\be,J_s:s\in S,f_1,\ldots,f_k}\bigr)\bigr]&,\\
(k,\la,\mu)\ne(1,0,0)&.
\end{aligned}
\end{align*}
The analogue of Proposition \ref{il9prop} holds. In our modification
of Definition \ref{il9dfn2} we assign $(-1)^{n+1}H^{\sst S}$ to each
internal edge, and then the analogue of Theorem \ref{il9thm1} holds,
giving the $A_{N,0}$ algebra $(\Q\X^{\sst S}_N,{\cal G},\m^{\sst
S})$.

The strict $A_{N,0}$-morphisms $\p^{\sst(0,1)}:(\Q\X^{\sst
S}_N,{\cal G},\m^{\sst S})\ra (\Q\X^{\sst(0,1)}_N,{\cal
G},\m^{\sst(0,1)})$ and $\hat\p^{\sst(0,1)}:(\Q\X^{\sst S}_N,{\cal
G},\m^{\sst S})\ra (\Q\hat\X{}^{\sst(0,1)}_N,{\cal
G},\hat\m^{\sst(0,1)})$ are defined as in Definition \ref{il9dfn3},
but using the projections $P^{\sst(0,1)}:\Q\X^{\sst S}_i\ra
\Q\X^{\sst(0,1)}_i$ and $\hat P^{\sst(0,1)}:\Q\X^{\sst S}_i\ra
\Q\hat\X{}^{\sst(0,1)}_i$ defined by $P^{\sst(0,1)}=\Pi^{\sst(0,1)}
\ci\bar\pd^{\sst S}$ and $\hat P^{\sst(0,1)}=-\hat\Pi^{\sst(0,1)}
\ci\bar\pd^{\sst S}$, where $\bar\pd^{\sst S}$ is the boundary
operator on $\Q\bar\X{}_i^{\sst S}$ and
$\Pi^{\sst(0,1)},\hat\Pi^{\sst(0,1)}$ are the projections to
$\Q\X^{\sst(0,1)}_i,\Q\hat\X{}^{\sst(0,1)}_i$. The difference in
signs here is because in oriented manifolds we have $\pd S=E\amalg
-\hat E$, where the orientations on $E,\hat E$ are determined by
their identifications with~$[0,1]$.

Then the analogue of Theorem \ref{il9thm2} holds, so that
$\p^{\sst(0,1)},\hat\p^{\sst(0,1)}$ are strict, surjective $A_{N,0}$
morphisms. Using \eq{il10eq2} and the natural isomorphism
\begin{equation*}
H_*^\rsi\bigl(S\!\t\!(L\!\amalg\!R),\pd
S\!\t\!(L\!\amalg\!R);\Q\bigr) \,{\buildrel\cong\over\longra}\,
H_{*-1}^\rsi\bigl(E\!\t\!(L\!\amalg\!R),\{(\pm 1,0)\}\!\t\!(L\amalg
R);\Q\bigr),
\end{equation*}
we find that $\p^{\sst(0,1)}$ is a weak homotopy equivalence, and
similarly so is~$\hat\p^{\sst(0,1)}$.

Equation \eq{il10eq1} now follows immediately from the identities
\begin{equation}
P^{\sst 0}\ci P^{\sst(0,1)}=\hat P^{\sst 0}\ci\hat P^{\sst(0,1)}
\qquad\text{and}\qquad P^{\sst 1}\ci P^{\sst(0,1)}=\hat P^{\sst
1}\ci\hat P^{\sst(0,1)}.
\label{il10eq3}
\end{equation}
To prove these, suppose that $f:\De_a\ra S\t(L\amalg R)$ lies in
$\X{}_N^{\sst S}$ with $P^{\sst 0}\ci P^{\sst(0,1)}(f)\ne 0$. Then
there exist $b=0,\ldots,a$ with $f\ci F_b^a\in\X^{\sst(0,1)}_N$ and
$c=0,\ldots,a-1$ with $f\ci F_b^a\ci F_c^{a-1}\in\X^{\sst 0}_N$,
where $b,c$ are unique by the conditions on $\X{}_N^{\sst S}$ above
and the conditions on $\X^{\sst(0,1)}_N$ in Theorem
\ref{il8thm}$(N1)$(b). Therefore $P^{\sst(0,1)}(f)=(-1)^{b}f\ci
F_b^a$, and $P^{\sst 0}\ci P^{\sst(0,1)}(f)=(-1)^{1+b+c} f\ci
F_b^a\ci F_c^{a-1}$, as $P^{\sst(0,1)}=\Pi^{\sst(0,1)}
\ci\bar\pd^{\sst S}$ and~$P^{\sst 0}=-\Pi^{\sst
0}\ci\pd^{\sst[0,1]}$.

If $c<b$ define $b'=c$ and $c'=b-1$, and if $c\ge b$ define $b'=c+1$
and $c'=b$. Then $f\ci F_b^a\ci F_c^{a-1}=f\ci F_{b'}^a\ci
F_{c'}^{a-1}$, so $(f\ci F_{b'}^a)\ci F_{c'}^{a-1}\in\X^{\sst 0}_N$.
The conditions on $\X{}_N^{\sst S}$ above give $f\ci
F_{b'}^a\notin\X{}_N^{\sst S}$, and also $f\ci
F_{b'}^a\notin\X{}_N^{\sst 0},\X{}_N^{\sst 1}$. Thus $f\ci F_{b'}^a$
lies in $\X^{\sst(0,1)}_N$ or $\hat\X{}^{\sst(0,1)}_N$. But $f\ci
F_b^a\in\X^{\sst(0,1)}_N$, $b\ne b'$ and uniqueness of $b$ in the
conditions on $\X{}_N^{\sst S}$ above imply that $f\ci
F_{b'}^a\notin \X^{\sst(0,1)}_N$. Hence $f\ci
F_{b'}^a\in\hat\X{}^{\sst(0,1)}_N$. The argument above now gives
$\hat P{}^{\sst(0,1)}(f)=(-1)^{1+b'} f\ci F_{b'}^a$, as $\hat
P{}^{\sst(0,1)}=-\hat\Pi{}^{\sst(0,1)} \ci\bar\pd^{\sst S}$, and
$\hat P{}^{\sst 0}\ci\hat P{}^{\sst(0,1)}(f) =(-1)^{b'+c'}f\ci
F_{b'}^a\ci F_{c'}^{a-1}=(-1)^{1+b+c} f\ci F_b^a\ci F_c^{a-1}=
P^{\sst 0}\ci P^{\sst(0,1)}(f)$, as $\hat P{}^{\sst
0}=-\hat\Pi{}^{\sst 0}\ci\hat\pd{}^{\sst[0,1]}$. Therefore if
$P^{\sst 0}\ci P^{\sst(0,1)}(f)\ne 0$ then $P^{\sst 0}\ci
P^{\sst(0,1)}(f)=\hat P{}^{\sst 0}\ci\hat P{}^{\sst(0,1)}(f)$. By
the same reasoning, if $\hat P{}^{\sst 0}\ci\hat P{}^{\sst(0,1)}
(f)\ne 0$ then $P^{\sst 0}\ci P^{\sst(0,1)}(f)=\hat P{}^{\sst
0}\ci\hat P{}^{\sst(0,1)}(f)$. This proves the first equation of
\eq{il10eq3}. The second is similar.
\end{proof}

Here is the main result of this section.

\begin{thm} In the situation above, there exists a homotopy
$\H:\f^{\sst 01}\Ra\hat\f^{\sst 01}$ between the $A_{N,0}$ morphisms
$\f^{\sst 01},\hat\f^{\sst 01}:(\Q\X^{\sst 0}_N,{\cal G},\m^{\sst
0})\ra (\Q\X^{\sst 1}_N,{\cal G},\m^{\sst 1})$.
\label{il10thm2}
\end{thm}

\begin{proof} As $\p^{\sst(0,1)},\hat\p^{\sst(0,1)}$ are weak
homotopy equivalences by Theorem \ref{il10thm2}, they are homotopy
equivalences by Theorem \ref{il3thm5}(c), so they have homotopy
inverses ${\mathfrak i}^{\sst(0,1)},\ab\hat{\mathfrak
i}^{\sst(0,1)}$. Write $\f\sim\g$ when two $A_{N,0}$ morphisms are
homotopic. Then we have
\begin{gather*}
\f^{\sst 01}=\p^{\sst 1}\ci{\mathfrak i}^{\sst 0}=\p^{\sst
1}\ci\id_{\Q\X_N^{\sst(0,1)}}\ci{\mathfrak i}^{\sst 0}\sim \p^{\sst
1}\ci\p^{\sst(0,1)}\ci{\mathfrak i}^{\sst(0,1)}\ci
{\mathfrak i}^{\sst 0}=\\
\hat\p^{\sst 1}\ci\hat\p^{\sst(0,1)}\ci{\mathfrak i}^{\sst(0,1)}\ci
{\mathfrak i}^{\sst 0}\sim \hat\p^{\sst
1}\ci\hat\p^{\sst(0,1)}\ci\hat{\mathfrak i}^{\sst(0,1)}\ci
\hat{\mathfrak i}^{\sst 0}\sim \hat\p^{\sst
1}\ci\id_{\Q\hat\X_N^{\sst(0,1)}}\ci\hat{\mathfrak i}^{\sst 0}=
\hat\p^{\sst 1}\ci\hat{\mathfrak i}^{\sst 0}=\hat\f^{\sst 01}.
\end{gather*}
Here in the third step $\p^{\sst(0,1)}\ci{\mathfrak
i}^{\sst(0,1)}\sim\id_{\Q\X_N^{\sst(0,1)}}$ as
$\p^{\sst(0,1)},{\mathfrak i}^{\sst(0,1)}$ are homotopy inverses,
the sixth step is similar, and in the fourth step we use
\eq{il10eq1}. For the fifth step, ${\mathfrak
i}^{\sst(0,1)}\ci{\mathfrak i}^{\sst 0}\sim\hat{\mathfrak
i}^{\sst(0,1)}\ci \hat{\mathfrak i}^{\sst 0}$ since these are
homotopy inverses for $\p^{\sst 0}\ci\p^{\sst(0,1)},\hat\p^{\sst
0}\ci\hat\p^{\sst(0,1)}$, which are equal by \eq{il10eq1}. Thus $\H$
exists, as homotopy is an equivalence relation.
\end{proof}

If ${\mathfrak i}^{\sst 0},\hat{\mathfrak i}^{\sst 0}$ are
constructed by sums over planar trees as in the $A_{N,0}$ version of
Corollary \ref{il3cor2}, then we can construct $\H$ explicitly as a
(complicated) sum over trees using the techniques of Markl
\cite{Mark}. Fukaya et al.\ \cite[\S 19.2 \& \S 30.10]{FOOO} prove
results analogous to Theorem \ref{il10thm2} by a rather more
elaborate method. Their proof involves a family of almost complex
structures $J_{\rho,s}$ for $(\rho,s)\in[0,1]^2$, four $A_{N,K}$
algebras of chains on $L$, and one $A_{N,K}$ algebra of chains
on~$(-\ep,1+\ep)\t L$.

To construct one of the $A_{N,0}$ morphisms between these, they
define \cite[eq.~(19.27)]{FOOO} complicated moduli spaces
$\M_{k+1}^\ma(M',L',\{J_{\rho,s}\}_{\rho,s}:\be;{\rm top}(\rho),{\rm
twp}(s))$, which are in effect disjoint unions over planar trees $T$
with $k$ leaves of multiple fibre products over $T$ of Kuranishi
spaces, with $\oM_{n+1}^\ma(\be_v,J_{\rho,s}:\rho,s\in[0,1])$ at
each internal vertex, and $\{(\rho_1,\rho_2)
\in[0,1]^2:\rho_1\le\rho_2\}$ at each internal edge. This sum over
trees roughly speaking constructs an explicit homotopy inverse for
the strict surjective $A_{N,K}$ morphism $\p^{\sst
0}\ci\p^{\sst(0,1)}=\hat\p^{\sst 0}\ci\hat\p^{\sst(0,1)}$ in our
notation, using the method of~\S\ref{il33}.

\subsection{Compositions of $\f^{\sst 01}$ in Corollary
\ref{il9cor} up to homotopy}
\label{il102}

Let $J^a,J^b,J^c$ be complex structures on $M$ compatible with
$\om$. Fix $N\ge 0$, $N'=N(N+2)$ and $\cal G$, which must satisfy
some conditions below, once and for all. Suppose $(\Q\X^{\sst
a}_N,{\cal G},\m^{\sst a})$, $(\Q\X^{\sst b}_N,{\cal G},\m^{\sst
b})$, $(\Q\X^{\sst c}_N,{\cal G},\m^{\sst c})$ are possible outcomes
for the $A_{N,0}$ algebra of Theorem \ref{il7thm} with
$J=J^a,J^b,J^c$ respectively, and $N,N',{\cal G}$ as above.

Suppose $J_t^{ab},J_t^{bc},J_t^{ac}$ for $t\in[0,1]$ are smooth
1-parameter families of almost complex structures on $M$ compatible
with $\om$ with $J_0^{ab}=J_0^{ac}=J^a$, $J_1^{ab}=J_0^{bc}=J^b$,
$J_1^{bc}=J_1^{ac}=J^c$. Let $(\Q\X^{\sst ab}_N,{\cal G},\m^{\sst
ab})$ be the $A_{N,0}$ algebra of Theorem \ref{il9thm1} using
$J_t^{ab}:t\in[0,1]$. Write $\p^{\sst ab,a},\p^{\sst
ab,b},{\mathfrak i}^{\sst a,ab},\f^{\sst ab}$ respectively for the
$A_{N,0}$ morphisms $\p^{\sst 0},\p^{\sst 1},{\mathfrak i}^{\sst
0},\f^{\sst 01}$ of Theorem \ref{il9thm2} and Corollary \ref{il9cor}
for $J_t^{ab}:t\in[0,1]$, so that $\p^{\sst ab,a}:(\Q\X^{\sst
ab}_N,{\cal G},\m^{\sst ab})\ra(\Q\X^{\sst a}_N,{\cal G},\m^{\sst
a})$, and so on. Use the analogous notation for
$J_t^{bc},J_t^{ac}:t\in[0,1]$. Then $\f^{\sst ac}$ and $\f^{\sst
bc}\ci\f^{\sst ab}$ are both $A_{N,0}$ morphisms $(\Q\X^{\sst
a}_N,{\cal G},\m^{\sst a})\ra (\Q\X^{\sst c}_N,{\cal G},\m^{\sst
c})$. We shall construct a {\it homotopy} $\H:\f^{\sst
ac}\Ra\f^{\sst bc}\ci\f^{\sst ab}$, using a very similar method
to~\S\ref{il101}.

To construct $\H$ we choose a 2-parameter family of almost complex
structures $J_t:t\in T$ interpolating between $J_t^{ab},
J_t^{bc},J_t^{ac}$ for $t\in[0,1]$. Let $T$ be the triangle
\begin{equation*}
T=\bigl\{(x,y)\in\R^2:x\le 1,\;\> y\ge 0,\;\> x\ge y\bigr\},
\end{equation*}
and $J_{(x,y)}:(x,y)\in T$ a smooth family of almost complex
structures on $M$ compatible with $\om$, with the boundary
conditions
\begin{equation*}
J_{(0,0)}\!=\!J^a,\; J_{(1,0)}\!=\!J^b,\; J_{(1,1)}\!=\!J^c,\;
J_{(t,0)}\!=\!J_t^{ab},\; J_{(1,t)}\!=\!J_t^{bc},\;
J_{(t,t)}\!=\!J_t^{ac},\; t\in[0,1].
\end{equation*}
This is illustrated in Figure \ref{il10fig2}(a). We need the family
$J_t:t\in T$ to be compatible with $\cal G$ in the sense that if
$\be\in H_2(M,\io(L);\Z)$ and $\oM_1^\ma(\be,J_t:t\in
T)\ne\emptyset$ then $\bigl([\om]\cdot\be,
\ha\mu_L(\be)\bigr)\in{\cal G}$. We can ensure this as
in~\S\ref{il101}.

\begin{figure}[htb]
\centerline{
\begin{footnotesize}
\raisebox{-75pt}{$\begin{xy}
0;<1.1mm,0mm>: ,(-2,-1)*{J^a} ,(22,-1)*{J^b} ,(22,21)*{J^c}
,(10,-2)*{J_t^{ab}:t\in[0,1]} ,(20.5,12)*!L{J_t^{bc}:}
,(20.8,8)*!L{t\in[0,1]} ,(1,16)*!L{J_t^{ac}:} ,(1,12)*!L{t\in[0,1]}
,(10,7)*!L{J_t:} ,(10,3)*!L{t\in T} ,(0,0)*{\bullet}
,(20,0)*{\bullet} ,(20,20)*{\bullet} ,\ar@{-}(0,0);(20,0)
,\ar@{-}(20,0);(20,20) ,\ar@{-}(20,20);(0,0)
\end{xy}$}\!\!\!\!\!\!\!\!\!\!\!\!\!\!
$\displaystyle \xymatrix@!0@R=20pt@C=50pt{ &&&&(\Q\hat\X{}^{\sst
c}_N,{\cal G},\hat\m^{\sst c})\\
\\
&& (\Q\hat\X{}^{\sst ac}_N,{\cal G},\hat\m^{\sst ac})
\ar[uurr]^{\p^{\sst ac,c}} \ar[ddll]_{\p^{\sst ac,a}} &&
(\Q\hat\X{}^{\sst bc}_N,{\cal G},\hat\m^{\sst bc}) \ar[uu]_{\p^{\sst
bc,c}} \ar[dd]^{\p^{\sst bc,b}}
\\
&&& (\Q\hat\X{}^{\sst T}_N,{\cal G},\hat\m^{\sst T})
\ar[ul]^{\p^{\sst ac}} \ar[ur]^(0.4){\p^{\sst bc}}
\ar[dl]^(0.4){\p^{\sst ab}}
\\
(\Q\hat\X{}^{\sst a}_N,{\cal G},\hat\m^{\sst a}) &&
(\Q\hat\X{}^{\sst ab}_N,{\cal G},\hat\m^{\sst ab})
\ar[rr]^(0.6){\p^{\sst ab,b}} \ar[ll]_{\p^{\sst ab,a}} &&
(\Q\hat\X{}^{\sst b}_N,{\cal G},\hat\m^{\sst b}). }$
\end{footnotesize}}
\caption{\!\!\!(a) $J_t:t\in T$\qquad (b) $A_{N,0}$ algebras and
morphisms}
\label{il10fig2}
\end{figure}
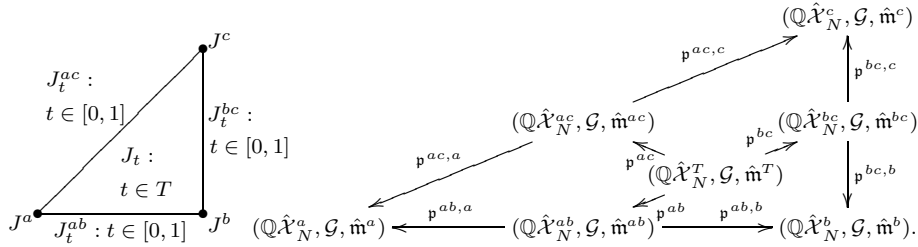

Then we prove analogues of Theorems \ref{il10thm1} and
\ref{il10thm2} by the same methods:

\begin{thm} In the situation above, we can define an $A_{N,0}$
algebra $(\Q\X^{\sst T}_N,{\cal G},\ab\m^{\sst T})$ and strict,
surjective $A_{N,0}$ morphisms $\p^{\sst ab}:(\Q\X^{\sst T}_N,{\cal
G},\m^{\sst T})\ra(\Q\X^{\sst ab}_N,{\cal G},\m^{\sst ab}),$
$\p^{\sst bc}:(\Q\X^{\sst T}_N,{\cal G},\m^{\sst T})\ra(\Q\X^{\sst
bc}_N,{\cal G},\m^{\sst bc}),$ $\p^{\sst ac}:(\Q\X^{\sst T}_N,{\cal
G},\m^{\sst T})\ra(\Q\X^{\sst ac}_N,{\cal G},\m^{\sst ac})$ which
are weak homotopy equivalences, such that Figure
{\rm\ref{il10fig2}(b)} is a commutative diagram.
\label{il10thm3}
\end{thm}

\begin{thm} In the situation above, there exists a homotopy
$\H:\f^{\sst ac}\Ra\f^{\sst bc}\ci\f^{\sst ab}$ between the
$A_{N,0}$ morphisms $\f^{\sst ac},\f^{\sst bc}\ci\f^{\sst
ab}:(\Q\X^{\sst a}_N,{\cal G},\m^{\sst a})\ra (\Q\X^{\sst c}_N,{\cal
G},\m^{\sst c})$.
\label{il10thm4}
\end{thm}

Fukaya et al.\ \cite[\S 19.3]{FOOO} prove related results by a
different method. In our notation, they suppose that the families
$J_t^{ab},J_t^{bc},J_t^{ac}$ satisfy $J_t^{ac}=J_{2t}^{ab}$ for
$t\le\ha$ and $J_t^{ac}=J_{2t-1}^{bc}$ for $t\ge\ha$, and show that
one can make choices in the constructions of $\f^{\sst ab},\f^{\sst
bc},\f^{\sst ac}$ so that $\f^{\sst ac}=\f^{\sst bc}\ci\f^{\sst
ab}$. Then for more general choices of $J_t^{ab},J_t^{bc},J_t^{ac}$
and $\f^{\sst ab},\f^{\sst bc},\f^{\sst ac}$, Theorem \ref{il10thm4}
follows from Theorem~\ref{il10thm2}.

\section{Gapped filtered $A_\iy$ algebras from immersed Lagrangians}
\label{il11}

We can now, at last, associate a gapped filtered $A_\iy$ algebra
to~$L$.

\begin{dfn} Suppose $(M,\om)$ is a compact symplectic manifold, and
$\io:L\ra M$ a compact immersed Lagrangian in $M$ with only
transverse double self-intersections. Let $J$ be an almost complex
structure on $M$ compatible with $\om$. Choose a relative spin
structure for $\io:L\ra M$ and orientations $o_{(p_-,p_+)}$ of the
$\Ker\bar\pd_{\la_{(p_-,p_+)}}$ as in \S\ref{il5}. Let ${\cal
G}\subset[0,\iy)\t\Z$ satisfy conditions (i),(ii) of~\S\ref{il6}.

For each $N=0,1,2,\ldots$, let $(\Q\X_N,{\cal G},\m_N)$ be an
$A_{N,0}$ algebra constructed in Theorem \ref{il7thm} for these
$J,\cal G$; we write $\m_N$ rather than $\m$ to make clear the
dependence on $N$. We assume no relation between the choices made in
constructing $(\Q\X_N,{\cal G},\m_N)$ and $(\Q\X_{N'},{\cal
G},\m_{N'})$ for $N\ne N'$, so the sets of simplices, perturbation
data, and so on, can all be different.

As in \S\ref{il7}, any $A_{N+1,0}$ algebra $(A,{\cal G},\bar\m)$ can
be truncated to an $A_{N,0}$ algebra $(A,{\cal G},\m)$ by taking
$\m$ to be the subset of $\bar\m_k^{\smash{\la,\mu}}$ with
$\nm{(\la,\mu)}+k-1\le N$. Write $(\Q\X_{N+1},{\cal G},\m_{N+1})_N$
for the truncation of $(\Q\X_{N+1},{\cal G},\m_{N+1})$ to an
$A_{N,0}$ algebra. Then $(\Q\X_N,{\cal G},\m_N)$ and
$(\Q\X_{N+1},{\cal G},\m_{N+1})_N$ are both possible outcomes for
$A_{N,0}$ algebras constructed in Theorem \ref{il7thm} using $J,\cal
G$. Applying the results of \S\ref{il8}--\S\ref{il9} with $J_t=J$
for $t\in[0,1]$, Corollary \ref{il9cor} constructs an $A_{N,0}$
morphism $\f^{\sst 01}$ that we will write as $\f^N:(\Q\X_N,{\cal
G},\m_N)\ra(\Q\X_{N+1},{\cal G},\m_{N+1})_N$, which is a homotopy
equivalence. Putting $J_s\equiv J$ for $s\in S$ in \S\ref{il10},
Theorem \ref{il10thm2} implies that $\f^N$ is independent of choices
up to homotopy.

Set $\X=\X_0$. By induction on $N=0,1,2,\ldots$ we shall construct
$\m^N$ such that $(\Q\X,{\cal G},\m^N)$ is an $A_{N,0}$ algebra, and
an $A_{N,0}$ morphism $\g^N:(\Q\X,{\cal G},\m^N)\ra(\Q\X_N,{\cal
G},\m_N)$ which is a weak homotopy equivalence, satisfying the
conditions:
\begin{itemize}
\setlength{\itemsep}{0pt}
\setlength{\parsep}{0pt}
\item[(i)] $\m^0=\m_0$ and $\g^0=\id_{\Q\X}$;
\item[(ii)] $\m^{N+1}$ extends $\m^N$ for all $N\ge 0$, that
is, the truncation $(\Q\X,{\cal G},\m^{N+1})_N$ of $(\Q\X,{\cal
G},\m^{N+1})$ to an $A_{N,0}$ algebra is $(\Q\X,{\cal G},\m^N)$;
\item[(iii)] The truncation $(\g^{N+1})_N:(\Q\X,{\cal
G},\m^{N+1})_N\ra(\Q\X_{N+1},{\cal G},\m_{N+1})_N$ of $g^{N+1}$ to
an $A_{N,0}$ morphism satisfies $(\g^{N+1})_N=\f^N\ci\g^N$ for all
$N\ge 0$, using $(\Q\X,{\cal G},\m^{N+1})_N=(\Q\X,{\cal G},\m^N)$
from~(ii).
\end{itemize}
For the first step, $\m^0,\g^0$ are given in (i). For the inductive
step, suppose we have constructed $\m^N,\g^N$ satisfying (i)--(iii)
for $N=0,1,\ldots,P$. Then $\f^P\ci\g^P:(\Q\X,{\cal
G},\m^P)\ra(\Q\X_{P+1},{\cal G},\m_{P+1})_P$ is an $A_{P,0}$
morphism which is a weak homotopy equivalence, since $\f^P,\g^P$
are. Theorem \ref{il3thm6}(a) with $N=P$, $\bar N=P+1$ now shows
that there exists an $A_{P+1,0}$ algebra $(\Q\X,{\cal G},\m^{P+1})$
extending $(\Q\X,{\cal G},\m^P)$ and an $A_{P+1,0}$ morphism
$\g^{P+1}:(\Q\X,{\cal G},\m^{P+1})\ra(\Q\X_{P+1},{\cal G},\m_{P+1})$
extending $\f^P\ci\g^P$ which is a weak homotopy equivalence. This
proves the inductive step.

For all $k\ge 0$ and $(\la,\mu)\in{\cal G}$, define
$\m_k^{\smash{\la,\mu}}:{\buildrel {\ulcorner\,\,\,\text{$k$ copies
} \,\,\,\urcorner} \over {\vphantom{m}\smash{\Q\X\t\cdots\t
\Q\X}}}\ra\Q\X$ by $\m_k^{\smash{\la,\mu}}=
\m_k^{\smash{N,\la,\mu}}$, where $N=\max(\nm{(\la,\mu)}+k-1,0)$ and
$\m_k^{\smash{N,\la,\mu}}$ is the $(k,\la,\mu)$ term in $\m^N$. Then
(ii) implies that $\m_k^{\smash{\la,\mu}}=
\m_k^{\smash{N',\la,\mu}}$ for any $N'\ge N$. Since $(\Q\X,{\cal
G},\m^N)$ is an $A_{N,0}$ algebra for all $N\ge 0$, equation
\eq{il3eq10} holds for the $\m_k^{\smash{N,\la,\mu}}$, so by
independence of $N$, the $\m_k^{\smash{\la,\mu}}$ satisfy
\eq{il3eq10} for all $k\ge 0$, $(\la,\mu)\in{\cal G}$ and pure
$a_1,\ldots,a_k\in\Q\X$.

Define $\Q$-multilinear maps $\m_k:{\buildrel
{\ulcorner\,\,\,\text{$k$ copies } \,\,\,\urcorner} \over
{\vphantom{m}\smash{(\Q\X\ot\La^0_\nov)\t\cdots\t(\Q\X\ot
\La^0_\nov)}}}\ra\Q\X\ot\La^0_\nov$ for $k=0,1,\ldots$ by
$\smash{\m_k=\sum_{ (\la,\mu)\in{\cal G}}T^\la
e^\mu\m_k^{\smash{\la,\mu}}}$. Write $\m=(\m_k)_{k\ge 0}$. Then
Definition \ref{il3dfn10} implies that $(\Q\X\ot\La^0_\nov,\m)$ is a
{\it gapped filtered\/ $A_\iy$ algebra}.
\label{il11dfn}
\end{dfn}

Definition \ref{il11dfn} is similar to Fukaya et al.\ \cite[\S
30.8]{FOOO}. Here is one of our main results, analogous
to~\cite[Th.s 10.11, 14.1 \& 14.2]{FOOO}.

\begin{thm}{\bf(a)} In Definition {\rm\ref{il11dfn},}
$(\Q\X\ot\La^0_\nov,\ab\m)$ depends up to canonical homotopy
equivalence only on $(M,\om),$ $\io:L\ra M$ and its relative spin
structure, and the indices $\eta_{(p_-,p_+)}$ in {\rm\S\ref{il43},}
and is independent of\/ $J,{\cal G},$ changes of paths
$\la_{(p_-,p_+)}$ in {\rm\S\ref{il43}} which fix $\eta_{(p_-,p_+)},$
the orientations $o_{(p_-,p_+)}$ on $\Ker\bar\pd_{\la_{(p_-,p_+)}}$
in {\rm\S\ref{il52},} and all other choices.

That is, if\/ $(\Q\X\ot\La^0_\nov,\m),(\Q\ti\X\ot\La^0_\nov,\ti\m)$
are outcomes in Definition {\rm\ref{il11dfn}} depending on $J,{\cal
G},\la_{(p_-,p_+)}, o_{(p_-,p_+)},\ldots$ and\/ $\ti J,\ti{\cal
G},\ab\ti \la_{(p_-,p_+)},\ab\ti o_{(p_-,p_+)},\ldots,$ we can
construct a gapped filtered\/ $A_\iy$ morphism ${\mathfrak
j}:(\Q\X\ot\La^0_\nov,\m)\!\ra\! (\Q\ti\X\ot\La^0_\nov,\ti\m)$ which
is a homotopy equivalence. If\/ ${\mathfrak j},{\mathfrak j}'$ are
possibilities for ${\mathfrak j}$ there is a homotopy~$\H:{\mathfrak
j}\!\Ra\!{\mathfrak j}'$.
\smallskip

\noindent{\bf(b)} If\/ $(\Q\X\ot\La^0_\nov,\m),(\Q\ti\X\ot
\La^0_\nov,\ti\m),(\Q\check\X\ot \La^0_\nov,\check\m)$ are possible
outcomes in Definition {\rm\ref{il11dfn}} and\/ ${\mathfrak
j}:(\Q\X\ot\La^0_\nov,\m)\ra (\Q\ti\X\ot\La^0_\nov,\ti\m),$
${\mathfrak j}':(\Q\ti\X \ot\La^0_\nov,\ti\m)\ra(\Q\check\X
\ot\ab\La^0_\nov,\ab\check\m),$ ${\mathfrak j}'':(\Q\X\ot
\La^0_\nov,\m)\ra(\Q\check\X\ab\ot\La^0_\nov,\ab\check\m)$ are
corresponding gapped filtered\/ $A_\iy$ morphisms in part\/
{\rm(a),} then there is a homotopy~$\H:{\mathfrak j}''\Ra{\mathfrak
j}'\ci{\mathfrak j}$.
\label{il11thm}
\end{thm}

\begin{proof} First we explain how to construct $\mathfrak
j$ in (a) when $\cal G,\la_{(p_-,p_+)},o_{(p_-,p_+)}$ are fixed, but
other choices $J,\ldots$ vary. Suppose $(\Q\X\ot\La^0_\nov, \m)$,
$(\Q\ti\X\ot\La^0_\nov,\ti\m)$ are constructed using ${\cal
G},\la_{(p_-,p_+)},o_{(p_-,p_+)}$ and other choices $J,\ldots$ and
$\ti J,\ldots$. Let $(\Q\X_N,\ab{\cal
G},\ab\m_N),\ab\f^N,\ab(\Q\X,{\cal G},\m^N),\g^N$ and
$(\Q\ti\X_N,{\cal G},\ti\m_N),\ti\f^N, (\Q\ti\X,\ab{\cal
G},\ab\ti\m^N),\ab \ti\g^N$, be the corresponding choices in
Definition \ref{il11dfn}. Let $\ti{\mathfrak d}^N,\ti{\mathfrak
e}^N$ be homotopy inverses for~$\ti\f^N,\ti\g^N$.

Let $J_t:t\in[0,1]$ be a smooth family of almost complex structures
on $M$ compatible with $\om$, with $J_0=J$ and $J_1=\ti J$. Suppose
that $\cal G$ satisfies conditions (i),(ii) of \S\ref{il8} for
$J_t:t\in[0,1]$; this implies that $\cal G$ also satisfies
conditions (i),(ii) of \S\ref{il6} for $J,\ti J$. If $\cal G$ does
not satisfy (i),(ii), we can use the third part of the proof to
change to a new $\cal G$ which does. Then Corollary \ref{il9cor}
constructs an $A_{N,0}$ morphism $\f^{\sst 01}$ that we will write
as $\h^N:(\Q\X_N,{\cal G},\m_N)\ra(\Q\ti\X_N,{\cal G},\ti\m_N)$,
which is a homotopy equivalence. Also, as $(\Q\ti\X_{N+1},{\cal
G},\ti\m_{N+1})_N$ is also a possible $A_{N,0}$ algebra from Theorem
\ref{il7thm} with $\ti J$, Corollary \ref{il9cor} constructs an
$A_{N,0}$ morphism ${\mathfrak i}^N:(\Q\X_N,{\cal
G},\m_N)\ra(\Q\ti\X_{N+1}, {\cal G},\ti\m_{N+1})_N$, which is a
homotopy equivalence. Thus we obtain the diagram of $A_{N,0}$
morphism homotopy equivalences:
\begin{equation}
\begin{gathered}
\xymatrix@R=20pt{ (\Q\X,{\cal G},\m^N) \ar[r]_{\g^N}
\ar@/^1.5pc/[rr]_{(\g^{N+1})_N} \ar@{{}-->}[d]^(0.4){{\mathfrak
j}^N} & (\Q\X_N,{\cal G},\m_N) \ar[r]_{\f^N} \ar[d]_(0.4){{\mathfrak
h}^N} \ar[dr]^(0.6){{\mathfrak i}^N} &
(\Q\X_{N+1},{\cal G},\m_{N+1})_N \ar[d]^(0.4){({\mathfrak h}^{N+1})_N} \\
(\Q\ti\X,{\cal G},\ti\m^N) \ar@<.5pc>[r]^{\ti\g^N}
\ar@<-1ex>@/_1.2pc/[rr]^{(\ti\g^{N+1})_N} & (\Q\ti\X_N,{\cal
G},\ti\m_N) \ar@<.5pc>[r]^(0.4){\ti\f^N} \ar[l]^{\ti{\mathfrak e}^N}
& (\Q\ti\X_{N+1},{\cal G},\ti\m_{N+1})_N \ar[l]^{\ti{\mathfrak d}^N}
\ar@<2ex>@/^1.2pc/[ll]^{(\ti{\mathfrak e}^{N+1})_N}}
\end{gathered}
\label{il11eq1}
\end{equation}

Write $\f\sim\g$ when two $A_{N,0}$ morphisms are homotopic. Then we
have
\begin{equation}
\begin{gathered}
\ti{\mathfrak e}^N\ci{\mathfrak h}^N\ci \g^N\sim \ti{\mathfrak
e}^N\ci\ti{\mathfrak d}^N\ci\ti\f^N \ci{\mathfrak h}^N\ci \g^N\sim
\ti{\mathfrak e}^N\ci\ti{\mathfrak d}^N\ci{\mathfrak i}^N\ci\g^N\sim\\
\ti{\mathfrak e}^N\ci\ti{\mathfrak d}^N\ci({\mathfrak h}^{N+1})_N\ci
\f^N\ci\g^N\sim(\ti{\mathfrak e}^{N+1})_N\ci({\mathfrak
h}^{N+1})_N\ci(\g^{N+1})_N\\
=(\ti{\mathfrak e}^{N+1}\ci{\mathfrak h}^{N+1}\ci\g^{N+1})_N,
\end{gathered}
\label{il11eq2}
\end{equation}
using $\ti{\mathfrak d}^N,\ti\f^N$ homotopy inverses in the first
step, $\ti\f^N\ci{\mathfrak h}^N\sim{\mathfrak i}^N$ by Theorem
\ref{il10thm4} in the second, $({\mathfrak h}^{N+1})_N\ci
\f^N\sim{\mathfrak i}^N$ by Theorem \ref{il10thm4} in the third, and
$(\g^{N+1})_N=\f^N\ci\g^N$ and $\ti{\mathfrak e}^N\ci\ti{\mathfrak
d}^N\sim(\ti{\mathfrak e}^{N+1})_N$ which follows from
$(\ti\g^{N+1})_N=\ti\f^N\ci\ti\g^N$ and $\ti{\mathfrak
d}^N,\ti{\mathfrak e}^N,\ti{\mathfrak e}^{N+1}$ homotopy inverses
for $\f^N,\g^N,\g^{N+1}$ in the fourth.

By induction on $N=0,1,2,\ldots$ we now choose $A_{N,0}$ morphisms
${\mathfrak j}^N:(\Q\X,{\cal G},\m^N)\ab\ra (\Q\ti\X,{\cal
G},\ti\m^N)$ which are homotopy equivalences, satisfying the
conditions:
\begin{itemize}
\setlength{\itemsep}{0pt}
\setlength{\parsep}{0pt}
\item[(i)] ${\mathfrak j}^N$ is homotopic to
$\ti{\mathfrak e}^N\ci{\mathfrak h}^N\ci\g^N$; and
\item[(ii)] The truncation $({\mathfrak j}^{N+1})_N:(\Q\X,{\cal
G},\m^{N+1})_N\ra(\Q\ti\X,{\cal G},\ti\m^{N+1})_N$ of ${\mathfrak
j}^{N+1}$ to an $A_{N,0}$ morphism satisfies $({\mathfrak
j}^{N+1})_N={\mathfrak j}^N$ for all $N\ge 0$, using $(\Q\X,{\cal
G},\ab\m^{N+1})_N=(\Q\X,{\cal G},\m^N)$ and $(\Q\ti\X,{\cal
G},\ti\m^{N+1})_N=(\Q\ti\X,{\cal G},\ti\m^N)$.
\end{itemize}
For the first step, we take ${\mathfrak j}^0=\ti{\mathfrak
e}^N\ci{\mathfrak h}^N\ci\g^N$, so that (i) for $N=0$ is trivial.
For the inductive step, suppose we have chosen ${\mathfrak j}^N$
satisfying (i),(ii) for $N=0,1,\ldots,P$. We shall construct
${\mathfrak j}^{P+1}$. Since ${\mathfrak j}^P$ is homotopic to
$\ti{\mathfrak e}^P\ci{\mathfrak h}^P\ci\g^P$ by (i), and
$\ti{\mathfrak e}^P\ci{\mathfrak h}^P\ci\g^P$ is homotopic to
$(\ti{\mathfrak e}^{P+1}\ci{\mathfrak h}^{P+1}\ci\g^{P+1})_P$ by
\eq{il11eq2}, ${\mathfrak j}^P$ is homotopic to $(\ti{\mathfrak
e}^{P+1}\ci{\mathfrak h}^{P+1}\ci\g^{P+1})_P$. So Theorem
\ref{il3thm6}(b) with $N=P$, $\bar N=P+1$, $\f={\mathfrak j}^P$ and
$\bar\g=\ti{\mathfrak e}^{P+1}\ci{\mathfrak h}^{P+1}\ci\g^{P+1}$
gives ${\mathfrak j}^{P+1}$ satisfying (i),(ii). Therefore by
induction ${\mathfrak j}^N$ exists for all~$N$.

There is now a unique gapped filtered $A_\iy$ morphism ${\mathfrak
j}:(\Q\X\ot\La^0_\nov,\m)\ra(\Q\ti\X\ot\La^0_\nov,\ti\m)$ whose
truncation to $A_{N,0}$ algebras is ${\mathfrak j}^N$ for
$N=0,1,2,\ldots$. It is a weak homotopy equivalence as the
${\mathfrak j}^N$ are, and so is a homotopy equivalence by Theorem
\ref{il3thm3}(c). Regarding $\g^N,\ti\g^N$ as fixed, $\ti{\mathfrak
e}^N$ above is independent of choices up to homotopy, and by Theorem
\ref{il10thm2}, so is ${\mathfrak h}^N$. Thus, ${\mathfrak j}^N$ is
independent of choices up to $A_{N,0}$ homotopy. As this holds for
all $N$, ${\mathfrak j}$ is independent of choices up to homotopy.
That is, if ${\mathfrak j},{\mathfrak j}'$ are possible choices for
${\mathfrak j}$ then there is a homotopy $\H:{\mathfrak
j}\ra{\mathfrak j}'$. We construct $\H$ as the union of a family of
$A_{N,0}$ homotopies $\H^N:{\mathfrak j}^N\Ra{\mathfrak j}^{\prime
N}$ with $(\H^{N+1})_N=\H_N$, chosen using an analogue of Theorem
\ref{il3thm6}(b) for homotopies. This proves (a) with $\cal G$ and
$\la_{(p_-,p_+)},o_{(p_-,p_+)}$ for $(p_-,p_+)\in R$ fixed.

Secondly, we prove (b) with $\cal G,\la_{(p_-,p_+)},o_{(p_-,p_+)}$
fixed. Suppose $(\Q\X\ot\ab\La^0_\nov,\ab\m)$,
$(\Q\ti\X\ot\La^0_\nov,\ab\ti\m),\ab(\Q\check\X\ot
\La^0_\nov,\check\m)$ and ${\mathfrak j},{\mathfrak j}',{\mathfrak
j}''$ are as in (b), all constructed using the same $\cal
G,\la_{(p_-,p_+)},o_{(p_-,p_+)}$. Then with the obvious notation we
have a diagram of $A_{N,0}$ morphism homotopy equivalences:
\begin{equation*}
\xymatrix@R=10pt{ (\Q\X,{\cal G},\m^N) \ar@<.2pc>[r]_{\g^N}
\ar[d]^(0.4){{}\,{\mathfrak j}^N} \ar@<-6ex>@/^-1pc/[dd]_{{\mathfrak
j}^{\prime\prime N}} & (\Q\X_N,{\cal G},\m_N)
\ar[d]_(0.4){{\mathfrak h}^N} \ar@<6ex>@/^1pc/[dd]^{{\mathfrak
h}^{\prime\prime N}} \\
(\Q\ti\X,{\cal G},\ti\m^N) \ar@<.3pc>[r]^{\ti\g^N}
\ar[d]^(0.4){{}\,{\mathfrak j}^{\prime N}} & (\Q\ti\X_N,{\cal
G},\ti\m_N) \ar[l]^{\ti{\mathfrak e}^N}
\ar[d]_(0.4){{\mathfrak h}^{\prime N}}\\
(\Q\check\X,{\cal G},\check\m^N) \ar@<.3pc>[r]^{\check\g^N} &
(\Q\check\X_N,{\cal G},\check\m_N) \ar[l]^{\check{\mathfrak e}^N} }
\end{equation*}
Theorem \ref{il10thm4} shows that ${\mathfrak h}^{\prime\prime
N}\sim{\mathfrak h}^{\prime N}\ci{\mathfrak h}^N$. Since ${\mathfrak
j}^N\sim\ti{\mathfrak e}^N\ci{\mathfrak h}^N\ci\g^N$, ${\mathfrak
j}^{\prime N}\sim\check{\mathfrak e}^N\ci{\mathfrak h}^{\prime
N}\ci\ti\g^N$ and $\ti\g^N,\ti{\mathfrak e}^N$ are homotopy
inverses, this implies that ${\mathfrak j}^{\prime\prime
N}\sim{\mathfrak j}^{\prime N}\ci{\mathfrak j}^N$. That is, the
$A_{N,0}$ truncations of ${\mathfrak j}''$ and ${\mathfrak
j}'\ci{\mathfrak j}$ are $A_{N,0}$ homotopic for all $N=0,1,\ldots$.
We can now construct $\H:{\mathfrak j}''\Ra{\mathfrak
j}'\ci{\mathfrak j}$ as in the end of the first part of the proof.

Thirdly, we explain how to change $\cal G$ in (a) and (b). Suppose
that ${\cal G}\subseteq\ti{\cal G}\subset[0,\iy)\t\Z$, and ${\cal
G},\ti{\cal G}$ are closed under addition, such that ${\cal
G}\cap(\{0\}\t\Z)=\ti{\cal G}\cap(\{0\}\t\Z)=\{(0,0)\}$ and ${\cal
G}\cap([0,C]\t\Z),\ti{\cal G}\cap([0,C]\t\Z)$ are finite for any
$C\ge 0$. We shall define a functor from the 2-category of $A_{N,0}$
algebras with fixed $\cal G$ to the 2-category of $A_{N,0}$ algebras
with fixed $\ti{\cal G}$, which we call $\ti{\cal G}$-{\it
truncation}.

If $(\la,\mu)\in{\cal G}$ then as ${\cal G}\subseteq\ti{\cal G}$, in
\eq{il3eq16} we can define $\nm{(\la,\mu)}$ using either $\cal G$ or
$\ti{\cal G}$. Write these as $\nm{(\la,\mu)}_{\cal
G},\nm{(\la,\mu)}_{\ti {\cal G}}$ to distinguish them. Then ${\cal
G}\subseteq\ti{\cal G}$ implies that $\nm{(\la,\mu)}_{\cal
G}\le\nm{(\la,\mu)}_{\ti {\cal G}}$, as $(\la,\mu)$ can be split
into more pieces in $\ti{\cal G}$ than in $\cal G$. Thus for $k,N$
given, $\nm{(\la,\mu)}_{\ti{\cal G}}+k-1\le N$ implies
that~$\nm{(\la,\mu)}_{\cal G}+k-1\le N$.

Suppose $(A,{\cal G},\m)$ is an $A_{N,0}$ algebra, so that
$\m=\bigl(\m^{\smash{\la,\mu}}_k:k\ge 0$, $(\la,\mu)\in{\cal G}$,
$\nm{(\la,\mu)}_{\cal G}+k-1\le N\bigr)$. Define an $A_{N,0}$
algebra $(A,\ti{\cal G},\ti\m)$, where
$\ti\m=\bigl(\ti\m^{\smash{\la,\mu}}_k:k\ge 0$,
$(\la,\mu)\in{\ti{\cal G}}$, $\nm{(\la,\mu)}_{\ti{\cal G}}+k-1\le
N\bigr)$ by $\ti\m^{\smash{\la,\mu}}_k=0$ if $(\la,\mu)\in\ti{\cal
G}\sm{\cal G}$, and $\ti\m^{\smash{\la,\mu}}_k=
\m^{\smash{\la,\mu}}_k$ if $(\la,\mu)\in{\cal G}$. Since
$(\la,\mu)\in{\cal G}$ and $\nm{(\la,\mu)}_{\ti{\cal G}}+k-1\le N$
implies that $\nm{(\la,\mu)}_{\cal G}+k-1\le N$, this is
well-defined, and \eq{il3eq10} holds for the
$\ti\m^{\smash{\la,\mu}}_k$ as it does for the
$\m^{\smash{\la,\mu}}_k$. So $(A,\ti{\cal G},\ti\m)$ is an $A_{N,0}$
algebra.

Write $(A,{\cal G},\m)_{\ti{\cal G}}=(A,\ti{\cal G},\ti\m)$, that
is, $(A,{\cal G},\m)_{\ti{\cal G}}$ is the $\ti{\cal G}$-{\it
truncation} of $(A,{\cal G},\m)$. In a similar way, if $\f:(A,{\cal
G},\m)\ra(B,{\cal G},\n)$ is an $A_{N,0}$ morphism of $A_{N,0}$
algebras with $\cal G$, then the $\ti{\cal G}$-{\it truncation}
$\f_{\ti{\cal G}}=\ti\f:(A,{\cal G},\m)_{\ti{\cal G}}\ra(B,{\cal
G},\n)_{\ti{\cal G}}$ is an $A_{N,0}$ morphism of $A_{N,0}$ algebras
with $\ti{\cal G}$, where $\ti\f^{\smash{\la,\mu}}_k=0$ if
$(\la,\mu)\in\ti{\cal G}\sm{\cal G}$, and
$\ti\f^{\smash{\la,\mu}}_k=\f^{\smash{\la,\mu}}_k$ if
$(\la,\mu)\in{\cal G}$. If $\H:\f\ra\g$ is a homotopy of $A_{N,0}$
morphisms $\f,\g:(A,{\cal G},\m)\ra(B,{\cal G},\n)$, then the
$\ti{\cal G}$-truncation $\H_{\ti{\cal G}}=\ti\H:\f_{\ti{\cal
G}}\Ra\g_{\ti{\cal G}}$ is a homotopy, where $\ti\H^{\smash{\la,
\mu}}_k=0$ if $(\la,\mu)\in\ti{\cal G}\sm{\cal G}$, and
$\ti\H^{\smash{\la,\mu}}_k=\H^{\smash{\la,\mu}}_k$
if~$(\la,\mu)\in{\cal G}$.

Now suppose that $(\Q\X\ot\La^0_\nov,\m)$ is a gapped filtered
$A_\iy$ algebra constructed in Definition \ref{il11dfn} using data
$J,{\cal G},\ldots$. We shall show how to construct {\it exactly the
same} gapped filtered $A_\iy$ algebra using $\ti{\cal G}$ instead of
$\cal G$. Use all the notation $(\Q\X_N,{\cal G},\m_N),\f^N,\g^N,
\m^N,\ldots$ of Definition \ref{il11dfn}. Then it is easy to see
that we may go through Definition \ref{il11dfn} replacing $\cal G$
by $\ti{\cal G}$, and all the $A_{N,0}$ algebras, morphisms and
homotopies by their $\ti{\cal G}$-truncations, and get a valid set
of choices. That is, we replace $(\Q\X_N,{\cal G},\m_N)$ by
$(\Q\X_N,\ti{\cal G},\ti\m_N)=(\Q\X_N,{\cal G},\m_N)_{\ti{\cal G}}$,
$\f^N,\g^N$ by $\ti\f^N=(\f^N)_{\ti{\cal G}},\ti\g^N=
(\g^N)_{\ti{\cal G}}$, and $(\Q\X,{\cal G},\m^N)$ by $(\Q\X,{\cal
G},\ti\m^N)=(\Q\X,{\cal G},\m^N)_{\ti{\cal G}}$.

Since $\ti{\cal G}$-truncation commutes with truncation of
$A_{N+1,0}$ algebras to $A_{N,0}$ algebras, these satisfy
$(\ti\g^{N+1})_N=\ti\f^N\ci\ti\g^N$, and so on. Thus, we obtain a
gapped filtered $A_\iy$ algebra $(\Q\X\ot\La^0_\nov,\ti\m)$ using
$\ti{\cal G}$, whose truncation to an $A_{N,0}$ algebra with
$\ti{\cal G}$ is $(\Q\X,{\cal G},\ti\m^N)=(\Q\X,{\cal
G},\m^N)_{\ti{\cal G}}$ for all $N=0,1,\ldots$. Clearly this implies
that $\ti\m=\m$, and $(\Q\X\ot\La^0_\nov,\ti\m)=(\Q\X
\ot\La^0_\nov,\m)$. So we are always free to enlarge $\cal G$ to
$\ti{\cal G}$, and obtain not just two homotopic, but the same,
gapped filtered $A_\iy$ algebras.

To extend the proofs of the first two parts to allow $\cal G$ to
vary, suppose in (a) that $(\Q\X\ot\La^0_\nov,\m)$,
$(\Q\ti\X\ot\La^0_\nov,\ti\m)$ are constructed using $J,{\cal G},\ab
\la_{(p_-,p_+)},\ab o_{(p_-,p_+)},\ab\ldots$ and $\ti J,\ti{\cal
G},\ab\la_{(p_-,p_+)},\ab o_{(p_-,p_+)},\ab\ldots$, with possibly
different ${\cal G},\ti{\cal G}$. Choose a smooth 1-parameter family
of almost complex structures $J_t:t\in[0,1]$ on $M$ compatible with
$\om$, with $J_0=J$ and $J_1=\ti J$. Choose some $\check{\cal
G}\subset[0,\iy)\t\Z$ such that ${\cal G}\subseteq\check{\cal G}$,
and $\ti{\cal G}\subseteq\check{\cal G}$, and conditions (i),(ii) of
\S\ref{il8} hold for $\check{\cal G}$ and $J_t:t\in[0,1]$. This is
possible, and there is a unique smallest such~$\check{\cal G}$.

Now regard $(\Q\X\ot\La^0_\nov,\m)$, $(\Q\ti\X\ot\La^0_\nov, \ti\m)$
as having been constructed using $\check{\cal G}$ rather than ${\cal
G},\ti{\cal G}$, as above. Then we can use the first part of the
proof with $\check{\cal G}$ in place of $\cal G$ to construct
${\mathfrak j}:(\Q\X\ot\La^0_\nov,\m)\ra(\Q\ti\X
\ot\La^0_\nov,\ti\m)$ and prove (a). The extension of (b) to varying
$\cal G$ is similar; we must choose $\check{\cal G}$ to contain
${\cal G},{\cal G}',{\cal G}''$, and the choices of $\cal G$ used to
define ${\mathfrak j},{\mathfrak j}',{\mathfrak j}''$, and to be
compatible with the family of almost complex structures $J_t:t\in T$
used in \S\ref{il102} to construct homotopies.

Finally we explain how to change the paths $\la_{(p_-,p_+)}$ and
orientations $o_{(p_-,p_+)}$ on $\Ker\bar\pd_{\smash{
\la_{(p_-,p_+)}}}$ for $(p_-,p_+)\in R$. Let $\ti\la_{(p_-,p_+)},\ti
o_{(p_-,p_+)}$ be an alternative set of choices, which yield the
same indices $\eta_{(p_-,p_+)}$. Then Proposition \ref{il5prop}
shows how the orientation of $\oM_{k+1}^\ma(\al,\be,J,f_1,
\ldots,f_k)$ changes for these new choices, in terms of
$\xi_{(p_-,p_+)}=\pm 1$ for $(p_-,p_+)\in R$. Let
$(\Q\X\ot\La^0_\nov,\m)$ be constructed in Definition \ref{il11dfn}
using the $\la_{(p_-,p_+)}, o_{(p_-,p_+)}$, and
$(\Q\X\ot\La^0_\nov,\ti\m)$ be constructed using
$\ti\la_{(p_-,p_+)},\ti o_{(p_-,p_+)}$, but otherwise using {\it
exactly the same choices}. That is, the chain complexes
$\Q\X_N,\Q\X$ and choices of perturbation data are unchanged, but
the other data of virtual chains, $\m_N,\f^N,\g^N,\m^N,\m,\ldots$
change to $\ti\m_N,\ti\f^N,\ti\g^N,\ti\m^N,\ti\m,\ldots$ with
various sign changes depending on the $\xi_{(p_-,p_+)}$.

But $\oM_{k+1}^\ma(\al,\be,J,f_1,\ldots,f_k)=\emptyset$ unless
$f_i:\De_{a_i}\ra L\amalg R$ maps to $\al(i)\in R$ if $i\in I$ and
to $L$ if $i\notin I$, and $\bev:\oM_{k+1}^\ma(\al,\be,J,f_1,\ldots,
f_k)\ra L\amalg R$ maps to $\si\ci\al(0)$ if $0\in I$ and to $L$ if
$0\notin I$. Because of this, if we define linear
$\Xi:\Q\X_i\ra\Q\X_i$ by
\begin{equation*}
\Xi(f)=\begin{cases} \xi_{\si(p_-,p_+)}f, &
f:\De_a\ra\{(p_-,p_+)\}\subset R, \\
f, & f:\De_a\ra L, \end{cases}
\end{equation*}
then in Definition \ref{il7dfn1} we have
$\ti\m_{k,\geo}^{\la,\mu}\bigl(\Xi(f_1),\ldots,\Xi(f_k)\bigr)=
\Xi\ci\m_{k,\geo}^{\la,\mu}(f_1,\ldots,f_k)$, as
$\m_{k,\geo}^{\la,\mu},\ti\m_{k,\geo}^{\la,\mu}$ are constructed
from virtual chains for $\oM_{k+1}^\ma(\al,\be,J,f_1,\ldots,f_k)$,
which change signs as in Proposition~\ref{il5prop}.

Going through the constructions of \S\ref{il7}--\S\ref{il10} and
Definition \ref{il11dfn}, we find that everything commutes with
$\Xi$ in this way, so that eventually $(\Q\X\ot\La^0_\nov,\m)$ and
$(\Q\X\ot\La^0_\nov,\ti\m)$ satisfy
$\ti\m_k\bigl(\hat\Xi(f_1),\ldots,\hat\Xi(f_k)\bigr)=
\hat\Xi\ci\m_k(f_1,\ldots,f_k)$, where $\hat\Xi:\Q\X\ot\ab
\La^0_\nov\ab\ra\Q\X\ot\La^0_\nov$ is the $\La_\nov^0$-linear map
induced by $\Xi:\Q\X\ra\Q\X$. Thus $\hat\Xi$ induces a strict
$A_\iy$ isomorphism $\bs\Xi:(\Q\X\ot\La^0_\nov,\m)\ra
(\Q\X\ot\La^0_\nov,\ti\m)$. To include change of
$\la_{(p_-,p_+)},o_{(p_-,p_+)}$ in (a), we compose $\mathfrak j$
constructed above for fixed $o_{(p_-,p_+)}$ with this $\bs\Xi$ to
get the new $\mathfrak j$. The same idea works for~(b).
\end{proof}

\begin{rem} In Theorem \ref{il11thm}(a), it is {\it nearly} true
that $(\Q\X\ot\La^0_\nov,\ab\m)$ is also independent of the indices
$\eta_{(p_-,p_+)}$ in \S\ref{il43} up to canonical homotopy
equivalence. This would be true if we relaxed the definition of
gapped filtered $A_\iy$ morphism in Definition \ref{il3dfn11}
slightly. For $(p_-,p_+)\in R$, let $\la_{(p_-,p_+)},
\ti\la_{(p_-,p_+)}$ be possible choices in \S\ref{il43}, let
$\eta_{(p_-,p_+)},\ab \ti\eta_{(p_-,p_+)}$ be the corresponding
indices \eq{il4eq8}, and let $o_{(p_-,p_+)},\ti o_{(p_-,p_+)}$ be
orientations on $\Ker\bar\pd_{\smash{\la_{(p_-,p_+)}}},\Ker
\bar\pd_{\smash{\ti\la_{(p_-,p_+)}}}$. As at the end of
\S\ref{il54}, we have $\ti\eta_{(p_-,p_+)}=\eta_{(p_-,p_+)}+
2d_{(p_-,p_+)}$ for $d_{(p_-,p_+)}\in\Z$.

We can now try to adapt the final part of the proof of Theorem
\ref{il11thm}, as follows. Suppose $(\Q\X\ot\La^0_\nov,\m)$ is
constructed in Definition \ref{il11dfn} using
$\la_{(p_-,p_+)},\eta_{(p_-, p_+)},\ab o_{(p_-,p_+)}$, and
$(\Q\X\ot\La^0_\nov,\ti\m)$ is constructed using
$\ti\la_{(p_-,p_+)},\ti\eta_{(p_-, p_+)},\ti o_{(p_-,p_+)}$, but
otherwise using {\it exactly the same choices}. When we change from
$\la_{(p_-,p_+)},\eta_{(p_-, p_+)},\ab o_{(p_-,p_+)}$ to
$\ti\la_{(p_-,p_+)},\ti\eta_{(p_-, p_+)},\ti o_{(p_-,p_+)}$, the
orientations of $\oM_{k+1}^\ma(\al,\be,J,f_1,\ldots,f_k)$ change as
in Proposition \ref{il5prop}, in terms of $\xi_{(p_-,p_+)}=\pm 1$
for $(p_-,p_+)\in R$, and $\deg f$ in \eq{il4eq13} changes by $\deg
f\mapsto\deg f+2d_{(p_-,p_+)}$ if $f:\De_a\ra\{(p_-,p_+)\}$. Define
a $\La^0_\nov$-linear map $\hat\Xi:\Q\X\ot\La^0_\nov\ra\Q\X
\ot\La^0_\nov$ by
\begin{equation*}
\hat\Xi(f)=\begin{cases} e^{-d_{(p_-,p_+)}}\xi_{\si(p_-,p_+)}f, &
f:\De_a\ra\{(p_-,p_+)\}\subset R, \\
f, & f:\De_a\ra L, \end{cases}
\end{equation*}
where $e$ is the formal variable in $\La^0_\nov$ from \S\ref{il34}.

Note that $\Q\X\ot\La^0_\nov$ is graded differently in
$(\Q\X\ot\La^0_\nov,\m)$ and $(\Q\X\ot\La^0_\nov,\ti\m)$, because of
the change in $\deg f$. Since $e$ has degree 2, the correction
$e^{-d_{(p_-,p_+)}}$ ensures that $\hat\Xi$ is graded of degree 0 as
a map $(\Q\X\ot\La^0_\nov,\m)\ra (\Q\X\ot\La^0_\nov,\ti\m)$. As in
the final part of the proof of Theorem \ref{il11thm}, we find that
$\ti\m_k\bigl(\hat\Xi(f_1),
\ldots,\hat\Xi(f_k)\bigr)=\hat\Xi\ci\m_k(f_1,\ldots,f_k)$ for all
$f_1,\ldots,f_k\in\Q\X\ot\La^0_\nov$.

We would like to define a strict gapped filtered $A_\iy$ isomorphism
$\bs\Xi:(\Q\X\ot\La^0_\nov,\ab\m) \ab\ra(\Q\X\ot\La^0_\nov,\ti\m)$
by $\bs\Xi_1= \hat\Xi$ and $\bs\Xi_k=0$ for $k\ne 1$, which would
prove that $(\Q\X\ot\La^0_\nov,\ab\m)$ is independent of
$\la_{(p_-,p_+)},\eta_{(p_-,p_+)}$ up to canonical homotopy
equivalence. However, this $\bs\Xi$ has
$\bs\Xi^{0,-d_{\smash{(p_-,p_+)}}}_{\smash{1}}\ne 0$ for all
$(p_-,p_+)\in R$, which contradicts the conditions on ${\cal G}'$ in
Definition \ref{il3dfn11}(i) if $d_{(p_-,p_+)}\ne 0$. We could
weaken Definition \ref{il3dfn11}(i) to make $\bs\Xi$ a gapped
filtered $A_\iy$ morphism, but this would cause problems elsewhere,
in particular, the definition of weak homotopy equivalence would no
longer make sense.
\label{il11rem}
\end{rem}

By Theorem \ref{il3thm4}, the gapped filtered $A_\iy$ algebra
$(\Q\X\ot\La^0_\nov,\m)$ of Definition \ref{il11dfn} admits a {\it
minimal model\/} $({\cal H}\ot\La^0_\nov,\n)$ with ${\cal H}\cong
H^*(\Q\X,\m_1^{\smash{0,0}})$. Here Theorem \ref{il6thm}$(N1)$(b)
implies that $H^*(\Q\X,\m_1^{\smash{0,0}})\cong H^\rsi_*(L\amalg
R;\Q)$ as an ungraded vector space, and the grading is given by
shifted cohomological degree in \eq{il4eq13}. As
$(\Q\X\ot\La^0_\nov, \m)$ is unique up to canonical homotopy
equivalence by Theorem \ref{il11thm}, $({\cal H}\ot\La^0_\nov,\n)$
is unique up to canonical gapped filtered $A_\iy$ isomorphism. Thus
we deduce:

\begin{cor} The gapped filtered\/ $A_\iy$ algebra
$(\Q\X\ot\La^0_\nov,\m)$ of Definition {\rm\ref{il11dfn}} has a
minimal model\/ $({\cal H}\ot\La^0_\nov,\n),$ with graded\/
$\Q$-vector space ${\cal H}\!=\!\smash{\bigop\limits_{d\in\Z}\!{\cal
H}^d}$ given by
\begin{equation}
{\cal H}^d=H_{n-d-1}(L;\Q)\op\bigop\nolimits_{\begin{subarray}{l}
(p_-,p_+)\in R:\\ d=\eta_{(p_-,p_+)}-1\end{subarray}}\Q(p_-,p_+),
\label{il11eq4}
\end{equation}
where $\Q(p_-,p_+)\cong H_0\bigl(\{(p_-,p_+)\};\Q\bigr)$ is the
$\Q$-vector space with basis~$\{(p_-,p_+)\}$.

This $({\cal H}\ot\La^0_\nov,\n)$ depends up to canonical gapped
filtered\/ $A_\iy$ isomorphism only on $(M,\om),$ $\io:L\ra M$ and
its relative spin structure, and the indices $\eta_{(p_-,p_+)},$ and
is otherwise independent of\/ $J,{\cal G},\la_{(p_-,p_+)},o_{(p_-,
p_+)}$ and other choices. That is, if\/ $({\cal
H}\ot\La^0_\nov,\n),$ $({\cal H}\ot\La^0_\nov,\ti\n)$ are two
possible outcomes, we can construct a gapped filtered\/ $A_\iy$
isomorphism ${\mathfrak j}:({\cal H}\ot\La^0_\nov,\n)\ra({\cal
H}\ot\La^0_\nov,\ti\n),$ and\/ $\mathfrak j$ is unique up to
homotopy. Furthermore, ${\mathfrak j}_1^{\smash{0,0}}:{\cal
H}\ra{\cal H}$ is the identity on $H_*(L;\Q),$ and $\pm 1$ on each\/
$(p_-,p_+)$ in~$R$.
\label{il11cor}
\end{cor}

This is similar to Fukaya et al.~\cite[Th.~A, \S 1.2]{FOOO}.

\section{Calabi--Yau manifolds and graded Lagrangian submanifolds}
\label{il12}

We now explain how the material of \S\ref{il4}--\S\ref{il11}
simplifies when $(M,\om)$ is {\it Calabi--Yau}, and the Lagrangian
$L$ is {\it graded}. Floer cohomology of graded Lagrangian
submanifolds in Calabi--Yau manifolds is important because of its
r\^ole in the Homological Mirror Symmetry Conjecture of Kontsevich
\cite{Kont}. For background on Calabi--Yau manifolds, special
Lagrangian submanifolds, and Mirror Symmetry see Joyce \cite{Joyc1},
and for graded Lagrangian submanifolds and Floer cohomology see
Seidel \cite{Seid} and Fukaya \cite[Def.~2.9]{Fuka}. The framework
we give can be generalized in various ways; see Joyce \cite[\S
8.4]{Joyc1} on {\it almost Calabi--Yau manifolds}, and Seidel
\cite{Seid} for a more general notion of grading, expressed in terms
of covering spaces of bundles of Lagrangian Grassmannians.

\begin{dfn} A {\it Calabi--Yau $n$-fold\/} is a quadruple
$(M,J,\om,\Om)$ where $(M,J)$ is a compact $n$-dimensional complex
manifold, $\om$ is the K\"ahler form of a K\"ahler metric $g$ on
$M$, and $\Om$ is a non-vanishing holomorphic $(n,0)$-form on $M$
satisfying
\begin{equation}
\om^n/n!=(-1)^{n(n-1)/2}(i/2)^n\Om\w\bar\Om.
\label{il12eq1}
\end{equation}
This implies that $g$ is Ricci-flat with holonomy group contained in
SU$(n)$. Note that $(M,\om)$ is a compact symplectic manifold, and
$J$ is an (almost) complex structure on $M$ compatible with~$\om$.

If $(M,J)$ is a compact complex manifold with trivial canonical
bundle $K_M$, then by Yau's proof of the Calabi Conjecture, every
K\"ahler class on $M$ contains a unique Ricci-flat K\"ahler metric
$g$, with K\"ahler form $\om$. There exists $\Om$, unique up to
phase change $\Om\mapsto{\rm e}^{i\th}\Om$, such that
$(M,J,\om,\Om)$ is Calabi--Yau. One can construct many examples of
such $(M,J)$ using complex algebraic geometry.

Now let $\io:L\ra M$ be an oriented, immersed Lagrangian. Then
$\io^*(\Om)$ is a complex $n$-form on $L$, and the normalization
\eq{il12eq1} implies that $\md{\io^*(\Om)}\equiv 1$, where
$\md{\,.\,}$ is computed using $\io^*(g)$. Hence $\io^*(\Om)=u
\vol_L$ for some smooth $u:L\ra\U(1)$, where $\vol_L$ is the volume
form on $L$ defined using $\io^*(g)$ and the orientation. We call
$L$ {\it special Lagrangian with phase\/} ${\rm e}^{i\th}$ for
$\th\in[0,2\pi)$ if $u\equiv{\rm e}^{i\th}$.

A {\it grading} on $L$ is a choice of smooth function $\phi:L\ra\R$
such that $u\equiv{\rm e}^{i\phi}$. We call $(L,\phi)$ a {\it graded
Lagrangian submanifold}. If a grading exists it is unique up to
$\phi\mapsto\phi+2\pi k$ for $k\in\Z$, provided $L$ is connected.
Special Lagrangian submanifolds with phase ${\rm e}^{i\th}$ are
automatically graded, with $\phi\equiv\th$ constant. Let $\al\in
H^1(\U(1);\Z)$ be the generator with $\int_{\U(1)}\al=1$. Then
$u^*(\al)\in H^1(L;\Z)$ is called the {\it Maslov class}, and $L$
admits a grading if and only if $u^*(L)=0$ in $H^1(L;\Z)$, that is,
if and only if $L$ is {\it Maslov zero}.
\label{il12dfn}
\end{dfn}

Suppose that $(M,J,\om,\Om)$ is Calabi--Yau and $(L,\phi)$ is an
embedded graded Lagrangian in $M$. Then the {\it Maslov index}
$\mu_L(\be)$ of Definition \ref{il4dfn4} is zero for all $\be\in
H_2(M,L;\Z)$. This is because $\mu_L(\be)=\be\cdot
c_1\bigl(M,\io(L)\bigr)$, where $c_1\bigl(M,\io(L)\bigr)\in
H^2\bigl(M,\io(L);\Z\bigr)$ is the {\it relative first Chern
class\/} for $\om$ on $(M,L)$, and the Calabi--Yau and graded
conditions imply that~$c_1\bigl(M,\io(L)\bigr)=0$.

To extend this to immersed graded Lagrangians, we must require the
paths $\la_{(p_-,p_+)}$ in Definition \ref{il4dfn3} to lift to paths
$(\la_{(p_-,p_+)},\psi_{(p_-,p_+)})$ in {\it graded\/} Lagrangian
subspaces of $T_pM$. That is, $\la_{(p_-,p_+)}=\{\la_{(p_-,p_+)}
(x,y)\}_{(x,y)\in\pd Y}$ is a smooth family of oriented Lagrangian
subspaces of $T_pM$, where $p=\io(p_-)=\io(p_+)$, and
$\psi_{(p_-,p_+)}:\pd Y\ra\R$ is a smooth map, such that
$\Om_p\vert_{\la_{(p_-,p_+)}(x,y)}= {\rm
e}^{i\psi_{(p_-,p_+)}(x,y)}\ab\vol_{\la_{(p_-,p_+)}(x,y)}$ for all
$(x,y)\in\pd Y$, and
\begin{equation*}
\la_{(p_-,p_+)}(x,y)=\begin{cases}
\d\io(T_{p_-}L), & \mbox{if }y=1,\\
\d\io(T_{p_+}L), & \mbox{if }y=-1,
\end{cases}
\quad \psi_{(p_-,p_+)}(x,y)=\begin{cases}
\phi(p_-), & \mbox{if }y=1,\\
\phi(p_+), & \mbox{if }y=-1.
\end{cases}
\end{equation*}
Then the same argument ensures that $\mu_L(\be)=0$ for all~$\be\in
H_2\bigl(M,\io(L);\Z\bigr)$.

Requiring the $\la_{(p_-,p_+)}$ to lift to paths
$(\la_{(p_-,p_+)},\psi_{(p_-,p_+)})$ in graded Lagrangians
determines the index $\eta_{(p_-,p_+)}$ in \eq{il4eq8} uniquely,
independently of choices in $(\la_{(p_-,p_+)},\ab\psi_{(p_-,p_+)})$.
Calculation shows that we can give a simple local formula
for~$\eta_{(p_-,p_+)}$.

\begin{prop} Let\/ $(M,J,\om,\Om)$ be a Calabi--Yau $n$-fold, and\/
$(\io:L\ra M,\phi)$ be an immersed graded Lagrangian submanifold
with only transverse double self-intersections. Suppose $p_-,p_+\in
L$ with\/ $p_-\ne p_+$ and\/ $\io(p_-)=\io(p_+)=p$. Then for any
choice of path\/ $(\la_{(p_-,p_+)},\psi_{(p_-,p_+)})$ in graded
Lagrangian subspaces of\/ $T_pM$ as above, the index
$\eta_{(p_-,p_+)}$ in Definition \ref{il4dfn3} may be computed as
follows.

One can choose holomorphic coordinates $(z^1,\ldots,z^n)$ near $p$
in $M$ in which
\begin{equation}
\begin{split}
&\om\vert_p=\ts\frac{i}{2}(\d z^1\w\d\overline z^1+\cdots+\d
z^n\w\d\overline z^n),\qquad \Om\vert_p=\d
z^1\w\cdots\w\d z^n,\\
&\d\io(T_{p_-}L)=\bigl\{({\rm e}^{i\phi^1_-}x^1,\ldots,{\rm
e}^{i\phi^n_-}x^n):x^1,\ldots,x^n\in\R\bigr\},\quad\text{and\/}\\
&\d\io(T_{p_+}L)=\bigl\{({\rm e}^{i\phi^1_+}x^1,\ldots,{\rm
e}^{i\phi^n_+}x^n):x^1,\ldots,x^n\in\R\bigr\},
\end{split}
\label{il12eq2}
\end{equation}
where $\phi^1_\pm,\ldots,\phi^n_\pm\in\R$ satisfy
$\phi^1_\pm+\cdots+\phi^n_\pm=\phi(p_\pm)$ and\/
$\phi^j_+-\phi^j_-\notin\pi\Z$ for $j=1,\ldots,n$. For $x\in\R,$
write $[x]$ for the greatest integer $m$ with\/ $m\le x$. Then
\begin{equation}
\eta_{(p_-,p_+)}=n+\sum\nolimits_{j=1}^n\Bigl[\frac{\phi_+^{\smash{j}}
-\phi_-^{\smash{j}}}{\pi}\Bigr].
\label{il12eq3}
\end{equation}
\label{il12prop}
\end{prop}

Since $\phi^j_+-\phi^j_-\notin\pi\Z$, we have
$\bigl[\frac{\phi_+^{\smash{j}}
-\phi_-^{\smash{j}}}{\pi}\bigr]+\bigl[\frac{\phi_-^{\smash{j}}
-\phi_+^{\smash{j}}}{\pi}\bigr]=-1$ for $j=1,\ldots,n$. Thus
exchanging $p_-,p_+$ and $\phi^{\smash{j}}_-,\phi^{\smash{j}}_+$ we
see from \eq{il12eq3} that $\eta_{(p_-,p_+)}+\eta_{(p_+,p_-)}=n$, as
in \eq{il4eq9}. Recall that in \S\ref{il46} we assumed that
$\eta_{(p_-,p_+)}\ge 0$ for all $(p_-,p_+)\in R$. This is {\it not
compatible} with requiring $\la_{(p_-,p_+)}$ to lift to graded
Lagrangians, since then $\eta_{(p_-,p_+)}$ is determined by
\eq{il12eq3}, and need not satisfy~$\eta_{(p_-,p_+)}\ge 0$.

In fact we only used $\eta_{(p_-,p_+)}\ge 0$ to define the modified
moduli spaces $\tM_{k+1}^\ma(\al,\ab\be,\ab J),\ab
\tM_{k+1}^\ma(\al,\be,J,f_1,\ldots,f_k)$, which were only for
motivation, and in the orientation calculations of \S\ref{il5}. But
as we explained in \S\ref{il54}, changing the $\eta_{(p_-,p_+)}$
does not affect any of the signs in \S\ref{il5}, as the
$\eta_{(p_-,p_+)}$ change by even numbers, and Proposition
\ref{il5prop} explains how changing
$\la_{(p_-,p_+)},\eta_{(p_-,p_+)}, o_{(p_-,p_+)}$ affects the
orientations on $\oM_{k+1}^\ma(\al,\be,J,f_1,\ldots,f_k)$. Using
this, we can define the orientations on $\oM_{k+1}^\ma(\al,\be,J,
f_1,\ldots,f_k)$ using choices $\ti\la_{(p_-,p_+)}$ inducing indices
$\ti\eta_{(p_-,p_+)}\ge 0$, and then replace $\ti\eta_{(p_-,p_+)}$
by $\eta_{(p_-,p_+)}$ in \eq{il12eq3} defined using graded paths
$(\la_{(p_-,p_+)},\ab\psi_{(p_-,p_+)})$, and the results of
\S\ref{il5} such as Theorem \ref{il5thm6} will still be valid.

To summarize our discussion so far: when $(M,J,\om,\Om)$ is
Calabi--Yau and $(\io:L\ra M,\phi)$ is an immersed graded Lagrangian
with only transverse double self-intersections, by using graded
paths $(\la_{(p_-,p_+)},\ab\psi_{(p_-,p_+)})$ in \S\ref{il43} the
indices $\eta_{(p_-,p_+)}$ are uniquely determined by \eq{il12eq3},
for all $\be\in H_2\bigl(M,\io(L);\Z\bigr)$ the Maslov index
$\mu_L(\be)$ is zero, and the orientation results of \S\ref{il5}
still hold.

We can now go through the whole of \S\ref{il6}--\S\ref{il11} working
over the Calabi--Yau Novikov ring $\La_\CY^0$ of \S\ref{il34},
rather than over $\La_\nov^0$. The point is that terms $T^\la e^\mu$
in $\La_\nov^0$ are to keep track of holomorphic discs with area
$\la$ and Maslov index $2\mu$. But for graded Lagrangians all Maslov
indices are zero, so we can work just with terms $T^\la$ in
$\La_\CY^0$. Thus we prove analogues of Theorem \ref{il11thm} and
Corollary~\ref{il11cor}:

\begin{thm} Let\/ $(M,J,\om,\Om)$ be a Calabi--Yau $n$-fold and\/
$(\io:L\ra M,\phi)$ a compact, immersed, graded Lagrangian with only
transverse double self-intersections. Choose a relative spin
structure for $\io:L\ra M$. Then
\begin{itemize}
\setlength{\itemsep}{0pt}
\setlength{\parsep}{0pt}
\item[{\bf(a)}] By an analogue of Definition \ref{il11dfn}, we can
construct a gapped filtered\/ $A_\iy$ algebra
$(\Q\X\ot\La^0_\CY,\ab\m),$ which depends up to canonical homotopy
equivalence only on $(M,\om),$ $\io:L\ra M,$ and its relative spin
structure.

That is, if\/ $(\Q\X\ot\La^0_\CY,\m)$ and\/ $(\Q\ti\X\ot
\La^0_\CY,\ti\m)$ are outcomes depending on $J,\ab{\cal
G},\ab\la_{(p_-,p_+)},\ldots$ and\/ $\ti J,\ti{\cal
G},\ab\ti\la_{(p_-,p_+)},\ldots,$ we can construct a gapped
filtered\/ $A_\iy$ morphism ${\mathfrak
j}:(\Q\X\ot\La^0_\CY,\m)\!\ra\! (\Q\ti\X\ot\La^0_\CY,\ti\m)$ which
is a homotopy equivalence. If\/ ${\mathfrak j},{\mathfrak j}'$ are
possibilities for $\mathfrak j$ there is a homotopy~$\H:{\mathfrak
j}\!\Ra\!{\mathfrak j}'$.
\item[{\bf(b)}] If\/ $(\Q\X\ot\La^0_\CY,\m), (\Q\ti\X\ot
\La^0_\CY,\ti\m),(\Q\check\X\ot \La^0_\CY,\check\m)$ and\/
${\mathfrak j}:(\Q\X\ot\La^0_\CY,\m)\ra
(\Q\ti\X\ot\La^0_\CY,\ti\m),$ ${\mathfrak j}':(\Q\ti\X
\ot\La^0_\CY,\ti\m)\ra(\Q\check\X \ot\ab\La^0_\CY,\ab\check\m),$
${\mathfrak j}'':(\Q\X\ot
\La^0_\CY,\m)\ra(\Q\check\X\ab\ot\La^0_\CY,\ab\check\m)$ are as in
{\rm(a),} there is a homotopy~$\H:{\mathfrak j}''\Ra{\mathfrak
j}'\ci{\mathfrak j}$.
\item[{\bf(c)}] The gapped filtered\/ $A_\iy$ algebra
$(\Q\X\ot\La^0_\CY,\m)$ in {\rm(a)} has a minimal model\/ $({\cal
H}\ot\La^0_\CY,\n),$ with\/ ${\cal H}=\bigop_{d\in\Z}{\cal H}^d$
given by~\eq{il11eq4}.
\end{itemize}
\label{il12thm}
\end{thm}

\section{Bounding cochains and Lagrangian Floer cohomology}
\label{il13}

Finally we apply our results to define bounding cochains and
Lagrangian Floer cohomology for immersed Lagrangians. We do this for
one and two Lagrangians over $\La_\nov^0,\La_\nov$ in
\S\ref{il131}--\S\ref{il132}, and for graded Lagrangians in
Calabi--Yau manifolds over $\La_\CY^0,\La_\CY$ in \S\ref{il133}.
Sections \ref{il134}--\ref{il135} suggest some questions and
conjectures for future research, concerning the invariance of Floer
cohomology under {\it local Hamiltonian equivalence} of immersed
Lagrangians, and on whether there exists a theory of {\it Legendrian
Floer cohomology} for embedded Legendrians in contact manifolds
which are $\U(1)$-bundles over symplectic manifolds, that is
invariant under embedded Legendrian isotopy.

\subsection{Bounding cochains, and the Floer cohomology of one Lagrangian}
\label{il131}

As in \S\ref{il36}, given a gapped filtered $A_\iy$ algebra
$(A\ot\La_\nov^0,\m)$, we can define {\it bounding cochains} $b$ for
$(A\ot\La_\nov^0,\m)$, and form {\it cohomology groups}
$H^*(A\ot\La_\nov^0,\m_1^b)$ and $H^*(A\ot\La_\nov,\m_1^b)$ over
$\La_\nov^0,\La_\nov$. We can apply these ideas either to
$(\Q\X\ot\La_\nov^0,\m)$ in Definition \ref{il11dfn}, or to its
canonical model $({\cal H}\ot\La_\nov^0,\n)$ in Corollary
\ref{il11cor}. The results will be the same in both cases, since up
to equivalence, bounding cochains and cohomology depend only on the
homotopy type of the gapped filtered $A_\iy$ algebra. We choose to
work with $({\cal H}\ot\La_\nov^0,\n)$, as the geometric
interpretation is clearer, and the notion of equivalence of bounding
cochains is better behaved.

\begin{dfn} Let $({\cal H}\ot\La_\nov^0,\n)$ be a gapped
filtered $A_\iy$ algebra in Corollary \ref{il11cor}, constructed
from $(M,\om)$ and $\io:L\ra M$. As in Definition \ref{il3dfn13}, a
{\it bounding cochain} $b$ for $({\cal H}\ot\La_\nov^0,\n)$ is $b\in
F^\la({\cal H}\ot\La_\nov^0)^{(0)}$ for some $\la>0$, satisfying
$\sum_{k\ge 0}\n_k(b,\ldots,b)=0$. Fix some bounding cochain $b$
for~$({\cal H}\ot\La_\nov^0,\n)$.

We shall define Lagrangian Floer cohomology over both Novikov rings
$\La_\nov^0$ and $\La_\nov$. For brevity we will use $\La_\nov^*$ to
mean either $\La_\nov^0$ or $\La_\nov$, the same for each
occurrence. Define graded $\La_\nov^*$-multilinear maps
$\n_k^b:{\buildrel {\ulcorner\,\,\,\text{$k$ copies }
\,\,\,\urcorner} \over {\vphantom{m}\smash{({\cal
H}\ot\La_\nov^*)\t\cdots\t({\cal H}\ot\La_\nov^*)}}}\ab\ra {\cal
H}\ot\La_\nov^*$ for $k=0,1,2,\ldots$, of degree $+1$, by
\begin{equation}
\begin{gathered}
\n_k^b(a_1,\ldots,a_k)=\!\!\sum_{n_0,\ldots,n_k\ge
0}\!\!\!\n_{k+n_0+\cdots+n_k}\begin{aligned}[t]
\bigl({\buildrel{\ulcorner\,\,\text{$n_0$} \,\,\urcorner}
\over{\vphantom{m}\smash{b,\ldots,b}}},a_1, {\buildrel
{\ulcorner\,\,\text{$n_1$}\,\,\urcorner}
\over{\vphantom{m}\smash{b,\ldots,b}}},a_2,{\buildrel
{\ulcorner\,\,\text{$n_2$}\,\,\urcorner}
\over{\vphantom{m}\smash{b,\ldots,b}}}&,\\[-6pt]
\ldots, {\buildrel {\ulcorner\,\,\text{$n_{k-1}$}\,\,\urcorner}
\over{\vphantom{m}\smash{b,\ldots,b}}},a_k,{\buildrel
{\ulcorner\,\,\text{$n_k$} \,\,\urcorner}
\over{\vphantom{m}\smash{b,\ldots,b}}}&\bigr).
\end{aligned}
\end{gathered}
\label{il13eq1}
\end{equation}
Then the $\n_k^b$ satisfy the $A_\iy$ relations \eq{il3eq1}, and
$\n_0^b=0$ as $b$ is a bounding cochain, so for pure
$a_1,a_2,a_3\in{\cal H}\ot\La_\nov^*$ we have
\begin{equation}
\begin{gathered}
(\n_1^b)^2=0,\quad \n_2^b\bigl(\n_1^b(a_1),a_2\bigr)+ (-1)^{\deg
a_1}\n_2^b\bigl(a_1,\n_1^b(a_2)\bigr)+\n_1^b\ci\n_2^b(a_1,a_2)=0,\\
\n_3^b\bigl(\n_1^b(a_1),a_2,a_3\bigr)+ (-1)^{\deg a_1}
\n_3^b\bigl(a_1,\n_1^b(a_2),a_3\bigr)+\\
(-1)^{\deg a_1+\deg a_2}\n_3^b\bigl(a_1,a_2,\n_1^b(a_3)\bigr)
+\n_2^b\bigl(\n_2^b(a_1,a_2),a_3\bigr)+\\
(-1)^{\deg a_1}\n_2^b\bigl(a_1,\n_2^b(a_2,a_3)\bigr)+
\n_1^b\ci\n_3^b(a_1,a_2,a_3)=0.
\end{gathered}
\label{il13eq2}
\end{equation}

The first equation of \eq{il13eq2} implies that $({\cal
H}\ot\La_\nov^*,\n_1^b)$ is a complex. Define the {\it Lagrangian
Floer cohomology groups} $HF^*\bigl((L,b);\La_\nov^0\bigr)$ and
$HF^*\bigl((L,b);\La_\nov\bigr)$ by
\begin{equation}
HF^k\bigl((L,b);\La_\nov^*\bigr)=H^{k-1}\bigl({\cal
H}\ot\La_\nov^*,\n_1^b\bigr).
\label{il13eq3}
\end{equation}
The grading is motivated by \eq{il11eq4} and the isomorphism
$H_k(L;\Q)\cong H^{n-k}(L;\Q)$ as $L$ is oriented of dimension $n$,
and implies that $HF^k\bigl((L,b);\La_\nov^*\bigr)$ is a modified
version of ordinary cohomology $H^k(L;\La_\nov^*)$. Define a
$\La_\nov^*$-bilinear product
$\bu:HF^k\bigl((L,b);\La_\nov^*\bigr)\t
HF^l\bigl((L,b);\La_\nov^*\bigr)\ra
HF^{k+l}\bigl((L,b);\La_\nov^*\bigr)$ by
\begin{equation}
\bigl(a_1+\Im\n_1^b\bigr)\bu\bigl(a_2+\Im\n_1^b\bigr)=
(-1)^{k(l+1)}\n_2^b(a_1,a_2)+\Im\n_1^b.
\label{il13eq4}
\end{equation}

Here since $\n_1^b(a_1)=\n_1^b(a_2)=0$, the second equation of
\eq{il13eq2} implies that $\n_1^b\bigl(\n_2^b(a_1,a_2)\bigr)=0$, so
the right hand side of \eq{il13eq4} does lie in
$HF^{k+l}\bigl((L,b);\La_\nov^*\bigr)$. Using the second equation of
\eq{il13eq2} we see that replacing $a_1\mapsto a_1+\n_1^b(c_1)$
changes $\n_2^b(a_1,a_2)\mapsto\n_2^b(a_1,a_2)-\n_1^b\bigl(
\n_2^b(c_1,a_2)\bigr)$. So the right hand side of \eq{il13eq4} is
independent of the choice of representative $a_1$ for
$a_1+\Im\n_1^b$, and similarly for $a_2$. Thus $\bu$ is
well-defined. Using the third equation of \eq{il13eq2} we can show
that $\bu$ is {\it associative}. It is a modified version of the cup
product on~$H^*(L;\La_\nov^*)$.

One can also construct a {\it unit\/} for
$\bigl(HF^*((L,b);\La_\nov^*),\bu\bigr)$, making it into a
$\La_\nov^*$-algebra. There is a complicated procedure for doing
this in Fukaya et al.\ \cite[\S 8]{FOOO}, involving first finding a
homotopy unit for $(\Q\X\ot\La_\nov^0,\m)$ in Definition
\ref{il11dfn}. We will not explain it, as the immersed case
introduces no new issues.
\label{il13dfn1}
\end{dfn}

\begin{rem} Although $HF^*\bigl((L,b);\La_\nov^*\bigr)$ is graded by
$k\in\Z$, multiplication by $e^d\in\La_\nov^*$ induces an
isomorphism $HF^k\bigl((L,b);\La_\nov^*\bigr)\ra
HF^{k+2d}\bigl((L,b);\La_\nov^*\bigr)$. So there are really only two
groups $HF^0\bigl((L,b);\La_\nov^*\bigr),
HF^1\bigl((L,b);\La_\nov^*\bigr)$, and it would be better to regard
$HF^*\bigl((L,b);\La_\nov^*\bigr)$ as {\it graded over} $\Z_2$,
rather than over~$\Z$.

We could rewrite most of the paper using $\Z_2$-graded spaces rather
than $\Z$-graded spaces, and this would achieve some
simplifications. In \S\ref{il3} we would work with $\Z_2$-graded
vector spaces $A=A^0\op A^1$ rather than $A=\bigop_{d\in\Z}A^d$, and
we would replace $\La_\nov,\La_\nov^0$ by $\La_\CY,\La_\CY^0$
throughout. For computing orientations and degrees, we would regard
$\eta_{(p_-,p_+)},\deg f$ as lying in $\Z_2$ rather than $\Z$. Then
$\eta_{(p_-,p_+)}\in\Z_2$ becomes independent of choice of
$\la_{(p_-,p_+)}$, and the problem in Remark \ref{il11rem}
disappears. We have not done this to keep our paper consistent with
Fukaya et al.~\cite{FOOO}.

In \S\ref{il133} we will see that for graded Lagrangians $(L,\phi)$
in Calabi--Yau manifolds, Floer cohomology
$HF^*\bigl((L,\phi,b);\La_\CY^*\bigr)$ is truly $\Z$-graded rather
than $\Z_2$-graded.
\label{il13rem1}
\end{rem}

Next we explain in which sense Floer cohomology is independent of
choices.

\begin{dfn} Let $({\cal H}\ot\La_\nov^0,\n)$ be as in Corollary
\ref{il11cor}. Write $\hM_{{\cal H},\n}$ for the set of bounding
cochains $b$ for $({\cal H}\ot\La_\nov^0,\n)$. Define $G_{{\cal
H},\n}$ to be the group of gapped filtered $A_\iy$ isomorphisms
${\mathfrak j}:({\cal H}\ot\La_\nov^0,\n)\ra({\cal
H}\ot\La_\nov^0,\n)$ which are homotopic to the identity. We call
$G_{{\cal H},\n}$ the {\it gauge group}. For ${\mathfrak j}\in
G_{{\cal H},\n}$ and $b\in\hM_{{\cal H},\n}$, define ${\mathfrak
j}\cdot b\in ({\cal H}\ot\La_\nov^0)^{(0)}$ by ${\mathfrak j}\cdot
b=\sum_{k\ge 0}{\mathfrak j}_k(b,\ldots,b)$. By summing \eq{il3eq11}
with ${\mathfrak j},\n$ in place of $\f,\m$ and $a_1=\cdots=a_k=b$
over all $k=0,1,\ldots$, we find that
\begin{equation*}
\ts\sum_{k=0}^\iy\n_k({\mathfrak j}\cdot b,\ldots,{\mathfrak j}\cdot
b)=\sum_{l,m=0}^\iy{\mathfrak j}_{l+m+1}\bigl(
{\buildrel{\ulcorner\,\,\text{$l$} \,\,\urcorner}
\over{\vphantom{m}\smash{b,\ldots,b}}},\ts\sum_{k=0}^\iy
\n_k(b,\ldots,b),{\buildrel{\ulcorner\,\,\text{$m$} \,\,\urcorner}
\over{\vphantom{m}\smash{b,\ldots,b}}}\bigr)=0,
\end{equation*}
as $b$ is a bounding cochain. Thus ${\mathfrak j}\cdot b$ is a
bounding cochain, so ${\mathfrak j}\cdot b\in\hM_{{\cal H},\n}$, and
this defines an action of $G_{{\cal H},\n}$ on $\hM_{{\cal H},\n}$.
Define the {\it moduli space of bounding cochains} to be~$\M_{{\cal
H},\n}=\hM_{{\cal H},\n}/G_{{\cal H},\n}$.

For ${\mathfrak j},b$ as above, define linear ${\mathfrak
j}_1^b:{\cal H}\ot\La_\nov^*\ra{\cal H}\ot\La_\nov^*$ by
\begin{equation}
{\mathfrak j}_1^b(a)=\ts\sum_{l,m=0}^\iy{\mathfrak j}_{l+m+1}\bigl(
{\buildrel{\ulcorner\,\,\text{$l$} \,\,\urcorner}
\over{\vphantom{m}\smash{b,\ldots,b}}},a,{\buildrel{\ulcorner\,\,\text{$m$}
\,\,\urcorner} \over{\vphantom{m}\smash{b,\ldots,b}}}\bigr).
\label{il13eq5}
\end{equation}
Now $\mathfrak j$ has an inverse ${\mathfrak j}^{-1}$ in $G_{{\cal
H},\n}$, and calculation shows that $({\mathfrak j}^{-1})_1^b\ci
{\mathfrak j}_1^b=\id$, so ${\mathfrak j}_1^b$ is an isomorphism. By
summing \eq{il3eq11} with ${\mathfrak j},\n$ in place of $\f,\m$,
$k=l+m+1$ and $a_j=a$ for $j=l+1$ and $a_j=b$ otherwise over all
$l,m\ge 0$ we find that ${\mathfrak j}_1^b\ci\n_1^b=\n_1^{{\mathfrak
j}\cdot b}\ci{\mathfrak j}_1^b$. Thus ${\mathfrak j}_1^b:\bigl({\cal
H}\ot\La_\nov^*,\n_1^b\bigr)\ra\bigl({\cal
H}\ot\La_\nov^*,\n_1^{{\mathfrak j}\cdot b}\bigr)$ is an isomorphism
of complexes, and induces an isomorphism $({\mathfrak
j}_1^b)_*:HF^*\bigl((L,b);\La_\nov^*\bigr)\ra
HF^*\bigl((L,{\mathfrak j}\cdot b);\La_\nov^*\bigr)$. As $\mathfrak
j$ is homotopic to the identity, this $({\mathfrak j}_1^b)_*$ is
independent of the choice of $\mathfrak j$ for fixed $b,{\mathfrak
j}\cdot b$. Thus, Floer cohomology
$HF^*\bigl((L,b);\La_\nov^*\bigr)$ depends up to canonical
isomorphism only on $G_{{\cal H},\n}\cdot b \in\M_{{\cal H},\n}$,
rather than on~$b\in\hM_{{\cal H},\n}$.

Now let $({\cal H}\ot\La^0_\nov,\n)$ and $({\cal
H}\ot\La^0_\nov,\ti\n)$ be two possible outcomes in Corollary
\ref{il11cor}. Then the corollary gives a gapped filtered $A_\iy$
isomorphism ${\mathfrak j}:({\cal H}\ot\La^0_\nov,\n)\ra({\cal
H}\ot\La^0_\nov,\ti\n)$, unique up to homotopy. For $b\in\hM_{{\cal
H},\n}$, define ${\mathfrak j}\cdot b$ as above. Then the same proof
shows that ${\mathfrak j}\cdot b$ is a bounding cochain for $({\cal
H}\ot\La^0_\nov,\ti\n)$. This defines a map ${\mathfrak
j}\,\cdot:\hM_{{\cal H},\n}\ra\hM_{{\cal H},\ti\n}$. It is a 1-1
correspondence, with inverse $({\mathfrak j}^{-1})\cdot$, and it
intertwines the actions of $G_{{\cal H},\n},G_{{\cal H},\ti\n}$ on
$\hM_{{\cal H},\n},\hM_{{\cal H},\ti\n}$, and thus induces a 1-1
correspondence~${\mathfrak j}_*:\M_{{\cal H},\n}\ra\M_{{\cal
H},\ti\n}$.

As $\mathfrak j$ is unique up to homotopy, this ${\mathfrak j}_*$ is
independent of the choice of $\mathfrak j$, for fixed $\n,\ti\n$.
Defining ${\mathfrak j}_1^b$ as in \eq{il13eq5}, the same proofs
show ${\mathfrak j}_1^b:\bigl({\cal H}\ot\La_\nov^*,
\n_1^b\bigr)\ra\bigl({\cal H}\ot\La_\nov^*,\ti\n_1^{{\mathfrak
j}\cdot b}\bigr)$ is an isomorphism of complexes, and induces an
isomorphism $({\mathfrak j}_1^b)_*:HF^*\bigl((L,b);\ab
\La_\nov^*\bigr)\ra HF^*\bigl((L,{\mathfrak j}\cdot
b);\La_\nov^*\bigr)$, which is independent of the choice of
$\mathfrak j$ for fixed $\n,\ti\n,b,{\mathfrak j}\cdot b$. We can
also use Theorem \ref{il11thm}(b) to check that, given three choices
$\n,\ti\n,\ti{\ti\n}$, the corresponding isomorphisms $({\mathfrak
j}_1^b)_*$ form commutative triangles.

This implies that the moduli space of bounding cochains $\M_{{\cal
H},\n}$ is {\it independent of choice of\/ $\n$ up to canonical
bijection}, and that under these bijections, Lagrangian Floer
cohomology $HF^*\bigl((L,b);\La_\nov^*\bigr)$, regarded as depending
on $G_{{\cal H},\n}\cdot b\in\M_{{\cal H},\n}$, is also {\it
independent of the choice of\/ $\n$ up to canonical isomorphism}. So
by Corollary \ref{il11cor}, in this sense, the moduli space
$\M_{{\cal H},\n}$ and associated Floer cohomology groups
$HF^*\bigl((L,b);\La_\nov^*\bigr)$ {\it depend only on $(M,\om),$
$\io:L\ra M$ and its relative spin structure, and the indices
$\eta_{(p_-,p_+)},$ and are independent of all other choices}.

In Remark \ref{il11rem} we showed that if
$(\Q\X\ot\La^0_\nov,\m),(\Q\X\ot\La^0_\nov,\ti\m)$ are constructed
in Definition \ref{il11dfn} using different indices
$\eta_{(p_-,p_+)},\ti\eta_{(p_-,p_+)}$, but otherwise exactly the
same choices, then we can construct
$\bs\Xi:(\Q\X\ot\La^0_\nov,\ab\m) \ab\ra(\Q\X\ot\La^0_\nov,\ti\m)$
which is {\it almost\/} a strict gapped filtered $A_\iy$
isomorphism. In the same way, if $({\cal
H}\ot\La_\nov^0,\n),(\ti{\cal H}\ot\La_\nov^0,\ti\n)$ are
constructed in Corollary \ref{il11cor} using different choices of
indices $\eta_{(p_-,p_+)},\ti\eta_{(p_-,p_+)}$, but otherwise
exactly the same choices, then we can construct $\bs\Xi:({\cal
H}\ot\La_\nov^0,\n)\ab\ra(\ti{\cal H}\ot\La_\nov^0,\ti\n)$, which is
almost a strict gapped filtered $A_\iy$ isomorphism, but does not
satisfy all of Definition~\ref{il3dfn11}(i).

Then $\bs\Xi_1:{\cal H}\ot\La_\nov^0\ra\ti{\cal H}\ot\La_\nov^0$
takes bounding cochains to bounding cochains, so $\bs\Xi_1:
\hM_{{\cal H},\n}\ra\hM_{\ti{\cal H},\ti\n}$ is a bijection which
induces a bijection $(\bs\Xi_1)_*:\M_{{\cal H},\n}\ra\M_{\ti{\cal
H},\ti\n}$. If $b\in\hM_{{\cal H},\n}$, so that $\bs\Xi_1(b)\in
\M_{\ti{\cal H},\ti\n}$, then $\bs\Xi_1:\bigl({\cal
H}\ot\La_\nov^*,\n_1^b\bigr)\ra \bigl(\ti{\cal
H}\ot\La_\nov^*,\ti\n_1^{\bs\Xi_1(b)}\bigr)$ is an isomorphism of
complexes, and induces an isomorphism
$(\bs\Xi_1)_*:HF^*\bigl((L,b);\ab \La_\nov^*\bigr)\ra
HF^*\bigl((L,\bs\Xi_1(b));\La_\nov^*\bigr)$. Thus, in the same sense
as above, $\M_{{\cal H},\n}$ and $HF^*\bigl((L,b);\La_\nov^*\bigr)$
are also {\it independent of the choice of
indices}~$\eta_{(p_-,p_+)}$.
\label{il13dfn2}
\end{dfn}

We state our conclusions as:

\begin{thm} In Definitions {\rm\ref{il13dfn1}} and\/
{\rm\ref{il13dfn2},} the moduli space of bounding cochains
$\M_{{\cal H},\n}$ depends up to canonical bijection only on
$(M,\om),$ $\io:L\ra M,$ and its relative spin structure, and the
Floer cohomology groups $HF^*\bigl((L,b);\La_\nov^*\bigr)$ also
depend as a $\La_\nov^*$-algebra up to canonical isomorphism only on
$(M,\om),$ $\io:L\ra M$ and its relative spin structure, and the
canonical bijection equivalence class of the point\/ $G_{{\cal
H},\n}\cdot b\in\M_{{\cal H},\n}$. They are independent in this
sense of all other choices, including the almost complex structure
$J,$ $\cal G,\X,\m,{\cal H},\n,$ and\/
$\la_{(p_-,p_+)},\eta_{(p_-,p_+)},o_{(p_-,p_+)}$ for~$(p_-,p_+)\in
R$.
\label{il13thm1}
\end{thm}

\subsection{The Floer cohomology of two Lagrangians}
\label{il132}

Now let $(M,\om)$ be a compact symplectic manifold and $\io_0:L_0\ra
M$, $\io_1:L_1\ra M$ be compact immersed Lagrangians in $(M,\om)$
with only transverse double self-intersections, which intersect
transversely in finitely many points $\io_0(L_0)\cap\io_1(L_1)$ in
$M$, that are not self-intersection points of $L_0$ or $L_1$. Let
$(\Q\X_0\ot \La_\nov^0,\m^0),(\Q\X_1\ot\La_\nov^0,\m^1)$ be gapped
filtered $A_\iy$ algebras in Definition \ref{il11dfn} for
$\io_0:L_0\ra M$, $\io_1:L_1\ra M$, constructed using almost complex
structures $J_0,J_1$, and let $({\cal H}_0\ot\La_\nov^0,\n^0),
({\cal H}_1\ot\La_\nov^0,\n^1)$ be the corresponding gapped filtered
$A_\iy$ algebras in Corollary \ref{il11cor}. Let $b_0,b_1$ be
bounding cochains for $({\cal H}_0\ot\La_\nov^0,\n^0),({\cal
H}_1\ot\La_\nov^0,\n^1)$ respectively.

Then following Fukaya et al.\ \cite[\S 12]{FOOO}, one can define
Lagrangian Floer cohomology $HF^*\bigl((L_0,b_0),(L_1,b_1);
\La_\nov^*\bigr)$ for the pair of immersed Lagrangians $L_0,L_1$.
Doing this in the immersed rather than the embedded case raises no
new issues that we have not already dealt with above. In fact, as we
explain below, for immersed Lagrangians one can easily recover Floer
cohomology for two Lagrangians $L_0,L_1$ from the Floer cohomology
for one Lagrangian $L_0\amalg L_1$ in \S\ref{il131}. Therefore on
this issue we will simply quote the conclusions of \cite{FOOO} with
brief explanations.

Write $CF(L_0,L_1;\La_\nov^*)$ for the free $\La_\nov^*$-module with
basis $\io_0(L_0)\cap\io_1(L_1)$, where each
$p\in\io_0(L_0)\cap\io_1(L_1)$ is graded in a similar way to the
$\eta_{(p_-,p_+)}$ in \S\ref{il43}. Then by choosing a smooth family
$J_t:t\in[0,1]$ of almost complex structures on $M$ compatible with
$\om$ interpolating between $J_0$ and $J_1$, and considering
\cite[\S 12.4]{FOOO} moduli spaces $\oM_{k_1,k_0}\bigl(
L^1,L^0;[\ell_p,w_1],[\ell_p,w_2]\bigr)$ of stable maps of
holomorphic discs into $M$ with boundary in
$\io_0(L_0)\cup\io_1(L_1)$, which are holomorphic w.r.t.\ the family
$J_t:t\in[0,1]$ in a certain sense, one can give
$CF(L_0,L_1;\La_\nov^*)$ the structure of a {\it gapped filtered\/
$A_\iy$ bimodule} over~$(\Q\X_0\ot\La_\nov^0,\m^0),(\Q\X_1
\ot\La_\nov^0,\m^1)$.

Passing to canonical models, one can also give
$CF(L_0,L_1;\La_\nov^*)$ the structure of a gapped filtered $A_\iy$
bimodule over $({\cal H}_0\ot\La_\nov^0,\n^0),({\cal
H}_1\ot\La_\nov^0,\n^1)$, \cite[Th.~F, \S 1.2]{FOOO}. This bimodule
structure is independent of the choice of bounding cochains. But
once we choose bounding cochains $b_0,b_1$ for $({\cal
H}_0\ot\La_\nov^0,\n^0),({\cal H}_1\ot\La_\nov^0,\n^1)$, we can
define a differential $\de^{b_0,b_1}$ on $CF(L_0,L_1;\La_\nov^*)$,
so that $\bigl(CF(L_0,L_1;\La_\nov^*),\de^{b_0,b_1}\bigr)$ is a
complex. We then define
$HF^*\bigl((L_0,b_0),(L_1,b_1);\La_\nov^*\bigr)$ to be the
cohomology of $\bigl(CF(L_0,L_1;\La_\nov^*), \de^{b_0,b_1}\bigr)$,
graded in the same way as~\eq{il13eq3}.

In this way we obtain an analogue of Theorem~\ref{il13thm1}:

\begin{thm} In the situation above, $HF^*\bigl((L_0,b_0),
(L_1,b_1);\La_\nov^*\bigr)$ depends as a $\La_\nov^*$-module up to
canonical isomorphism only on $(M,\om),$ $\io_0:L_0\ra M,$
$\io_1:L_1\ra M$ and their relative spin structures, and the
canonical bijection equivalence classes of the points $G_{{\cal
H}_0,\n^0}\cdot b_0\in\M_{{\cal H}_0,\n^0}$ and\/~$G_{{\cal
H}_1,\n^1}\cdot b_1\in\M_{{\cal H}_1,\n^1}$.
\label{il13thm2}
\end{thm}

Actually, if we take $J^0,J^1$ and $J^t$ for $t\in[0,1]$ to be some
fixed almost complex structure $J$, the definition of Floer
cohomology $HF^*\bigl((L_0,b_0),(L_1,b_1);\La_\nov^*\bigr)$ for two
Lagrangians is implicit in our definition of Floer cohomology
$HF^*\bigl((L,b);\La_\nov^*\bigr)$ for one immersed Lagrangian in
\S\ref{il131}. Take $L=L_0\amalg L_1$ with immersion
$\io=\io_0\amalg\io_1:L\ra M$. Then bounding cochains $b_0,b_1$ for
$L_0,L_1$ give a bounding cochain $b$ for $L$, and there is a
canonical isomorphism
\begin{equation}
\begin{gathered}
HF^*\bigl((L,b);\La_\nov^*\bigr)\cong
HF^*\bigl((L_0,b_0);\La_\nov^*\bigr)\op
HF^*\bigl((L_1,b_1);\La_\nov^*\bigr) \op\\
HF^*\bigl((L_0,b_0),(L_1,b_1);\La_\nov^*\bigr)\op
HF^*\bigl((L_1,b_1),(L_0,b_0);\La_\nov^*\bigr).
\end{gathered}
\label{il13eq6}
\end{equation}

Thus, Floer cohomology for two Lagrangians $L_0,L_1$ is just a
sector of Floer cohomology for one Lagrangian $L_0\amalg L_1$, and
one can deduce Theorem \ref{il13thm2} from Theorem \ref{il13thm1}
with little effort. This works only for immersed Lagrangians, since
even if $L_0,L_1$ are embedded, $L_0\amalg L_1$ is immersed unless
$\io_0(L_0)\cap\io_1(L_1)=\emptyset$.

Although it is not covered in \cite{FOOO}, it follows from the
framework of Fukaya \cite{Fuka} that if $L_0,L_1,L_2$ are immersed
Lagrangians in $(M,\om)$ with only transverse double
self-intersections, which intersect pairwise transversely as above,
with no triple self-intersections, and $b_0,b_1,b_2$ are bounding
cochains for $L_0,L_1,L_2$, then we can define a
$\La_\nov^*$-bilinear product
\begin{equation}
\begin{split}
\bu_{012}:HF^*\bigl((L_0,b_0),(L_1,b_1);\La_\nov^*\bigr)\t
HF^*\bigl((L_1,b_1),(L_2,b_2);\La_\nov^*\bigr)&\\
\longra HF^*\bigl((L_0,b_0),(L_2,b_2);&\La_\nov^*\bigr).
\end{split}
\label{il13eq7}
\end{equation}
This is basically composition of morphisms between objects
$(L_0,b_0),(L_1,b_1)$ and $(L_2,b_2)$ of the derived Fukaya category
of~$(M,\om)$.

As in \eq{il13eq6}, $HF^*\bigl((L_i,b_i),(L_j,b_j);\La_\nov^*\bigr)$
for $i,j=0,1,2$ are all sectors of the one-Lagrangian Floer
cohomology $HF^*\bigl((L,b);\La_\nov^*\bigr)$ for $L=L_0\amalg
L_1\amalg L_2$, and then $\bu_{012}$ in \eq{il13eq7} is just the
product $\bu$ on $HF^*\bigl((L,b); \La_\nov^*\bigr)$ in Definition
\ref{il13dfn1} restricted to these sectors. For four such
Lagrangians $L_0,\ldots,L_3$, associativity of $\bu$ for
$L=L_0\amalg\cdots\amalg L_3$ gives the associativity property
\begin{equation*}
\bu_{023}\ci\bigl(\bu_{012}\t\id_{HF^*(L_2,L_3)}\bigr)=
\bu_{013}\ci\bigl(\id_{HF^*(L_0,L_1)}\t\bu_{123}\bigr).
\end{equation*}

When we work over $\La_\nov$ rather than $\La_\nov^0$, Lagrangian
Floer cohomology has very important invariance properties under
Hamiltonian isotopy, most of which is proved by Fukaya et al.\
\cite[Th.~G, \S 1.2]{FOOO} in the embedded case:

\begin{thm} Let\/ $(M,\om)$ be a compact symplectic manifold, and\/
$\psi_t:t\in[0,1]$ be a smooth\/ $1$-parameter family of Hamiltonian
equivalent symplectomorphisms of\/ $(M,\om),$ with\/ $\psi_0=\id_M$.
Then:
\begin{itemize}
\setlength{\itemsep}{0pt}
\setlength{\parsep}{0pt}
\item[{\rm(a)}] Let\/ $\io_0:L_0\ra M$ be a compact immersed
Lagrangian in $(M,\om)$ and\/ $\io_1:L_1\ra M$ be the image of\/
$\io_0:L_0\ra M$ under $\psi_1,$ that is, $L_1=L_0$ and\/
$\io_1=\psi_1\ci\io_0$. Let\/ $({\cal H}_0\ot\La_\nov^0,\n^0),
({\cal H}_1\ot\La_\nov^0,\n^1)$ be gapped filtered\/ $A_\iy$
algebras in Corollary {\rm\ref{il11cor}} for $L_0,L_1$. Then using
$\psi_t:t\in[0,1]$ we can define a gapped filtered\/ $A_\iy$
isomorphism $\bs\Psi:({\cal H}_0\ot\La_\nov^0,\n^0)\ra({\cal
H}_1\ot\La_\nov^0,\n^1),$ unique up to homotopy. This induces a
unique bijection $\bs\Psi_*:\M_{{\cal H}_0,\n^0}\ra\M_{{\cal
H}_1,\n^1}$.
\item[{\rm(b)}] In {\rm(a),} if\/ $L_0,L_1$ intersect transversely in
$M,$ then whenever $b_0\in\hM_{{\cal H}_0,\n^0}$ and\/
$b_1\in\hM_{{\cal H}_1,\n^1}$ with\/ $\bs\Psi_*(G_{{\cal
H}_0,\n^0}\cdot b_0)=G_{{\cal H}_1,\n^1}\cdot b_1,$ there is a
canonical isomorphism
\begin{equation}
HF^*\bigl((L_0,b_0);\La_\nov\bigr)\cong
HF^*\bigl((L_0,b_0),(L_1,b_1);\La_\nov\bigr).
\label{il13eq8}
\end{equation}
\item[{\rm(c)}] In {\rm(a),} if\/ $\io_2:L_2\ra M$ is another compact
immersed Lagrangian in $(M,\om)$ which intersects $L_0,L_1$
transversely, with\/ $({\cal H}_2\ot\La_\nov^0,\n^2)$ in Corollary
{\rm\ref{il11cor},} and\/ $b_0\in\hM_{{\cal H}_0,\n^0},
b_1\in\hM_{{\cal H}_1,\n^1}$ and\/ $b_2\in\hM_{{\cal H}_2,\n^2}$
with\/ $\bs\Psi_*(G_{{\cal H}_0,\n^0}\cdot b_0)=G_{{\cal
H}_1,\n^1}\cdot b_1,$ there is a canonical isomorphism
\begin{equation}
HF^*\bigl((L_0,b_0),(L_2,b_2);\La_\nov\bigr)\cong
HF^*\bigl((L_1,b_1),(L_2,b_2);\La_\nov\bigr).
\label{il13eq9}
\end{equation}
\end{itemize}
\label{il13thm3}
\end{thm}

Here part (a) is immediate from \S\ref{il131}, since $\psi_1$ is an
isomorphism from $M,\om,\io_0:L_0\ra M$ to $M,\om,\io_1:L_1\ra M$.
The nontrivial statements are (b),(c).

\begin{rem}{\bf(i)} Equations \eq{il13eq8} and \eq{il13eq9} do {\it
not\/} hold in general for Floer cohomology over $\La_\nov^0$. In
particular, from $HF^*\bigl((L_0,b_0),(L_1,b_1); \La_\nov^0\bigr)$
we can recover the $\Q$-vector space with basis
$\io_0(L_0)\cap\io_1(L_1)$. Thus, if \eq{il13eq9} held over
$\La_\nov^0$ it would force $\bmd{\io_0(L_0)\cap\io_2(L_2)}=
\bmd{\io_1(L_1)\cap\io_2(L_2)}$, which is false in general.
\smallskip

\noindent{\bf(ii)} In the embedded case, it is well known that
Theorem \ref{il13thm3} has important consequences in symplectic
geometry. Using (b) one can deduce the {\it Arnold Conjecture} for
compact monotone symplectic manifolds.
\smallskip

\noindent{\bf(iii)} The only place where we use compactness of $M$
is to ensure that moduli spaces of $J$-holomorphic curves
$\M_{k+1}(\al,\be,J)$ are compact. If $M$ is {\it noncompact\/} but
$J$ has suitable convexity properties at infinity which ensure
compactness of $\M_{k+1}(\al,\be,J)$, then Lagrangian Floer
cohomology is well-defined and Theorem \ref{il13thm3} holds. This
can be done for cotangent bundles $T^*L$ and $\C^n$, for instance.

By taking $M=T^*L$ for $L$ a compact $n$-manifold, and $L_0$ to be
the zero section of $T^*L$, part (b) implies another conjecture of
Arnold on cotangent bundles.

Taking $M=\C^n$, if $\io_0:L_0\ra\C^n$ is a compact immersed
Lagrangian, then by choosing $\psi_1$ to be a large translation in
$\C^n$ we can arrange that $\io_0(L_0)\cap\io_1(L_1)=\emptyset$.
Thus $CF(L_0,L_1;\La_\nov)=\{0\}$, so $HF^*\bigl((L_0,b_0),
(L_1,b_1);\La_\nov\bigr)=\{0\}$, and (b) implies that
$HF^*\bigl((L_0,b_0);\La_\nov\bigr)=\{0\}$ for any bounding cochain
$b_0$ for $L_0$.
\label{il13rem2}
\end{rem}

\subsection{Floer cohomology for graded Lagrangians in Calabi--Yau
$n$-folds}
\label{il133}

As in \S\ref{il12}, suppose $(M,J,\om,\Om)$ is a Calabi--Yau
$n$-fold and $(\io:L\ra M,\phi)$ an immersed graded Lagrangian with
only transverse double self-intersections. Choose a relative spin
structure for $\io:L\ra M$. Theorem \ref{il12thm} constructs gapped
filtered $A_\iy$ algebras $(\Q\X\ot\La^0_\CY,\m)$ and $({\cal
H}\ot\La^0_\CY,\n)$. We can then go through the whole of
\S\ref{il131} and \S\ref{il132} using graded Lagrangians, and
working over the Calabi--Yau Novikov rings $\La_\CY^0,\La_\CY$
rather than $\La_\nov^0,\La_\nov$.

Use $\La_\CY^*$ to mean $\La_\CY^0$ or $\La_\CY$. Write triples
$(L,\phi,b)$ as a shorthand for an immersed graded Lagrangian
$(\io:L\ra M,\phi)$ together with a bounding cochain $b$ for $({\cal
H}\ot\La^0_\CY,\n)$ in Theorem \ref{il12thm}(c). Then we may define
{\it Lagrangian Floer cohomology groups} $HF^*\bigl((L,\phi,b);
\La_\CY^*\bigr)$ for one graded Lagrangian as in \S\ref{il131}, and
$HF^*\bigl((L_0,\phi_0,b_0),(L_1,\phi_1,b_1);\La_\CY^*\bigr)$ for
two graded Lagrangians as in~\S\ref{il132}.

\begin{thm} The analogues of Theorems {\rm\ref{il13thm1},
\ref{il13thm2}} and\/ {\rm\ref{il13thm3}} hold for Lagrangian Floer
cohomology of immersed graded Lagrangians in Calabi--Yau $n$-folds,
over the Novikov rings~$\La_\CY^0,\La_\CY$.
\label{il13thm4}
\end{thm}

We explained in Remark \ref{il13rem1} that $HF^k\bigl((L,b);
\La_\nov^*\bigr)\cong HF^{k+2d}\bigl((L,b);\La_\nov^*\bigr)$ for
$d\in\Z$, so one should regard $HF^*\bigl((L,b);\La_\nov^*\bigr)$ as
graded over $\Z_2$ rather than $\Z$. In contrast, $HF^*\bigl(
(L,\phi,b);\La_\CY^*\bigr)$ really is graded over $\Z$, and this
makes Floer cohomology for graded Lagrangians a more powerful tool,
as Seidel \cite{Seid} points out.

In particular, we can give useful criteria for existence and
uniqueness of bounding cochains. Since $\La_\CY^0$ is graded of
degree 0, a bounding cochain $b$ for $({\cal H}\ot\La_\CY^0,\n)$
lies in $b\in F^\la({\cal H}^0\ot\La_\CY^0)$ for some $\la>0$ and
must satisfy $\sum_{k\ge 0}\n_k(b,\ldots,b)=0$ in ${\cal
H}^1\ot\La_\CY^0$. But \eq{il11eq4} gives
\begin{equation}
\begin{split}
{\cal H}^0&=H_{n-1}(L;\Q)\op\ts\bigop_{(p_-,p_+)\in R:\;
\eta_{(p_-,p_+)}=1}\Q(p_-,p_+), \\
{\cal H}^1&=H_{n-2}(L;\Q)\op\ts\bigop_{(p_-,p_+)\in R:\;
\eta_{(p_-,p_+)}=2}\Q(p_-,p_+).
\end{split}
\label{il13eq10}
\end{equation}
Thus we deduce:

\begin{prop} Suppose $(M,J,\om,\Om)$ is a Calabi--Yau $n$-fold,
$(\io:L\ra M,\phi)$ is an immersed graded Lagrangian with only
transverse double self-intersections, and\/ $({\cal
H}\ot\La^0_\CY,\n)$ is as in Theorem {\rm\ref{il12thm}(c)}. Then
\begin{itemize}
\setlength{\itemsep}{0pt}
\setlength{\parsep}{0pt}
\item[{\bf(a)}] If\/ $b_{n-2}(L)=0$ and\/ $\eta_{(p_-,p_+)}\ne 2$ for
all\/ $(p_-,p_+)\in R,$ then every $b\in F^\la({\cal
H}^0\ot\La_\CY^0)$ for $\la>0$ is a bounding cochain; and
\item[{\bf(b)}] If\/ $b_{n-1}(L)=0$ and\/ $\eta_{(p_-,p_+)}\ne 1$ for
all\/ $(p_-,p_+)\in R,$ then $0$ is the only possible bounding
cochain.
\end{itemize}
\label{il13prop1}
\end{prop}

Since $\io:L\ra M$ has a relative spin structure, $L$ is oriented,
so $b_{n-1}(L)=0$ in (b) is equivalent to $b^1(L)=0$, which is a
sufficient condition for an immersed Lagrangian $\io:L\ra M$ to
admit a grading $\phi$. As in Remark \ref{il13rem2}(iii), we can
also apply the theory to {\it noncompact\/} Calabi--Yau manifolds
$(M,J,\om,\Om)$, provided $J$ is convex at infinity. For example,
$M=\C^n$ with the Euclidean $J,\om,\Om$ will do.

In the noncompact case we may suppose $(M,\om)$ is an exact
symplectic manifold, that is, $\om=\d\xi$ for some 1-form $\xi$ on
$M$. If $\io:L\ra M$ is an immersed Lagrangian then $\io^*(\xi)$ is
a closed 1-form on $L$, and we call $L$ {\it exact\/} if
$\io^*(\xi)$ is exact. If $L$ is exact, then there can be no
nonconstant holomorphic discs in $M$ whose boundaries lie in
$\io(L)$ and lift continuously to $L$, as Stokes' Theorem shows that
their area would be zero. This implies that the component of $\n_0$
in $H_{n-2}(L;\Q)\ot\La_\CY^0$ is zero. If also $\eta_{(p_-,p_+)}\ne
2$ for all $(p_-,p_+)\in R$ then $\n_0=0$, so 0 is a bounding
cochain, giving:

\begin{prop} Suppose $(M,J,\om,\Om)$ is a noncompact, exact
Calabi--Yau $n$-fold, with\/ $J$ convex at infinity, $(\io:L\ra
M,\phi)$ is an exact immersed graded Lagrangian with only transverse
double self-intersections, and\/ $\eta_{(p_-,p_+)}\ne 2$ for all\/
$(p_-,p_+)\in R$. Then $0$ is a bounding cochain for $({\cal
H}\ot\La^0_\CY,\n)$ in Theorem~{\rm\ref{il12thm}(c)}.
\label{il13prop2}
\end{prop}

Now let $(\io:L\ra M,\phi)$ be a compact immersed graded Lagrangian
in $\C^n$. Propositions \ref{il13prop1}(a) and \ref{il13prop2} give
two sufficient conditions for 0 to be a bounding cochain for $L$.
Then $HF^*\bigl((L,\phi,0);\La_\CY\bigr)$ is well-defined, and
Remark \ref{il13rem2}(iii) shows that $HF^*\bigl((L,\phi,0);\La_\CY
\bigr)=\{0\}$. But $HF^*\bigl((L,\phi,0);\La_\CY \bigr)$ is the
cohomology of the complex $({\cal H}\ot\La_\CY,\n_1)$. To have zero
cohomology imposes constraints upon the ranks over $\La_\CY$ of the
graded pieces of a free $\La_\CY$-complex. For instance, we have:

\begin{cor} Let\/ $(\io:L\ra M,\phi)$ be a compact, immersed, graded
Lagrangian in $\C^n,$ with transverse double self-intersections.
Suppose that $\eta_{(p_-,p_+)}\ne 2$ for all\/ $(p_-,p_+)\in R,$ and
either $b_{n-2}(L)=0$ or $L$ is exact. Then $\dim{\cal H}^d\le
\dim{\cal H}^{d-1}+\dim{\cal H}^{d+1}$ for all\/ $d\in\Z,$ with\/
${\cal H}^d$ given in~\eq{il11eq4}.
\label{il13cor1}
\end{cor}

Here is an example.

\begin{ex} Define a curve in $\C$ by
$C=\{s+it:s,t\in\R$, $t^2=s^2-s^4\}$. This is sketched in Figure
\ref{il13fig}. It is an immersed circle in $\R^2$, the shape of an
$\iy$ sign, with one self-intersection point at $0$. For $n\ge 1$,
define
\begin{equation*}
L_n=\bigl\{(\la x_1,\ldots,\la x_n\bigr):\la\in C,\;
x_1,\ldots,x_n\in\R,\; x_1^2\!+\!\cdots\!+\!x_n^2\!=\!1\bigr\}.
\end{equation*}
It is easy to see that $L_n$ is the image of an immersed Lagrangian
sphere $\io:{\cal S^n}\ra\C^n$, which has one transverse
self-intersection point at $0\in\C^n$ with $\io(p_-)=\io(p_+)=0$,
where $p_\pm=(\pm 1,0,\ldots,0)\in{\cal S}^n$. Note that $L_n$ is
SO$(n)$-invariant, and we can choose $\io$ to be equivariant with
respect to the actions of SO$(n)$ on ${\cal S}^n$ fixing $p_\pm$,
and on $\C^n$. The tangent spaces to $\io({\cal S}^n)$ at the
self-intersection point are
\begin{equation}
\begin{split}
\d\io(T_{p_-}{\cal S}^n)&=\bigl\{({\rm e}^{-i\pi/4}x_1,\ldots,{\rm
e}^{-i\pi/4}x_n):x_1,\ldots,x_n\in\R\bigr\},\\
\d\io(T_{p_+}{\cal S}^n)&=\bigl\{({\rm e}^{i\pi/4}x_1,\ldots,{\rm
e}^{i\pi/4}x_n):x_1,\ldots,x_n\in\R\bigr\}.
\end{split}
\label{il13eq11}
\end{equation}

\begin{figure}[htb]
\centerline{\begin{xy}
0;<1.7mm,0mm>: ,(0,0)*+{} ,(20,-1)*+!U{} **\crv{(10,10)&(20,10)}
,(0,0)*+{} ,(20,0)*+{} **\crv{(10,-10)&(20,-10)} ,(0,0)*+{}
,(-20,-1)*+!U{} **\crv{(-10,10)&(-20,10)} ,(0,0)*+{} ,(-20,0)*+{}
**\crv{(-10,-10)&(-20,-10)} ,(0,0)*+{} ,(-25,0)*+{} ,(-25,0)*+{ };
(25,0)*+{} **@{-} ?>* \dir{>} ,(0,-10)*+{ }; (0,10)*+{} **@{-} ?>*
\dir{>} ,(24,0.2)*+!D{s} ,(0.2,9)*+!L{t}
\end{xy}}
\caption{The curve $C$ in $\C$}
\label{il13fig}
\end{figure}

We shall calculate the index $\eta_{(p_-,p_+)}$ using Proposition
\ref{il12prop}. Despite the comparison between \eq{il12eq2} and
\eq{il13eq11}, we are not free to put $\phi^j_-=-\frac{\pi}{4}$ and
$\phi^j_+=\frac{\pi}{4}$, since \eq{il13eq11} only determines the
$\phi^j_\pm$ up to addition of $\pi\Z$. We have to choose a framing
$\phi:{\cal S}^n\ra\R$ for $\io:{\cal S}^n\ra\C^n$, and choose the
$\phi^j_\pm$ to satisfy $\phi^1_\pm+\cdots+\phi^n_\pm=\phi(p_\pm)$.

Consider the path $p:[-\frac{\pi}{2},\frac{\pi}{2}]\ra{\cal S}^n$
defined by $p(u)=(\sin u,\cos u,0,\ldots,0)$. Then
$p(\pm\frac{\pi}{2})=p_\pm$, and $\io\ci p(u)=(\la(u),0,\ldots,0)$,
where $\la:[-\frac{\pi}{2},\frac{\pi}{2}]\ra C$ sweeps out the right
hand lobe $s\ge 0$ of $C$ in the anticlockwise direction.
Calculation shows that for $u\in(-\frac{\pi}{2},\frac{\pi}{2})$ we
have
\begin{equation*}
\d\io(T_{p(u)}{\cal S}^n)=\bigl\{(\ts\frac{\d\la}{\d u}(u)
x_1,\la(u)x_2,\la(u)x_n):x_1,\ldots,x_n\in\R\bigr\}.
\end{equation*}
From Figure \ref{il13fig} we see that $\arg\frac{\d\la}{\d u}(u)$
increases continuously from $-\frac{\pi}{4}$ to $\frac{5\pi}{4}$ and
$\arg\la(u)$ increases continuously from $-\frac{\pi}{4}$ to
$\frac{\pi}{4}$ over~$(-\frac{\pi}{2},\frac{\pi}{2})$.

Therefore $\io:{\cal S}^n\ra\C^n$ has a framing $\phi:{\cal
S}^n\ra\R$ with $\phi(p_-)=-\frac{n\pi}{4}$ and
$\phi(p_-)=\frac{n\pi}{4}+\pi$, and in Proposition \ref{il12prop} we
may take $\phi^j_-=-\frac{\pi}{4}$ for $j=1,\ldots,n$,
$\phi^1_+=\frac{5\pi}{4}$, and $\phi^j_+=\frac{\pi}{4}$ for
$j=2,\ldots,n$. Hence $\bigl[\frac{\phi_+^j-\phi_-^j}{\pi}\bigr]$ is
1 for $j=1$ and 0 for $j=2,\ldots,n$, and \eq{il12eq3} gives
$\eta_{(p_-,p_+)}=n+1$, and similarly $\eta_{(p_+,p_-)}=-1$. Thus
\eq{il11eq4} gives ${\cal H}^d=\Q$ if $d=-2,-1,n-1,n$, and ${\cal
H}^d=0$ otherwise.

When $n>2$, Proposition \ref{il13prop1} implies that 0 is the unique
bounding cochain for $\io:{\cal S}^n\ra\C^n$. When $n=2$ Proposition
\ref{il13prop1}(a) does not apply, but this is an exact Lagrangian,
so Propositions \ref{il13prop1}(b) and \ref{il13prop2} show that 0
is the unique bounding cochain for $\io:{\cal S}^2\ra\C^2$. Thus as
above $HF^*\bigl(({\cal S}^n,\phi,0);\La_\CY\bigr)$ is well-defined,
and zero. Corollary \ref{il13cor1} holds.
\label{il13ex1}
\end{ex}

If $(M,J,\om,\Om)$ is a compact Calabi--Yau $n$-fold and $p\in M$,
then by shrinking the example above by a homothety and locally
identifying $\C^n$ near 0 with $M$ near $p$ using Darboux' Theorem,
we can construct Lagrangian immersions $\io:{\cal S}^n\ra M$. The
same arguments then prove:

\begin{prop} Let\/ $(M,J,\om,\Om)$ be a compact Calabi--Yau
$n$-fold for $n>1,$ and\/ $p\in M$. Then there exists an immersed,
graded Lagrangian $(\io:{\cal S}^n\ra M,\phi)$ with exactly one
transverse double self-intersection point at\/
$p\!=\!\io(p_-)\!=\!\io(p_+),$ with\/
$\eta_{(p_-,p_+)}\!=\!n\!+\!1$. It has unique bounding cochain $0,$
and\/ $HF^*\bigl(({\cal
S}^n,\phi,0);\La_\CY\bigr)\!=\!\{0\}$.\!\!\!\!
\label{il13prop3}
\end{prop}

Thus there exist many immersed Lagrangians which have unobstructed
Floer cohomology, but which are zero objects in the derived immersed
Fukaya category.

\subsection{Local Hamiltonian equivalence of immersed Lagrangians}
\label{il134}

For immersed Lagrangians, there are two different notions of
Hamiltonian equivalence.

\begin{dfn} Let $(M,\om)$ be a symplectic manifold, and
$\io:L\ra M$, $\io':L'\ra M$ be compact, immersed Lagrangians in
$M$. Then
\begin{itemize}
\setlength{\itemsep}{0pt}
\setlength{\parsep}{0pt}
\item[(i)] We say that $\io:L\ra M$, $\io':L'\ra M$ are {\it globally
Hamiltonian equivalent\/} if there exists a diffeomorphism $h:L\ra
L'$ and a smooth $1$-parameter family $\psi_t:t\in[0,1]$ of
Hamiltonian equivalent symplectomorphisms of $(M,\om)$ with
$\psi_0=\id_M$, such that $\psi_1\ci\io\equiv \io'\ci h$.
\item[(ii)] We say that $\io:L\ra M$, $\io':L'\ra M$ are {\it locally
Hamiltonian equivalent\/} if there exists a diffeomorphism $h:L\ra
L'$ and a smooth $1$-parameter family $\io_t:t\in[0,1]$ of
Lagrangian immersions $\io_t:L\ra M$, such that $\io_0=\io$ and
$\io_1=\io'\ci h$, and for each $t\in[0,1]$ the 1-form
$\d\io_t^*\bigl(\frac{\d\io_t}{\d t}\cdot \io_t^*(\om)\bigr)$ on $L$
is {\it exact}.

Here $\frac{\d\io_t}{\d t},\io_t^*(\om)$ and $\frac{\d\io_t}{\d
t}\cdot\io_t^*(\om)$ are sections of the vector bundles
$\io_t^*(TM),\ab\io_t^*(\La^2T^*M),\ab\io_t^*(T^*M)$ over $L$,
respectively, $\d\io_t:TL\ra\io_t^*(TM)$ is the derivative of
$\io_t$, and $\d\io_t^*:\io_t^*(T^*M)\ra T^*L$ the dual map. It
follows from $\io_t$ a Lagrangian immersion for $t\in[0,1]$ that
$\d\io_t^*\bigl(\frac{\d\io_t}{\d t}\cdot \io_t^*(\om)\bigr)$ is a
closed 1-form.
\end{itemize}
\label{il13dfn3}
\end{dfn}

By setting $\io_t=\psi_t\ci\io$, we see that global implies local
Hamiltonian equivalence. For embedded Lagrangians, if the
$\io_t:L\ra M$ are embeddings for all $t\in[0,1]$ then we can find a
family $\psi_t:t\in[0,1]$ as in (i) such that $\io_t=\psi_t\ci\io$,
so that local implies global Hamiltonian equivalence. Thus, {\it for
embedded Lagrangians, global and local Hamiltonian equivalence is
the same}. But for immersed Lagrangians, local Hamiltonian
equivalence can slide sheets of $L$ over each other, change the
number of self-intersection points, and so on, but global
Hamiltonian equivalence cannot. Hence, {\it for immersed
Lagrangians, local Hamiltonian equivalence is weaker than global
Hamiltonian equivalence}.

Theorem \ref{il13thm3} shows that Floer cohomology over $\La_\nov$
has strong invariance properties under global Hamiltonian
equivalence. So it makes sense to ask:

\begin{quest} Does Floer cohomology $HF^*\bigl((L_0,b_0);
\La_\nov\bigr),HF^*\bigl((L_0,b_0),\ab(L_1,\ab b_1);\ab
\La_\nov\bigr)$ have any useful invariance properties under
(possibly restricted classes of) {\it local\/} Hamiltonian
equivalences of $\io_0:L_0\ra M$ and~$\io_1:L_1\ra M$?
\label{il13qn1}
\end{quest}

For arbitrary local Hamiltonian equivalences, the answer to this
must be no. The {\it Lagrangian $h$-principle}, due to Gromov
\cite[p.~60-61]{Grom} and Lees \cite{Lees}, states that two
Lagrangian immersions $\io_0:L\ra M$, $\io_1:L\ra M$ are homotopic
through (possibly {\it exact\/}) Lagrangian immersions $\io_t:L\ra
M$ for $t\in[0,1]$ if and only if $\io_0,\io_1$ are homotopic in a
weaker sense, that is, $(\io_0,\d\io_0),(\io_1,\d\io_1)$ should be
homotopic through pairs $(\io,\ti\io)$, where $\io:L\ra M$ is smooth
and $\ti\io:TL\ra TM$ is a bundle map covering $\io$ which embeds
$TL$ as a bundle of Lagrangian subspaces in~$TM$.

Thus, the Lagrangian $h$-principle implies that two immersed
Lagrangians are locally Hamiltonian equivalent (at least when either
$b^1(L)=0$, so that \cite[Th.~1]{Lees} applies, or $M=\C^n$, so that
\cite[p.~60-61]{Grom} applies, and probably more generally) if and
only if they are homotopic in a weak sense which can be well
understood using homotopy theory. But Floer cohomology detects
`quantum' information not visible to classical algebraic topology
--- this is its whole point. So arbitrary local Hamiltonian
equivalence is too coarse an equivalence relation to preserve Floer
cohomology.

However, it could still be true that Floer cohomology over
$\La_\nov$ is in some sense invariant under some special class of
local Hamiltonian equivalences more general than global Hamiltonian
equivalences. For example, in Theorem \ref{il13thm3}(c),
$\io_0\amalg\io_2:L_0\amalg L_2\ra M$ and
$\io_1\amalg\io_2:L_1\amalg L_2\ra M$ are immersed Lagrangians which
are locally Hamiltonian equivalent but generally not globally so ---
for instance, if $\bmd{\io_0(L_0)\cap\io_2(L_2)}\ne
\bmd{\io_1(L_1)\cap\io_2(L_2)}$ then $L_0\amalg L_2$ and $L_1\amalg
L_2$ have different numbers of self-intersection points, and cannot
be globally Hamiltonian equivalent. But \eq{il13eq6} and Theorem
\ref{il13thm3}(c) imply that there is a canonical isomorphism
\begin{equation*}
HF^*\bigl((L_0\amalg L_2,b_0\amalg b_2);\La_\nov\bigr)\cong
HF^*\bigl((L_1\amalg L_2,b_1\amalg b_2);\La_\nov\bigr).
\end{equation*}

Another possibility: in the Calabi--Yau, graded Lagrangian case,
Proposition \ref{il13prop1} suggests that only self-intersections
with $\eta_{(p_-,p_+)}=1$ or 2 are relevant to existence of bounding
cochains. So we could consider only local Hamiltonian equivalences
through immersions $\io_t:L\ra M$ which have no self-intersections
with $\eta_{(p_-,p_+)}=1$ or 2, and perhaps these will preserve
Floer cohomology over~$\La_\CY$.

We shall now describe a mechanism for how the moduli spaces of
bounding cochains $\M_{{\cal H},\n}$ can change under local
Hamiltonian equivalence.

\begin{ex} Suppose $(M,\om)$ is a compact symplectic $2n$-manifold,
$L$ is a compact $n$-manifold, and $\io_t:L\ra M$ for $t\in[0,1]$ is
a smooth family of Lagrangian immersions, which have only transverse
double self-intersections for all $t\in[0,1]$. This implies that the
number of self-intersections of $\io_t:L\ra M$ is independent of
$t$. Therefore we can choose a smooth family of diffeomorphisms
$\de_t:M\ra M$ with $\de_0=\id_M$, such that $\io_t=\de_t\ci\io_0$.
So $\de_t^{-1}$ identifies $(M,\om),\io_t:L\ra M$ with
$(M,\de_t^*(\om)),\io_0:L\ra M$. That is, we can work with a fixed
immersion $\io_0:L\ra M$, but a 1-parameter family of symplectic
forms $\de_t^*(\om)$ on $M$ for~$t\in[0,1]$.

Let $t>0$ be small. Then $\om$ and $\de_t^*(\om)$ are $C^0$ close as
2-forms on $M$. Dimension calculations show that we can choose an
almost complex structure $J_0$ on $M$ compatible with both $\om$ and
$\de_t^*(\om)$. Write $J_t=(\de_t)_*(J_0)$, so that $J_t$ is
compatible with $\om$ as $J_0$ is compatible with $\de_t^*(\om)$.
Then $\de_t$ identifies $M,\io_0:L\ra M,J_0$ with $M,\io_t:L\ra
M,J_t$. Thus, $\de_t$ takes $J_0$-holomorphic curves in $M$ with
boundary in $\io_0(L)$ to $J_t$-holomorphic curves in $M$ with
boundary in $\io_t(L)$. However, $\de_t$ need not preserve the areas
of the curves computed using~$\om$.

Let $(\Q\X_0\ot\La^0_\nov,\m^0)$, $({\cal H}_0\ot\La^0_\nov, \n^0)$
be the gapped filtered $A_\iy$ algebras in Theorem \ref{il11thm} and
Corollary \ref{il11cor}, associated to $(M,\om)$ and $\io_0:L\ra M$
with almost complex structure $J_0$. Let
$(\Q\X_t\ot\La^0_\nov,\m^t),({\cal H}_t\ot\La^0_\nov,\n^t)$ be the
corresponding gapped filtered $A_\iy$ algebras associated to
$(M,\om)$ and $\io_t:L\ra M$ with almost complex structure $J_t$,
where {\it the choices made to construct\/ $\X^t,\m^t,{\cal
H}_t,\n^t$ are the images under $\de_t$ of the choices made to
construct\/} $\X^0,\m^0,{\cal H}_0,\n^0$. That is, we have
$\X_t=\{\de_t\ci f:f\in\X_0\}$, and then $\de_t$ induces
isomorphisms of Kuranishi spaces
\begin{equation}
\oM_{k+1}(\al,\be,J_0,f_1,\ldots,f_k)\cong
\oM_{k+1}(\al,(\de_t)_*(\be),J_t,\de_t\ci f_1,\ldots,\de_t\ci f_k),
\label{il13eq12}
\end{equation}
and we choose all orientations and perturbation data compatible with
these.

The difference between $(\Q\X_0\ot\La^0_\nov,\m^0),({\cal
H}_0\ot\La^0_\nov,\n^0)$ and $(\Q\X_t\ot\La^0_\nov, \m^t)$, $({\cal
H}_t\ot\La^0_\nov,\n^t)$ is that $\de_t$ changes the areas of $J_0$-
and $J_t$-holomorphic curves, and this changes the coefficients
$\la$ in the multilinear maps $\m_k^{\smash{\la,\mu}},
\n_k^{\smash{\la,\mu}}$ which make up $\m^0,\ab\n^0,\ab\m^t,\ab
\n^t$. The changes in areas of curves can be expressed like this:
there exist constants $c_{(p_-,p_+)}\in\R$ for all $(p_-,p_+)\in R$,
with $c_{(p_-,p_+)}+c_{(p_+,p_-)}=0$, such that if
$\oM_{k+1}(\al,\be,J_0,f_1,\ldots,f_k)\ne\emptyset$ then
\begin{equation}
(\de_t)_*(\be)\cdot[\om]_{M,\io_t(L)}=\be\cdot[\om]_{M,\io_0(L)}+
\ts\sum_{i\in I}c_{\al(i)},
\label{il13eq13}
\end{equation}
where $[\om]_{M,\io_0(L)},[\om]_{M,\io_t(L)}$ are the classes of
$\om$ in~$H^2(M,\io_0(L);\R),H^2(M,\io_t(L);\R)$.

By \eq{il11eq4} we have ${\cal H}_0={\cal H}_t=H_*(L;\Q)\op
\bigop_{(p_-,p_+)\in R}\Q(p_-,p_+)$. Using similar ideas to Remark
\ref{il11rem}, define a $\La_\nov$-linear map $\hat\Xi_t:{\cal
H}_0\ot\La_\nov\ra{\cal H}_t\ot\La_\nov$ to be the identity on
$H_*(L;\Q)$ and to satisfy $\hat\Xi_t(p_-,p_+)=
T^{-c_{(p_-,p_+)}}(p_-,p_+)$, where $T$ is the formal variable in
$\La_\nov$ from \S\ref{il34}. Then using \eq{il13eq12} and
\eq{il13eq13} we see that $\m_k^t\bigl(\hat\Xi_t(h_1),\ldots,
\hat\Xi_t(h_k)\bigr)=\hat\Xi_t\ci\m_k^0(h_1,\ldots,h_k)$ for
all~$h_1,\ldots,h_k\in{\cal H}_0\ot\La_\nov$.

Thus, as in Remark \ref{il11rem}, it is {\it nearly} true that
setting $\bs\Xi_1=\hat\Xi_t$ and $\bs\Xi_k=0$ for $k\ne 1$ defines a
strict gapped filtered $A_\iy$ isomorphism $\bs\Xi:({\cal
H}_0\ot\La^0_\nov,\n^0)\ra({\cal H}_t\ot\La^0_\nov,\n^t)$. The
problem is that if $c_{(p_-,p_+)}>0$ for some $(p_-,p_+)\in R$ then
$(p_-,p_+)\in{\cal H}_0\ot\La^0_\nov$ but
$\hat\Xi_t(p_-,p_+)=T^{-c_{(p_-,p_+)}}(p_-,p_+)\notin{\cal
H}_t\ot\La^0_\nov$, so $\hat\Xi_t$ does not map ${\cal
H}_0\ot\La_\nov^0\ra{\cal H}_t\ot\La_\nov^0\subset{\cal
H}_t\ot\La_\nov$.

However, if $b\in{\cal H}_0\ot\La_\nov^0$ is a bounding cochain for
$({\cal H}_0\ot\La^0_\nov,\n^0)$, and $\hat\Xi_t(b)$ lies in
$F^\la({\cal H}_t\ot\La_\nov)$ for some $\la>0$, then $\hat\Xi_t(b)$
is a bounding cochain for $({\cal H}_t\ot\La^0_\nov,\n^t)$. Also
$\hat\Xi_t$ is an isomorphism of complexes $({\cal
H}_0\ot\La_\nov,\n_1^{0,b})\ra({\cal H}_t\ot\La_\nov,\n_1^{t,b})$,
and so induces an isomorphism of Floer cohomology over $\La_\nov$
(though not over $\La_\nov^0$):
\begin{equation*}
(\hat\Xi_t)_*:HF^*(\io_0:L\ra M,b;\La_\nov)\longra HF^*(\io_t:L\ra
M,\hat\Xi_t(b);\La_\nov).
\end{equation*}

We have discovered a kind of {\it wall-crossing phenomenon}. When
$t\in[0,\ep)$ for some $\ep>0$ we can map bounding cochains $b$ for
$\io_0:L\ra M$ to bounding cochains $\hat\Xi_t(b)$ for $\io_t:L\ra
M$, and this map induces canonical isomorphisms on Lagrangian Floer
cohomology. We have $\hat\Xi_t(b)\in F^{\la(t)}({\cal
H}_t\ot\La_\nov)$, where we take $\la(t)$ as large as possible. For
$\hat\Xi_t(b)$ to be a bounding cochain we need $\la(t)>0$. However,
it may happen that at $t=\ep$ we have $\la(\ep)=0$, and for $t>\ep$
we have $\la(t)>0$. Then at $t=\ep$ we cross a `wall' where the
bounding cochain for $\io_0:L\ra M$ no longer corresponds to any
bounding cochain for $\io_t:L\ra M$ when~$t\ge\ep$.
\label{il13ex2}
\end{ex}

This example suggests the following conjectural picture:

\begin{conj} Suppose that\/ $(M,\om)$ is a compact symplectic
manifold, and that\/ $\io_t:L\ra M$ for $t\in[0,1]$ is a smooth\/
$1$-parameter family of compact Lagrangian immersions satisfying the
exactness condition of Definition {\rm\ref{il13dfn3}(ii)}. Let\/
$S\subseteq[0,1]$ be the open subset of\/ $t\in[0,1]$ for which\/
$\io_t:L\ra M$ has only transverse double self-intersections.
Suppose for simplicity that\/ $L$ is oriented and spin; this induces
relative spin structures for $\io_t:L\ra M$ for all\/ $t\in[0,1],$
as in {\rm\S\ref{il51}}. Then for all\/ $t\in S,$ we have the moduli
space of bounding cochains $\M_{{\cal H}_t,\n^t}$ for $\io_t:L\ra
M,$ which is independent of choices up to canonical bijection by
Theorem~{\rm\ref{il13thm1}}.

We conjecture that for all $s,t\in S$ there should exist open
subsets $O_{s,t}\subseteq\M_{{\cal H}_s,\n^s}$ and homeomorphisms
$\Phi_{s,t}:O_{s,t}\ra O_{t,s}$ with\/ $\Phi_{t,s}=\Phi_{s,t}^{-1},$
and whenever $G_{{\cal H}_s,\n^s}\cdot b_s\in O_{s,t},$ $G_{{\cal
H}_t,\n^t}\cdot b_t\in O_{t,s}$ with\/ $\Phi_{s,t}(G_{{\cal
H}_s,\n^s}\cdot b_s)=G_{{\cal H}_t,\n^t}\cdot b_t,$ there should
exist canonical isomorphisms
\begin{gather*}
HF^*(\io_s:L\ra M,b_s;\La_\nov)\cong HF^*(\io_t:L\ra
M,b_t;\La_\nov),\\
HF^*\bigl((\io_s:L\ra M,b_s),(L',b');\La_\nov\bigr)\cong
HF^*\bigl((\io_t:L\ra M,b_t),(L',b');\La_\nov\bigr),
\end{gather*}
for any compact immersed Lagrangian $\io':L'\ra M$ with transverse
double self-intersections intersecting $\io_s(L),\io_t(L)$
transversely, and bounding cochain~$b'$.

Furthermore, for any $G_{{\cal H}_s,\n^s}\cdot b_s\in\M_{{\cal
H}_s,\n^s}$ the set\/ $T_s=\{t\in S:G_{{\cal H}_s,\n^s}\cdot b_s\in
O_{s,t}\}$ is an open subset of\/ $S$ containing $s,$ and at the
boundary of\/ $T_s$ in $S,$ a wall-crossing phenomenon like that in
Example {\rm\ref{il13ex2}} occurs.
\label{il13conj1}
\end{conj}

\subsection{Immersed Lagrangians and embedded Legendrians}
\label{il135}

We now develop the ideas of \S\ref{il134} further in the context of
{\it contact geometry} and {\it Legendrian submanifolds}. Let
$(M,\om)$ be a compact symplectic $2n$-manifold, and suppose
$[\om]\in H^2(M;\R)$ lies in the image of $H^2(M;\Z)\ra H^2(M;\R)$.
Then there exists a principal $\U(1)$-bundle $P\ra M$ with first
Chern class $c_1(P)=2\pi[\om]$, and a connection $A$ on $P$ with
curvature $2\pi\om$. Write the $\U(1)$ action on $P$ as $({\rm
e}^{\sqrt{-1}\,\th},p) \mapsto {\rm e}^{\sqrt{-1}\,\th}\cdot p$, and
let $v\in C^\iy(TP)$ be the vector field of the $\U(1)$-action, so
that ${\rm e}^{\sqrt{-1}\,\th}$ acts as $\exp(\th v):P\ra P$. Write
$\pi:P\ra M$ for the natural projection whose fibres are
$\U(1)$-orbits $\U(1)\cdot p$ for $p\in P$. Let $\ga$ be the 1-form
of the connection on $P$, so that $\ga\in C^\iy(T^*P)$ is
$\U(1)$-invariant with $v\cdot\ga\equiv 1$
and~$\d\ga\equiv\pi^*(2\pi\om)$.

Then $P$ has the structure of a {\it contact\/ $(2n+1)$-manifold},
with contact 1-form $\ga$ and Reeb vector field $v$. An immersed
$n$-manifold $\ti\io:L\ra P$ is called a {\it Legendrian
submanifold\/} if $\ti\io^*(\ga)\equiv 0$. If $\ti\io:L\ra P$ is
Legendrian then $\pi\ci\ti\io:L\ra P$ is a Lagrangian immersion.
Conversely, if $\io:L\ra M$ is a Lagrangian immersion, then
$\io^*(P)\ra L$ is a $\U(1)$-bundle with a flat $\U(1)$-connection,
and there exists a Legendrian immersion $\ti\io:L\ra P$ with
$\io=\pi\ci\ti\io$ if and only if this flat $\U(1)$-connection has a
constant section, that is, if it is trivial. Since flat
$\U(1)$-connections are classified by morphisms $H^1(L;\Z)\ra\U(1)$,
a sufficient condition for an immersed Lagrangian $\io:L\ra M$ to
lift to an immersed Legendrian $\ti\io:L\ra P$ is
that~$H^1(L;\Z)=\{0\}$.

If $\ti\io:L\ra P$ is an embedding we identify $L$ with
$\ti\io(L)\subset P$ and regard $L$ as a subset of $P$, with
$\ga\vert_L\equiv 0$. Generic Legendrians in $P$ are embedded. If
$L\subset P$ is an embedded Legendrian then $\pi=\pi\vert_L:L\ra M$
is an immersed Lagrangian, which in general is {\it not\/} embedded.

We call two Legendrian immersions $\ti\io:L\ra P$, $\ti\io':L'\ra P$
{\it immersed Legendrian isotopic\/} if there exists a
diffeomorphism $h:L\ra L'$ and a smooth 1-parameter family
$\ti\io_t:t\in[0,1]$ of Legendrian immersions $\ti\io_t:L\ra P$,
such that $\ti\io_0=\ti\io$ and $\ti\io_1=\ti\io'\ci h$. If
$\ti\io,\ti\io'$ are embeddings, we call $\ti\io:L\ra P$,
$\ti\io':L'\ra P$ {\it embedded Legendrian isotopic\/} if there
exist $\ti\io_t:t\in[0,1]$ as above with each $\ti\io_t:L\ra P$ an
embedding. Clearly, embedded Legendrian isotopic implies immersed
Legendrian isotopic.

If $\ti\io:L\ra P$, $\ti\io':L'\ra P$ are Legendrian immersions and
$h:L\ra L'$, $\ti\io_t:t\in[0,1]$ is an immersed Legendrian isotopy
between them, then $\pi\ci\ti\io:L\ra M$, $\pi\ci\ti\io':L'\ra M$
are Lagrangian immersions, and $h:L\ra L'$, $\pi\ci\ti\io_t:
t\in[0,1]$ is a local Hamiltonian equivalence between them, in the
sense of Definition \ref{il13dfn3}(ii). Conversely, if $\io:L\ra M$,
$\io':L'\ra M$ are locally Hamiltonian equivalent Lagrangian
immersions, then there exists a Legendrian lift $\ti\io:L\ra P$ with
$\io\equiv\pi\ci\ti\io$ if and only if there exists a Legendrian
lift $\ti\io':L'\ra P$ with $\io'\equiv\pi\ci\ti\io'$, and then
$h,\io_t:t\in[0,1]$ in Definition \ref{il13dfn3}(ii) lift to an
immersed Legendrian isotopy $h,\ti\io_t:t\in[0,1]$ between
$\ti\io:L\ra P$ and $\ti\io':L'\ra P$. So {\it local Hamiltonian
equivalence in $M$ corresponds exactly to immersed Legendrian
isotopy in}~$P$.

Now embedded Legendrian isotopies are a special class of immersed
Legendrian isotopies, and so project to a special class of local
Hamiltonian equivalences. Question \ref{il13qn1} asked whether Floer
cohomology is invariant under any special classes of local
Hamiltonian equivalences. So it makes sense to ask:

\begin{quest} In the situation above, let $L_0,L_1\subset P$ be
compact embedded Legendrians. Suppose that the Lagrangian immersions
$\pi:L_0\ra M$, $\pi:L_1\ra M$ have only transverse double
self-intersections. Is Floer cohomology $HF^*\bigl((\pi:L_0\ra
M,b_0);\La_\nov\bigr),HF^*\bigl((\pi:L_0\ra M,b_0),\ab(\pi:L_1\ra
M,b_1);\ab \La_\nov\bigr)$ preserved under embedded Legendrian
isotopies of~$L_0,L_1$?
\label{il13qn2}
\end{quest}

The authors expect the problem to be better behaved if we work over
a smaller Novikov ring $\La_\nov^{\Z}$. Suppose $L\subset P$ is a
compact embedded Legendrian, and $\pi:L\ra M$ has only transverse
double points. Define $R$ as in \S\ref{il41}. If $(p_-,p_+)\in R$
then $p_-,p_+\in L$ with $p_-\ne p_+$ and $\pi(p_-)=\pi(p_+)$ in
$M$. Thus $p_-,p_+$ are distinct points in the same $\U(1)$-orbit,
and $p_+={\rm e}^{\sqrt{-1}\,\th}\cdot p_-$ for some unique
$\th\in(0,2\pi)$. Define $a_{(p_-,p_+)}=\frac{\th}{2\pi}$. Then
$a_{(p_-,p_+)}\in(0,1)$, and~$a_{(p_-,p_+)}+a_{(p_+,p_-)}=1$.

The areas of $J$-holomorphic curves in $M$ with boundaries in
$\pi(L)$ have an integrality property involving the $a_{(p_-,p_+)}$
for $(p_-,p_+)\in R$. We can express it like this: if
$\oM_{k+1}(\al,\be,J)\ne\emptyset$ and $[\om]_{M,\pi(L)}$ is the
class of $\om$ in $H^2(M,\pi(L);\R)$ then
\begin{equation}
\be\cdot[\om]_{M,\pi(L)}-\ts\sum_{i\in I}a_{\al(i)}\in\Z.
\label{il13eq14}
\end{equation}
To prove \eq{il13eq14}, suppose $[\Si,\vec
z,u,l,\bar{u}]\in\oM_{k+1}(\al,\be,J)$, and for simplicity take
$\Si\cong D^2$ nonsingular. Then $\bar u:{\cal S}^1\sm\{\ze_i:i\in
I\}\ra L$ is smooth, with $(\lim_{\th\uparrow
0}\bar{u}(e^{\sqrt{-1}\,\th}\ze_i), \ab\lim_{\th\downarrow
0}\bar{u}(e^{\sqrt{-1}\,\th}\ze_i))=\al(i)$ in $R$, for all~$i\in
I$.

Modify this $\bar u$ to a piecewise smooth map $\ti u:{\cal S}^1\ra
P$ by inserting at each $\ze_i$ for $i\in I$, the line segment
$[0,2\pi a_{(p_-,p_+)}]\ra P$ mapping $\th\mapsto{\rm
e}^{\sqrt{-1}\,\th} \cdot p_-$, where $\al(i)=(p_-,p_+)$. Then
$\int_{{\cal S}^1}\ti u^*(\ga)=2\pi\sum_{i\in I}a_{\al(i)}$, since
$\ga\vert_L\equiv 0$ and $v\cdot\ga\equiv 1$. Now consider the
$\U(1)$-bundle $u^*(P)\ra\Si$. It has a connection $u^*(\ga)$ with
curvature $2\pi u^*(\om)$, and we have in effect constructed a
section $\ti u$ of $u^*(P)\vert_{\pd\Si}$ with $\int_{\pd\Si}\ti
u^*(\ga)=2\pi\sum_{i\in I}a_{\al(i)}$. But $\int_\Si 2\pi\,
u^*(\om)=\int_{\pd\Si}\ti u^*(\ga)+2\pi c_1\bigl(u^*(P);\ti
u\bigr)$, where $c_1\bigl(u^*(P);\ti u\bigr)\in\Z\cong
H^2(\Si,\pd\Si;\Z)$ is the first Chern class of the $\U(1)$-bundle
$u^*(P)\ra\Si$ relative to the trivialization of
$u^*(P)\vert_{\pd\Si}$ induced by $\ti u$. Putting all this together
gives~\eq{il13eq14}.

By analogy with \eq{il3eq6}--\eq{il3eq9}, define {\it Novikov rings}
\begin{align*}
\La_\nov^\Z&=\bigl\{\ts\sum_{i=0}^\iy a_iT^{\la_i}e^{\mu_i}:
\text{$a_i\in\Q$, $\la_i\in\Z$, $\mu_i\in\Z$, $\lim_{i\ra\iy}\la_i
=\iy$}\bigr\}, \\
\La^\N_\nov&=\bigl\{\ts\sum_{i=0}^\iy
a_iT^{\la_i}e^{\mu_i}: \text{$a_i\in\Q$, $\la_i\in\N$, $\mu_i\in\Z$,
$\lim_{i\ra\iy}\la_i =\iy$}\bigr\},\\
\La^\Z_\CY&=\bigl\{\ts\sum_{i=0}^\iy a_iT^{\la_i}: \text{$a_i\in\Q$,
$\la_i\in\Z$, $\lim_{i\ra\iy}\la_i =\iy$}\bigr\},\\
\La^\N_\CY&=\bigl\{\ts\sum_{i=0}^\iy a_iT^{\la_i}: \text{$a_i\in\Q$,
$\la_i\in\N$, $\lim_{i\ra\iy}\la_i =\iy$}\bigr\},
\end{align*}
where $\N=\{0,1,2,\ldots\}\subset\Z$. Then in the situation of
\S\ref{il11}, having constructed $\X$ define $\widetilde{\Q\X}$ to
be the $\Q$-vector space with basis $f$ for $f\in\X$ with
$f:\De_a\ra L$, and $T^{a_{(p_-,p_+)}}f$ for $f\in\X$ with
$f:\De_a\ra\{(p_-,p_+)\}\subset R$. Similarly, modifying
\eq{il11eq4}, define a $\Q$-vector space $\ti{\cal
H}=\bigop_{d\in\Z}\ti{\cal H}^d$ by
\begin{equation*}
\ti{\cal H}^d=H_{n-d-1}(L;\Q)\op\bigop\nolimits_{\begin{subarray}{l}
(p_-,p_+)\in R:\\ d=\eta_{(p_-,p_+)}-1\end{subarray}}\Q\cdot
T^{a_{(p_-,p_+)}}(p_-,p_+).
\end{equation*}

We can then go through \S\ref{il7}--\S\ref{il13} using
$\La_\nov^\Z,\La_\nov^\N,\La_\CY^\Z,\La_\CY^\N$ in place of
$\La_\nov,\ab\La_\nov^0,\ab\La_\CY,\ab\La_\CY^0$, and
$\widetilde{\Q\X}\ot\La_\nov^\N,\ti{\cal H}\ot\La_\nov^\N, \ti{\cal
H}\ot\La_\nov^\Z,\widetilde{\Q\X}\ot\La_\CY^\N, \ti{\cal
H}\ot\La_\CY^\N,\ti{\cal H}\ot\La_\CY^\Z$ in place of
$\Q\X\ot\La_\nov^0,{\cal H}\ot\La_\nov^0,{\cal
H}\ot\La_\nov,\Q\X\ot\La_\CY^0,{\cal H}\ot\La_\CY^0, {\cal
H}\ot\La_\CY$, respectively. The integrality condition \eq{il13eq14}
and the definitions of $\widetilde{\Q\X},\ti{\cal H}$ ensure we can
choose $\m_k$ to map $\bigl(\widetilde{\Q\X}\ot
\La_\nov^\N\bigr){}^k\ra\widetilde{\Q\X}\ot\La_\nov^\N$, and
similarly for $\n_k$. That is, only powers $T^l$ or
$T^{l+a_{(p_-,p_+)}}$ for $l\in\N$ and $(p_-,p_+)\in R$ occur in
$\widetilde{\Q\X}\ot\La_\nov^\N$, and in the terms $T^\la
e^\mu\m_k^{\smash{\la,\mu}}$ in $\m_k$, the only allowed values for
$\la\in\R$ are those which take possible total powers of $T$ in
$\bigl(\widetilde{\Q\X}\ot\La_\nov^\N\bigr){}^k$ to possible powers
of $T$ in~$\widetilde{\Q\X}\ot\La_\nov^\N$.

Thus, in \S\ref{il11} we construct gapped filtered $A_\iy$ algebras
$(\widetilde{\Q\X}\ot\La_\nov^\N,\m)$ and $(\ti{\cal H}\ot
\La_\nov^\N,\n)$ over $\La_\nov^\N$, and in the graded case of
\S\ref{il12} we construct $(\widetilde{\Q\X}\ot\La_\CY^\N,\m)$ and
$(\ti{\cal H}\ot\La_\CY^\N,\n)$ over $\La_\CY^\N$. Then as in
\S\ref{il131}--\S\ref{il133}, we define Lagrangian Floer cohomology
$HF^*\bigl((L,b);\La_\nov^\N\bigr),HF^*\bigl((L,b);\La_\nov^\Z\bigr)$,
over $\La_\nov^\N$ or $\La_\nov^\Z$, and similarly for two
Lagrangians, and for graded Lagrangians over
$\La_\CY^\N,\La_\CY^\Z$. Several of the definitions of gapped
filtered $A_\iy$ algebras, morphisms, etc.\ require minor
modification to allow for inclusion of factors $T^{a_{(p_-,p_+)}}$
in~$\widetilde{\Q\X},\ti{\cal H}$.

We can now make our most important point. Consider the wall-crossing
phenomenon described in Example \ref{il13ex2}. This occurs when, for
a family of immersed Lagrangians $\io_t:L\ra M$ for $t\in[0,1]$, we
have a family of bounding cochains $b_t\in F^{\la(t)}({\cal
H}_t\ot\La_\nov)$, where $\la(t)>0$ is necessary for $b_t$ to be a
bounding cochain. If $\la(\ep)=0$ then at $t=\ep$ we cross a `wall'
where $b_t$ ceases to be a bounding cochain.

Now if $\io_t=\pi\ci\ti\io_t$ for a smooth family of Legendrian
embeddings $\ti\io_t:L\ra M$, then the only allowed powers of $T$ in
bounding cochains $b(t)$ are $T^l$ for $l=1,2,\ldots$ and
$T^{l+a_{(p_-,p_+)}(t)}$ for $l=0,1,\ldots$, where
$a_{(p_-,p_+)}(t)\in(0,1)$. Thus, the leading power of $T$ in $b_t$
could only deform continuously to zero at $t=\ep$ if
$a_{(p_-,p_+)}(t)\ra 0$ as $t\ra\ep$. But $a_{(p_-,p_+)}(\ep)=0$
implies that $\ti\io_\ep(p_-)=\ti\io_\ep(p_+)$, that is,
$\ti\io_\ep:L\ra P$ is an immersion, but not an embedding.

This shows that {\it the wall-crossing phenomenon in Example
\ref{il13ex2} cannot happen for bounding cochains for $(\ti{\cal
H}\ot \La_\nov^\N,\n)$ under embedded Legendrian isotopy}. If, as
Conjecture \ref{il13conj1} claims, this is the only mechanism by
which Floer cohomology changes under local Hamiltonian equivalence,
then Floer cohomology over $\La_\nov^\Z$ should be unchanged under
embedded Lagrangian isotopy. So we conjecture:

\begin{conj} In the situation above, suppose that\/ $\ti\io_t:L\ra P$
for $t\in[0,1]$ is a smooth\/ $1$-parameter family of Legendrian
embeddings with\/ $L$ compact, oriented, and spin, and that\/
$\pi\ci\ti\io_0:L\ra M$ and\/ $\pi\ci\ti\io_1:L\ra M$ have only
transverse double self-intersections. Then there should exist a
canonical bijection $\Psi:\M_{{\cal H}_0,\n^0}\ra\M_{{\cal
H}_1,\n^1}$ between the moduli spaces of bounding cochains for
$\pi\ci\ti\io_0:L\ra M$ and\/ $\pi\ci\ti\io_1:L\ra M$. Let\/
$b_0\in\hM_{{\cal H}_0,\n^0}$ and\/ $b_1\in\hM_{{\cal H}_1,\n^1}$
with\/ $\Psi(G_{{\cal H}_0,\n^0}\cdot b_0)=G_{{\cal H}_1,\n^1}\cdot
b_1,$ and suppose $L_2$ is a compact embedded Legendrian in $P,$
such that\/ $\pi:L_2\ra M$ has only transverse double
self-intersections, and\/ $b_2$ is a bounding cochain for
$\pi:L_2\ra M$. Then there are canonical isomorphisms
\begin{gather*}
HF^*\bigl((\pi\ci\ti\io_0:L\ra M,b_0);\La_\nov^\Z\bigr)\cong
HF^*\bigl((\pi\ci\ti\io_1:L\ra M,b_1);\La_\nov^\Z\bigr),\\
HF^*\bigl((\pi\ci\ti\io_0:L\ra M,b_0),(\pi:L_2\ra
M,b_2);\La_\nov^\Z\bigr)\cong \\
HF^*\bigl((\pi\ci\ti\io_1:L\ra M,b_1),(\pi:L_2\ra
M,b_2);\La_\nov^\Z\bigr).
\end{gather*}
\label{il13conj2}
\end{conj}

This conjecture suggests there should exist a theory of {\it
Legendrian Floer cohomology} for embedded Legendrians in contact
manifolds $P$ which are $\U(1)$-bundles over symplectic manifolds
$(M,\om)$. This should clearly be related to the theory of {\it
Legendrian contact homology}, which was described informally by
Eliashberg, Givental and Hofer \cite[\S 2.8]{EGH}, and by Chekanov
\cite{Chek} for Legendrian knots in $\R^3$, and has been developed
rigorously by Ekholm, Etnyre and Sullivan \cite{EES1,EES2}, for
embedded Legendrians $L$ in $\R^{2n+1}$ and in $M\t\R$ for $(M,\om)$
an exact symplectic manifold.

In particular, for $(M,\om)$ exact one can compare our
$HF_*(L,b;\La_\nov^\Z)$ for embedded Legendrians in $M\t{\cal S}^1$,
and Ekholm et al.'s $HC_*(L,J)$ for embedded Legendrians $L$ in
$M\t\R$, \cite{EES2}. It seems that $HC_*(L,J)$ should be a sector
of $HF_*(L,b;\La_\nov^\Z)$, but not the whole thing, since
$HC_*(L;J)$ is the homology of a complex involving $H_1(L;\Z)$ and
the set of double points of $\pi(L)$ in $M$, but
$HF_*(L,b;\La_\nov^\Z)$ is the cohomology of a complex involving all
of $H_*(L;\Q)$ and $R$, which has two points $(p_-,p_+),(p_+,p_-)$
for each double point $p$ of $\pi(L)$ in $M$. We hope our conjecture
will lead to progress in Legendrian contact homology.

\medskip

\noindent{\small\sc The Mathematical Institute, 24-29 St. Giles,
Oxford, OX1 3LB, U.K.}

\noindent{\small\sc E-mail: \tt joyce@maths.ox.ac.uk}
\medskip

\noindent{\small\sc Department of Mathematics, Tokyo Metropolitan
University, Tokyo, Japan.}

\noindent{\small\sc E-mail: \tt akaho@comp.metro-u.ac.jp }

\end{document}